\documentclass[reqno,12pt,oneside]{report}

\usepackage{float}
\usepackage{xypic}
\usepackage{color}
\usepackage{combelow}
\usepackage{rac}       
\usepackage{aas_macros} 
\usepackage[intlimits]{amsmath} 
\usepackage{amsxtra}  
\usepackage{amsthm}
\usepackage{amssymb}
\usepackage{amsfonts}
\usepackage{graphicx}   
\usepackage{rotating}
\usepackage{color}
\usepackage{ucs}
\usepackage[utf8x]{inputenc}
\usepackage{epsfig}
\usepackage{subfigure}  
\usepackage{verbatim}
\usepackage{natbib}   
\usepackage[printonlyused]{acronym} 
\usepackage{setspace}   
\theoremstyle{plain}
\newtheorem{theorem}{Theorem}
\newtheorem{proposition}[theorem]{Proposition}

\newtheorem{lemma}[theorem]{Lemma}

\newtheorem{conjecture}[theorem]{Conjecture}

\theoremstyle{definition}
\newtheorem{definition}[theorem]{Definition}

\newtheorem{example}[theorem]{Example}

\newtheorem{exercise}[theorem]{Exercise}

\theoremstyle{remark}
\newtheorem{remark}[theorem]{Remark}

\numberwithin{theorem}{chapter}   

\makeatletter
\def\cleardoublepage{\clearpage\if@twoside \ifodd\c@page\else
\hbox{}
\thispagestyle{empty}
\newpage
\if@twocolumn\hbox{}\newpage\fi\fi\fi}
\makeatother

\newcommand{\nc}{\newcommand}
\nc{\tab}{{\text{}\\}}
\nc{\homdeg}{{\text{hom deg }}}

\nc{\fa}{{\mathfrak{a}}}
\nc{\fb}{{\mathfrak{b}}}
\nc{\fg}{{\mathfrak{g}}}
\nc{\fh}{{\mathfrak{h}}}
\nc{\fj}{{\mathfrak{j}}}
\nc{\fn}{{\mathfrak{n}}}
\nc{\fu}{{\mathfrak{u}}}
\nc{\fp}{{\mathfrak{p}}}
\nc{\fr}{{\mathfrak{r}}}
\nc{\ft}{{\mathfrak{t}}}
\nc{\fsl}{{\mathfrak{sl}}}
\nc{\fgl}{{\mathfrak{gl}}}

\nc{\fA}{{\mathfrak{A}}}
\nc{\fB}{{\mathfrak{B}}}
\nc{\fC}{{\mathfrak{C}}}
\nc{\fG}{{\mathfrak{G}}}
\nc{\fK}{{\mathfrak{K}}}
\nc{\fP}{{\mathfrak{P}}}
\nc{\fR}{{\mathfrak{R}}}
\nc{\fV}{{\mathfrak{V}}}
\nc{\fZ}{{\mathfrak{Z}}}

\nc{\pol}{{\text{Poles}}}
\nc{\asyt}{{\text{ASYT}}}

\nc{\BA}{{\mathbb{A}}}
\nc{\BC}{{\mathbb{C}}}
\nc{\BM}{{\mathbb{M}}}
\nc{\BN}{{\mathbb{N}}}
\nc{\BQ}{{\mathbb{Q}}}
\nc{\BF}{{\mathbb{F}}}
\nc{\tBF}{{\widetilde{\mathbb{F}}}}
\nc{\Bk}{{\mathbb{k}}}
\nc{\BK}{{\mathbb{K}}}
\nc{\BP}{{\mathbb{P}}}
\nc{\BR}{{\mathbb{R}}}
\nc{\BZ}{{\mathbb{Z}}}

\nc{\CA}{{\mathcal{A}}}
\nc{\CB}{{\mathcal{B}}}
\nc{\CC}{{\mathcal{C}}}
\nc{\CE}{{\mathcal{E}}}
\nc{\CF}{{\mathcal{F}}}
\nc{\CG}{{\mathcal{G}}}
\nc{\CI}{{\mathcal{I}}}
\nc{\CL}{{\mathcal{L}}}
\nc{\CM}{{\mathcal{M}}}
\nc{\CMa}{\mathcal{M}}
\nc{\CH}{{\mathcal{H}}}
\nc{\CN}{{\mathcal{N}}}
\nc{\CO}{{\mathcal{O}}}
\nc{\CP}{{\mathcal{P}}}
\nc{\CQ}{{\mathcal{Q}}}
\nc{\CR}{{\mathcal{R}}}
\nc{\CS}{{\mathcal{S}}}
\nc{\CT}{{\mathcal{T}}}
\nc{\CU}{{\mathcal{U}}}
\nc{\CV}{{\mathcal{V}}}
\nc{\CW}{{\mathcal{W}}}
\nc{\CX}{{\mathcal{X}}}
\nc{\CZ}{{\mathcal{Z}}}

\nc{\uu}{{U_q(\hgl_n)}}
\nc{\uup}{{U_q^+(\hgl_n)}}
\nc{\uug}{{U_q^\geq(\hgl_n)}}
\nc{\uus}{{U_q(\hsl_n)}}
\nc{\uuu}{{U_q(\hgl_1)}}

\nc{\nni}{{\mathbb{N}}^I}
\nc{\zzi}{{\mathbb{Z}}^I}
\nc{\qqi}{{\mathbb{Q}}^I}
\nc{\rri}{{\mathbb{R}}^I}

\nc{\nn}{{\mathbb{N}}^n}
\nc{\zz}{{\mathbb{Z}}^n}
\nc{\qq}{{\mathbb{Q}}^n}
\nc{\rr}{{\mathbb{R}}^n}

\nc{\tT}{{T}}
\nc{\Ka}{{K}}

\nc{\wfZ}{{\widetilde{\fZ}}}
\nc{\wT}{{\widetilde{T}}}

\nc{\od}{{\overline{d}}}
\nc{\rg}{{\textrm{R}\Gamma}}
\nc{\erg}{{\emph{R}\Gamma}}
\nc{\id}{{\textrm{id}}}

\def\ph{\varphi}
\def\pt{\textrm{pt}}

\def\Ext{\textrm{Ext}}
\def\Hom{\textrm{Hom}}

\def\ph{\varphi}

\def\e{\varepsilon}

\def\vs{\varsigma}

\def\ind{\text{ind}}
\def\eind{\emph{ind}}

\def\high{\text{ht }}
\def\wide{\text{wd }}

\def\wt{\text{wt }}
\def\ehigh{\emph{ht }}
\def\ewide{\emph{wd }}

\def\op{\text{op}}

\def\sym{\text{Sym}}

\def\estab{\emph{Stab}}
\def\stab{\text{Stab}}

\def\su{{U_q(\dot{\fsl}_n)}}
\def\sup{{U_q^+(\dot{\fsl}_n)}}

\def\uu{{U_q(\dot{\fgl}_n)}}

\def\uuu{{U_q(\dot{\fgl}_1)}}

\def\uup{{U_q^+(\dot{\fgl}_n)}}
\def\uum{{U_q^-(\dot{\fgl}_n)}}
\def\uul{{U_q^\leq(\dot{\fgl}_n)}}

\def\uug{{U_q^\geq(\dot{\fgl}_n)}}

\def\UU{{U_{q,t}(\ddot{\fsl}_n)}}
\def\UUm{{U^-_{q,t}(\ddot{\fsl}_n)}}
\def\UUp{{U^+_{q,t}(\ddot{\fsl}_n)}}
\def\UUl{{U^\leq_{q,t}(\ddot{\fsl}_n)}}
\def\UUg{{U^\geq_{q,t}(\ddot{\fsl}_n)}}

\def\bla{{\boldsymbol{\lambda}}}

\def\bmu{{\boldsymbol{\mu}}}
\def\bnu{{\boldsymbol{\nu}}}
\def\bth{{\boldsymbol{\theta}}}

\def\ov{\overline}
\def\sh{\ast}
\def\woo{\widehat{\otimes}}

\def\ba{{\textbf{a}}}
\def\bb{{\textbf{b}}}
\def\bc{{\textbf{c}}}
\def\bk{{\textbf{k}}}
\def\bl{{\textbf{l}}}
\def\bm{{\textbf{m}}}

\def\bv{{\textbf{v}}}
\def\bw{{\textbf{w}}}

\def\bs{{\boldsymbol{\vs}}}

\def\ebv{{\emph{\textbf{v}}}}
\def\ebw{{\emph{\textbf{w}}}}
\def\ebm{{\emph{\textbf{m}}}}
\def\ebk{{\emph{\textbf{k}}}}

\def\od{\overline{d}}
\def\loccit{\emph{loc. cit. }}
\def\loc{\text{loc}}
\def\lamu{{\lambda \backslash \mu}}

\def\blamu{{\bla \backslash \bmu}}
\def\blanu{{\bla \backslash \bnu}}
\def\bnumu{{\bnu \backslash \bmu}}

\def\cray{{\bla^+ \backslash \bla^-}}

\def\ma{{\max}}

\def\la{{\lambda}}
\def\th{{\theta}}
\def\sq{{\square}}
\def\bsq{{\blacksquare}}
\def\col{{\text{color }}}

\def\ld{\text{l.d. }}

\def\eld{\emph{l.d. }}

\def\mindeg{\text{min deg }}
\def\maxdeg{\text{max deg }}
\def\emindeg{\emph{min deg }}
\def\emaxdeg{\emph{max deg }}

\def\homdeg{\text{hom deg }}

\def\La{{\Lambda}}

\def\b{\textbf}

\def\leaf{\text{Leaf}}

\def\sslash{\mathbin{/\mkern-6mu/}}

\def\un{\text{un}}
\def\ss{\text{ss}}
\def\s{\text{s}}

\def\pic{\text{Pic}}
\def\rep{\text{Rep}}

\def\ochi{\widetilde{\chi}}

\def\oP{P}
\def\oQ{Q}
\def\oZ{\widetilde{Z}}

\nc{\tCF}{{\widetilde{\CF}}}
\nc{\oCF}{{\overline{\CF}}}
\nc{\oCN}{{\widetilde{\CN}}}
\nc{\oCV}{{\overline{\CV}}}
\nc{\oCW}{{\overline{\CW}}}

\def\oij{{\overrightarrow{ij}}}

\begin{document}

\bibliographystyle{agu04} 

\titlepage{Quantum Algebras and Cyclic Quiver Varieties}{Andrei Negu\cb t}


\frontispiece{\includegraphics[width=5.8in]{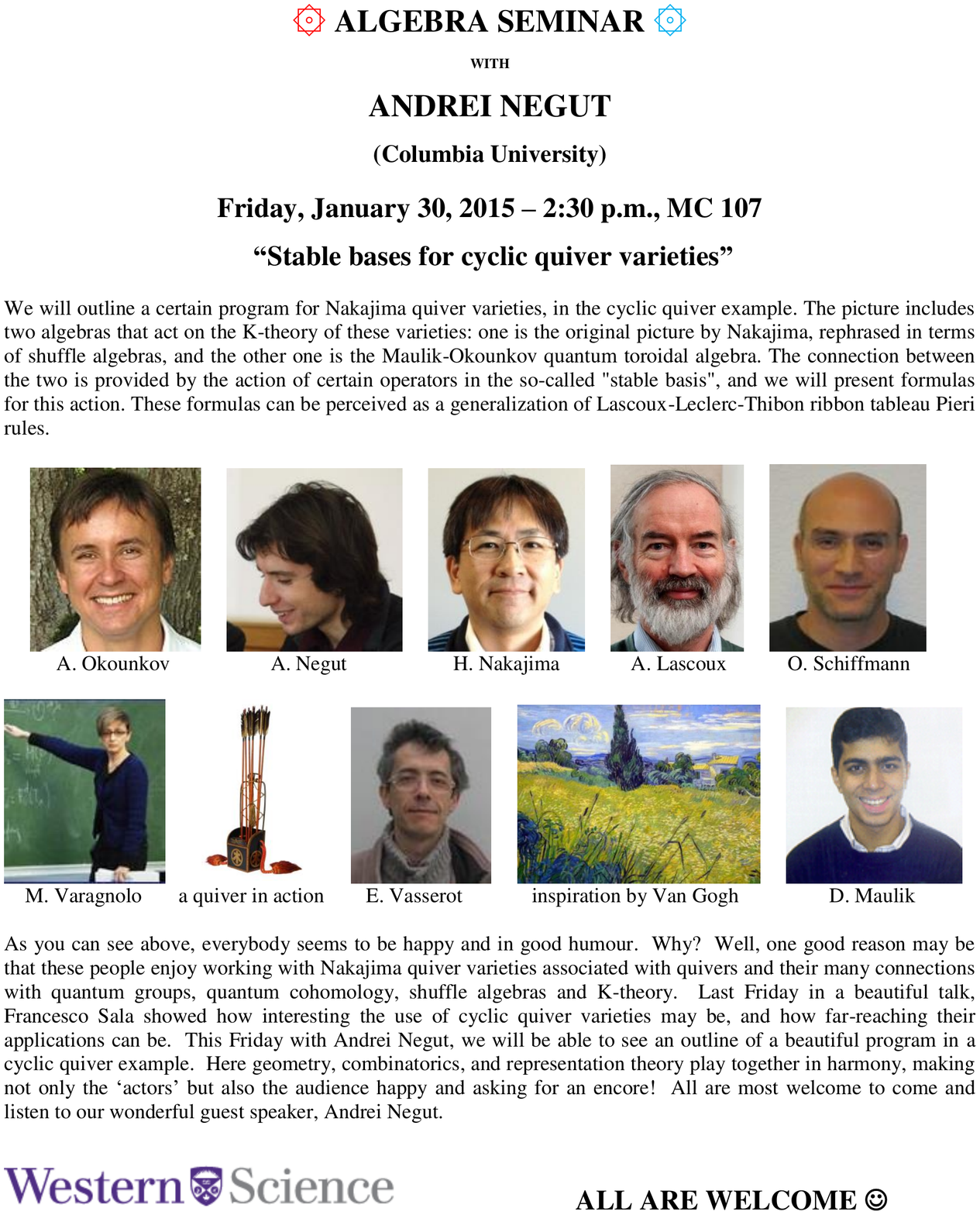} \\
\centering{\footnotesize{Poster, courtesy of Ján Mináč and Leslie Hallock, for my talk at the University of Western Ontario}}}

\copyrightpage{Andrei Negu\cb t}

\startabstractpage
{Quantum Algebras and Cyclic Quiver Varieties}{Andrei Negu\cb t}
\noindent The purpose of this thesis is to present certain viewpoints on the geometric representation theory of Nakajima cyclic quiver varieties, in relation to the Maulik-Okounkov stable basis. Our main technical tool is the shuffle algebra, which arises as the $K-$theoretic Hall algebra of the double cyclic quiver. We prove the isomorphism between the shuffle algebra and the quantum toroidal algebra $\UU$, and identify the quotients of Verma modules for the shuffle algebra with the $K-$theory groups of Nakajima cyclic quiver varieties, which were studied by Nakajima and Varagnolo-Vasserot. 

\tab 
The shuffle algebra viewpoint allows us to construct the universal $R-$matrix of the quantum toroidal algebra $\UU$, and to factor it in terms of pieces that arise from subalgebras isomorphic to quantum affine groups $U_q(\dot{\fgl}_{m})$, for various $m$. This factorization generalizes constructions of Khoroshkin-Tolstoy to the toroidal case, and matches the factorization that Maulik-Okounkov produce via the stable basis in the $K-$theory of Nakajima quiver varieties. We connect the two pictures by computing formulas for the root generators of $\UU$ acting on the stable basis, which provide a wide extension of Murnaghan-Nakayama and Pieri type rules from combinatorics.

\label{Abstract}

\initializefrontsections

\makeatletter
\if@twoside \setcounter{page}{2} \else \setcounter{page}{0} \fi
\makeatother

\tableofcontents    
\listoffigures      

\startacknowledgementspage
I would first and foremost like to thank my advisor, Andrei Okounkov, for all his guidance over the past 8 years. Studying with him has been a great opportunity and the source of much inspiration. I would also like to thank the many wonderful mathematicians who have mentored and taught me over the years, including Roman Bezrukavnikov, Alexander Braverman, Ivan Cherednik, Boris Feigin, Michael Finkelberg, Dennis Gaitsgory, Adriano Garsia, Yulij Ilyashenko, Ivan Losev, Davesh Maulik, Hiraku Nakajima, Nikita Nekrasov, Rahul Pandharipande, Olivier Schiffmann, as well as the math teachers from my formative years: Dorela F\unichar{259}ini\cb si, Radu Gologan, Sever Moldoveanu, C\unichar{259}lin Popescu, Dan Schwartz. I would like to thank my collaborators Claudiu Raicu, Sachin Gautam, Vadim Gorin, Eugene Gorsky, Michael McBreen, Leonid Rybnikov, Francesco Sala, Peter Samuelson, Aaron Silberstein, Bhairav Singh, Andrey Smirnov, Alexander Tsymbaliuk, Shintaro Yanagida and many others, for incredibly useful discussions about mathematics and many great collaborations. 

\tab 
My deepest gratitude goes to Michelle for her love and support during my PhD studies, an experience which I've heard (rightfully) described as ``monastic". My parents Livia and Nicu, as well as my brother Radu, have shaped me into the person I am today. That is a priceless gift, one which transcends the many great things I have learned in graduate school. Many thanks are due to the people who supported me throughout my graduate studies, especially Terrance Cope for the amazing way he took care of all his graduate students at Columbia University. I would also like to thank the designers of this Dissertation format from the University of Michigan.
\label{Acknowledgements}

\dedicationpage{şi-ntocmai cum cu razele ei albe luna \\
nu micşorează, ci tremurătoare \\
măreşte şi mai tare taina nopţii, \\
aşa îmbogăţesc şi eu întunecata zare \\
cu largi fiori de sfânt mister \\ 
şi tot ce-i neînţeles \\
se schimbă-n neînţelesuri şi mai mari \\
sub ochii mei \\
\flushright{\emph{Lucian Blaga, 1919}}} 


\startthechapters 

\chapter{Introduction}
\label{chap:intro}
 \section{Overview}
\label{sec:overview}

\noindent The present thesis attempts to outline a program pertaining to the geometric representation theory of \b{symplectic resolutions}. By definition, these are holomorphic symplectic varieties $(X,\omega)$ endowed with a proper resolution of singularities:
$$
\pi: X \longrightarrow X_0
$$
to an affine variety $X_0$. All our symplectic resolutions will be \b{conical}, in the sense that there is a $\BC^*-$action with respect to which $\pi$ is equivariant. This action is assumed to contract $X_0$ to a point, and to scale the symplectic form $\omega$ by the character $t \rightarrow t^{2}$. 

\tab
Important examples of conical symplectic resolutions include cotangent bundles to flag varieties, hypertoric varieties, and transverse slices to Schubert varieties inside affine Grassmannians. However, the most important example of symplectic resolutions for us will be a class of moduli spaces known as \b{Nakajima quiver varieties}. These were introduced by \citep{Nak0} as certain ``cotangent bundles" to moduli spaces of framed representations of quivers. Consider the \b{cyclic quiver}: \\

\begin{figure}[H]
	
\begin{picture}(200,100)(-30,-10)

\put(120,40){$\widehat{A}_n \ = $}

\put(200,85){\circle*{3}}
\put(200,5){\circle*{3}}
\put(235,65){\circle*{3}}
\put(235,25){\circle*{3}}
\put(165,65){\circle*{3}}
\put(165,25){\circle*{3}}

\put(203,85){\vector(3,-2){29}}
\put(235,63){\vector(0,-1){36}}
\put(232,24){\vector(-3,-2){29}}
\put(197,5){\vector(-3,2){29}}
\put(165,27){\vector(0,1){36}}
\put(168,66){\vector(3,2){29}}

\put(198,89){$1$}
\put(196,-2){$...$}
\put(238,68){$2$}
\put(157,68){$n$}


\end{picture}

\caption[The cyclic quiver]

\label{fig:quiver}

\end{figure}

\noindent and the corresponding Nakajima quiver varieties $\CN_{\bv,\bw}^\bth$ will be defined in Chapter \ref{chap:nakajima}. The construction depends on the data $\bv,\bw \in \nn$, as well as the choice of a so-called \b{stability condition} $\bth \in \rr$, which for most of this paper will be taken to be:
$$
\bth \ = \ (1,...,1) 
$$
Nakajima cyclic quiver varieties generalize three other important symplectic resolutions: cotangent bundles to type $A$ partial flag varieties, moduli spaces of framed sheaves on $\BP^2$, and Hilbert schemes of resolutions of type $A$ singularities.

\tab
The numerical invariants of moduli spaces play an important role in mathematical physics, as they can be identified with correlation functions in various quantum field theories. In particular, the $K-$theory of Nakajima quiver varieties is an area of significant current interest. Following Nakajima, let us define:
\begin{equation}
\label{eqn:jonsnow}
K(\bw) \ = \ \bigoplus_{\bv\in \nn} K_T(\CN^\bth_{\bv,\bw})
\end{equation}
with respect to a torus action $T \curvearrowright \CN^\bth_{\bv,\bw} \ $ which will be properly introduced in Chapter \ref{chap:nakajima}. Starting with the work of \citep{Nak1} and \citep{VV}, it became clear that $K(\bw)$ should be thought of as a representation of a certain algebra $\CA$. Among other things, this allows one to interpret various integrals of $K-$theory classes (which arise in physics as correlation functions) as characters of operators in the algebra $\CA$, which can be studied via representation theory.

\tab
The algebra $\CA$ has been interpreted as the quantum toroidal algebra $\UU$ in \citep{VV}, who proved that:
\begin{equation}
\label{eqn:vv}
\UU \ \curvearrowright \ K(\bw)
\end{equation} 
for any $\bw \in \nn$. In \citep{Nak1}, the construction was done for all quivers without loops, although this generality falls outside the scope of the present thesis. Quantum toroidal algebras admit a presentation as shuffle algebras $\CS$, a view which we will explain in Chapter \ref{chap:shuffle}. This connection was observed by \citep{E}, based on a construction of \citep{FO} for finite dimensional Lie algebras. When $\CA$ is interpreted as a shuffle algebra, its action \eqref{eqn:vv} on $K(\bw)$ was constructed independently by \citep{FT} and \citep{SV} when $n=1$ and $\bw = (1)$, which is the case of the Hilbert scheme. One of the technical results we will prove in the present thesis is: \\

\begin{theorem}
\label{thm:iso0}

The map $\UUp \hookrightarrow \CS^+$ defined by Enriquez is surjective. Since it is preserves the bialgebra structures on these algebras, it gives rise to an isomorphism: 
\begin{equation}
\label{eqn:iso0}
\UU \cong \CS
\end{equation}
of their Drinfeld doubles. 

\end{theorem}





\tab
We shall henceforth write $\CA$ for either $\UU$ or $\CS$. The definition of Drinfeld doubles will be reviewed in Section \ref{sec:basicquantum}, and the proof of Theorem \ref{thm:iso0} will occupy most of Chapter \ref{chap:subalgebras}. The main idea is to assign to each $\bm \in \qq$ a subalgebra:
\begin{equation}
\label{eqn:subalg}
\CA \ \supset \ \CB_\bm \ \cong \ \bigotimes_{h=1}^g U_q(\dot{\fgl}_{l_h})
\end{equation}
where the natural numbers $g,l_1,...,l_g$ depending on $\bm$ will be constructed in Section \ref{sec:arcs}. In order to avoid double hats on quantum toroidal algebras, we will always use dots instead: hence $U_q(\dot{\fgl}_{l_h})$ is the quantum affine group more commonly denoted by $U_q(\widehat{\fgl}_{l_h})$. The embeddings \eqref{eqn:subalg} induce a factorization:
\begin{equation}
\label{eqn:kraftwerk}
\CA^\pm \ = \ \bigotimes_{r \in \BQ}^{\rightarrow} \CB^\pm_{\bm + r \bth} 
\end{equation}
of the positive/negative halves of the algebra $\CA$, for any fixed $\bm \in \qq$. In other words, the subalgebras $\CB_{\bm+r\bth}$ are the building blocks of the quantum toroidal algebra $\CA$, as the \b{slope} $r$ varies over the rational numbers. The product \eqref{eqn:kraftwerk} may be infinite, but only finitely many of the factors have non-trivial matrix coefficients in all representations studied in this thesis. The factorization \eqref{eqn:kraftwerk} has a very important consequence for the universal $R-$matrix of $\CA$:
\begin{equation}
\label{eqn:defr0}
\CR \ \in \ \CA \ \widehat{\otimes} \ \CA
\end{equation}
which is characterized by the property:
\begin{equation}
\label{eqn:defr1}
\CR \cdot \Delta(a) \ = \ \Delta^{\op}(a)\cdot \CR \qquad \qquad \forall \ a \in \CA
\end{equation}
Specifically, we will show in Section \ref{sec:factorizing} that the universal $R-$matrix of the quantum toroidal algebra $\CA$ factors in terms of $R-$matrices of the quantum groups $\CB_{\bm+r\bth}$:
\begin{equation}
\label{eqn:rfac1}
\CR \ = \ \left( \prod_{r \in \BQ}^{\rightarrow}  \CR_{\CB_{\bm+ r \bth}} \right) \cdot \CR_{\CB_{\infty\cdot \bth}} 
\end{equation}
for any fixed $\bm \in \qq$. This factorization is a toroidal analogue of the constructions of \citep{KT} for finite-dimensional and affine quantum groups. As in the philosophy of \emph{loc. cit.}, the idea is to break up the complicated quantum toroidal algebra $\CA$ into simpler pieces. In the toroidal case, these pieces are precisely $\CB_\bm$, as $\bm$ ranges over $\qq$. By analogy with \emph{loc. cit.}, we will refer to $\CB_\bm$ as \b{root subalgebras} and their $R-$matrices $\CR_{\CB_\bm}$ will be called \b{root $R$-matrices}.

\tab 
The object \eqref{eqn:defr0} may seem purely formal, but its action is well-defined on tensor products of Verma modules:
\begin{equation}
\label{eqn:defverma}
\CA \ \curvearrowright \ M(\bw) 
\end{equation}
\begin{equation}
\label{eqn:r}
R \ : \ M(\bw) \otimes M(\bw') \longrightarrow M(\bw) \otimes^{\op} M(\bw')
\end{equation}
intertwining the $\CA-$module structures induced by $\Delta$ and $\Delta^\op$. Verma modules for the shuffle algebra will be reviewed in Section \ref{sec:verma}, where we will construct their Shapovalov forms. The maximal quotient of $M(\bw)$ on which the form is non-degenerate:
\begin{equation}
\label{eqn:robbstark}
L(\bw) \ = \ \frac {M(\bw)}{\text{kernel of Shapovalov form}}
\end{equation}
is the toroidal analogue of irreducible modules in category $\CO$. The above constructions will be defined purely algebraically in Chapters \ref{chap:shuffle} and \ref{chap:subalgebras}, but they are strongly motivated by geometry. As discovered by \citep{SV} for the Jordan quiver, shuffle algebras such as $\CS$ can be interpreted as (parts) of the $K-$theory of $T^*(\text{a certain stack})$, as we will recall in Section \ref{sec:k}. Then both the algebra structure of $\CS$, as well as the action $\CS \curvearrowright L(\bw)$, can be interpreted as: 
\begin{equation}
\label{eqn:jeyne}
T^*(\text{stack}) \ \text{ acts on itself and on Nakajima quiver varieties}
\end{equation}
The above construction is called the $K-$\b{theoretic Hall algebra} by \emph{loc. cit.} It is based on the evolution from the usual Hall algebra of polynomials, to be recalled in Section \ref{sec:k}, to the study of similar algebra structures in cohomology and $K-$theory. This builds on the work of numerous mathematicians, such as \citep{KL}, \citep{GV} and \citep{Gr}.

\tab 
While we will not take this path in the present thesis, let us mention the cohomological Hall algebra of \citep{KS}, which has philosophical similarities with the above construction. For this connection and an application of the above ideas to more general formal cohomological theories, see \citep{YZ}. We will review the principle \eqref{eqn:jeyne} in Section \ref{sec:k}, and although we do not have a complete framework of the geometry of the stacks involved, we will prove the identification between the actions of the shuffle algebra and the quantum toroidal algebra on the $K-$theory groups of Nakajima cyclic quiver varieties: \\

\begin{theorem}
\label{thm:act0}
For any $\ebw \in \nn$, the following actions are compatible:
\begin{equation}
\label{eqn:act0}
\xymatrix{
\UU \ar[d]^\cong & \curvearrowright & K(\ebw) \ar@{_{(}->}[d] \\
\CS & \curvearrowright & L(\ebw)} 
\end{equation}
where the isomorphism on the left is the content of Theorem \ref{thm:iso0}, the top action is \eqref{eqn:vv} and the bottom action is \eqref{eqn:jeyne}.

\tab 
The map on the right is an isomorphism after localization. We construct it by identifying the Shapovalov form on $L(\ebw)$ with the Euler characteristic pairing on $K(\ebw)$.

\end{theorem}


\tab 
As Theorem \ref{thm:act0} establishes a connection between the geometry of $K(\bw)$ and the representation theory of Verma modules $M(\bw) \twoheadrightarrow L(\bw)$, one may ask about the representation theoretic interpretation of the $R-$matrices \eqref{eqn:r}. This construction is closely related to the Bethe ansatz for quantum toroidal algebras, as studied recently by Feigin, Jimbo, Miwa and Mukhin in \citep{FJMM} and \citep{FJMM2}. 

\tab 
The geometric incarnation of trigonometric $R-$matrices was developed in great generality by \citep{MO2}, and their work serves as the motivation for much of this thesis. In a sense, they carry the above framework in reverse order: first they define the $R-$matrices \eqref{eqn:r} for every pair of degree vectors $\bw,\bw' \in \nn$, then they use a general procedure to define a quasi-triangular Hopf algebra $\CA^{\text{MO}}$ which acts on the $K-$theory of Nakajima quiver varieties. While proving this goes beyond our scope, one expects:
\begin{equation}
\label{eqn:expect1}
\CA \ \cong \ \CA^{\text{MO}} \qquad \text{and their actions on }K(\bw) \text{ coincide}
\end{equation}
The main idea of \loccit is to construct the $R-$matrix \eqref{eqn:r} as a composition:
\begin{equation}
\label{eqn:ned}
R_\bm^{+ , -} \ : \ K(\bw) \otimes K(\bw') \stackrel{\stab^-_\bm}\longrightarrow K(\bw + \bw') \stackrel{\left(\stab^+_\bm\right)^{-1}}\longrightarrow K(\bw) \otimes K(\bw')
\end{equation}
where the maps $\stab^\pm_\bm$ are called \b{the stable basis}, and are defined with respect to a one dimensional flow on Nakajima quiver varieties, in the positive ($+$) or negative $(-)$ direction. The above transformations are quite difficult to compute in general, but \loccit breaks them down into more elementary transformations denoted by:
\begin{equation}
\label{eqn:defr}
R^\pm_{\bm , \bm'} \ : \  \ K(\bw) \otimes K(\bw') \stackrel{\stab^\pm_\bm}\longrightarrow K(\bw + \bw') \stackrel{\left(\stab^\pm_{\bm'}\right)^{-1}}\longrightarrow K(\bw) \otimes K(\bw')
\end{equation}
and called root $R-$matrices. Note that, as opposed from \eqref{eqn:ned}, both stable maps in \eqref{eqn:defr} are taken with respect to the same flow: either positive or negative. On general grounds, the $R-$matrices \eqref{eqn:defr} correspond to sub-quasitriangular Hopf algebras:
$$
\CB_\bm^{\text{MO}} \ \subset \ \CA^{\text{MO}} 
$$
which will be called root subalgebras. The reason for the choice of terminology is the following factorization, which we will review in Lemma \ref{lem:stablefac} in greater generality:
\begin{equation}
\label{eqn:rfac2}
R_\bm^{+ , -} \ = \ \prod_{r \in \BQ_+}^{\rightarrow}  R^+_{\bm + r \bth , \bm + (r+\e)\bth} \cdot \text{diagonal matrix} \cdot \prod_{r \in \BQ_-}^{\leftarrow}  R^-_{\bm + (r+\e) \bth , \bm+ r \bth}
\end{equation}
In other words, inverting the direction of the flow from $-$ to $+$ can be achieved by infinitely many small increases of the index $\bm \in \nn$. Looking back at the expectation \eqref{eqn:expect1}, we see that the left hand side of \eqref{eqn:rfac1} coincides with the limit case $\bm \rightarrow  - \infty \cdot \bth$ of \eqref{eqn:rfac2}, when evaluated in the representation $K(\bw) \otimes K(\bw')$. Then one expects that the root $R-$matrices in the right hand sides of \eqref{eqn:rfac1} and \eqref{eqn:rfac2} also coincide, and in fact, this arises from an isomorphism of the corresponding root subalgebras:
\begin{equation}
\label{eqn:expect2}
\CB_\bm \ \cong \ \CB_\bm^{\text{MO}} 
\end{equation}
Stable bases for the $K-$theory of cotangent bundles to partial flag varieties were also constructed in \citep{RTV}, where the authors use them to study Bethe subalgebras. Though the geometric setup of the present thesis is philosophically very close to that of \emph{loc. cit.}, we prefer to recall the construction of \citep{MO2} in full generality in Chapter \ref{chap:stable}, for the reference of the interested reader.

\tab
As we said, proving \eqref{eqn:expect1} and \eqref{eqn:expect2} falls outside the scope of this thesis. However, we will now describe certain explicit formulas which partially confirm \eqref{eqn:expect2}. The stable basis construction produces not only maps as in \eqref{eqn:defr}, but also elements:
\begin{equation}
\label{eqn:defstable}
s_\bla^{\pm,\bm} \ \in \ K(\bw) 
\end{equation}
These elements form a basis of the group of $K-$theory classes supported on certain stable Lagrangians, as we will recall in Chapter \ref{chap:stable}, hence the name ``stable basis". The indexing set is over all $\bw-$tuples of partitions $\bla$, the combinatorics of which will be recalled in Section \ref{sec:basicpar}. Moreover, the subalgebras $\CB_\bm$ of \eqref{eqn:subalg} are generated by the root generators of quantum groups $\uu$, as defined in Section \ref{sec:quantum}:
\begin{equation}
\label{eqn:genp}
P_{\pm [i;j)}^\bm \quad \text{and} \quad Q_{\pm [i;j)}^\bm \quad \in \quad \CB_\bm
\end{equation}
In the above, $i < j \in \BN$ go over $\bm-$integral arcs in the cyclic quiver \eqref{fig:quiver}, i.e. arcs such that $m_i+...+m_{j-1} \in \BZ$. The relations between these generators are very meaningful, e.g. for any pair of minimal $\bm-$integral arcs $[i;j)$ and $[i';j')$ we have:
\begin{equation}
\label{eqn:commutator}
\left[P_{[i;j)}^\bm, Q_{-[i';j')}^\bm \right] \ = \ \delta_{[i';j')}^{[i;j)} (q^2-1)\left( \ph_{[i;j)} - \ph_{[i;j)}^{-1} \right)
\end{equation}
where $\ph_{[i;j)} = \ph_i...\ph_{j-1}$ are Cartan generators. The above is the key relation within $U_q(\dot{\fsl}_{n})$, and the complete set of relations that hold in the quantum groups $U_q(\dot{\fgl}_{n})$ will be recalled in Section \ref{sec:quantum}. Note that $P_{-[i;j)}^\bm$ and $Q_{[i;j)}^\bm$ are the antipodes of $Q_{-[i;j)}^\bm$ and $P_{[i;j)}^\bm$, and so they also satisfy relation \eqref{eqn:commutator}. We will realize the generators \eqref{eqn:genp} in terms of the shuffle algebra in Section \ref{sec:explicit}, and use them to prove Theorem \ref{thm:iso0}. According to Theorem \ref{thm:act0}, the elements $P_{\pm [i;j)}^\bm$ and $Q_{\pm [i;j)}^\bm$ of $\CA$ act on the group $K(\bw)$ for any $\bw \in \nn$. We identify this action with the so-called \b{eccentric correspondences} on Nakajima varieties, that we define in Section \ref{sec:act}. The main purpose of Chapter \ref{chap:pieri} is to obtain formulas for $P_{\pm [i;j)}^\bm, Q_{\pm [i;j)}^\bm \curvearrowright K(\bw)$ in the stable basis \eqref{eqn:defstable}: \\






\begin{theorem}
\label{thm:pieri}

For any $\ebm \in \qq$ and any $\ebm-$integral arc $[i;j)$, we have:
\begin{equation}
\label{eqn:pieri1}
P_{[i;j)}^\ebm \cdot \frac {s_\bmu^{+, \ebm}}{o_\bmu^\ebm} \ = \ \sum^{\blamu \ = \ C \text{ is a type }[i;j)}_{\text{cavalcade of }\ebm-\text{ribbons }} \frac {s_\bla^{+, \ebm}}{o_\bla^\ebm} \cdot (1 - q^2)^{\#_C} (-q)^{N_C^+} q^{\ehigh C + \eind_{C}^{\ebm}} 
\end{equation}
\begin{equation}
\label{eqn:pieri2}
Q_{-[i;j)}^{\ebm} \cdot \frac {s_\bla^{+, \ebm}}{o_\bla^\ebm} \ = \ \sum^{\blamu \ = \ S \text{ is a type }[i;j)}_{\text{stampede of }\ebm-\text{ribbons }} \frac {s_\bmu^{+, \ebm}}{o_\bmu^\ebm} \cdot (1 - q^2)^{\#_S} (-q)^{N_S^-} q^{\ewide S - \eind_{S}^{\ebm} - j+i} 
\end{equation}
By Exercise \ref{ex:opposite}, the above also give formulas for $P_{-[i;j)}^\ebm$ and $Q_{[i;j)}^\ebm$ acting on $s_\bla^{-,\ebm}$.


\end{theorem}



\tab 
The combinatorial notions that appear in the right hand sides of \eqref{eqn:pieri1} and \eqref{eqn:pieri2} will be defined properly in Section \ref{sec:basicpar}. In short, a ``cavalcade" $C$ is a disjoint union of $\#_C$ non-adjacent ribbons, strung together according to the $\bm-$integral arc $[i;j)$. A ``stampede" $S$ is what happens when a cavalcade goes wrong, in that ribbons are now going backwards and they are allowed to touch. However, when the $\bm-$integral arc $[i;j)$ is minimal, both a cavalcade and a stampede consist of a single ribbon, and formulas \eqref{eqn:pieri1} and \eqref{eqn:pieri2} look quite similar. In the particular case $\bm = (0,0,...,0)$ and $\bw = (1,0,...,0)$, such ribbons consist of a single box and we obtain the action:  
$$
\su \ \curvearrowright \ \Lambda \qquad \text{where} \qquad \Lambda \ = \ \text{ring of symmetric functions}
$$ 
constructed by \citep{Hay} and \citep{MM} (for $q \mapsto - \frac 1q$). We will show the connection between this particular case and the general formulation of Theorem \ref{thm:pieri} in Section \ref{sec:maya}, by using Maya diagrams. If we consider all arcs $[i;j)$, not just the minimal ones, then Theorem \ref{thm:pieri} gives rise to an action $\uu \curvearrowright \Lambda$. This action is a different presentation of the one constructed by \citep{LLT}.  

\tab 
For general $\bm \in \qq$ and $\bw \in \nn$, we define the integer $\ind_{C}^\bm$ in \eqref{eqn:ind} below, and Theorem \ref{thm:pieri} provides an action of quantum affine groups on tensor products of Fock spaces. When $n = 1$ and $\bw = (1)$, the above procedure yields an operator:
$$
P_{\pm k}^m \ : \ \Lambda \ \longrightarrow \ \Lambda
$$
for any $k\in \BN$ and $m\in \BQ$. In \citep{Npieri}, we proved formulas for a certain plethystic modification of the above operators in the stable basis $s_\lambda^{+,m}$. Since in the particular case of $m=0$, the stable basis $s_\lambda^{+,0}$ consists of Schur functions and the given operators are multiplication by elementary symmetric functions, we called the main Theorem of \loccit the ``$m-$Pieri rule", for any $m \in \BQ$. 

\tab 
The structure of this thesis is the following. In the remainder of this Chapter, we recall generalities and notations pertaining to $K-$theory, partitions, quantum algebras, symplectic varieties and geometric invariant theory. In Chapter \ref{chap:nakajima}, we recall the construction of Nakajima quiver varieties and their relations with moduli spaces of sheaves, and introduce well-known geometric constructions. In Chapter \ref{chap:stable}, we recall the definition of the $K-$theoretic stable basis of \citep{MO2} and their construction of geometric $R-$matrices, and recall the proof of the factorization result \eqref{eqn:rfac2}. In Chapter \ref{chap:shuffle}, we recall the double shuffle algebra and its interpretation as a Hall algebra. We then construct Verma modules for the shuffle algebra, and naturally identify them with the $K-$theory groups of Nakajima quiver varieties. In Chapter \ref{chap:subalgebras}, we define the subalgebras $\CB_\bm$ that will allow us to prove the isomorphism of Theorem \ref{thm:iso0} and to construct the factorization of $R-$matrices \eqref{eqn:rfac1}. In Chapter \ref{chap:pieri}, we prove Theorem \ref{thm:pieri} concerning the root generators $P_{[i;j)}^\bm$ and $Q_{-[i;j)}^\bm$ acting on the stable basis $s^{+,\bm}_\bla$. Since the main goal of this thesis is to present a mathematical landscape to a wide audience, we leave certain technical and straightforward results as Exercises to the interested reader. These are strewn across all Chapters, and those which are not well-known results will be proved in the final Chapter \ref{chap:proofs}. 

	
	
	
	

\section{Basic notions on $K-$theory}
\label{sec:basick}

\noindent This section contains a basic treatment of torus equivariant algebraic $K-$theory, which can be skipped by the more experienced reader. We suggest having a look at formula \eqref{eqn:modified}, where we write down the renormalized Euler characteristics that will be studied throughout this thesis. The main reference for this Section is \citep{CG}. Given an algebraic variety with a torus action $T \curvearrowright X$, one defines its $T-$\b{equivariant }$K-$\b{theory ring}:
$$
K_T(X)
$$
as the Grothendieck group of the category of $T-$equivariant coherent sheaves on $X$. When $X$ is smooth, this group is generated by isomorphism classes of $T-$equivariant vector bundles $\CV$ on $X$. One imposes the relation:
$$
\CV = \CW_1 + \CW_2
$$
for any $T-$equivariant short exact sequence $0 \rightarrow \CW_1 \rightarrow \CV \rightarrow \CW_2 \rightarrow 0$. The multiplication on $K_T(X)$ is given by tensor product. Note that $K_T(X)$ is a module over:
$$
K_T(\pt) \ = \ \text{Rep}(T) \ = \ \BZ[\chi]_{\chi:T\rightarrow \BC^*}
$$
since one can always ``twist" any vector bundle by tensoring it with $\CO_{\chi} = $ the trivial line bundle endowed with a $T-$equivariant structure via the character $\chi:T \rightarrow \BC^*$. \\

\begin{example}

One of the basic non-trivial examples of equivariant $K-$theory is:
\begin{equation}
\label{eqn:constantin}
T \ = \ \BC^* \times \BC^* \ \curvearrowright \ \BC^2 \qquad \qquad (t_1,t_2)\cdot (x,y) = (t_1x,t_2y)
\end{equation}
which induces an action $T \curvearrowright \BP^1$. There are two fixed points with respect to this action, namely $0 = [0:1]$ and $\infty = [1:0]$. The ring $K_{T}(\BP^1)$ is generated by the line bundle $\CO(1)$, with fibers given by: 
$$
\CO(1)|_{[x:y]} \ = \ \Big\{ \text{dual of the line } \{xb = ya\} \subset \BC^2 \Big\}
$$
The fibers of $\CO(1)$ at the fixed points are one-dimensional vector space endowed with a $T-$action, so they are identified with characters:
\begin{equation}
\label{eqn:jean}
\CO(1)|_\infty = q_1, \qquad \qquad \CO(1)|_0 = q_2
\end{equation}
where $q_1,q_2$ denote those characters of $T$ which are dual to the basis $t_1,t_2$ of \eqref{eqn:constantin}. The well-known Euler sequence is defined as:
\begin{equation}
\label{eqn:ses}
0 \longrightarrow q_1 q_2 \CO \stackrel{\alpha}\longrightarrow q_1\CO(1) \oplus q_2 \CO(1) \stackrel{\beta}\longrightarrow \CO(2) \longrightarrow 0
\end{equation}
$$
\alpha(v) = (yv,xv), \qquad \beta(v_1,v_2) = x v_1 - y v_2
$$
The factors of $q_1,q_2$ need to be inserted in the above short exact sequence in order to make the maps $T-$equivariant. One can check that the factors are correct by restricting the above sequence to the fixed points, as in \eqref{eqn:jean}. Relation \eqref{eqn:ses} implies:
$$
(q_1+q_2)l \ = \ q_1q_2 + l^{2}
$$
where we write $l = \CO(1)$. In fact, this is the only relation in the $K-$theory group: 
\begin{equation}
\label{eqn:kp1}
K_T(\BP^1) = \frac {\BZ\left[q_1^{\pm 1}, q_2^{\pm 1}, l^{\pm 1}\right]}{(l - q_1)(l - q_2)}
\end{equation}
Besides the class $l$ of the ample line bundle, one may also consider the skyscraper sheaves at the torus fixed points:
$$
0 \longrightarrow q_1^{-1} \CO(-1) \longrightarrow \CO \longrightarrow I_0 \longrightarrow 0 \qquad \quad \Longrightarrow \quad \qquad I_0 = 1 - \frac 1{lq_1}
$$
$$
0 \longrightarrow q_2^{-1} \CO(-1) \longrightarrow \CO \longrightarrow I_\infty \longrightarrow 0 \qquad \quad \Longrightarrow \quad \qquad I_\infty = 1 - \frac 1{lq_2}
$$
The reason why one needs to twist $\CO(-1)$ differently in the above short exact sequences is that the sheaf of functions which vanish at $0$ (respectively $\infty$) is generated by the function $\frac xy$ (respectively $\frac yx$) and these have different equivariant weights. 

\end{example}

\tab 
The above example allows us to observe a very important phenomenon. The localization of \eqref{eqn:kp1} over the field of equivariant constants:
$$
K_T(\BP^1)_\loc \ := \ K_T(\BP^1) \bigotimes_{\BZ\left[q_1^{\pm 1}, q_2^{\pm 1}\right]} \BQ(q_1,q_2)
$$
has dimension two, as it is generated by the classes $1$ and $l$. On the other hand, it is equally well generated by the classes $I_0$ and $I_\infty$ of skyscraper sheaves at the fixed points. This principle is called \b{equivariant localization}, and roughly states that equivariant $K-$theory is concentrated at the fixed points. More precisely and generally, suppose we are given a variety with a torus action $T \curvearrowright X$. Assume that the fixed point set breaks up into proper connected components as:
$$
X^T \ = \ F_1 \sqcup ... \sqcup F_t
$$
and let $\iota_s : F_s \hookrightarrow X$ denote the various inclusion maps. The Thomason equivariant localization theorem, inspired by Atiyah-Bott localization in cohomology, says that:
\begin{equation}
\label{eqn:loc0}
\alpha = \sum_{s=1}^t \iota_{s*} \left(\frac {\alpha|_{F_s}}{\Lambda^*(N_{F_s \subset X}^\vee)}\right) \ \in \ K_T(X)_\loc: = K_T(X) \bigotimes_{K_T(\pt)} \text{Frac}(K_T(\pt)) 
\end{equation}
Note that the above equality only makes sense in localized equivariant $K-$theory, because of the presence of denominators. We will use the above formula to compute equivariant Euler characteristics of various sheaves on $X$, which are defined as:
$$
\chi_T(X,\alpha) \ = \ \sum_{i=0}^{\infty} (-1)^i \Big(T-\text{character of } H^i(X,\alpha) \Big) \ \in \ \text{Rep}(T) = K_T(\pt)
$$
The above sum makes sense even if there are infinitely many terms, as long as we obtain a convergent geometric series. Since the equivariant Euler characteristic is the same as the push-forward map to a point, we can use \eqref{eqn:loc0} to obtain:
$$
\chi_T(X,\alpha) \ = \ \sum_{s=1}^t \chi_T\left(F_s, \ \frac {\alpha|_{F_s}}{\Lambda^*(N_{F_s \subset X}^\vee)} \right)
$$
In other words, the equivariant Euler characteristic of a class can be computed simply by studying the restriction of that class to the torus fixed locus. In the particular case when $X^T$ consists of discretely many points, the above becomes:
$$
\chi_T(X,\alpha) \ = \ \sum_{p\in X^T} \frac {\alpha|_p}{\Lambda^*(T^\vee_p X)} \ = \ \sum_{p\in X^T} \frac {\alpha|_p}{(1 - w^{-1}_1)...(1-w_d^{-1})} 
$$
where in the second term, $w_1+...+w_d$ is the $T-$character in the tangent space $T_p X$. We will encounter many such localization formulas, and we will prefer to slightly change the denominators in order to make them more symmetric:
$$
\chi_T\left(X,\alpha \cdot K_X^{-1/2}\right) \ = \ \sum_{p\in X^T} \frac {\alpha|_p \cdot (w_1...w_d)^{-1/2}}{(1 - w^{-1}_1)...(1-w_d^{-1})} \ = \ \sum_{p\in X^T} \frac {\alpha|_p}{[T_pX]}
$$
where we write $K_X = \Lambda^dT_X$ and:
\begin{equation}
\label{eqn:quantum}
[v] \ := \ v^{\frac 12} - v^{- \frac 12} \quad \text{and generalize it to} \quad [V] \ := \ \frac {[v_1]...[v_d]}{[v_1']...[v'_{d'}]} 
\end{equation}
for a virtual $T-$representation $V = v_1+...+v_d - v_1' - ... - v_{d'}'$. Here and throughout the paper, the square root of the canonical bundle is simply a formal device meant to make our formulas look better. It is not an essential part of our argument, and in fact it could be eliminated by resorting to polarizations, as defined in \citep{MO}. For the sake of keeping the presentation clear and concise, we will gloss over this imprecision and only work with the renormalized Euler characteristic:
\begin{equation}
\label{eqn:modified}
\ochi_T(X,\alpha) \ := \ \chi_T\left(X, \alpha \cdot K_X^{-1/2}\right) 
\end{equation}
whose effect on our formulas will simply be to ``center" the denominators, i.e. replacing the ``quantum numbers" $1-w^{-1}$ by their more symmetric form $w^{\frac 12} - w^{-\frac 12}$. Formula \eqref{eqn:modified} establishes an important feature behind our notation throughout this paper: we always introduce formal square roots of all $T-$characters, and even more so, we introduce formal square roots of line bundles. We will abuse notation and not always mention this explicitly, especially since all our formulas will be of the form:
$$
l^{\frac 12}\cdot (\text{actual }K-\text{theory class})
$$
Taking the artifice one step further, one can define modified push-forward maps as:
\begin{equation}
\label{eqn:modpush}
\widetilde{\pi}_* \ : \ K_T(X) \longrightarrow K_T(Y)
\end{equation}
$$
\widetilde{\pi}_*(\alpha) \ = \ \pi_*\left(\alpha \cdot \text{det}\left[\text{Cone}(TX \stackrel{d\pi}\rightarrow \pi^*TY)\right]^{-\frac 12} \right)
$$
for l.c.i. morphisms $\pi:X \rightarrow Y$. We stress the fact that the above does \b{not} pretend to be a theory of square roots of line bundles on algebraic varieties, but instead is merely a way for us to make our formulas more symmetric. We could have done away without the modification $\pi \mapsto \widetilde{\pi}$, but only at the expense of simplicity and conciseness, which are important attributes of an expository text.

\tab 
The equivariant Euler characteristic is simply the case $Y = \pt$ of \eqref{eqn:modpush}. Then the equivariant localization theorem \eqref{eqn:loc0} can be written as:
\begin{equation}
\label{eqn:loc}
\alpha \ = \ \sum_{s=1}^t \widetilde{\iota}_{s*} \left(\frac {\alpha|_{F_s}}{[N_{F_s \subset X}]} \right)
\end{equation}
for any class $\alpha \in K_T(X)$, where $\iota_s:F_s \hookrightarrow X$ denote the inclusions of the fixed loci and $N_{F_s \subset X}$ denotes the normal bundles. Here we may define the denominator by: 
\begin{equation}
\label{eqn:quantumvectorbundle}
[\CV] \ := \ \frac {[\CL_1] \cdot ... \cdot [\CL_d]}{[\CL_1'] \cdot ... \cdot [\CL'_{d'}]} \ = \ \frac {\left(\CL^{\frac 12}_1 - \CL^{-\frac 12}_1 \right) \cdot ... \cdot \left(\CL^{\frac 12}_d - \CL^{-\frac 12}_d \right)}{\left(\CL^{'\frac 12}_1 - \CL^{'-\frac 12}_1 \right) \cdot ... \cdot \left(\CL^{'\frac 12}_{d'} - \CL^{'-\frac 12}_{d'} \right)}
\end{equation}
for any alternating sum of line bundle classes $\CV = \sum_{i=1}^d \CL_i - \sum_{i'=1}^{d'} \CL'_{i'} \in K_T(X)$. The fact that any vector bundle can be written as a linear combination of line bundles in the $K-$group is a consequence of the splitting principle, though in the present thesis, most vector bundles we will be concerned with will be naturally expressed in this form. \\


\section{Basic notions on partitions}
\label{sec:basicpar}

\noindent A \b{partition} of $v$ is an unordered sequence of natural numbers which sum up to $v$:
$$
\lambda \ \vdash \ v \qquad \text{if }\ \lambda = (\lambda_1 \geq \lambda_2 \geq...) \quad \text{such that} \quad \lambda_1+\lambda_2+... = v
$$
For example, $(4,3,1)$ is a partition of the natural number 8. There is a one-to-one correspondence between partitions and \b{Young diagrams}, the latter being simply stacks of $1\times 1$ boxes placed in the corner of the first quadrant of the plane:
\begin{figure}[H]
	
\begin{picture}(110,130)(-135,0)

\put(17,17){$1$}
\put(14,57){$q_2$}
\put(10,97){$q_2^2$}
\put(57,17){$q_1$}
\put(54,57){$q_1q_2$}
\put(94,17){$q_1^2$}
\put(92,57){$q_1^2q_2$}
\put(134,17){$q_1^3$}

\put(0,0){\line(1,0){160}}
\put(0,40){\line(1,0){157}}
\put(0,80){\line(1,0){117}}
\put(0,120){\line(1,0){37}}

\put(0,0){\line(0,1){120}}
\put(40,0){\line(0,1){117}}
\put(80,0){\line(0,1){80}}
\put(120,0){\line(0,1){77}}
\put(160,0){\line(0,1){37}}

\put(40,120){\circle{5}}
\put(120,80){\circle{5}}
\put(160,40){\circle{5}}

\put(160,0){\circle*{5}}
\put(120,40){\circle*{5}}
\put(40,80){\circle*{5}}
\put(0,120){\circle*{5}}


\end{picture}

\caption[A Young diagram with weights and inner/outer corners represented]

\label{fig:boxes}

\end{figure}

\noindent The above Young diagram represents the partition $(4,3,1)$, because it has 4 boxes on the first row, 3 boxes on the second row, and 1 box on the third row. The circled points in Figure 1.2 denote the corners of the partition: the full circles will be called \b{inner corners}, while the hollow circles will be called \b{outer corners}. The monomials displayed in Figure 1.2 are called the \textbf{weights} of the boxes they are located in:
\begin{equation}
\label{eqn:weight0}
\chi_\square \ = \ q_1^x q_2^y \ = \ q^{x+y} t^{x-y}
\end{equation}
where $q_1 = qt$, $q_2 = \frac qt$, and $(x,y)$ are the coordinates of the southwest corner of the box $\sq$. The coefficient of $t$ in the above expression, namely $x-y$, will be called the \textbf{content} of the box, and note that it is constant across diagonals. In this paper, the word ``diagonal" will only refer to the ones in southwest-northeast direction. Recall the \b{dominance partial order} on partitions $\lambda$ and $\mu$ of the same size $|\la| = |\mu|$:
$$
\la \ \unrhd \ \mu \quad \Leftrightarrow \quad \la_1+...+\la_j \geq \mu_1+...+\mu_j \qquad \forall \ j
$$
At the level of the corresponding Young diagrams, the above condition is equivalent to the fact that we can obtain $\la$ from $\mu$ by rolling boxes from northwest to southeast. We will now define another partial order between partitions. If the Young diagram of $\mu$ is a subset of the Young diagram of $\la$, we will indicate this as:
$$
\la \ \geq \ \mu \qquad \Leftrightarrow \qquad \la_j \ \geq \ \mu_j \quad \forall \ j
$$
The difference $\lamu$ is called a \b{skew} partition, and the corresponding set of boxes is called a skew diagram. If a skew diagram is connected and contains no $2\times 2$ boxes, then we call it a \b{ribbon}. The boxes of a ribbon are labeled from northwest to southeast by their contents:
\begin{figure}[H]

\begin{picture}(100,80)(-40,5)

\put(78,67){\footnotesize{$i$}}
\put(255,8){\footnotesize{$j$}}

\put(70,60){\line(0,1){20}}
\put(70,80){\line(1,0){40}}
\put(110,80){\line(0,-1){60}}
\put(130,40){\line(0,-1){20}}
\put(150,40){\line(0,-1){20}}
\put(170,20){\line(0,-1){20}}
\put(190,20){\line(0,-1){20}}
\put(210,20){\line(0,-1){20}}
\put(230,20){\line(0,-1){20}}
\put(250,20){\line(0,-1){20}}
\put(90,40){\line(1,0){80}}
\put(170,40){\line(0,-1){20}}
\put(150,20){\line(1,0){100}}
\put(250,20){\line(0,-1){20}}
\put(250,0){\line(-1,0){100}}
\put(150,0){\line(0,1){20}}
\put(150,20){\line(-1,0){60}}
\put(90,20){\line(0,1){60}}
\put(110,60){\line(-1,0){40}}

\put(40,25){$R \ :$}

\end{picture}

\caption[A ribbon of type $[i;j)$]

\label{fig:ribbon}

\end{figure}


\noindent The box labeled $i$ is the head, and the box labeled $j$ is the tail of the ribbon. Since the contents of the boxes of a ribbon must be consecutive integers, we refer to $[i;j)$ modulo $\BZ(n,n)$ as the \b{type} of the ribbon. In other words, a $[i;j)-$ribbon and a $[i+n;j+n)-$ribbon are the same thing. The \b{height} (respectively \b{width}) of a ribbon is defined as the difference in vertical (respectively horizontal) coordinate between its first and last box, so the ribbon $R$ in Figure \ref{fig:ribbon} has height 3 and width 8:
\begin{equation}
\label{eqn:height}
\high R = 3 \qquad \qquad \wide R = 8
\end{equation}
Note that the height and width of an $[i;j)-$ribbon always add up to $j-i-1$. Given a vector of rational numbers $\bm = (m_1,...,m_n)\in \qq$, we call $R$ an $\bm-$\b{integral} ribbon if $m_i+...+m_{j-1} \in \BZ$. The indices of $m$ will always be taken modulo $n$, so $m_a := m_{a\text{ mod } n}$ for all $a\in \BZ$. For a $\bm-$integral ribbon $R$ of type $[i;j)$, we define:
\begin{equation}
\label{eqn:ind}
\ind_{R}^\bm \ = \ \sum_{a=i}^{j-2} \pm \Big(m_i + ... + m_{a} - \left \lfloor m_i+...+m_{a} \right \rfloor \Big) 
\end{equation}
where the sign is $+$ or $-$ depending on whether the $(a+1)$-st box in the ribbon is to the right or below the $a$-th box, respectively. A \b{cavalcade} of $\bm-$integral ribbons is a skew diagram that consists of disjoint and non-adjacent ribbons:
\begin{equation}
\label{eqn:little}
R_a \quad \text{of type} \quad [i_{a-1};i_{a}) \qquad \qquad m_{i_{a-1}}+...+m_{i_a-1} \in \BZ
\end{equation}
for $a \in \{1,...,k\}$, strung together in order from the northwest to the southeast, as in Figure \ref{fig:cavalcade} below. The arc $[i_0;i_k) = [i_0;i_1) + ... + [i_{k-1};i_k)$ mod $\BZ(n,n)$ will be called the \b{type} of the cavalcade.


\begin{figure}[H]
	
\begin{picture}(100,110)(-40,-5)

\put(71,82){\footnotesize{$i_{a-1}$}}
\put(138,62){\footnotesize{$i_a$}}

\put(70,75){\line(0,1){20}}
\put(70,95){\line(1,0){40}}
\put(110,95){\line(0,-1){20}}
\put(130,75){\line(-1,0){20}}
\put(90,75){\line(-1,0){20}}
\put(130,75){\line(0,-1){20}}
\put(130,55){\line(-1,0){40}}
\put(90,55){\line(0,1){20}}

\put(176,23){\footnotesize{$i_{a}$}}
\put(272,2){\footnotesize{$i_{a+1}$}}

\put(170,35){\line(1,0){60}}
\put(230,35){\line(0,-1){20}}
\put(270,15){\line(-1,0){40}}
\put(210,15){\line(-1,0){40}}
\put(270,-5){\line(-1,0){60}}
\put(210,-5){\line(0,1){20}}
\put(170,15){\line(0,1){20}}
\put(270,15){\line(0,-1){20}}
	
\end{picture}
	
\caption[A cavalcade of ribbons]
	
\label{fig:cavalcade}
	
\end{figure}

\noindent  Recall that the type $[i_a;i_{a+1})$ of a ribbon is taken modulo $\BZ(n,n)$, and that is why there is no contradiction in having two boxes labeled $i_a$ in the above picture. 

\tab 
Note that a skew diagram can be presented as a cavalcade in at most one way. On the other hand, a \b{stampede} of $\bm-$integral ribbons is a collection of disjoint ribbons $R_1,...,R_k$, positioned as in Figure \ref{fig:stampede}. The defining condition on these ribbons is:
\begin{equation}
\label{eqn:ribbonstampede}
\la = \nu_0 \geq \nu_1 \geq ... \geq \nu_k = \mu \quad \text{are partitions such that} \quad R_a = \nu_{a-1} / \nu_{a}
\end{equation}
and the head of the ribbon $R_{a}$ has content $<$ than the tail of the ribbon $R_{a-1}$:

\begin{figure}[H]
	
\begin{picture}(100,150)(-40,-50)

\put(72,82){\footnotesize{$i_{a+1}$}}
\put(118,62){\footnotesize{$i_{a+2}$}}

\put(70,75){\line(0,1){20}}
\put(90,75){\line(0,-1){20}}
\put(90,55){\line(1,0){20}}
\put(70,75){\line(1,0){20}}
\put(70,95){\line(1,0){40}}
\put(110,95){\line(0,-1){40}}

\put(152,22){\footnotesize{$i_{a-1}$}}
\put(232,-18){\footnotesize{$i_{a}$}}

\put(150,35){\line(1,0){60}}
\put(210,35){\line(0,-1){20}}
\put(210,15){\line(1,0){20}}
\put(230,15){\line(0,-1){40}}
\put(230,-25){\line(-1,0){20}}
\put(210,-25){\line(0,1){20}}
\put(210,-5){\line(-1,0){20}}
\put(190,-5){\line(0,1){20}}
\put(190,15){\line(-1,0){40}}
\put(150,15){\line(0,1){20}}

\put(174,2){\footnotesize{$i_{a}$}}
\put(272,-38){\footnotesize{$i_{a+1}$}}

\put(170,15){\line(0,-1){40}}
\put(170,-25){\line(1,0){20}}
\put(190,-25){\line(0,-1){20}}
\put(190,-45){\line(1,0){80}}
\put(270,-45){\line(0,1){20}}
\put(270,-25){\line(-1,0){40}}
	
\end{picture}
	
\caption[A stampede of ribbons]
	
\label{fig:stampede}
	
\end{figure}

\noindent Note that there can be several stampedes of ribbons on any skew diagram $\lamu$, but there is a unique stampede if we fix the types of the ribbons involved:
$$
R_a \quad \text{of type} \quad [i_{a-1};i_{a}) \qquad \qquad m_{i_{a-1}}+...+m_{i_a-1} \in \BZ 
$$
The arc $[i_0;i_k) = [i_0;i_1) + ... + [i_{k-1};i_k)$ mod $\BZ(n,n)$ will be called the \b{type} of the stampede. If $C$ is either a cavalcade or stampede, the number of constituent ribbons will be denoted by $\#_C$, and we set:
$$
\high C = \sum^{\text{ribbon}}_{R\ \in \ C} \high R \qquad \ \ \wide C = \sum^{\text{ribbon}}_{R \ \in \ C} \wide R \qquad \ \ \ind_{C}^\bm = \sum^{\text{ribbon}}_{R \ \in \ C} \ind_{R}^\bm 
$$
Let us now fix any natural number $w$. A collection of partitions:
\begin{equation}
\label{eqn:wpartition}
\bla = (\lambda^1,...\lambda^w)
\end{equation}
will be called a $w-$\b{partition}. The weight of a box $\sq \in \lambda^i \subset \bla$ is defined as the following generalization of \eqref{eqn:weight0}:
\begin{equation}
\label{eqn:weight}
\chi_{\sq} \ = \ (qu_i) \cdot q_1^xq_2^y \ = \ u_i q^{x+y+1} t^{x-y}
\end{equation}
where $(x,y)$ are the coordinates of $\sq$, and $u_1,...,u_w$ are formal variables that keep track of which constituent partition of $\bla$ the box $\sq$ lies in. The content of a box $\sq \in \lambda^i \subset \bla$ is defined as:
\begin{equation}
\label{eqn:content}
c_{\sq} = a_i + x - y
\end{equation}
where $a_1,...,a_w$ are also formal variables. We will represent $w-$partitions on a single picture, very far away from each other, with $\lambda^i$ northwest of $\lambda^{j}$ for any $i<j$. The situation of $w=2$ is represented in the following picture:

\begin{figure}[H] 
	
\begin{picture}(100,220)(-30,-70)

\put(70,135){\line(0,-1){70}}
\put(70,95){\line(1,0){20}}
\put(90,95){\line(0,-1){10}}
\put(90,85){\line(1,0){20}}
\put(110,85){\line(0,-1){20}}
\put(190,65){\line(-1,0){120}}
\put(190,65){\line(0,-1){120}}
\put(190,-15){\line(1,0){10}}
\put(200,-15){\line(0,-1){10}}
\put(200,-25){\line(1,0){20}}
\put(220,-25){\line(0,-1){30}}
\put(280,-55){\line(-1,0){90}}

\put(80,73){$\lambda^1$}
\put(198,-45){$\lambda^2$}


\end{picture}

\caption[Representing a $w-$partition as far apart Young diagrams] 

\label{fig:partition}

\end{figure} 

\noindent The dominance ordering can be defined for $w-$partitions, i.e. we set $\bla \unrhd \bmu$ if we can obtain $\bla$ from $\bmu$ by rolling boxes from the northwest to the southeast in Figure \ref{fig:partition}. We can also define the ordering $\bla \geq \bmu$ for $w-$partitions, which simply means that: 
$$
\la^j \ \geq \ \mu^j \qquad \forall \ j\in \{1,...,w\}
$$
Going one step further, we can talk about $w-$skew diagrams, which are simply collections of $w$ skew diagrams. Since a ribbon is \b{connected}, it can occupy boxes only in a single constituent partition of a $w-$skew diagram. However, a cavalcade or stampede of ribbons is allowed to occupy boxes in more than one constituent partitions of a $w-$skew diagram.


\tab
The final step is to consider the above notions modulo some fixed $n > 1$. The \b{color} of a box is defined as its content modulo $n$, and the intuition behind this is that we paint all the southwest-northeast diagonals periodically in $n$ colors. Given a vector:
\begin{equation}
\label{eqn:framing}
\bw = (w_1,...,w_n) \ \in \ \nn
\end{equation}
of total size $|\bw| = w_1 + ... + w_n$, define a $\bw-$partition as:
\begin{equation}
\label{eqn:defwpart}
\bla = (\lambda^1,...,\lambda^{\bw})
\end{equation}
where each constituent partition $\lambda^i$ is assigned a color $\sigma_i \in \BZ/n\BZ$ in such a way that $\# \{\sigma_i = k\} = w_k$. These colors should be interpreted as shifts modulo $n$, where the shift $\sigma_i$ of $\lambda^i$ is remembered by demanding that it be congruent to $a_i$ modulo $n$, where $a_i$ is the formal variable of \eqref{eqn:content}. We will not hesitate to write $\lambda^1,...,\lambda^{\bw}$ instead of the more appropriate $\lambda^1,...,\lambda^{|\bw|}$ in \eqref{eqn:defwpart}, since this will remind us that a $\bw-$partition also keeps track of the color shift of each of its constituent partitions. Finally, we define:
\begin{equation}
\label{eqn:ind3}
o_\blamu \ = \ \prod_{\sq \in \blamu} \chi_\sq \qquad \text{and} \qquad o_\blamu^\bm \ = \ \prod_{\sq \in \blamu} \chi^{m_{\text{color of }\sq}}_\sq
\end{equation}
for any $\bm \in \qq$. When $\bmu = \emptyset$, these are the constants which appear in Theorem \ref{thm:pieri}. Finally, to any box $\bsq$ and $\bw-$partition $\bla$, we define the numbers:
\begin{equation}
\label{eqn:defnumbers1}
N_{\bsq|\bla}^+ = \sum^{\sq \text{ corner of }\bla}_{\text{with }c_\sq > c_\bsq} \delta^\sq_{\text{inner}} - \delta^\sq_{\text{outer}}
\end{equation}
\begin{equation}
\label{eqn:defnumbers2}
N_{\bsq|\bla}^- = \sum^{\sq \text{ corner of }\bla}_{\text{with }c_\sq < c_\bsq} \delta^\sq_{\text{outer}} - \delta^\sq_{\text{inner}}
\end{equation}
Given a cavalcade of ribbons $C$, we set $N_{C}^+ = \sum_{\bsq \in C} N_{\bsq|\bla}^+$. For a stampede of ribbons $S$, the corresponding notion is a bit more subtle. We recall that a stampede of ribbons comes with a specified flag of Young diagrams as in \eqref{eqn:ribbonstampede}, and we set:
\begin{equation}
\label{eqn:slightly}
N_{S}^- = \sum_{i=1}^k \sum_{\sq \in R_i} N^-_{\bsq|\bnu_i}
\end{equation}

\section{Basic notions on quantum algebras}
\label{sec:basicquantum}


\noindent In this paper, we will consider \b{bialgebras} that are associative, coassociative, unital, counital, and the product and coproduct will always be:
$$
A \otimes A \ \stackrel{*}\longrightarrow \ A, \qquad A \ \stackrel{\Delta}\longrightarrow \ A \otimes A \qquad \text{such that} \qquad \Delta(a*b) = \Delta(a) * \Delta(b)
$$
We will often use Sweedler notation $\Delta(a) = a_1 \otimes a_2$ for the coproduct, where a sum over tensors is implied, but not explicitly mentioned. A bialgebra is called a \b{Hopf algebra} if it comes endowed with an antipode map:
\begin{equation}
\label{eqn:antipode}
S \ : \ A \longrightarrow A \qquad \text{such that} \qquad S(a_1)a_2 = a_1S(a_2) = \e(a)\cdot 1
\end{equation}
where $\e \in A^\vee$ is the counit. The antipode is an anti-homomorphism of algebras and of coalgebras, and in all examples of the present thesis, it is completely determined by condition \eqref{eqn:antipode}. While not necessary for most of our definitions, one reason why it is good to have the antipode around is the fact that it allows to write down the action of $A$ on Verma modules. Given two bialgebras $A^-$ and $A^+$, a bilinear form:
$$
A^- \otimes A^+ \ \longrightarrow \ \text{base field}
$$
will be called a \textbf{bialgebra pairing} if it satisfies:
\begin{equation}
\label{eqn:bialg}
\langle aa',b \rangle = \langle a\otimes a', \Delta(b) \rangle  \qquad \qquad \langle a, bb' \rangle = \langle \Delta^{\op}(a), b \otimes b'\rangle  
\end{equation}
for all $a,a' \in A^-$ and $b,b' \in A^+$. All our pairings will be Hopf, meaning that they are compatible with the antipode via $\langle Sa ,b \rangle = \langle a, S^{-1}b \rangle$. Given such a bialgebra pairing, we may construct the \textbf{Drinfeld double} of the two bialgebras in question:
\begin{equation}
\label{eqn:dd}
A \ = \ A^- \otimes A^+
\end{equation}
as a vector space, and then require that the two tensor factors be sub-bialgebras which satisfy the extra relation:
\begin{equation}
\label{eqn:reldrinfeld}
\langle a_1,b_1 \rangle a_2 \cdot b_2  \ = \ b_1 \cdot a_1 \langle a_2,b_2 \rangle \qquad \qquad \forall \ a \in A^-, \ b\in A^+
\end{equation}
If we consider the antipode, then the above relation is equivalent to:
\begin{equation}
\label{eqn:reldrinfeld2}
a \cdot b \ = \ \langle S a_1,b_1 \rangle \ b_2 \cdot a_2 \ \langle a_3, b_3 \rangle \qquad \qquad \forall \ a \in A^-, \ b\in A^+
\end{equation}
Formula \eqref{eqn:reldrinfeld2} prescribes how to normally order $A^+$ and $A^-$.


\tab 
Our basic example of a Drinfeld double is the quantum toroidal algebra $\UU$, which we now define. Let us note that our definition has one of the central elements of the quantum toroidal algebra set equal to 1. \footnote{See for example \citep{FT2} for the definition of the quantum toroidal algebra with this extra central element, as well as a check of the fact that $\UU$ is a Drinfeld double} Moreover, our $q$ is equal to the usual $-q$ in the theory, so as to match our conventions from geometry. Let us fix a natural number $n > 1$ and define the following symmetric bilinear forms on elements $\bk = (k_1,...,k_n)$ of either $\nn$, $\zz$ or $\qq$. These are the \b{scalar product}:
\begin{equation}
\label{eqn:scalar}
\bk \cdot \bl = \sum_{i=1}^n k_il_i
\end{equation}
and the \textbf{Killing form}:
\begin{equation}
\label{eqn:kill}
(\bk,\bl) = \sum_{i=1}^n 2k_il_i - k_il_{i+1} - k_i l_{i-1}
\end{equation}
The terminology is supported by the fact that \eqref{eqn:kill} is the Killing form of the root system for the cyclic quiver of Figure \ref{fig:quiver}. Explicitly, the \textbf{positive roots} in our setup are defined for all integers $i < j$ as:
\begin{equation}
\label{eqn:root}
[i;j) = \bs^i + ... + \bs^{j-1} \in \nn
\end{equation}
where $\bs^i \in \nn$ is the \textbf{simple root} $(0,...,0,1,0,...,0)$ with a single 1 at the $i-$th position. Note that the kernel of \eqref{eqn:kill} is spanned by the \b{imaginary root}:
\begin{equation}
\label{eqn:imaginary}
\bth = (1,...,1) \ \in \ \nn
\end{equation}
Many variables that will appear in this thesis will be assigned a certain \b{color} $\in \BZ/n\BZ$. Then we will often encounter the following color-dependent rational function:
\begin{equation}
\label{eqn:defzeta}
\zeta\left( \frac {x_i}{x_j} \right) \ = \ \frac {\left[\frac {x_j}{qtx_i} \right]^{\delta_{j-1}^i} \left[\frac {tx_j}{qx_i} \right]^{\delta_{j+1}^i} }{\left[\frac {x_j}{x_i} \right]^{\delta_j^i} \left[\frac {x_j}{q^2x_i} \right]^{\delta_j^i}}
\end{equation}
where $x_i$ and $x_j$ are variables of color $i$ and $j$, respectively, and the quantum numbers are $[x] = x^{\frac 12} - x^{-\frac 12}$. Since the colors $i,j$ in the above formula are only defined modulo $n$, so are the Kronecker delta functions $\delta_j^i$. With this in mind, we can define the quantum toroidal algebra $\UU$ (following, for example, \citep{FJMM}\footnote{The parameters of \loccit are connected with ours via:
$$
q_1^{\text{their}} = -\frac 1{qt}, \quad q_2^{\text{their}} = q^2, \quad q_3^{\text{their}} = -\frac tq, \quad d^{\text{their}} = t^{-1}
$$
Moreover, our algebra has one less central element than that of \loccit}) as:

$$
\UU \ = \ \Big \langle \{e^\pm_{i,d}\}_{1\leq i \leq n}^{d\in \BZ}, \ \{ \ph^\pm_{i,d} \}_{1\leq i \leq n}^{d\in \BN_0} \Big \rangle
$$
The relations between the above generators are best expressed if we package the generators into currents: 
$$
e^\pm_{i}(z) = \sum_{d \in \BZ} e^\pm_{i,d} z^{-d} \qquad \qquad \ph^\pm_{i}(z) = \sum_{d=0}^\infty \ph^\pm_{i,d} z^{\mp d}
$$
and require that the $\ph^\pm_{i,d}$ commute among themselves, as well as:
$$
e_i^\pm(z) \ph_j^{\pm'}(w) \cdot \zeta \left(\frac {w^{\pm 1}}{z^{\pm 1}} \right) \ = \ \ph_j^{\pm'}(w) e^\pm_i(z) \cdot \zeta \left(\frac {z^{\pm 1}}{w^{\pm 1}} \right)
$$
$$
e_i^\pm(z) e_j^\pm(w) \cdot \zeta \left(\frac {w^{\pm 1}}{z^{\pm 1}} \right) \ \ = \ \ e_j^\pm(w) e^\pm_i(z)  \cdot \zeta \left(\frac {z^{\pm 1}}{w^{\pm 1}} \right)
$$
and:
\begin{equation}
\label{eqn:comm}
[e_i^+(z), e_{j}^-(w)] = \delta_i^j \delta \left(\frac zw \right) \cdot \frac {\ph^+_i(z) - \ph_{i}^-(w)}{q-q^{-1}} 
\end{equation}
for all signs $\pm,\pm'$ and all $i,j \in \{1,...,n\}$. In the above, $z$ and $w$ are variables of color $i$ and $j$, respectively. We also impose the Serre relation: 
\begin{equation}
\label{eqn:serre}
e_{i}^\pm (z_1) e_{i}^\pm(z_2) e^\pm_{i \pm' 1}(w)  + \left( q+q^{-1} \right) e^\pm_{i}(z_1) e^\pm_{i \pm' 1}(w) e^\pm_{i}(z_2) + e^\pm_{i \pm' 1}(w) e^\pm_{i}(z_1) e^\pm_{i}(z_2) + 
\end{equation}
$$
+ e_{i}^\pm (z_2) e_{i}^\pm(z_1) e^\pm_{i \pm' 1}(w) + \left( q+q^{-1} \right) e^\pm_{i}(z_2) e^\pm_{i \pm' 1}(w) e^\pm_{i}(z_1) + e^\pm_{i \pm' 1}(w) e^\pm_{i}(z_2) e^\pm_{i}(z_1)  = 0
$$
for all signs $\pm,\pm'$ and all $i\in \{1,...,n\}$. The above slightly differs from the usual Serre relation because our parameter $q$ equals the usual $-q$. We impose the extra relation:
$$
\ph_{i,0}^+ \ = \ \left(\ph_{i,0}\right)^{-1}
$$
and denote the above Cartan element by $\ph_i$. As we mentioned, the usual quantum toroidal algebra is defined as a central extension of $\UU$, where the central extension governs the failure of $\ph_{i,d}^+$ and $\ph_{i,d}^-$ to commute. Introducing this extra extension would complicate our computations significantly, and would not shed any further light on our geometric constructions. We therefore choose to ignore it in this thesis.

\tab
Note that the quantum group $\su$ is the subalgebra of $\UU$ generated by the constant terms of the currents: $e_i^\pm := e_{i,0}^\pm$ and the Cartan generators $\ph_{i}$. These satisfy the following relations for all signs $\pm,\pm'$ and all indices $i,j \in \{1,...,n\}$:
$$
\ph_i \ph_j = \ph_j \ph_i \quad \text{and} \quad \ph_j e_i^{\pm}  = (-q)^{\mp (\bs_i,\bs_j)} e_i^{\pm} \ph_j 
$$
$$
[e^\pm_i, e^\pm_j] \ = \ 0 \quad \text{unless} \ j\equiv i\pm 1
$$
$$
[e_i^+, e_{j}^-] \ = \ \delta_i^j \cdot \frac {\ph_i - \ph_{i}^{-1}}{q-q^{-1}} 
$$
$$
e^{\pm}_{i} e^{\pm}_{i} e^\pm_{i \pm' 1} + \left( q+q^{-1} \right) e^\pm_{i} e^\pm_{i \pm' 1} e^\pm_{i} + e^\pm_{i \pm' 1} e^\pm_{i} e^\pm_{i} = 0
$$
We will consider the following subalgebras of $\UU$:
$$
\UUp \ = \ \left \langle e^+_{i,d} \right \rangle_{1\leq i \leq n}^{d\in \BZ} \qquad \qquad \UUm \ = \ \left \langle e^-_{i,d} \right \rangle_{1\leq i \leq n}^{d\in \BZ}
$$
$$
\UUg \ = \ \left \langle e^+_{i,d}, \ph^+_{i,d'} \right \rangle_{1\leq i \leq n}^{d\in \BZ, d' \in \BN_0} \qquad  \qquad \UUl \ = \ \left \langle e^-_{i,d}, \ph^-_{i,d'} \right \rangle_{1\leq i \leq n}^{d\in \BZ, d'\in \BN_0}
$$
The latter two subalgebras are actually bialgebras with respect to the coproduct: 
\begin{equation}
\label{eqn:copquant}
\Delta \ : \ \UU \ \longrightarrow \ \UU \ \widehat{\otimes} \ \UU
\end{equation}
$$
\Delta \left(e_{i}^+(z) \right) = \ph^+_i(z) \otimes e_{i}^+(z) + e_{i}^+(z) \otimes 1 \qquad \qquad \Delta\left(\ph^+_i(z) \right) = \ph^+_i(z) \otimes \ph^+_i(z)
$$
$$
\Delta \left(e_{i}^-(z) \right) = 1 \otimes e_{i}^-(z) + e_{i}^-(z) \otimes \ph^-_i(z) \qquad \quad \ \Delta\left(\ph^-_i(z)\right) = \ph^-_i(z) \otimes \ph^-_i(z)
$$
for all $i\in \{1,...,n\}$. Moreover, there exists a bialgebra pairing:
\begin{equation}
\label{eqn:pairingtor}
\UUl \ \otimes \ \UUg \longrightarrow \BQ(q,t)
\end{equation}
completely determined by:
$$
\langle \ph^-_{i}(z), \ph^+_j(w) \rangle = \frac {\zeta(w/z)}{\zeta(z/w)} \qquad \qquad \langle e_{i,d}^-, e_{j,d'}^+ \rangle = \frac {\delta_i^j \delta_{d+d'}^0}{q^{-1}-q} 
$$
and the properties \eqref{eqn:bialg}. In the first formula above, we think of the variables $z$ and $w$ as having colors $i$ and $j$ respectively. We leave it as an exercise to the interested reader to show that $\UU$ is the Drinfeld double of its positive and negative halves, as in \eqref{eqn:dd}. Finally, let us note that we have isomorphisms of algebras:
\begin{equation}
\label{eqn:isos}
\UU \cong \UU^{\op} \qquad \qquad \UU \cong \UU|_{t \rightarrow \frac 1t}
\end{equation}
under the maps $e_{i,d}^\pm \rightarrow e_{i,d}^\mp$, $\ph_{i,d}^\pm \rightarrow \ph_{i,d}^\pm$ and $e_{i,d}^\pm \rightarrow e_{i,-d}^\mp$, $\ph_{i,d}^\pm \rightarrow \ph_{i,d}^\mp$, respectively. The second map of \eqref{eqn:isos} is also an anti-isomorphism of bialgebras. \\

\section{Basic notions on symplectic varieties and GIT}
\label{sec:symplectic}

\noindent A smooth algebraic variety $X$ is called \b{symplectic} if it comes endowed with a closed non-denegerate $2-$form:
$$
\omega \ \in \ \Gamma(X,T^*X \wedge T^*X) \qquad \Leftrightarrow \qquad \omega \ : \ TX \wedge TX \longrightarrow \CO_X
$$
called the symplectic form. A smooth subvariety $L \subset X$ is called \b{Lagrangian} if it is middle-dimensional and $\omega|_{TL} = 0$. 

\tab The non-degeneracy of the symplectic form is the statement that the induced map $\overline{\omega}:T^*X \longrightarrow TX$ is an isomorphism, so we could alternatively interpret $\overline{\omega}$ as a bilinear form $\widetilde{\omega}$ on $T^*X$. This gives rise to a \b{Poisson structure} on $X$, namely a Lie bracket:
\begin{equation}
\label{eqn:poisson}
\{\cdot, \cdot \}: \CO_X \wedge \CO_X \longrightarrow \CO_X, \qquad \qquad \{f,g\} = \widetilde{\omega}(df,dg)
\end{equation}
which satisfies the Leibniz rule $\{fg,h\} = f \{g,h\} + g \{f,h\}$. 

\tab
We will consider group actions $G \curvearrowright X$ which preserve the symplectic form, i.e.:
$$
g^*\omega \ = \ \omega, \qquad \qquad \forall g\in G
$$
The above is equivalent to the fact that the Lie derivative of $\omega$ is 0 in the direction of any $\xi \in \text{Im} \left(\fg \mapsto \text{Vect}(X) \right)$. Another equivalent condition is that the 1-form:
$$
\omega(\xi, \cdot) \ \in \ \Gamma(X,T^*X), \qquad \qquad \forall \ \xi \in \text{Im} \left(\fg \mapsto \text{Vect}(X) \right)
$$ 
is closed. We will assume a stronger condition, namely that the above 1-form is exact for all $\xi$. More specifically, we assume that the group action is \b{Hamiltonian}, which means that for any $\xi \in \fg$ there exists a function:
\begin{equation}
\label{eqn:hamiltonianfunction}
H_\xi \ \in \ \Gamma(X,\CO_X)
\end{equation}
such that $\omega(\xi,\cdot) = d H_\xi$, and the assignment $\xi \mapsto H_\xi$ is a Lie algebra homomorphism with respect to the Poisson bracket \eqref{eqn:poisson}. If we are in this situation, we can define a \b{moment map} by setting:
\begin{equation}
\label{eqn:defmoment}
\mu \ : \ X \longrightarrow \fg^\vee \qquad \qquad \mu(x)(\xi) \ = \ H_\xi(x), \qquad \forall \ \xi \in \fg 
\end{equation}

\begin{example}
\label{ex:cotangent}
	
The basic example of symplectic varieties are cotangent bundles to smooth varieties, namely $X = T^*M$, since we can present their tangent bundle as:
$$
TX \ \cong \ TM \oplus T^*M
$$
and define the natural symplectic form:
$$
\omega(\xi \wedge \xi') = \omega(\lambda \wedge \lambda') = 0 \qquad \qquad \omega(\xi \wedge \lambda) = \lambda(\xi)
$$
for any $\xi,\xi' \in TM$ and $\lambda,\lambda' \in T^*M$. Any action $G \curvearrowright M$ extends to the cotangent bundle, and it is easy to see that the resulting action $G \curvearrowright T^*M$ is symplectic. In fact, it is naturally Hamiltonian, with:
$$
\mu : T^*M \longrightarrow \fg^\vee \qquad \text{defined as the dual of} \qquad \mu^\vee: \fg \longrightarrow TM
$$
where the latter is the infinitesimal action of $\fg$ on $M$ that is induced by the $G-$action.
	
\end{example}

\tab
Let us return to the general setup of \eqref{eqn:defmoment} and note that the moment map $\mu$ is $G-$equivariant, with respect to the coadjoint action of $G$ on $\fg^\vee$. Then we see that:
$$
G \ \curvearrowright \ \mu^{-1}(0)
$$
Moreover, it is a simple exercise to show that the symplectic form $\omega$ descends to:
\begin{equation}
\label{eqn:hamiltonianreduction}
Y \ = \ \mu^{-1}(0) / G
\end{equation}
if the quotient is smooth and $G$ is reductive. The reason why \eqref{eqn:hamiltonianreduction} is symplectic is that the normal directions to $\mu^{-1}(0)$ (in other words, tangent directions that lie in the kernel of $d\mu$) are precisely dual to the tangent directions to $G-$orbits. The symplectic variety $Y$ is called the \b{Hamiltonian reduction} of $X$ with respect to $G$.

\tab 
However, in algebraic geometry we will be faced with many ways of taking the quotient \eqref{eqn:hamiltonianreduction}, and not all of them will be algebraic varieties, much less smooth. We will now recall \b{geometric invariant theory} (GIT, see \citep{M}), which will allow us to define quotients with good properties. Let us assume we have a reductive group action on a variety $G \curvearrowright X$, and we wish to make sense of the quotient space $X/G$. The categorically-minded reader might first think about the \b{quotient stack}:
$$
X \ \longrightarrow \ [X/G]
$$
which parametrizes $G-$bundles with a $G-$equivariant map to $X$. The category of coherent sheaves on $[X/G]$ is defined as the category of $G-$equivariant coherent sheaves on $X$, and the same can be said about their $K-$theory groups:
\begin{equation}
\label{eqn:kstack}
K\left([X/G]\right) \ := \ K_G(X)
\end{equation}
However, the quotient stack is rarely a variety, so we will need to look for other quotients. An easy solution is the \b{affine quotient}:
\begin{equation}
\label{eqn:affinequotient}
X \ \longrightarrow \ X / G \  := \ \text{Spec} \left(\Gamma(X,\CO_X)^G \right)
\end{equation}
In other words, functions on the affine quotient correspond to $G-$invariant functions on the prequotient. This seems reasonable at first glance, but one runs into problems when trying to glue these affine varieties. Another issue is that points of the affine quotient do not parametrize $G-$orbits in $X$, as one can see from one of the simplest possible examples of reductive group actions:
$$
\BC^* \ \curvearrowright \ \BC^2, \qquad t\cdot(x_1,x_2) = (tx_1,tx_2)
$$ 
It is easy to see that:
$$
\BC^2 / \BC^* = \text{Spec} \left(\BC[x,y]^{\BC^*} \right) = \text{Spec} \left(\BC \right) = \pt
$$
so the affine quotient does not convey any information on the various $\BC^*-$orbits on the plane. Indeed, the problem quickly appears to be the point $0 \in \BC^2$, which lies in the closure of any orbit. Therefore, the solution appears to be to remove the point 0, and work instead with:
\begin{equation}
\label{eqn:projectivespace}
(\BC^2 \backslash 0) \ \stackrel{\pi}\longrightarrow \ (\BC^2 \backslash 0) \sslash \BC^* \ =: \ \BP^1
\end{equation}
where the above quotient is now \b{geometric}: fibers of $\pi$ consist of entire $\BC^*$ orbits. Geometric invariant theory seeks to generalize the setup \eqref{eqn:projectivespace} to arbitrary actions $G \curvearrowright X$ of a reductive group on a projective over affine variety $X$. The construction depends on a \b{linearization} of this action, i.e. a lift of the $G-$action to an ample line bundle $L$ on $X$. Then we define:
$$
X = X^{\un} \sqcup X^{\ss}
$$
where $X^\un$ is the closed subset of \b{unstable} points $x$, whose defining property is that:
$$
\overline{G\cdot x_*} \cap \left(\text{zero section of }L \right) \neq \emptyset
$$
for some non-zero lift $x_*$  of $x$ in the total space of $L$. Points of the open complement $X^{\ss}$ are called \b{semistable}, and GIT claims the existence of a quotient:
$$
X^{\ss} \ \stackrel{\pi}\longrightarrow \ X \sslash_L G
$$
which is good, in the sense that locally on $X \sslash_L G$, the ring of functions coincides with the ring of $G-$invariant functions on $X^\ss$.  Furthermore, let:
$$
X^\s \ \subset \ X^\ss
$$
denote the subset of \b{stable} points, i.e. points with finite stabilizer whose $G-$orbits are closed in $X^\ss$. Then the restriction of the map $\pi$ to the open set $X^\s$ is a geometric quotient, meaning that points on the variety $X^\s \sslash_L G$ parametrize $G-$orbits on $X$. 

\tab 
The construction of unstable, semistable and stable points is geometric, but it can be presented algebraically in quite simple terms. Our choice of linearization means that we may construct the graded ring:
$$
R_L \ = \ \bigoplus_{k=0}^\infty \Gamma\left(X,L^{\otimes k}\right)^G
$$
Then we have:
$$
X \sslash_L G \ = \ \text{Proj}(R_L)
$$
which implicitly uses the following well-known description of semistable points: \\

\begin{lemma}
\label{lem:ss}
	
A point $x \in X$ is semistable for $G \curvearrowright L$ if and only if there exists a $G-$invariant section of some power of $L$ which does not vanish at $x$. 
	
\end{lemma}

\tab See \citep{P} for a proof of the above Lemma, as well as an overview of GIT in line with our approach in this thesis. Since $\Gamma(X,\CO_X)^G$ is the degree zero piece of $R_L$, we always have a proper map:
\begin{equation}
\label{eqn:goody}
X \sslash_L G \ \longrightarrow \ X / G
\end{equation}
which is an isomorphism for the trivial linearization $G \curvearrowright L$. Examples of quotients obtained by the above procedure include not only projective spaces such as \eqref{eqn:projectivespace}, but also toric varieties and moduli spaces of curves. In the next Chapter, we will apply this machinery to give rise to moduli spaces of quiver representations.

\chapter{Nakajima Quiver Varieties}
\label{chap:nakajima}



\section{The moduli space of framed sheaves}
\label{sec:gieseker}

\noindent Let us start by defining Nakajima quiver varieties in one of the most basic cases, namely the framed Jordan quiver:
\begin{figure}[H]

\begin{picture}(200,110)(-10,-30)

\put(200,31){\circle*{10}}
\put(200,55){\circle{60}}
\put(194,36){\vector(4,-1){5}}
\put(200,-10){\vector(0,1){36}}
\put(195,-20){\line(1,0){10}}
\put(195,-10){\line(1,0){10}}
\put(195,-20){\line(0,1){10}}
\put(205,-20){\line(0,1){10}}

\put(208,-20){$W$}
\put(208,25){$V$}
\put(188,6){$A$}
\put(169,40){$X$}

\end{picture} 

\caption[The Jordan quiver and its moduli of framed representations]

\label{fig:jordan}

\end{figure}

\noindent The black circle denotes a ``quiver vertex", while the white square denotes a ``framing vertex", and the two will play different roles in the following construction. We fix vector spaces $V,W$ of dimensions $v,w \in \BN$, and we consider pairs of linear maps corresponding to the arrows in the above quiver:
\begin{equation}
\label{eqn:vectorspace}
(X,A) \ \in \ \Hom(V,V) \oplus \Hom(W,V)
\end{equation}
We consider the action of $G_v := GL(V)$ on \eqref{eqn:vectorspace}, via $g\cdot (X,A) = (gXg^{-1},gA)$. This represents the main difference between quiver and framing vertices, in that we only consider the general linear group action that corresponds to the former, not the latter. The cotangent bundle to the affine space \eqref{eqn:vectorspace} is itself an affine space:
$$
N_{v,w} \ = \ \Hom(V,V) \oplus \Hom(V,V) \oplus \Hom(W,V) \oplus \Hom(V,W)
$$
We will denote elements of this vector space as quadruples $(X,Y,A,B)$. The $G_v$ action extends to the cotangent bundle as:
$$
g\cdot (X,Y,A,B) \ = \ \left(gXg^{-1}, gYg^{-1},gA,Bg^{-1} \right)
$$
and the moment map \eqref{eqn:defmoment} can be written explicitly as:
$$
\mu \ : \ N_{v,w} \ \longrightarrow \ \fg_v^\vee \ = \ \Hom(V,V)
$$
$$
\mu \left(X,Y,A,B\right) \ \ = \ \  [X,Y]+AB
$$
$$$$
\begin{definition}

The Nakajima quiver variety is the Hamiltonian reduction:
\begin{equation}
\label{eqn:nak0}
\CN^\theta_{v,w} = \mu^{-1}(0) \sslash_{\det^{-\theta}} G_v
\end{equation}
for any $\theta \in \BZ$. 

\end{definition}

\tab 
Indeed,  since $\mu^{-1}(0)$ is an affine variety, the linearization of the $G_v$ action will be topologically trivial. The notation $\sslash_{\det^{-\theta}}$ means that we take the trivial line bundle with $G_v-$action given by the power $(-\theta)$ of the determinant character. The following Exercise is a well-known characterization of the set of semistable points, which must be taken into account when defining the quotient $\CN_{v,w}^\th$ via Section \ref{sec:symplectic}. \\

\begin{exercise}
\label{ex:unstable} 
	
A quadruple $(X,Y,A,B)$ is semistable:
$$
\text{when }\theta > 0, \qquad \text{iff } \ \not \exists \ V' \subsetneq V \quad \text{such that} \quad X,Y : V' \rightarrow V' \quad \text{and} \quad A:W \rightarrow V'
$$
$$
\text{when }\theta < 0, \qquad \text{iff } \ \not \exists \ V \stackrel{\neq}\twoheadrightarrow V' \quad \text{such that} \quad X,Y : V' \rightarrow V' \quad \text{and} \quad B:V' \rightarrow W
$$
Moreover, the action of $G_v$ on the semistable locus is free.
\end{exercise}


\tab
Since $G_v$ acts freely on the semistable locus, we conclude that all semistable points are stable. Thus \eqref{eqn:nak0} is a geometric quotient, and according to \eqref{eqn:goody} we have a proper resolution of singularities:
$$
\rho \ : \ \CN_{v,w}^1 \longrightarrow \CN_{v,w}^0
$$
An important and well-known result in mathematical physics, the ADHM construction (\citep{ADHM} and \citep{Nakajima}), states that: 
$$
\CN_{v,w}^0 \ \cong \ \text{Uhlenbeck compactification of the moduli space of framed instantons}
$$
With this in mind, it should come as no surprise that $\CN^1_{v,w}$ is isomorphic to the moduli space of framed, degree $v$ and rank $w$ torsion-free sheaves on $\BP^2$:
\begin{equation}
\label{eqn:sheaf}
\CN^1_{v,w} \ \cong \ \{\CF \text{ torsion-free sheaf on } \BP^2 \text{ s. t. } c_2(\CF) = v, \ \CF|_\infty \cong \CO_\infty^{\oplus w}\}
\end{equation}
where $\infty \subset \BP^2$ denotes a fixed line. For a detailed construction of the above isomorphisms, see for example \citep{Nak}. When $w=1$, the framing determines an embedding of the sheaf $\CF$ into $\CO$, so it simply becomes an ideal sheaf of finite colength: the moduli space for $w=1$ is thus simply the Hilbert scheme of points of $\BC^2$. \\

\begin{remark}
\label{rem:quiver}
In general, a quiver is an oriented graph with vertex set denoted by $I$. Then the Nakajima quiver variety was defined in \citep{Nak0} as the Hamiltonian reduction of the cotangent bundle to the vector space:
\begin{equation}
\label{eqn:vectorspace2}
\bigoplus_{e = \oij} \Hom(V_i,V_j) \bigoplus_{i\in I} \Hom(W_i,V_i)
\end{equation}
where $\{V_i\}_{i\in I}$, $\{W_i\}_{i\in I}$ are vector spaces of dimensions $\bv = \{v_i\}_{i\in I}$, $\bw = \{w_i\}_{i\in I} \in \BN^I$. We write $G_\bv = \prod_{i\in I} GL(V_i)$ and let it act on the vector space \eqref{eqn:vectorspace2} by conjugation. The Nakajima quiver variety with stability condition $\bth = (\th_i)_{i\in I} \in \BZ^I$ is denoted by:
$$
\CN_{\bv,\bw}^\bth = \mu^{-1} (0) \sslash_{\det^{-\bth}} G_\bv
$$
where $\det^{-\bth}:G_\bv \rightarrow \BC^*$ is the character $(g_i)_{i\in I} \mapsto \prod_{i\in I} \det(g_i)^{-\th_i}$. For the Jordan quiver in Figure \ref{fig:jordan}, it is clear that we obtain precisely the varieties \eqref{eqn:nak0}.
\end{remark}


\tab
Let us return to the Jordan quiver. Since we will henceforth only work with $\th = 1$, we will denote Nakajima quiver varieties by $\CN_{v,w}$. The torus $T = T_w = \BC^* \times \BC^* \times (\BC^*)^w$ acts on the variety $\CN_{v,w}$: the first two factors act by scaling the plane $\BP^2$ in such a way that keeps the line $\infty$ invariant, and the last $w$ factors act on the trivialization at $\infty$. In terms of quadruples of linear maps, the action is explicitly given by:
\begin{equation}
\label{eqn:act}
(q,t,u_1,...,u_w) \cdot (X,Y,A,B) = (qtX,qt^{-1}Y,qAU, qU^{-1}B)
\end{equation}
where $U = \text{diag}(u_1,...,u_w)$. We will consider the $T-$equivariant $K-$theory groups of moduli spaces of framed sheaves. As in early work in cohomology by Nakajima and Grojnowski for Hilbert schemes, it makes sense to package these groups as:
$$
K(w) \ = \ \bigoplus_{v\in \BN} K_T(\CN_{v,w})
$$
which is a module over the ring: 
$$
\BF_w \ := \ K_T(\text{pt}) \ = \ \BZ[q^{\pm 1}, t^{\pm 1},u_1^{\pm 1},...,u_w^{\pm 1}]
$$
When $w=1$ and $u_1 = 1$, a well-known construction of Bridgeland-King-Reid implies:
\begin{equation}
\label{eqn:haiman}
K(1) \ \cong \ \text{Fock space} \ = \ \BZ[q^{\pm 1}, t^{\pm 1}][x_1,x_2,...]^{\sym}
\end{equation}
The work of \citep{Ha} establishes that the above correspondence sends:
\begin{equation}
\label{eqn:haimanesque}
K_T(\CN_{v,1}) \ \ni \ I_\lambda \ \leftrightarrow \ P^{q,t}_\lambda(x_1,x_2,...)
\end{equation}
for any partition $\lambda \vdash v$. By a slight abuse, the notation $I_\lambda$ refers to the skyscraper sheaf at the torus fixed point denoted by the same letter:
$$
I_\lambda \ = \ (x^{\lambda_1}, x^{\lambda_2}y,x^{\lambda_3}y^2,...) \ \subset \ \BC[x,y]
$$ 
while in the right hand side of \eqref{eqn:haimanesque} we have the well-known Macdonald polynomial $P_\lambda$ depending on the parameters \footnote{To be precise, the parameters that usually appear in the definition of Macdonald polynomials would be equal to $qt$ and $qt^{-1}$ in our notation. Statement \eqref{eqn:haimanesque} requires Macdonald polynomials to be modified as in \citep{GH}} $q$ and $t$. For general $w$, the constructions in Chapter \ref{chap:stable} imply that:
\begin{equation}
\label{eqn:upgrade}
K(w) \ \cong \ K(1)^{\otimes w} \ = \ \text{Fock space}^{\otimes w}
\end{equation}
These constructions are due to \citep{MO2}, who produce as many geometric isomorphisms \eqref{eqn:upgrade} as there are coproduct structures, and one has such a structure for every rational number $m \in \text{Pic}(\CN_{v,w}) \otimes \BQ = \BQ$. But as mere vector spaces, the isomorphism \eqref{eqn:upgrade} is easy to see. For one thing, one could construct it by observing that fixed points of $\CN_{v,w}$ are indexed by $w-$partitions as in \eqref{eqn:wpartition}: 
$$
\bla \ = \ (\lambda^1,...,\lambda^w)
$$
and are given by the direct sum of the $w$ monomial ideals:
\begin{equation}
\label{eqn:fixedpoint}
I_\bla \ = \ I_{\lambda^1} \oplus ... \oplus I_{\lambda^w} \ \in \ \text{Coh}_T(\CN_{v,w})
\end{equation}
As we have seen in Section \ref{sec:basick}, fixed points are important because they allow us to express $K-$theory classes by equivariant localization \eqref{eqn:loc}. A very important feature of localization formulas is the presence of the torus characters in the tangent spaces at the fixed points, so we will now compute these. \\ 

\begin{exercise}
\label{ex:tangent}
	
As $T-$characters, the tangent spaces to the fixed points of $\CN_{v,w}$ are:
\begin{equation}
\label{eqn:tangies0}
T_{\bla} \CN_{v,w} \ = \ \sum_{i=1}^w \sum_{\square \in \bla} \left( \frac {\chi_\square}{q u_i} + \frac {u_i}{q\chi_\square} \right) + \left(\frac 1{qt} + \frac tq - 1 - \frac 1{q^2} \right) \sum_{\sq, \square'\in \bla} \frac {\chi_{\square'}}{\chi_\square} 
\end{equation}
where $\chi_\sq$ denotes the weight of a box in a $w-$partition, as in \eqref{eqn:weight}. 
	
\end{exercise}


\tab In fact, formula \eqref{eqn:tangies0} can also be deduced from the fact that $\CN_{v,w}$ is a moduli space of sheaves, since the Kodaira-Spencer isomorphism implies that:
\begin{equation}
\label{eqn:ks0}
T_\CF \CN_{v,w} \ \cong \ \Ext^1(\CF,\CF(-\infty)) \ \cong \ - \chi(\CF,\CF(-\infty))
\end{equation}
The second equality holds because the corresponding $\Hom$ and $\Ext^2$ groups vanish. The former group vanishes because of the twist by $\CO(-\infty)$, while the latter vanishes because of Serre duality, which claims that $\Ext^2(\CF, \CF(-\infty)) = \Hom(\CF, \CF(-2\infty))^\vee$. \\



\section{Nakajima varieties for the cyclic quiver}
\label{sec:quiver}


\noindent Let us now fix a natural number $n>1$ and consider the finite group $H = \BZ/n\BZ$. Consider the action $H \curvearrowright \BP^2$ by:
\begin{equation}
\label{eqn:satin}
\xi \cdot [x:y:z] = [\xi^{-1} x: \xi y : z], \qquad \forall \ \xi^n=1
\end{equation}
where the coordinates are chosen such that $\infty = \{[x:y:0]\}$. Fix a homomorphism:
\begin{equation}
\label{eqn:sigma}
\sigma \ : \ H \longrightarrow (\BC^*)^w, \qquad \sigma(\xi) = (\xi^{-\sigma_1},...,\xi^{-\sigma_w}) 
\end{equation}
which induces a decomposition:
\begin{equation}
\label{eqn:decomposition}
W \ = \ W_1 \oplus ... \oplus W_n
\end{equation}
where $W_i$ is spanned by those basis vectors $\omega_j \in W$ such that $\sigma_j \equiv i$ modulo $n$. We will write $w_i = \dim W_i$ and record these numbers in the vector $\bw = (w_1,...,w_n) \in \nn$. Relations \eqref{eqn:satin} and \eqref{eqn:sigma} give rise to an action of $H$ on the moduli of sheaves $\CN_{v,w}$, which preserves the symplectic structure. Therefore, the fixed point locus $\CN^H_{v,w}$ is also symplectic, and we will now describe it. In terms of quadruples, the $H-$action sends:
$$
\xi \cdot (X,Y,A,B) \ = \ (\xi X, \xi^{-1} Y, A \sigma(\xi)^{-1}, \sigma(\xi) B) 
$$
Such a quadruple is $H-$fixed if and only if there exists $g \in G_v$ such that:

\begin{equation}
\label{eqn:property}
(\xi X, \xi^{-1} Y, A \sigma^{-1}(\xi), \sigma(\xi) B) = (gXg^{-1}, gYg^{-1}, gA, Bg^{-1})
\end{equation}
where $\xi = e^{\frac {2\pi i}n}$. Let us consider the decomposition of the space $V = \bigoplus_c V(c)$ into the generalized eigenspaces of $g$. Property \eqref{eqn:property} implies that:
$$
V(c) \ \mathop{\rightleftarrows}^X_Y \ V(\xi c) \qquad \qquad W_i \ \mathop{\rightleftarrows}^A_B \ V(\xi^{i})
$$
for all eigenvalues $c$ and all residues $i$ modulo $n$. The semistability property of Exercise \ref{ex:unstable} forces $V$ to be generated by the image of $A$ acted on by $X$ and $Y$, so this means that the only non-zero eigenspaces are $V_i := V(\xi^i)$ for $i \in \{1,...,n\}$. This gives us a decomposition of the vector space as:
$$
V = V_1 \oplus ... \oplus V_{n}
$$
and we will write $v_i = \dim V_i$ and $\bv = (v_1,...,v_n) \in \nn$. To summarize the above, an $H-$fixed point is given by a cycle of maps that matches the quiver of Figure \ref{fig:quiver}:
\begin{equation}
\label{eqn:cyclic}
\xymatrix{
& & W_1 \ar@/^/[d]^{A_1} & & \\
& & V_1 \ar@/^/[u]^{B_1} \ar@/^/[rd]^{X_1} \ar@/^/[ld]^{Y_n} & &  \\
W_n \ar@/^/[r]^{A_n} & V_n \ar@/^/[l]^{B_n} \ar@/^/[ru]^{X_n} \ar@/^/[dd]^{Y_{n-1}} & & V_2 \ar@/^/[r]^{B_2} \ar@/^/[dd]^{X_2} \ar@/^/[lu]^{Y_1} & W_2 \ar@/^/[l]^{A_2} \\
& & & & \\
W_{n-1} \ar@/^/[r]^{A_{n-1}} & V_{n-1} \ar@/^/[l]^{B_{n-1}} \ar@/^/[uu]^{X_{n-1}} \ar@/^/[rd] &  & V_3 \ar@/^/[r]^{B_3} \ar@/^/[uu]^{Y_2} \ar@/^/[ld] & W_3 \ar@/^/[l]^{A_3} \\
& & V_{...} \ar@/^/[ru] \ar@/^/[lu] \ar@/^/[d]  & & \\
& & W_{...} \ar@/^/[u] & &} \qquad
\end{equation}
Rigorously, a collection of maps as above is an element of the vector space:
\begin{equation}
\label{eqn:quad}
N_{\bv,\bw} = \bigoplus_{i=1}^n \Hom(V_i,V_{i+1}) \bigoplus_{i=1}^n \Hom(V_{i+1},V_{i}) \bigoplus_{i=1}^n \Hom(W_i,V_{i})\bigoplus_{i=1}^n \Hom(V_{i},W_{i})
\end{equation}
Such quadruples are required to lie in the kernel of the moment map:
\begin{equation}
\label{eqn:moment}
\mu:N_{\bv,\bw} \ \longrightarrow \ \fg_\bv^\vee \ = \ \bigoplus_{i=1}^n \Hom(V_i,V_i) \qquad 
\end{equation}
$$
\mu \Big( (X_i,Y_i,A_i,B_i)_{1\leq i \leq n} \Big) \ = \ \bigoplus_{i=1}^n \Big(X_{i-1}Y_{i-1} - Y_{i}X_{i}+A_iB_i)
$$
with the indices taken modulo $n$. According to Remark \ref{rem:quiver}, any such semistable quadruple is a point on the Nakajima variety for the cyclic quiver, hence:
\begin{equation}
\label{eqn:fixedgies}
\CN^H_{v,w} \ = \ \bigsqcup_{|\bv| = v} \CN^\bth_{\bv,\bw} \qquad \qquad \text{where} \qquad \CN^\bth_{\bv,\bw} = \mu^{-1}(0) \sslash_{\det^{-\bth}} G_\bv
\end{equation}
and the vector $\bw \in \nn$ of size $|\bw| = w$ keeps track of the decomposition \eqref{eqn:decomposition}. The stability condition that we will work with throughout this thesis is:
\begin{equation}
\label{eqn:stabquiver}
G_\bv = \prod_{i=1}^n GL(V_i) \ \longrightarrow \ \BC^* \qquad (g_1,...,g_n) \ \longrightarrow (\det g_1)^{-1} ... (\det g_n)^{-1}
\end{equation}
which is associated to the vector $\bth = (1,...,1) \in \nn$. We will henceforth drop $\bth$ from the notation for Nakajima quiver varieties. By analogy with Exercise \ref{ex:unstable}, it is easy to see that a collection of maps $(X_i,Y_i,A_i,B_i)_{1\leq i \leq n}$ is $\bth-$semistable if and only if the vector spaces $V_i$ are generated by the image of the $A_i$ acted on by the $X_i$ and the $Y_i$. In this thesis, we are interested in the $K-$theory group:
$$
K(\bw) \ = \ \bigoplus_{\bv \in \nn} K_{\bv,\bw} \qquad \text{where} \qquad K_{\bv,\bw} := K_T(\CN_{\bv,\bw})
$$
is a module over the ring of constants $K_T(\text{pt}) = \BF_\bw := \BZ[q^{\pm 1}, t^{\pm 1}, u_1^{\pm 1},...,u_{\bw}^{\pm 1}]$. Here we slightly abuse notation, since $\bw$ is not just a number, but a vector of natural numbers which keeps track of the dimensions of the framing vector spaces $\{W_i\}_{i\in \{1,...,n\}}$. If the $j-$th coordinate vector $\omega_j$ lies in $W_i \subset W$, then we think of the equivariant parameter $u_{j}$ as having color $i$ modulo $n$, as in Section \ref{sec:basicpar}. \\

\begin{remark}
For a different choice of stability condition, we would have obtained moduli of sheaves on the minimal resolution of the $A_{n-1}$ singularity. See \citep{Nagao} for a review, and for the explicit isomorphism between the $K-$theories involved. 
\end{remark}

\tab
Since Nakajima varieties for the cyclic quiver are fixed loci of the moduli of framed sheaves $N_{v,w}$, the two spaces have the same collection of torus fixed points. Comparing this with \eqref{eqn:fixedpoint}, we see that fixed points of $\CN_{\bv,\bw}$ are indexed by $\bw-$partitions:
$$
\bla \ = \ (\lambda^1,...,\lambda^{\bw}) 
$$
in the terminology introduced at the end of Section \ref{sec:basicpar}. The above notation means that we now we think of each constituent partition $\lambda^j$ as having a color $\in \{1,...,n\}$ associated to it, precisely the same color which was associated to the equivariant parameter $u_j$. The content $c_\sq$ of any box $\sq \in \bla$ is defined as in \eqref{eqn:content}, and the color of the box is simply $c_\sq$ modulo $n$. The explicit torus fixed collection $(X_i,Y_i,A_i,B_i)_{1\leq i \leq n} \in \CN_{\bv,\bw}$ which is associated to the $\bw-$partition $\bla$ is given by:
\begin{equation}
\label{eqn:fixed}
W_i = \bigoplus^{u_j \equiv i}_{j \in \{1,...,\bw\}} \BC \cdot \omega_j \qquad \qquad V_i = \bigoplus^{c_\square \equiv i}_{\square \in \bla} \BC \cdot \upsilon_{\square}
\end{equation}
where we set:
$$
X_i(\upsilon_\square) = \upsilon_{\square_\rightarrow} \ \qquad \qquad \qquad Y_{i-1}(\upsilon_\square) = \upsilon_{\square_\uparrow}
$$ 
$$
A_i(\omega_j) = \delta_{j\equiv i} \cdot \upsilon_{\text{root of partition }\lambda^j} \qquad \quad B_i = 0 
$$
for all $i\in \{1,...,n\}$, where $\square_\rightarrow$ and $\square_\uparrow$ denote the boxes immediately to the right and above $\square$, respectively. If these boxes do not appear in the constituent partition of $\bla$ that contains $\sq$, then the corresponding matrix coefficients of $X_i, Y_{i-1}$ are defined to be 0. The root of a partition refers to the box in the bottom left corner of Figure \ref{fig:boxes}. 


\tab
The $T-$character in the tangent space to a fixed point $\bla$ can be computed as the $\BZ/n\BZ$ fixed part of \eqref{eqn:tangies0}. In other words, since $\col q = 0$ and $\col t = 1$, then we only keep the monomials in \eqref{eqn:tangies0} which have overall color 0. This has the effect of only keeping monomials which involve $t^{nk}$ for some $k\in \BZ$:
\begin{equation}
\label{eqn:tanmod0}
T_{\bla} \CN_{\bv,\bw} = \sum_{i=1}^{\bw} \sum_{\square \in \bla}^{c_\square \equiv i} \left( \frac {\chi_\square}{qu_i}  +  \frac {u_i}{q \chi_\square} \right) +
\end{equation}
$$
+ \sum_{\sq,\sq' \in \bla} \left(  \delta_{c_{\sq}+1}^{c_{\sq'}} \cdot \frac {\chi_{\square'}}{qt \chi_\square} +  \delta_{c_{\sq}-1}^{c_{\sq'}} \cdot \frac {t\chi_{\square'}}{q\chi_\square} - \delta_{c_{\sq}}^{c_{\sq'}} \cdot \frac {\chi_{\square'}}{\chi_\square} - \delta_{c_{\sq}}^{c_{\sq'}} \cdot \frac {\chi_{\square'}}{q^2\chi_\square} \right)
$$
where we write $\delta_o^{o'}$ for the Kronecker delta function of $o-o'$ modulo $n$. Formula \eqref{eqn:tanmod0} is also the $\BZ/n\BZ$ invariant part of \eqref{eqn:ks0}:
\begin{equation}
\label{eqn:ks1}
T_\CF \CN_{\bv,\bw} \ \cong \ \Ext^1(\CF,\CF(-\infty))^{\BZ/n\BZ} \ \cong \ - \chi(\CF,\CF(-\infty))^{\BZ/n\BZ}
\end{equation}
where in the LHS we represent points of the Nakajima cyclic quiver variety as $\BZ/n\BZ$ fixed sheaves, according to \eqref{eqn:sheaf} and the action defined in \eqref{eqn:satin}. \\



\section{Tautological bundles}
\label{sec:taut}


\noindent We have already seen how the classes of skyscraper sheaves at the torus fixed points form an important basis of $K-$theory, but the problem with them is that they usually involve working with localization. If we want to work with integral (non-localized) $K-$theory classes, we will need to look to the \b{tautological vector bundles}:
$$
\CV_i \quad  \text{of rank }v_i \text{ on} \quad \CN_{\bv,\bw}
$$
whose fibers are precisely the vector spaces $V_i$ from \eqref{eqn:cyclic}. Since Nakajima quiver varieties arise as quotients of the group $G_\bv = \prod_{i=1}^n GL(V_i)$, the tautological bundles will be topologically non-trivial. The top exterior powers:
\begin{equation}
\label{eqn:picard}
\CO_i(1) \ = \ \Lambda^{v_i}\CV_i
\end{equation}
are called tautological line bundles, and their restrictions to fixed points \eqref{eqn:fixed} are:
$$
\CV_i|_\bla \ = \ \sum_{\square \in \bla}^{c_\square \equiv i} \chi_\square \qquad \Rightarrow \qquad \CO_i(1)|_\bla \ = \ o_\bla^{\bs^i} \ = \ \prod_{\square \in \bla}^{c_\square \equiv i} \chi_\square
$$
The line bundles \eqref{eqn:picard} generate the Picard group of $\CN_{\bv,\bw}$. More generally, any polynomial in the classes:
$$
[\CV_i], [\Lambda^2\CV_i],..., [\Lambda^{v_i}\CV_i] \ \in \ K_T(\CN_{\bv,\bw})
$$
will be called a \textbf{tautological class}. A more structured way to think about tautological classes is to consider the map:
\begin{equation}
\label{eqn:kirwan}
\xymatrix{
\Lambda_{\bv,\bw} \ar[r]^-{\text{constants}} & K_{T \times G_\bv}(N_{\bv,\bw}) \ar[d]^{\text{restriction}} & \\
& K_{T \times G_\bv}\left(\mu^{-1}(0)^{\bth-\text{ss}}\right) \ar@{=}[r] & K_{\bv,\bw}}
\end{equation}
where the top left corner is the ring of symmetric polynomials in variables of $n$ colors:
$$
\Lambda_{\bv,\bw} \ = \ K_{T \times G_\bv}(\pt) \ = \ \BF_\bw[...,x_{i1}^{\pm 1},...,x_{iv_i}^{\pm 1},...]^\sym_{1\leq i \leq n}
$$ 
namely the representation ring of $G_\bv$, with coefficients in the ring of equivariant parameters $\BF_\bw = \BZ[q^{\pm 1}, t^{\pm 1}, u_1^{\pm 1},..., u_\bw^{\pm 1}]$. We will write the map \eqref{eqn:kirwan} as:
\begin{equation}
\label{eqn:taut}
\Lambda_{\bv,\bw} \ \ni \ f \ \longrightarrow \ \ov{f} \ \in \ K_{\bv,\bw}
\end{equation}
Explicitly, the class $\ov{f}$ is given by taking Chern roots $\{x_{ia}\}^{1\leq i \leq n}_{1\leq a \leq v_i}$ of the tautological bundles $\{\CV_i\}_{1\leq i \leq n}$ and plugging them into the symmetric polynomial $f$. Tautological classes have restriction given by:
\begin{equation}
\label{eqn:local}
\ov{f}|_\bla = f(\chi_\bla) \quad \Longrightarrow \quad \ov{f} = \sum_{\bla} |\bla \rangle \cdot f(\chi_\bla)
\end{equation}
where we denote $\chi_\bla = \{\chi_{\square}\}_{\square \in \bla}$ and renormalize the classes of fixed points as:
\begin{equation}
\label{eqn:renormalized}
|\bla \rangle = \frac {I_\bla}{[T_\bla \CN_{\bv,\bw}]} \ \in \ K_T(\CN_{\bv,\bw})_\loc = K_{\bv,\bw} \bigotimes_{\BF_\bw} \tBF_\bw
\end{equation}
where $\tBF_\bw = \BQ(q,t,u_1,...,u_\bw)$. In \eqref{eqn:local}, the notation $f(\chi_\bla)$ presupposes that for all boxes $\sq \in \bla$, we plug the weight $\chi_\square$ into an input of $f$ of the same color as $\square$. The following result is a particular case of \b{Kirwan surjectivity} in symplectic geometry: \\


\begin{conjecture}
\label{conj:kirwan}

The map \eqref{eqn:kirwan} is surjective, i.e. the classes $\ov{f}$ span $K_{\ebv,\ebw}$.


\end{conjecture}

\tab 
In cohomology, \citep{HP} proved that the above conjecture is true after localizing with respect to the equivariant parameter $q$, i.e. working with the ring $K_{\bv,\bw} \otimes_{\BZ[q^{\pm 1}]} \BQ(q)$ instead of $K_{\bv,\bw}$. One expects that the $K-$theoretic version of their statement also holds. Note that the stable basis constructions in this thesis are unaffected by localization with respect to $q$, so for our purposes, we could do without Conjecture \ref{conj:kirwan}. However, we chose to present it here for completeness.


\tab 
Philosophically, Conjecture \ref{conj:kirwan} implies that one can describe $K_{\bv,\bw}$ by describing the kernel of the composition \eqref{eqn:kirwan}. In other words, we need to find those symmetric Laurent polynomials $f$ for which $\ov{f} = 0$. Equivariant localization \eqref{eqn:local} implies: 
$$
\ov{f} \ = \ 0 \qquad \Leftrightarrow \qquad f(\chi_\bla) \ = \ 0 \qquad \forall \ \bw-\text{partition }\bla
$$
This is a finite collection of linear conditions on the coefficients of the Laurent polynomial $f$. We can present this in a more qualitative way by appealing to the  pairing in $K-$theory produced by the modified Euler characteristic \eqref{eqn:modified}: 
\begin{equation}
\label{eqn:pairingmoduli}
(\cdot,\cdot) \ : \ K_{\bv,\bw} \otimes K_{\bv,\bw} \ \longrightarrow \ \BF_\bw \qquad  \qquad (\gamma,\gamma') = \ochi\left(\CN_{\bv,\bw}, \gamma\cdot \gamma' \right)
\end{equation}
Since $\CN_{\bv,\bw}$ is smooth with proper fixed point sets, the above pairing is non-degenerate. Therefore, we conclude that:
\begin{equation}
\label{eqn:kmoduli}
K_{\bv,\bw} \ = \ \frac {\Lambda_{\bv,\bw}}{\text{kernel of }(\cdot,\cdot)}
\end{equation}
so we need to obtain an understanding of the kernel of the pairing. Assume that the equivariant parameters are given by complex numbers with $|q|<1$, $|t|=1$ and $|u_1|=...=|u_\bw| = 1$. Let us write $\bv! = v_1!...v_n!$ and formally set: 
\begin{equation}
\label{eqn:placeholder}
X = X_1 + ... + X_n \qquad \qquad X_i = \sum_{a=1}^{v_i} x_{ia}
\end{equation}
for the alphabet of variables of a symmetric Laurent polynomial $f$.  This means that the summands of $X$ represent the inputs of $f$, and so we use the shorthand notation 
\begin{equation}
\label{eqn:shorthand}
f(X) \ = \ f(...,x_{ia},...)^{1\leq i \leq n}_{1\leq a \leq v_i}
\end{equation}
As the colors $i$, the indices of $X_i$ will always be taken modulo $n$, hence $X_i = X_{i+n}$. \\
 

\begin{proposition}
\label{prop:integral}

For any symmetric Laurent polynomial $f\in \Lambda_{\ebv,\ebw}$,  we have:
\begin{equation}
\label{eqn:integral}
\ochi\left(\CN_{\ebv,\ebw}, \ov{f}\right) \ = \ \frac 1{\ebv!} \left( \int_{|X| = 1} -  \int_{|X| \ll 1} \right) \frac {f(X) \cdot DX}{\prod_{i,j = 1}^n \zeta\left( \frac {X_i}{X_j} \right) \prod_{i=1}^\ebw \Big[ \frac {X_i}{qu_i}  \Big] \Big[ \frac {u_i}{qX_i}  \Big]} 
\end{equation}
where we use multiplicative notation: 
\begin{equation}
\label{eqn:multiplicative}
\zeta\left( \frac {X_i}{X_j} \right) = \prod_{a=1}^{v_i} \prod_{b=1}^{v_j} \zeta\left( \frac {x_{ia}}{x_{jb}} \right) \qquad \qquad  \left[ \frac {X_i^{\pm 1}}{qu^{\pm 1}_i}  \right] = \prod_{a=1}^{v_i} \left[ \frac {x^{\pm 1}_{ia}}{qu^{\pm 1}_i}  \right]
\end{equation}
The integral in \eqref{eqn:integral} goes over all variables in the alphabet $X = \sum x_{ia}$, each running independently of the others over a contour composed of the unit circle minus a small circle around $0$.\footnote{One should think of the choice of contours as throwing out the poles at 0. Following Nekrasov, this could alternatively be formalized as:
$$
\left(\int_{|x| = 1} - \int_{|x| \ll 1} \right) F(x) Dx \ := \ \lim_{\kappa \rightarrow 0} \int_{|x|=1}^{\text{Re }\kappa \gg 0} F(x) x^{\kappa} Dx
$$
for any rational function $F(x)$. In other words, we evaluate the integral for $\text{Re }\kappa \gg 0$, and then analytically continue it to $\kappa = 0$. For a general stability condition, the proper regularization is $x^{\kappa \th}$} We write $D X = \prod^{1\leq i \leq n}_{1\leq a \leq v_i} Dx_{ia}$, where $Dz = \frac {dz}{2\pi i z}$. 

\end{proposition}

\tab Note that $\zeta(x_{ia}/x_{ia})$ in \eqref{eqn:multiplicative} contains a factor of $1-1 = 0$ in the denominator, which we implicitly eliminate from the above products. This will be the case in all similar formulas in this thesis. \\
\begin{remark}

It is natural to conjecture that \eqref{eqn:integral} holds for an arbitrary Nakajima quiver variety, though the argument below only works for the case of isolated fixed points. As we will see in the proof of the above Proposition, the denominator of the fraction in \eqref{eqn:integral} is simply the $[\cdot]$ class of the tangent bundle to $\CN_{\bv,\bw}$ in terms of tautological classes. For arbitrary quiver varieties, \eqref{eqn:integral} is a consequence of a result known as Martin's theorem (see for example \citep{HP}) by the following argument: it is straightforward to prove \eqref{eqn:integral} for abelian Nakajima quiver varieties, i.e. those whose gauge group $G_\bv$ is a torus. Martin's theorem relates Euler characteristics on arbitrary Hamiltonian reductions to those of their abelianizations, and it would imply formula \eqref{eqn:integral}. Martin's theorem in $K-$theory follows from the situation in \loccit by taking the Chern character and applying Riemann-Roch. \\ 

\end{remark}

\begin{proof} Since $\ochi(I_\bla) = 1$, we obtain $\ochi(|\bla\rangle) = [T_\bla \CN_{\bv,\bw}]^{-1}$ and hence \eqref{eqn:local} implies:
$$
\ochi\left(\CN_{\bv,\bw}, \ov{f}\right) \ = \ \sum_\bla \frac {f(\chi_\bla)}{[T_\bla\CN_{\bv,\bw}]} 
$$
By \eqref{eqn:tanmod0}, the denominator $[T_\bla\CN_{\bv,\bw}]$ is given by:
$$
\prod_{i=1}^\bw \prod^{c_\square \equiv i}_{\square \in \bla} \left[\frac {\chi_\square}{qu_i} \right]\left[\frac {u_i}{q\chi_\square} \right] \prod_{\sq,\sq' \in \bla} \frac {\left[\frac {\chi_{\square'}}{qt\chi_\square} \right]^{\delta^{c_{\sq'}}_{c_{\sq}+1}} \left[\frac {t\chi_{\square'}}{q\chi_\square } \right]^{\delta^{c_{\sq'}}_{c_{\sq}-1}}}{\left[\frac {\chi_{\square'}}{\chi_\square} \right]^{\delta^{c_{\sq'}}_{c_{\sq}}} \left[\frac {\chi_{\square'}}{q^2\chi_\square} \right]^{\delta^{c_{\sq'}}_{c_{\sq}}}} = \zeta\left(\frac {\chi_\bla}{\chi_\bla} \right)\prod_{i=1}^\bw \left[ \frac {\chi_\bla}{qu_i}  \right] \left[ \frac {u_i}{q\chi_\bla}  \right] 
$$
where $\chi_\bla$ denotes the set of contents of the boxes in the $\bw-$partition $\bla$, and $\zeta\left(\frac {\chi_\bla}{\chi_\bla}\right)$ is multiplicative notation as in \eqref{eqn:multiplicative}. It implies that we apply $\zeta$ to all pairs of weights of two boxes in the Young diagram $\bla$. Therefore, in order to finish the proof, we need to establish the following equality:
$$
\frac 1{\bv!}\left( \int_{|X|=1} - \int_{|X|\ll 1} \right) \frac {f(X) \cdot DX}{\prod_{i,j=1}^n \zeta\left(\frac {X_i}{X_j} \right) \prod_{i=1}^\bw \Big[ \frac {X_i}{qu_i}  \Big] \Big[ \frac {u_i}{qX_i}  \Big]} =
$$
\begin{equation}
\label{eqn:residuecount}
 = \sum_\bla \frac {f(\chi_\bla)}{\zeta\left(\frac {\chi_\bla}{\chi_\bla} \right)\prod_{i=1}^\bw \left[ \frac {\chi_\bla}{qu_i}  \right] \left[ \frac {u_i}{q\chi_\bla}  \right]}
\end{equation}
In order for the denominator of the right hand side to be precise, we define $[x] = x^{\frac 12} - x^{-\frac 12}$ only if the color of $x$ is 0, otherwise we set $[x]=1$ throughout this thesis.

\tab
Since the denominator of the left hand side of \eqref{eqn:residuecount} consists of linear factors $x_{ia} - qt x_{jb}$ or $x_{ia} - qt^{-1} x_{jb}$, the residues one picks out are when $x_{ia} = \chi_{ia}$ for certain monomials $\chi_{ia}$ in $q,t,u_1,...,u_\bw$. Formula \eqref{eqn:residuecount} will be proved once we show that the only monomials which appear are the ones that correspond to the set of boxes in a $\bw-$diagram $\{\chi_{ia}\}^{1\leq i \leq n}_{1\leq a \leq v_i} = \chi_\bla$, and that such a residue comes from a simple pole in \eqref{eqn:residuecount} (so evaluating it will be well-defined regardless of the order we integrate out the variables).

\tab For a given residue, let us place as many bullets $\bullet$ at the box $\sq \in \bla$ as there are variables such that $x_{ia} = \chi_\sq$. The assumption $|q|<1$ and the choice of contours means that when we integrate over $x_{ia}$, we can only pick up poles of the form: 
$$
x_{ia} - qt x_{jb} = 0 \qquad \text{or} \qquad x_{ia} - qt^{-1} x_{jb} = 0
$$
There are as many such linear factors in the denominator of \eqref{eqn:residuecount} as there are bullets in the boxes directly south and west of the box $\sq$ of weight $\chi_{ia}$, plus one factor if $\chi_{ia}$ is the weight of the root of a partition. Meanwhile, there are as many factors:
$$
x_{ia} - x_{ib} \qquad \text{or} \qquad x_{ia} - q^2  x_{ib}
$$ 
in the numerator of \eqref{eqn:residuecount} as twice the number of bullets sharing a box with $\chi_{ia}$, plus the number of bullets in the box directly southwest of it. Therefore, in order to have a pole in the variable $x_{ia}$ at $\chi_\sq$, we must have:
$$
\delta_{\sq}^{\text{root}} + \# \text{ bullets directly south of }\sq + \# \text{ bullets directly west of }\sq \geq 
$$ 
\begin{equation}
\label{eqn:ciccone}
\geq 2 \# \text{ bullets at the box }\sq - 2 + \# \text{ bullets directly southwest of }\sq
\end{equation}
In particular, there can be no bullets outside the first quadrant, and there can be no multiple bullets in a single box. Indeed, if this were the case, one would contradict inequality \eqref{eqn:ciccone} by taking $\sq$ to be the southwestern most box which has multiple bullets. Finally, if we have bullets in three boxes of weight $\chi, \chi qt$ and $\chi q^2$, the inequality forces us to all have a bullet in the box of weight $\chi qt^{-1}$. This precisely establishes the fact that the bullets trace out a partition, and hence this contributes the residue at the pole $\{x_{ia}\} = \{\text{set of bullets}\}$ to \eqref{eqn:residuecount}. The pole is simple because the difference between the sides of the inequality \eqref{eqn:ciccone} is 1 for all boxes in a $\bw-$partition. Finally, the factor of $\bv!$ in \eqref{eqn:residuecount} arises since we can permute the indices $(i,a)$ arbitrarily. \\
\end{proof}

\section{Simple correspondences}
\label{sec:simple}

\noindent Among all geometric operators that act on $K(\bw)$, the most fundamental ones come from Nakajima's simple correspondences. To define these, consider any $i\in \{1,...,n\}$ and any pair of degrees such that $\bv^+ = \bv^- + \bs^i$, where $\bs^i \in \nn$ is the degree vector with $1$ on position $i$ and zero everywhere else. Then the simple correspondence: 
\begin{equation}
\label{eqn:simplecorr}
\fZ_{\bv^+,\bv^-,\bw} \ \hookrightarrow \ \CN_{\bv^+,\bw} \times \CN_{\bv^-,\bw}
\end{equation}
parametrizes pairs of quadruples $(X^\pm,Y^\pm,A^\pm,B^\pm)$ that respect a fixed collection of quotients $(V^+ \twoheadrightarrow V^-) = \{V_j^+ \twoheadrightarrow V_j^-\}_{j\in \{1,...,n\}}$ of codimension $\delta_j^i$:

\begin{equation}
\label{eqn:bigdiag}
\xymatrix{ W_{i-2} \ar[dd]_{A_{i-2}} & & W_{i-1} \ar[dd]_{A_{i-1}} & & W_i\ar[d]^{A_{i}} & & W_{i+1} \ar[dd]^{A_{i+1}} & & W_{i+2} \ar[dd]^{A_{i+2}}\\
& & & & \ar@/^/[dll]^{Y_{i-1}^+} V_i^+ \ar@{->>}[dd] \ar@/^/[drr]^{X_i^+} & & & & \\
... \ V_{i-2}^\pm \ar[dd]_{B_{i-2}} \ar@/^/[rr]^{X^\pm_{i-2}} & & \ar@/^/[ll]^{Y^\pm_{i-2}} V_{i-1}^\pm  \ar[dd]_{B_{i-1}} \ar@/^/[rru]^{X_{i-1}^+} \ar@/_/[rrd]_{X_{i-1}^-} & & & & V_{i+1}^\pm \ar[dd]^{B_{i+1}} \ar@/^/[rr]^{X^\pm_{i+1}} \ar@/^/[llu]^{Y_i^+} \ar@/_/[lld]_{Y_{i}^-} & & \ar[dd]^{B_{i+2}} \ar@/^/[ll]^{Y^\pm_{i+1}} V_{i+2}^\pm  \ ... \\
& & & & \ar@/_/[llu]_{Y_{i-1}^-} V_i^- \ar[d]^{B_{i}} \ar@/_/[rru]_{X_{i}^-} & & & & \\
W_{i-2} & & W_{i-1} & & W_i & & W_{i+1} & & W_{i+2} } \qquad
\end{equation}	
We only consider quadruples which are semistable and zeroes for the map $\mu$ of \eqref{eqn:moment}, and take them modulo the subgroup of $G_{\bv^+} \times G_{\bv^-}$ that preserves the fixed collection of quotients $V^+ \twoheadrightarrow V^-$. The variety $\fZ_{\bv^+,\bv^-,\bw}$ comes with the tautological line bundle:
$$
\CL|_{V^+ \twoheadrightarrow V^-} = \text{Ker}(V^+_i \twoheadrightarrow V^-_i)
$$
as well as projection maps:
\begin{equation}
\label{eqn:diagram}
\xymatrix{& \fZ_{\bv^+,\bv^-,\bw} \ar[dl]_{\pi_+} \ar[dr]^{\pi_-} \\
\CN_{\bv^+,\bw} & & \CN_{\bv^-,\bw} }
\end{equation}
With this data in mind, we may consider the following operators on $K-$theory:
\begin{equation}
\label{eqn:simpleop}
e_{i,d}^\pm \ : \ K_{\bv^\mp,\bw} \longrightarrow K_{\bv^\pm,\bw}
\end{equation}
$$
e_{i,d}^\pm(\alpha) \ = \ \widetilde{\pi}_{\pm *} \left( \CL^d \cdot \pi^*_{\mp}(\alpha ) \right)
$$
When the degrees $\bv,\bw$ will not be relevant, we will abbreviate the simple correspondence by $\fZ_i$, and interpret \eqref{eqn:simpleop} as operators $e_{i,d}^\pm : K(\bw) \rightarrow K(\bw)$ which have degree $\pm \bs^i$ in the grading $\bv$. We also define the following operators of degree 0:
\begin{equation}
\label{eqn:cartan0}
\ph_{i,d}^\pm \ : \ K(\bw) \longrightarrow K(\bw) \qquad \qquad \ph_i^\pm(z) = \sum_{d=0}^\infty \ph_{i,d}^\pm z^{\mp d}
\end{equation}
$$
\ph_i^\pm(z) \ = \ \text{ multiplication by the tautological class} \quad \overline{\frac {\zeta\left( \frac zX \right)}{\zeta\left(\frac Xz \right)}} \prod_{1\leq j \leq \bw}^{u_j \equiv i} \frac {\left[ \frac {u_j}{qz} \right]}{\left[ \frac {z}{qu_{j}} \right]}
$$
where the RHS must be expanded in negative or positive powers of the variable $z$ of color $i$, depending on whether the sign is $+$ or $-$. Recall that $\zeta\left(\frac {z^{\pm 1}}{X^{\pm 1}}\right)$ refers to multiplicative notation in the alphabet of variables $X = \sum x_{ia}$, as in \eqref{eqn:multiplicative}. Then the main result of \citep{VV}, as well as \citep{Nak1} for general quivers without loops, is: \\

\begin{theorem}
\label{thm:actvv}

For all $\ebw \in \nn$, the operators $e_{i,d}^\pm$ and $\ph_{i,d}^\pm$ give rise to an action:
$$
\UU \ \curvearrowright \ K(\ebw)
$$

\end{theorem}

\noindent The correspondence $\fZ_i$ is well-known to be smooth, as well as proper with respect to either projection map $\pi_\pm$. Hence the operators $e_{i,d}^\pm$ are well-defined in integral $K-$theory. For the remainder of this section, we will forgo this integrality and seek to compute them via equivariant localization, i.e. in the basis of torus fixed points. At the end of the Section, we will recover integrality from our formulas. Since torus fixed points of Nakajima cyclic quiver varieties are parametrized combinatorially by:
$$
\CN_{\bv^\pm, \bw}^{\text{fixed}} = \{I_{\bla^\pm}, \text{ where }\bla^\pm \text{ is a } \bw-\text{partition of size }\bv^\pm\}
$$
we conclude that fixed points of the correspondence $\fZ_i$ are parametrized by:
$$
\fZ_{\bv^+, \bv^-, \bw}^{\text{fixed}} = \{(I_{\bla^+} \subset I_{\bla^-}), \ \text{where }\bla^+ \geq_i \bla^-\}
$$
In the above, recall that we write $\bla^+ \geq \bla^-$ if the Young diagram of the former partition completely contains that of the latter. In the case of $\fZ_i$, their difference automatically consists of a single box of color $i$, and we will denote this by $\bla^+ \geq_i \bla^-$. \\

\begin{exercise}
\label{ex:tangent2}
The $T-$character in the tangent spaces to $\fZ_{\bv^+,\bv^-,\bw}$ is given by:
\begin{equation}
\label{eqn:tansim}
T_{\bla^+ \geq_i \bla^-} \fZ_{\bv^+,\bv^-,\bw} = \sum_{j = 1}^{\bw} \left( \sum_{\square \in \bla^+}^{c_\square \equiv j} \frac {\chi_\square}{qu_j} + \sum_{\square' \in \bla^-}^{c_{\square'} \equiv j} \frac {u_j}{q\chi_{\square'}} \right) +
\end{equation}
$$
+ \sum^{\sq \in \bla^+}_{\sq' \in \bla^-} \left(\delta_{c_{\sq'}}^{c_{\sq}-1} \cdot \frac {\chi_{\square}}{qt\chi_{\square'}} + \delta_{c_{\sq'}}^{c_{\sq}+1} \cdot \frac {t\chi_{\square}}{q\chi_{\square'}} - \delta_{c_{\sq'}}^{c_{\sq}} \cdot \frac {\chi_{\square}}{\chi_{\square'}} - \delta_{c_{\sq'}}^{c_{\sq}} \cdot \frac {\chi_{\square}}{q^2\chi_{\square'}} \right) - 1 
$$
for any $\bla^+ \geq_i \bla^-$. 
\end{exercise}

\tab
Let us now explain how to use formulas such as \eqref{eqn:tansim} to compute the matrix coefficients of operators such as \eqref{eqn:simpleop} in the basis of fixed points $|\bla\rangle$, renormalized as in \eqref{eqn:renormalized}. This principle will be used again in Section \ref{sec:act} for the more complicated eccentric correspondences. Given a correspondence:
$$
\fZ \ \subset \ X_+ \times X_-
$$
endowed with projection maps $\pi_\pm : \fZ \rightarrow X_\pm$, we wish to compute the operators:
$$
K(X_+) \ \mathop{\leftrightarrows}^{e^+}_{e^-} \ K(X_-) \qquad \qquad e^\pm(\alpha) = \widetilde{\pi}_{\pm *}\left(\pi^*_\mp(\alpha) \right)
$$ 
in the basis of fixed points: 
$$
|p^\pm\rangle \ := \ \frac {I_{p^\pm}}{[T_{p^\pm}X_\pm]} \ \in \ K(X_\pm)_\loc \qquad \text{where} \qquad \{ p^{\pm} \} \ = \ X_\pm^{\text{fixed}}
$$
Let us write $\fZ^{\text{fixed}} \subset \{(p^+,p^-), \ p^\pm \in X_\pm^{\text{fixed}}\}$ for the fixed locus of $\fZ$. Then we have:
$$
\pi_\mp^*\left(|p^\mp\rangle\right) \ = \ \sum_{(p^+,p^-) \in \fZ^{\text{fixed}}} |(p^+,p^-)\rangle
$$
Using the formula for the push-forward:
$$
\widetilde{\pi}_{\pm *} \left( |(p^+,p^-)\rangle \right) \ = \ |p^\pm \rangle \cdot [\text{Cone }d\pi_{\pm}]
$$
we obtain:
\begin{equation}
\label{eqn:formula}
e^\pm\left(|p^\mp \rangle \right) = \sum_{(p^+,p^-) \in \fZ^{\text{fixed}}} |p^\pm\rangle \cdot \frac {[T_{p^\pm}X_\pm]}{[T_{(p^+,p^-)}\fZ]}
\end{equation}
This formula establishes the fact that matrix coefficients $\langle p^\pm|e^\pm |p^\mp \rangle$ of correspondences in the basis of fixed points are given by the $[\cdot]$ class of the tangent bundle to $\fZ$ and to $X^\pm$. More precisely, we need to compute the $[\cdot]$ class of the difference between the tangent bundle of the base space and the tangent bundle of the correspondence. For the simple Nakajima correspondences of \eqref{eqn:simpleop}, this information is provided by \eqref{eqn:tanmod0} and \eqref{eqn:tansim}. Specifically, let us consider a fixed point $(\bla^+ \geq_i \bla^-)$ of $\fZ_i$, and let $\blacksquare$ denote the only box in $\bla^+ \backslash \bla^-$, which by definition has color $i$. Then we may combine \eqref{eqn:tanmod0} and \eqref{eqn:tansim} to obtain:
$$
[T_{\bla^+} \CN_{\bv^+,\bw}] - [T_{\bla^+ \geq_i \bla^-} \fZ_{\bv^+,\bv^-,\bw}] = 1 + \sum_{1\leq j \leq \bw}^{u_j\equiv i} \frac {u_j}{q\chi_\bsq} + 
$$
$$
+  \sum_{\square\in \bla^+}^{c_{\sq} \equiv i+1} \frac {\chi_\sq}{qt\chi_\bsq} + 
\sum_{\square\in \bla^+}^{c_{\sq} \equiv i-1} \frac {t\chi_\square}{q\chi_\bsq} -
\sum_{\square\in \bla^+}^{c_{\sq} \equiv i} \frac {\chi_{\sq}}{\chi_\bsq} -
\sum_{\square\in \bla^+}^{c_{\sq} \equiv i} \frac {\chi_\sq}{q^2\chi_\bsq}
$$
or:
$$
[T_{\bla^-} \CN_{\bv^-,\bw}] - [T_{\bla^+ \geq_i \bla^-} \fZ_{\bv^+,\bv^-,\bw}] = 1 - \sum_{1\leq j \leq \bw}^{u_j\equiv i} \frac {\chi_\bsq}{qu_j} -
$$
$$
-  \sum_{\square\in \bla^-}^{c_{\sq} \equiv i-1} \frac {\chi_\bsq}{qt\chi_\sq} - \sum_{\square\in \bla^-}^{c_{\sq} \equiv i+1} \frac {t\chi_\bsq}{q\chi_\sq} +
\sum_{\square\in \bla^-}^{c_{\sq} \equiv i} \frac {\chi_\bsq}{\chi_\sq} + \sum_{\square\in \bla^-}^{c_{\sq} \equiv i} \frac {\chi_\bsq}{q^2\chi_\sq}
$$
Then \eqref{eqn:formula} gives us the following formulas for the matrix coefficients of the operator \eqref{eqn:simpleop}, which are non-zero only if $\bla^+ \geq_i \bla^-$:
$$
\langle \bla^\pm | e_{i,d}^\pm |\bla^\mp \rangle = \chi_\bsq^d \cdot (1-1) \cdot \left( \frac {\prod_{\sq \in \bla^\pm}^{c_\sq \equiv i\pm 1} \Big[\frac {\chi_\sq^{\pm 1}}{qt\chi_\bsq^{\pm 1}} \Big] \prod_{\sq \in \bla^\pm}^{c_\sq \equiv i\mp 1} \Big[ \frac {t\chi_\sq^{\pm 1}}{q\chi_\bsq^{\pm 1}} \Big]}{\prod_{\sq \in \bla^\pm}^{c_\sq \equiv i} \Big[ \frac {\chi_\sq^{\pm 1}}{\chi_\bsq^{\pm 1}} \Big] \prod_{\sq \in \bla^\pm}^{c_\sq \equiv i}\Big[\frac {\chi_\sq^{\pm 1}}{q^2\chi_\bsq^{\pm 1}} \Big]} \right)^{\pm 1}  \prod^{u_j\equiv i}_{1\leq j \leq \bw} \left[\frac {u_j^{\pm 1}}{q\chi_\bsq^{\pm 1}} \right]^{\pm 1}
$$ 
The factor $1-1$ is meant to cancel a single factor of $1-1$ which appears in the denominator of the above expression. Let us rewrite the above formula without this rather strange implicit cancellation, by changing the products over $\sq \in \bla^\pm$ to products over $\sq \in \bla^\mp$:
$$
\langle \bla^\pm | e_{i,d}^\pm |\bla^\mp \rangle = \frac {\chi_\bsq^d}{[q^{-2}]} \cdot \left( \frac {\prod_{\sq \in \bla^\mp}^{c_\sq \equiv i\pm 1} \Big[\frac {\chi_\sq^{\pm 1}}{qt\chi_\bsq^{\pm 1}} \Big] \prod_{\sq \in \bla^\mp}^{c_\sq \equiv i\mp 1} \Big[ \frac {t\chi_\sq^{\pm 1}}{q\chi_\bsq^{\pm 1}} \Big]}{\prod_{\sq \in \bla^\mp}^{c_\sq \equiv i} \Big[ \frac {\chi_\sq^{\pm 1}}{\chi_\bsq^{\pm 1}} \Big] \prod_{\sq \in \bla^\mp}^{c_\sq \equiv i}\Big[\frac {\chi_\sq^{\pm 1}}{q^2\chi_\bsq^{\pm 1}} \Big]} \right)^{\pm 1}  \prod^{u_j\equiv i}_{1\leq j \leq \bw} \left[\frac {u_j^{\pm 1}}{q\chi^{\pm 1}_\bsq} \right]^{\pm 1} 
$$ 
Comparing the above with the definition of $\zeta$ in \eqref{eqn:defzeta}, we obtain:
\begin{equation}
\label{eqn:fixedpoints1}
\langle \bla^+ | e_{i,d}^+ |\bla^- \rangle \ = \ \frac {\chi_\bsq^d}{[q^{-2}]} \cdot \zeta\left(\frac {\chi_\bsq}{\chi_{\bla^-}} \right) \prod^{u_j\equiv i}_{1\leq j \leq \bw} \left[\frac {u_j}{q\chi_\bsq} \right] 
\end{equation}
\begin{equation}
\label{eqn:fixedpoints2}
\langle \bla^- | e_{i,d}^- |\bla^+ \rangle \ = \ \frac {\chi_\bsq^d}{[q^{-2}]} \cdot \zeta \left(\frac {\chi_{\bla^+}}{\chi_\bsq} \right)^{-1} \prod^{u_j\equiv i}_{1\leq j \leq \bw} \left[\frac {\chi_\bsq}{qu_j} \right]^{-1} 
\end{equation}
For an arbitrary symmetric Laurent polynomial $f \in \Lambda_{\bv,\bw}$, localization \eqref{eqn:local} yields:
\begin{equation}
\label{eqn:form1}
e_{i,d}^+ (\ov{f}) = \sum_{\bla^-} e_{i,d}^+ \left( |\bla^-\rangle \right) \cdot f(\chi_{\bla^-}) = 
\end{equation}
$$
= \sum_{\bla^+ \geq_i \bla^-}^{\bsq = \bla^+/\bla^-} | \bla^+ \rangle \cdot \frac {\chi_\bsq^d}{[q^{-2}]} \cdot f(\chi_{\bla^-}) \cdot \zeta \left( \frac {\chi_\bsq}{\chi_{\bla^-}} \right)  \prod^{u_j\equiv i}_{1\leq j \leq \bw} \Big[\frac {u_j}{q\chi_\bsq} \Big] 
$$
and:
\begin{equation}
\label{eqn:form2}
e_{i,d}^- (\ov{f}) = \sum_{\bla^+} e_{i,d}^+ \left( |\bla^+\rangle \right) \cdot f(\chi_{\bla^+}) = 
\end{equation}
$$
= \sum_{\bla^+ \geq_i \bla^-}^{\bsq = \bla^+/\bla^-} | \bla^- \rangle \cdot \frac {\chi_\bsq^d}{[q^{-2}]} \cdot f(\chi_{\bla^+}) \cdot \zeta \left(\frac {\chi_{\bla^+}}{\chi_\bsq} \right)^{-1} \prod^{u_j\equiv i}_{1\leq j \leq \bw} \Big[\frac {\chi_\bsq}{qu_j} \Big]^{-1}  
$$
Since $e_{i,d}^\pm$ is defined on integral $K-$theory, we know that the right hand sides of the above expressions are integral $K-$theory classes. However, this is not immediately apparent from the above formulas, which involve many denominators and the localized fixed point classes $|\bla^\pm \rangle$. We will soon see that the right hand sides of \eqref{eqn:form1} and \eqref{eqn:form2} arise as residue computations of a certain rational function, as in the following elementary formula pertaining to polynomials $p(x)$ in a single variable:
$$
\BZ[a_1,...,a_k] \ni \text{Res}_{x=\infty} \frac {p(x)}{(x-a_1)...(x-a_k)} = \sum_{i=1}^k \frac {p(a_i)}{(a_i-a_1)...(a_i-a_k)} \in \BQ(a_1,...,a_k)
$$
Indeed, while the right hand side looks like a rational function in the variables $a_1,...,a_k$, only when we interpret it as a residue around $\infty$ does it become apparent that it is actually a polynomial. Let $X$ be a placeholder for the infinite alphabet of variables $\{x_{i1},x_{i2},...\}_{1\leq i \leq n}$, in other words $X = \sum^{1\leq i \leq n}_{1\leq a < \infty} x_{ia}$. \\

\begin{exercise}
\label{ex:intcorr}
For any $f = f(X) \in \Lambda_{\bv,\bw}$, in the notation of \eqref{eqn:shorthand}, we have:
\begin{equation}
\label{eqn:formula1}
e_{i,d}^+ \left( \ov{f} \right) = \int z^d \cdot \ov{f(X - z) \zeta\Big(\frac zX \Big)} \cdot \prod_{1\leq j \leq \bw}^{u_j \equiv i} \Big[\frac {u_{i}}{qz} \Big] Dz
\end{equation}
\begin{equation}
\label{eqn:formula2}
e_{i,d}^- \left( \ov{f} \right) = \int z^d \cdot \ov{f(X + z) \zeta\Big( \frac Xz \Big)^{-1}} \cdot \prod_{1\leq j \leq \bw}^{u_j \equiv i} \Big[\frac {z}{qu_i} \Big]^{-1} Dz 
\end{equation}
where the integrals are taken over small contours around $0$ and $\infty$. Here, $X+z$ (respectively $X-z$) denotes the alphabet $X$ to which we adjoin (respectively remove) the variable $z$, according to plethystic notation for symmetric functions.

\end{exercise}

\tab
One could prove Exercise \ref{ex:intcorr} just like the case $k=1$ of Theorems 3.9 and 4.10 of \citep{Nflag}. The main idea therein is to exhibit $\fZ_{\bv^+,\bv^-,\bw}$ as a projective bundle over $\CN_{\bv^\pm,\bw}$. The setup in \loccit is for the Jordan quiver, so one needs to restrict attention to $\BZ/n\BZ$ fixed points to obtain Exercise \ref{ex:intcorr}. Because formulas \eqref{eqn:formula1} and \eqref{eqn:formula2} only sum the residues at 0 and $\infty$, they take values in \b{integral} $K-$theory, a fact which was not apparent from \eqref{eqn:form1} and \eqref{eqn:form2}. Indeed, all denominators of \eqref{eqn:formula1} and \eqref{eqn:formula2} which involve factors such as $(X-z\cdot \text{const})$ will change to either $X$ or $(-z\cdot \text{const})$, as $z$ approaches either $0$ or $\infty$.

\chapter{Stable Bases in $K-$theory}
\label{chap:stable}


\section{Torus actions and Newton polytopes}
\label{sec:torusaction}


\noindent Let us assume that we are given a torus $A$ that acts on a symplectic variety:
$$
A \ \curvearrowright \ X
$$
and preserves the symplectic form. Then we may consider the torus fixed point locus $X^A \subset X$, which inherits a symplectic structure from $\omega$. For generic directions $\sigma : \BC^* \rightarrow A$, we have $X^\sigma = X^A$, and we may consider the one dimensional flow on $X$ induced by $\sigma$. Define the \b{attracting correspondence} with respect to $\sigma$ as:
\begin{equation}
\label{eqn:attcorr}
Z^\sigma \ := \ \{(x,y) \text{ s.t.}\lim_{t \rightarrow 0} \ \sigma(t)\cdot x  = y\} \ \hookrightarrow \ X \times X^A
\end{equation}
which keeps track of which points of $X$ flow into which points of $X^A$ in the direction prescribed by $\sigma$. We think of $Z^\sigma$ as a correspondence via the projection maps:
$$
\xymatrix{
& Z^\sigma \ar[ld]_{\pi_1} \ar[rd]^{\pi_2} & \\ 
X & & X^A} \qquad \quad \pi_1(x,y) = x \qquad \quad \pi_2(x,y) = y
$$
When the fixed locus is disconnected, for any connected component $F \subset X^A$ we may consider the restriction $Z^\sigma_F = Z^\sigma \cap \left( X \times F \right)$. We define the \b{leaf} of $F$ as:
$$
\leaf^\sigma_F = \pi_1(Z^\sigma_F)
$$
to consist of all points of $X$ which flow into $F$ under $\sigma$. Since $X$ and $X^A$ are symplectic varieties, our choice of cocharacter $\sigma:\BC^* \rightarrow A$ gives rise to a splitting of the normal bundle to any fixed component $F \subset X^A$:
$$
N_{F \subset X} \ = \ N_{F \subset X}^+ \oplus N_{F \subset X}^-
$$
into attracting and repelling directions. The symplectic form $\omega$ is a perfect pairing between $N^+_{F \subset X}$ with $N^-_{F \subset X}$. Moreover, as shown in Lemma 3.2.4. of \citep{MO}, the subvariety $\leaf^\sigma_F$ is nothing but the total space of the affine bundle $N_{F \subset X}^+$ on $F$. Similar considerations apply to the correspondence \eqref{eqn:attcorr}:
$$
Z^\sigma_F \ \subset \ X \times F
$$
whose tangent space is $T_F \oplus N^+_{F \subset X}$. As such, we observe that $Z^\sigma_F$ is a Lagrangian correspondence, i.e. for any $\Gamma \in K_A(Z^\sigma_F)$, the operator:
$$
K_A(F) \ \longrightarrow \ K_A(X) \qquad \qquad \alpha \longrightarrow \pi_{1*} \Big(\Gamma\cdot \pi_2^*(\alpha) \Big)
$$ 
takes Lagrangian classes to Lagrangian classes. Here, the phrase ``Lagrangian class" refers to a $K-$theory class supported on a Lagrangian subvariety.


\tab
Let us now assume that we have a torus $T \curvearrowright X$, and only a certain subtorus $A \subset T$ preserves the symplectic form on $X$. We will now discuss Newton polytopes of $K-$theory classes, and we will begin with the situation of equivariant constants: 
$$
\alpha \in K_T(\pt) = \BZ[t_1^{\pm 1},...,t_n^{\pm 1}], \qquad \qquad \alpha = \sum_{c_i \in \BZ \backslash 0} c_i \cdot t_1^{x^{(i)}_1}...t_n^{x^{(i)}_n} 
$$
Then the \b{Newton polytope} of $\alpha$ is defined as:
$$
P_T(\alpha) \ = \ \{\text{convex hull of } (x^{(i)}_1,...,x^{(i)}_n)\} \ \subset \ \ft^\vee_\BR
$$
We will often write $P^\circ_T(\alpha)$ for the Newton polytope with the ``outermost" vertex removed, where the notion of outermost will be defined with respect to a direction in $\ft_\BR$ that will be spelled out explicitly in the next Section. Let us now mimic the above notation for a class $\alpha \in K_T(F)$, where the variety $F$ is fixed pointwise by the subtorus $A \subset T$. We have:
$$
K_T(F) \ \cong \ K_{T/A}(F) \bigotimes_{\BZ[T/A]} \BZ[T] \ \cong \ K_{T/A}(F) \bigotimes_{\BZ} \BZ[A]
$$
Note that the second isomorphism is not canonical, as it involves choosing a splitting $\BZ[T] \cong \BZ[T/A] \otimes \BZ[A]$. However, the Newton polytope of $\alpha \in K_T(F)$ is well-defined after projecting $\ft^\vee_\BR \twoheadrightarrow \fa^\vee_\BR$:
$$
P_{A}(\alpha), \ P_{A}^\circ(\alpha) \ \subset \ \fa^\vee_\BR
$$
because the projection is not altered by changing the splitting. Throughout this section, we will be faced with the question of when the Newton polytope of a class $\alpha$ lies inside (a translate of) the Newton polytope of a class $\beta$:
\begin{equation}
\label{eqn:inside}
P_A(\alpha) \ \subset \ P^\circ_A(\beta) + l_A
\end{equation}
for some $l_A\in \fa^\vee_\BQ$. We could ask the opposite question, namely when does the above inclusion fail? The answer is: when there exists a lattice point $p \in P_A(\alpha)$ which fails to land inside the polytope $P^\circ_A(\beta)+l_A$. This is equivalent to the existence of a one dimensional projection $\pi:\fa^\vee_\BZ \rightarrow \BZ^\vee$ such that $\pi(p)$ does not lie inside $\pi\left(P^\circ_A(\beta)+l_A\right)$. We conclude that \eqref{eqn:inside} holds if and only if:
\begin{equation}
\label{eqn:inside2}
P_\sigma(\alpha) \ \subset \ P^\circ_\sigma(\beta) + l_\sigma 
\end{equation}
for all rank one subtori $\sigma:\BC^* \subset A$, where $l_\sigma = \pi(l_A) \in \BQ^\vee$ is a rational number. Moreover, it is enough to only consider those $\sigma$ which have constant sign on $P^\circ_A(\beta)$, which corresponds to considering all $\sigma$ in a certain chamber inside $\fa^\vee_\BR$. The advantage of reducing to the rank one viewpoint is that with respect to the cocharacter $\sigma$, $K-$theory classes specialize to Laurent polynomials in a single variable $t$:
$$
\alpha|_\sigma \ = \ c' \cdot t^{\mindeg \alpha} + ... + c'' \cdot t^{\maxdeg \alpha}
$$
for various coefficients $c',...,c'' \in \BZ$, and their Newton polytopes are simply intervals:
\begin{equation}
\label{eqn:minmax}
P_\sigma(\alpha) = \Big[ \mindeg \alpha, \maxdeg \alpha \Big]
\end{equation}
The notions $\mindeg$ and $\maxdeg$ are defined with respect to the one dimensional torus $\sigma$, i.e. restricting equivariant weights according to the morphism $A^\vee \stackrel{\sigma^\vee}\longrightarrow \BZ$. The inclusion of intervals \eqref{eqn:inside2} reduces to the following properties:
\begin{equation}
\label{eqn:simpleinclusion1}
\mindeg \alpha > \mindeg \beta + l_\sigma \quad \text{or} \quad \mindeg \alpha \geq \mindeg \beta + l_\sigma
\end{equation}
\begin{equation}
\label{eqn:simpleinclusion2}
\maxdeg \alpha \leq \maxdeg \beta + l_\sigma \quad \text{or} \quad \maxdeg \alpha < \maxdeg \beta + l_\sigma
\end{equation}
where ``or" depends on whether the vertex we choose to exclude when defining $P_\sigma^\circ(\beta)$ is the leftmost or the rightmost endpoint of the interval \eqref{eqn:minmax}. Throughout this paper, we will make the former choice. Formulas such as \eqref{eqn:simpleinclusion1} and \eqref{eqn:simpleinclusion2} will be easier to prove than the inclusion of polytopes \eqref{eqn:inside}, essentially since the notions min deg and max deg satisfy the following additivity properties:
$$
\mindeg \alpha+\alpha' \geq \min\Big( \mindeg \alpha, \mindeg \alpha' \Big)
$$
$$
\maxdeg \alpha+\alpha' \leq \max\Big( \maxdeg \alpha, \maxdeg \alpha' \Big)  
$$
and multiplicativity properties:
$$
\mindeg \alpha \cdot \left(\alpha'\right)^{\pm 1} = \mindeg \alpha \pm \mindeg \alpha'
$$
$$
\maxdeg \alpha \cdot \left( \alpha' \right)^{\pm 1} = \maxdeg \alpha \pm \maxdeg \alpha'
$$
The quantity $\alpha / \alpha'$ is the kind of ratio which appears in localized $K-$theory, and we will work with such formulas in Chapter \ref{chap:pieri}. \\
 
 
\section{Definition of the stable basis}
\label{sec:stable}

\noindent Stable bases may be defined for all symplectic varieties $X$ which are acted on by a torus $T$. In this context, we consider a subtorus of $T$:
$$
A \ \subset \ T \ \curvearrowright \ X
$$
which preserves the symplectic form $\omega$. Let us consider a generic cocharacter:
$$
\sigma \ : \ \BC^* \longrightarrow A
$$
and recall the leaves of $X$ under the flow induced by $\sigma$, as in Section \ref{sec:torusaction}. We obtain an ordering on the connected components of the fixed locus $X^A$ by setting:
\begin{equation}
\label{eqn:flow}
F' \unlhd F \qquad \text{if} \qquad F' \cap \overline{\leaf^\sigma_F} \neq \emptyset
\end{equation}
In other words, $F' \unlhd F$ if there is a projective flow line going from a point in $F'$ to a point in $F$. We take the transitive closure of this ordering:
$$
F'' \unlhd F' \quad \text{and} \quad F' \unlhd F \quad \Longrightarrow \quad F'' \unlhd F
$$
and note that Section 3.2.3. of \citep{MO} shows that $F \unlhd F'$ and $F' \unlhd F$ implies $F = F'$. Then \eqref{eqn:flow} extends to a well-defined partial ordering on the connected components of the fixed locus. The ordering allows us to extend the attracting correspondence $Z^\sigma$ of \eqref{eqn:attcorr} by adding to it contributions from ``downstream" fixed components:
\begin{equation}
\label{eqn:attcorr2}
\oZ^\sigma \ := \ \{(x,y) \text{ s.t. } \exists \text{ chain }  x\rightarrow z_1 \rightarrow ... \rightarrow z_k = y \} \ \hookrightarrow \ X \times X^A
\end{equation}
Here, $x \rightarrow z_1$ means that $\lim_{t\rightarrow 0} \sigma(t)\cdot x  = z_1$. Furthermore, since $z_1,...,z_k$ are torus fixed points, the notion $z_i \rightarrow z_{i+1}$ means that there exists a projective line joining $z_i$ and $z_{i+1}$, flowing from the former to the latter under $\sigma$. Consider any rational line bundle $\CL \in \text{Pic}_T(X) \otimes \BQ$. \\

\begin{definition} 
	
To such $\sigma$ and $\CL$, \citep{MO2} associate a map:
\begin{equation}
\label{eqn:stab}
\stab^\sigma_\CL \ : \ K_T(X^A) \longrightarrow K_T(X),
\end{equation}
given by a $K-$theory class on the correspondence $\oZ^\sigma$, subject to the condition: \footnote{Note that our convention on the class of a subvariety, such as $Z^\sigma \hookrightarrow X \times X^A$, is defined with respect to the modified direct image \eqref{eqn:modpush}, and so differs from the actual direct image by the square root of the determinant of the normal bundle}
\begin{equation}
\label{eqn:leading}
\stab^\sigma_\CL \Big|_{F \times F}  =  \CO_{Z^\sigma}  \Big|_{F \times F} 
\end{equation}
and the following condition for all $F' \lhd F$:
\begin{equation}
\label{eqn:small}
P_{A}\left(\stab^\sigma_\CL \Big |_{F' \times F}\right) \ \subset \ P^\circ_{A}\left( \CO_{Z^\sigma}  \Big|_{F' \times F'} \right) + \wt \CL|_{F'} - \wt \CL|_F \quad \subset \quad \fa^\vee_\BR
\end{equation}
where the Newton polytope $P_{A}^\circ \subset \fa^\vee_\BR$ in the right hand side is formed by excluding the vertex on which the cocharacter $\sigma$ is minimal. This means that one dimensional projections of $P_A^\circ$ will be intervals open on the left and closed on the right, or equivalently, that we make the choices $>$ and $\leq$ in \eqref{eqn:simpleinclusion1} and \eqref{eqn:simpleinclusion2}, respectively.

\end{definition} 

\tab
In most examples, there will be a preferred  ample line bundle $\bth \in \text{Pic}(X)$, which has the property that the flow ordering induced by $\sigma$ on fixed components coincides with the following ``ample partial ordering" defined by pairing $\bth|_{F} \in T^\vee$ with the cocharacter $\sigma \in T$: 
\begin{equation}
\label{eqn:compatibility}
F' \ \unlhd \ F \quad \Leftrightarrow \quad \langle \sigma, \bth|_{F'} \rangle \ \leq \ \langle \sigma, \bth|_{F} \rangle
\end{equation}
See \citep{MO} for the general framework. In the case of Nakajima quiver varieties, this ample line bundle coincides with the stability condition \eqref{eqn:stabquiver}. Then our choice to exclude the ``left endpoint" of $P_A^\circ$ in \eqref{eqn:small} implies that the stable basis is unchanged by slightly moving $\CL$ in the negative direction of $\bth$:
\begin{equation}
\label{eqn:epsilon}
\stab^\sigma_\CL \ = \ \stab^\sigma_{\CL - \e \bth} \qquad \text{for small enough } \e = \e(\CL) >0
\end{equation}
While uniqueness is rather straightforward, it is not clear that a collection of maps satisfying properties \eqref{eqn:leading} and \eqref{eqn:small} exists, and in fact, the construction is not yet known to hold in full generality. The situation of quiver varieties we are concerned with in this thesis is a particular case of Hamiltonian reductions of vector spaces by reductive groups, in which case the existence of stable bases was proved by \citep{MO2} by abelianization, using techniques of \citep{Sh}.

\tab
One of the hallmarks of the stable basis is its \b{integrality}, i.e. the fact that it is defined on the actual $K-$theory ring, as opposed from a localization. If we had relaxed this requirement, it would be very easy to construct a multitude of maps $\stab$ satisfying \eqref{eqn:leading} and \eqref{eqn:small} in localized $K-$theory. One of these is:
\begin{equation}
\label{eqn:stabinfinity}
\stab^\sigma_\infty(\alpha) \ := \ \widetilde{\iota}^F_* \left( \frac {\alpha}{[N^+_{F \subset X}]}  \right) \ = \ (\iota^{F*})^{-1}(\alpha \cdot [N^-_{F \subset X}])
\end{equation}
for all $\alpha \in K_T(F)$, for all fixed components $\iota^F : F \hookrightarrow X^\sigma$. The map \eqref{eqn:stabinfinity} has the property that its restriction to $F' \times F$ is zero for all $F'\neq F$, which morally corresponds to condition \eqref{eqn:small} for $\CL = \infty \cdot \bth$, where $\bth$ is the ample line bundle of \eqref{eqn:compatibility}. \\

\begin{remark}
\label{rem:stablocal}
The construction \eqref{eqn:stabinfinity} shows that working in localized $K-$theory destroys the uniqueness of the stable basis construction, but there is a way to partially salvage this. If we only localize with respect to the one-dimensional torus that scales the symplectic form, i.e.:
$$
\text{replace} \qquad K_T(X) \qquad \text{by} \qquad K_T(X) \bigotimes_{\BZ[q^{\pm 1}]} \BQ(q)
$$
then the stable basis is still well-defined and unique. The reason is that condition \eqref{eqn:small} is an inclusion of Newton polytopes in the torus which preserves the symplectic form, and so it is unaffected by tensoring with arbitrary rational functions of $q$.
\end{remark}

	
	






\noindent Definition \eqref{eqn:stab} was given with respect to a generic $\sigma$, by which we mean those cocharacters such that $X^\sigma = X^A$. Those cocharacters for which this property fails determine a family of hyperplanes in $\fa_\BR$, and the complement of these hyperplanes partitions $\fa$ into \b{chambers}. As we vary $\sigma$ within a given chamber, the map \eqref{eqn:stab} does not change. However, as we move from one chamber to the next, we see that many things change: some attracting directions become repelling and vice-versa, and so the ordering $F' \unlhd F$ of certain components changes. Therefore, we observe that the stable basis fundamentally depends on the chamber containing $\sigma$:
$$
\fa_\BR \ \supset \ \fC \ \ni \ \sigma
$$
Similarly, the set of line bundles such that $\wt \CL_{F'} - \wt \CL_F \in \fa^\vee_\BZ$ for any two fixed components $F \neq F'$ determines a discrete collection of affine hyperplanes in $\text{Pic}(X) \otimes \BR$. The complement of these affine hyperplanes partitions $\text{Pic}(X) \otimes \BR$ into \b{alcoves}, and it is easy to see that condition \eqref{eqn:small} only depends on the alcove containing $\CL$:
$$
\pic(X) \ \supset \ \fA \ \ni \ \CL
$$
We conclude that the stable basis map actually depends on the discrete data of a chamber $\fC \subset \fa_\BR$ and an alcove $\fA \subset \pic(X) \otimes \BR$. The following Exercise explains the relation between the stable basis for two opposite chambers: \\


\begin{exercise}
\label{ex:opposite}

The maps $\stab_\CL^\sigma$ and $\stab_{\CL^{-1}}^{\sigma^{-1}}$ are inverse transposes of each other:
\begin{equation}
\label{eqn:pairing}
\left(\stab_\CL^\sigma(\alpha), \stab_{\CL^{-1}}^{\sigma^{-1}}(\beta) \right)_X = \left(\alpha, \beta\right)_{X^A}
\end{equation}
for all $\alpha,\beta \in K_T(X^A)$. In particular, if $X^A$ consists of finitely many points, the images of these points under $\stab_\CL^\sigma$ and $\stab_{\CL^{-1}}^{\sigma^{-1}}$ determine dual bases of $K_T(X)$. 
	
\end{exercise}

\section{From stable bases to $R-$matrices}
\label{sec:defr}

\noindent Let us now recall the Nakajima quiver varieties $\CN_{\bv,\bw}$ that were introduced in the previous Chapter. While the construction in the present Section applies to more general symplectic resolutions, our notation will be specific to cyclic quiver varieties. Recall that the torus which acts on $\CN_{\bv,\bw}$ is:
$$
T_\bw \ = \ \BC^*_q \times \BC^*_t \times \prod_{i=1}^\bw \BC^*_{u_i}
$$
where the indices denote the equivariant character corresponding to each factor. The factor $\BC^*_q$ scales the symplectic form, while all the other factors preserve it. Let us pick any collection of framing vectors $\bw^1,...,\bw^k \in \nn$ and set $\bw = \bw^1+...+\bw^k$. Consider the one dimensional subtorus:
\begin{equation}
\label{eqn:torus}
A \ \cong \ \BC^* \ \stackrel{\sigma}\longrightarrow \ T_\bw, \qquad \qquad a \ \rightarrow (1,1,\underbrace{a^{N_1},...,a^{N_1}}_{\bw^1 \text{ factors}},...,\underbrace{a^{N_k},...,a^{N_k}}_{\bw^k \text{ factors}})
\end{equation}
where $N_1 \ll N_2 \ll ... \ll N_k$ are integers. The fixed locus of $\sigma$ has been described in \citep{MO}:
\begin{equation}
\label{eqn:fix0}
\CN_{\bv,\bw} \ \stackrel{\iota}\hookleftarrow \ \CN_{\bv,\bw}^A \ \cong \ \bigsqcup_{\bv = \bv^1 + ... + \bv^k} \CN_{\bv^1,\bw^1} \times ... \times \CN_{\bv^k,\bw^k}
\end{equation}
and therefore the stable basis construction for $\sigma$ and $-\sigma$ gives rise to maps \eqref{eqn:stab}:
$$
K_T\left(\CN_{\bv^1,\bw^1} \times ... \times \CN_{\bv^k,\bw^k}\right) \  \stackrel{\stab^{\sigma^{-1}}_\CL}\longrightarrow \  K_T\left(\CN_{\bv,\bw}\right) \ \stackrel{\stab^{\sigma}_\CL}\longleftarrow K_T\left(\CN_{\bv^1,\bw^1} \times ... \times \CN_{\bv^k,\bw^k}\right)
$$
for any rational line bundle $\CL$. Since the Picard group of Nakajima quiver varieties is freely generated by the tautological line bundles $\CO_1(1),..., \CO_n(1)$, the rational line bundle will be of the form $\CL = \CO(\bm) := \prod_{i=1}^n \CO_i(m_i)$, and can be identified with a vector $\bm = (m_1,...,m_n) \in \qq$. Since $A$ is one dimensional, there are only two chambers for cocharacters, positive and negative, and therefore we will use the signs $+$ and $-$ instead of $\sigma$ and $\sigma^{-1}$. By taking the direct sum of the above maps over all vectors $\bv = \bv^1 + ... + \bv^k$, we obtain:
\begin{equation}
\label{eqn:stabs}
K(\bw^1) \otimes ... \otimes K(\bw^k) \ \stackrel{\stab^-_\bm}\longrightarrow \ K(\bw) \ \stackrel{\stab^+_\bm}\longleftarrow \ K(\bw^1) \otimes ... \otimes K(\bw^k)
\end{equation}
When $k=2$, the composition of the above maps is called a \b{geometric} $R-$\b{matrix}:
$$
R_{\bm}^{+,-} \ : \ K(\bw^1) \otimes K(\bw^2) \ \longrightarrow \ K(\bw^1) \otimes K(\bw^2)
$$
\begin{equation}
\label{eqn:geomr}
R_\bm^{+,-} \ = \ \left(\stab^+_\bm\right)^{-1} \circ \stab^-_\bm
\end{equation}
by \citep{MO2}, who first introduced it. Because of the presence of the inverse map in \eqref{eqn:geomr}, the map $R^{+,-}_\bm$ will have poles corresponding to half the normal weights of the inclusion \eqref{eqn:fix0}. The usual quantum parameters of $R-$matrices are the equivariant parameters $\{u_i\}$ that arise in the matrix coefficients of the operator \eqref{eqn:geomr}, or more precisely, ratios $\frac {u_i}{u_j}$. The term ``$R-$matrix" is justified by the fact that the above endomorphisms \eqref{eqn:geomr} satisfy the quantum Yang-Baxter equation, which is nothing but saying that two specific triple products of $R_{\bm}^{+,-}$'s are both equal to the composition \eqref{eqn:stabs} for $k=3$. The proof of this statement is quite straightforward from the uniqueness of the stable basis construction, as explained in Example 4.1.9. of \citep{MO} in the cohomological case. As explained in \loccit (see also \citep{Mc} for a survey), taking arbitrary matrix coefficients of \eqref{eqn:stabs} in the last tensor factor gives rise to a family of endomorphisms:
$$
\CA^{\text{MO}} \ \hookrightarrow \ \text{End} \Big( K(\bw^1) \otimes ... \otimes K(\bw^k) \Big) 
$$
for all $k$ and all framing vectors $\bw^1,...,\bw^k \in \nn$. This family of endomorphisms can be thought of as a quasi-triangular Hopf algebra, with coproduct denoted by $\Delta_\bm$, whose category of representations has objects $\{K(\bw)\}_{\bw \in \nn}$, and whose $R-$matrices are precisely \eqref{eqn:geomr}. We will not dwell upon this algebra any further, since its properties were described in great detail in \citep{MO}, and we will mostly be concerned with an alternative construction in the next Chapter. 

\tab
We will, however, focus on a factorization property of the geometric $R-$matrices \eqref{eqn:geomr} which is particular to the $K-$theoretic case. Let us work in the more general case of a symplectic flow:
$$
\sigma \ \curvearrowright \ X \qquad \text{with fixed point set} \qquad X^\sigma \ \stackrel{\iota}\hookrightarrow \ X
$$
and let us study the corresponding stable basis maps in $T-$equivariant $K-$theory. For any line bundle $\bm \in \text{Pic}(X) \otimes \BQ$, we construct the \b{change of stable basis} map:
\begin{equation}
\label{eqn:geomr2}
R_{\bm}^{+,-} : K_T(X^\sigma) \ \longrightarrow \ K_T(X), \qquad \qquad R_\bm^{+,-} = \left(\stab^+_\bm\right)^{-1} \circ \stab^-_\bm
\end{equation}
We could call the above a ``geometric $R-$matrix" by analogy with \eqref{eqn:geomr}, but this would be rather misleading, since $R-$matrices usually act between various tensor products representations of the same algebra, while \eqref{eqn:geomr2} is a general map. This map represents the change in stable basis as we replace the negative direction $\sigma^{-1}$ with the positive direction $\sigma$. However, we can achieve the same result by changing the line bundle $\bm$ in a prescribed direction $\bth \in \text{Pic}(X)$, which is required to be compatible with the flow $\sigma$ in the sense of \eqref{eqn:compatibility}. To be precise, we consider the composition from left to right:
\begin{equation}
\label{eqn:geomroot}
R^+_{\bm,\bm+\e\bth} \ : \ K_T(X^\sigma) \ \stackrel{\stab^+_{\bm+\e \bth}}\longrightarrow \ K_T(X) \ \stackrel{\stab^+_{\bm}}\longleftarrow \ K_T(X^\sigma)
\end{equation}
where $\e$ is a very small positive rational number. We interpret $\bm+\e\bth$ as the next alcove in $\text{Pic}(X) \otimes \BQ$ after the one containing $\bm$, as we move in the direction of the vector $\bth$. The map \eqref{eqn:geomroot} will be called an \b{infinitesimal change of stable basis}. We can consider the following infinite product of maps \eqref{eqn:geomroot}:
$$
\left(\stab^+_\bm\right)^{-1} \circ \stab^+_\infty \ = \ \prod_{r \in \BQ_+}^{\rightarrow} R^+_{\bm+r\bth,\bm+(r+\e)\bth} \ : \ K_T(X^\sigma) \ \longrightarrow \ K_T(X^\sigma)
$$ 
Intuitively, the infinite product goes over the positive half-line of slope $\bth$ starting at $\bm$, and picks up a factor every time we encounter a wall between two alcoves $\fA$ and $\fA'$ inside $\text{Pic}(X) \otimes \BR$. The corresponding factor is simply the change of stable basis between the two alcoves $\fA$ and $\fA'$. In particular, we encounter finitely many walls and so there will be finitely many non-trivial factors in the above product. Then we may use formula \eqref{eqn:stabinfinity} to obtain:
\begin{equation}
\label{eqn:bzzz1}
\left(\stab^+_\bm\right)^{-1} = \prod_{r \in \BQ_+}^{\rightarrow} R^+_{\bm+r\bth,\bm+(r+\e)\bth} \circ \left( \frac {\iota^*}{[N^-_{X^\sigma \subset X}]} \right)
\end{equation}
We can do the same constructions for the negative stable basis, i.e. define:
\begin{equation}
\label{eqn:geomroot2}
R^-_{\bm+\e\bth,\bm} \ : \ K_T(X^\sigma) \ \stackrel{\stab^-_{\bm}}\longrightarrow \ K_T(X) \ \stackrel{\stab^-_{\bm+r\bth}}\longleftarrow \ K_T(X^\sigma)
\end{equation}
which implies:
$$
\left(\stab^-_\infty\right)^{-1} \circ \stab^-_\bm \ = \ \prod_{r \in \BQ_-}^{\rightarrow} R^-_{\bm+(r+\e)\bth,\bm + r\bth} \ : \ K_T(X^\sigma) \ \longrightarrow \ K_T(X^\sigma)
$$ 
Using formula \eqref{eqn:stabinfinity} for the negative $\sigma^{-1}$ direction, we obtain:
\begin{equation}
\label{eqn:bzzz2}
\stab^-_\bm = (\iota^*)^{-1}\left([N^+_{X^\sigma \subset X}] \cdot \prod_{r \in \BQ_-}^{\rightarrow} R^-_{\bm+(r+\e)\bth,\bm + r\bth} \right)
\end{equation}
Multiplying \eqref{eqn:bzzz1} and \eqref{eqn:bzzz2} together we obtain the factorization of $R-$matrices constructed by \citep{MO2}: \\

\begin{lemma}
\label{lem:stablefac}
The change of stable basis \eqref{eqn:geomr2} factors in terms of the infinitesimal change maps \eqref{eqn:geomroot} and \eqref{eqn:geomroot2}:
\begin{equation}
\label{eqn:stablefac}
R_\ebm^{+,-} = \prod_{r \in \BQ_+}^{\rightarrow} R^+_{\ebm+r\bth,\ebm+(r+\e)\bth} \circ \frac {[N^+_{X^\sigma \subset X}]}{[N^-_{X^\sigma \subset X}]} \circ \prod_{r \in \BQ_-}^{\rightarrow} R^-_{\ebm+(r+\e)\bth,\ebm + r\bth} 
\end{equation}
where the middle term in the composition is the operator of multiplication by $\frac {[N^+_{X^\sigma \subset X}]}{[N^-_{X^\sigma \subset X}]}$. \\

\end{lemma}

\section{Fixed loci and the isomorphism}
\label{sec:iso}

\noindent We will now apply some of the considerations of the previous Section to the case when $X = \CN_{\bv,\bw}$ is a cyclic quiver variety and the fixed locus was seen in \eqref{eqn:fix0} to be $X^\sigma = \bigsqcup \CN_{\bv^1,\bw^1} \times ... \times \CN_{\bv^k,\bw^k}$. The middle term of formula \eqref{eqn:stablefac} has to do with the normal bundles to the fixed locus, and we will now compute these explicitly. Let us start from the Kodaira-Spencer presentation \eqref{eqn:ks1} of the tangent space to $X$, which was described in Section \ref{sec:gieseker}:
$$
T_{\CF}\CN_{\bv,\bw} = - \chi \left(\CF,\CF(-\infty) \right)
$$
If we restrict the above to the fixed locus $X^\sigma$, it means that we are considering sheaves of the form $\CF = \CF_1 \oplus ... \oplus \CF_k$, where each $\CF_i$ has rank prescribed by the framing vector $\bw^i$. Since the Euler characteristic is additive, we see that:
$$
T_{\CF_1 \oplus ... \oplus \CF_k}\CN_{\bv,\bw} = - \bigoplus_{a,b=1}^k \chi\left(\CF_a,\CF_b(-\infty) \right)
$$
The tangent space to the fixed locus consists of those summands with $a=b$, while:
\begin{equation}
\label{eqn:floc1}
N^+_{X^\sigma \subset X} |_{\CF_1 \oplus ... \oplus \CF_k} = - \bigoplus_{1\leq a < b \leq k} \chi\left(\CF_a,\CF_b(-\infty) \right)
\end{equation}
\begin{equation}
\label{eqn:floc2}
N^-_{X^\sigma \subset X} |_{\CF_1 \oplus ... \oplus \CF_k} = - \bigoplus_{1\leq b < a \leq k} \chi\left(\CF_a,\CF_b(-\infty) \right)
\end{equation}
The reason for this is that all characters $\chi(\CF_a,\CF_b(-\infty))$ contain a factor of $\frac {u_b}{u_a}$, which is attracting or repelling with respect to the torus \eqref{eqn:torus} depending on whether $a<b$ or $b<a$. Therefore, we are led to consider the sheaf:
\begin{equation}
\label{eqn:defe}
\xymatrix{\CE \ar@{.>}[d] \\ \CN_{\bv^1,\bw^1} \times  \CN_{\bv^2,\bw^2}} \quad \text{with fibers \quad} \CE|_{\CF^1, \CF^2} = \text{Ext}^1(\CF^1, \CF^2(-\infty))^{\BZ/n\BZ}
\end{equation}
which is a vector bundle because the corresponding $\Hom$ and $\Ext^2$ spaces vanish. Note that the Kodaira-Spencer isomorphism \eqref{eqn:ks1} implies that $\CE|_\Delta \cong T\CN_{\bv,\bw}$. Then formulas \eqref{eqn:floc1} and \eqref{eqn:floc2} can be interpreted as:
\begin{equation}
\label{eqn:milky}
N^+ = \sum_{1\leq a < b \leq k} \pi_{a,b}^*(\CE) \qquad \qquad N^- = \sum_{1\leq a < b \leq k} \pi_{b,a}^*(\CE)
\end{equation}
where $\pi_{a,b} : \CN_{\bv^1,\bw^1} \times ... \times \CN_{\bv^k,\bw^k} \rightarrow \CN_{\bv^a,\bw^a} \times \CN_{\bv^b,\bw^b}$ is the standard projection, and we use the abbreviated notation $N^+$ and $N^-$ for the attracting and repelling parts of the normal bundle to the fixed locus $X^\sigma \hookrightarrow X$. Instead of using the definition \eqref{eqn:defe}, we will prefer to use the following formula for the $K-$theory class of $\CE$, which is proved by analogy with \eqref{eqn:tanmod0}:
\begin{equation}
\label{eqn:league}
\CE \ = \ \sum_{j=1}^{\bw^1} \frac {\CV^2_j}{qu_j}  + \sum_{j=1}^{\bw^2} \frac {u_j}{q \CV^1_j} + \sum_{i=1}^n \left(\frac {\CV^2_{i+1}}{qt \CV^1_i} + \frac {t\CV^2_{i-1}}{q\CV^1_i} - \frac {\CV^2_i}{\CV^1_i} - \frac {\CV^2_i}{q^2\CV^1_i} \right)
\end{equation}
where $\{\CV^1_i\}_{1\leq i \leq n}$ and $\{\CV^2_i\}_{1\leq i \leq n}$ denote the tautological vector bundles pulled back from the two factors of $\CN_{\bv^1,\bw^1} \times \CN_{\bv^2,\bw^2}$. We will now use the above formulas to analyze the slope $\infty$ stable basis map \eqref{eqn:stabinfinity} with respect to the cocharacters $\sigma$ or $\sigma^{-1}$ defined in \eqref{eqn:torus}, for $k=2$:
\begin{equation}
\label{eqn:cretu}
K(\bw^1) \otimes K(\bw^2) \ \stackrel{\stab_\infty^\pm}\longrightarrow \ K(\bw) \qquad \qquad \stab^\pm_\infty(\alpha) = (\iota^*)^{-1}\left(\alpha \cdot [N^\mp]\right)
\end{equation}
where $\iota:\CN_{\bv^1,\bw^1} \times \CN_{\bv^2,\bw^2} \hookrightarrow \CN_{\bv,\bw}$ is the inclusion of the torus fixed locus. The two maps \eqref{eqn:cretu} each represent half of the geometric $R-$matrix \eqref{eqn:geomr} at slope $\infty$. They iteratively allow us to break the $K-$theory of Nakajima quiver varieties into tensor products of $K-$theories of quiver varieties with smaller framing vectors, until we reach those with framing given by the simple roots $\bs^i = (0,...,0,1,0,...,0)$. When $n=1$, this is precisely the decomposition hinted at in \eqref{eqn:upgrade}. However, we will now show that \eqref{eqn:cretu} can be promoted from a map of vector spaces to a map of representations. This result is implicitly contained in \citep{Ts1}, when $n=1$. \\

\begin{proposition}
\label{prop:sasha}

For any choice of framing vertices $\ebw = \ebw^1 + \ebw^2$, the map $\estab^\pm_\infty$ of \eqref{eqn:cretu} is a $\UU-$intertwiner, with respect to the action: 
$$
K(\ebw^1) \otimes K(\ebw^2) \ \curvearrowleft \ \UU \ \curvearrowright \ K(\ebw)
$$
of Theorem \ref{thm:act0}. The tensor product is made into a $\UU$ module with respect to the coproduct $\Delta^{\emph{op}}$ or $\Delta$ of \eqref{eqn:copquant}, depending on whether the sign is $+$ or $-$. \\

\end{proposition}

\begin{proof} We will prove the case of $\stab^+_\infty$, and leave the analogous case of $\stab^-_\infty$ to the interested reader. We need to prove that the following diagrams are commutative:
\begin{equation}
\label{eqn:covet}
\xymatrix{K(\bw^1) \otimes K(\bw^2) \ar[r]^-{\stab^+_\infty} \ar[d]_{\Delta^\op(\ph^\pm_i(z))} &  K(\bw) \ar[d]^{\ph^\pm_i(z)} \\
K(\bw^1) \otimes K(\bw^2) \ar[r]^-{\stab^+_\infty} & K(\bw)} \qquad \qquad 
\xymatrix{K(\bw^1) \otimes K(\bw^2) \ar[r]^-{\stab^+_\infty} \ar[d]_{\Delta^\op(e^\pm_i(z))} &  K(\bw) \ar[d]^{e^\pm_i(z)} \\
K(\bw^1) \otimes K(\bw^2) \ar[r]^-{\stab^+_\infty} & K(\bw)} 
\end{equation}
The first diagram is almost immediate, and it follows from the fact that $\ph_i^\pm(z)$ are operators of multiplication by the tautological class:
$$
\gamma_\bw^i(z) \ := \ \overline{\frac {\zeta\left( \frac zX \right)}{\zeta\left( \frac Xz \right)}} \prod_{1\leq j \leq \bw}^{u_j \equiv i} \frac {\left[ \frac {u_j}{qz} \right]}{\left[ \frac {z}{qu_{j}} \right]} \ \in \ K(\bw)[[z^{\mp 1}]]
$$
Since the restriction to the fixed locus $\iota : \CN_{\bv^1,\bw^1} \times \CN_{\bv^2,\bw^2} \hookrightarrow \CN_{\bv,\bw}$ respects tautological classes, we see that $\iota^*(\gamma_\bw^i(z)) =  \gamma_{\bw^1}^i(z) \boxtimes \gamma_{\bw^2}^i(z)$. Therefore, the first diagram of \eqref{eqn:covet} follows from the identities:
$$
\ph^\pm_i(z)\left(\stab^+_\infty(\alpha) \right) = \ph^\pm_i(z)\Big( (\iota^*)^{-1}\left(\alpha \cdot [N^-]\right) \Big) = \gamma_\bw^i(z) \cdot (\iota^*)^{-1}\left(\alpha \cdot [N^-]\right) = 
$$
$$
= (\iota^*)^{-1}\Big(\iota^* \left( \gamma_\bw^i(z) \right) \cdot \alpha \cdot [N^-] \Big) = (\iota^*)^{-1}\Big(\gamma_{\bw^1}^i(z) \cdot \gamma_{\bw^2}^i(z) \cdot \alpha \cdot [N^-] \Big) = 
$$
$$
= (\iota^*)^{-1}\Big(\ph^\pm_i(z) \otimes \ph^\pm_i(z) \left(\alpha \right) \cdot [N^-] \Big) = \stab^+_\infty \left( \Delta(\ph^\pm_i(z))(\alpha) \right)
$$
for any $\alpha \in K(\bw^1) \otimes K(\bw^2)$. Let us now turn to proving the second diagram of \eqref{eqn:covet}, and we will use the following explicit formula for the coproduct:
$$
\Delta^\op \left(e^+_i(z)\right) = e_i^+(z) \otimes \ph_i^+(z) + 1 \otimes e_i^+(z) \quad \qquad \Delta^\op\left(e^-_i(z)\right) = e_i^-(z) \otimes 1 + \ph_i^-(z) \otimes e_i^-(z)
$$
We will use formulas \eqref{eqn:formula1} and \eqref{eqn:formula2} for the action of the series $e_i^\pm(z)$	on $K(\bw)$:
\begin{equation}
\label{eqn:italian}
e_{i}^+(z) \left( \ov{f(X)} \right) = \ov{f( X - z ) \zeta\left( \frac zX \right)} \cdot \prod_{1\leq j \leq \bw}^{u_j \equiv i} \Big[\frac {u_{j}}{qz} \Big]
\end{equation}
\begin{equation}
\label{eqn:spiderman}
e_{i}^-(z) \left( \ov{f(X)} \right) = \ov{f( X + z ) \zeta\left( \frac Xz \right)^{-1}} \cdot \prod_{1\leq j \leq \bw}^{u_j \equiv i} \Big[\frac {z}{qu_j} \Big]^{-1} 
\end{equation}
where the integral is replaced by a series in $z$ because we replace the single operators $e_{i,d}^\pm$ by their generating series $e_i^\pm(z) = \sum_{d\in \BZ} e_{i,d}^\pm z^{-d}$. We will henceforth denote tautological classes on $K(\bw)$, $K(\bw^1)$, $K(\bw^2)$ by $\ov{f(X)}$, $\ov{f(X^1)}$, $\ov{f(X^2)}$. Here $X$, $X^1$, $X^2$ are three alphabets of variables such that $X = X^1 + X^2$. The repelling normal bundle can be written in terms of tautological classes by using \eqref{eqn:milky} and \eqref{eqn:league}:
$$
N^- \ = \ \prod_{j=1}^{\bw^2} \frac {\CV^1_j}{qu_j}  + \prod_{j=1}^{\bw^1} \frac {u_j}{q \CV^2_j} + \sum_{k=1}^n \left(\frac {\CV^1_{k+1}}{qt \CV^2_k} + \frac {t\CV^1_{k-1}}{q\CV^2_k} - \frac {\CV^1_k}{\CV^2_k} - \frac {\CV^1_k}{q^2\CV^2_k} \right) \qquad \Longrightarrow
$$
\begin{equation}
\label{eqn:repel}
\Longrightarrow \qquad [N^-] \ = \ \zeta\left(\frac {X^2}{X^1}\right) \prod_{1\leq j \leq \bw^2} \Big[\frac {X^1}{qu_j} \Big] \prod_{1\leq j \leq \bw^1} \Big[ \frac {u_j}{q X^2} \Big] 
\end{equation}
Moreover, in terms of such tautological classes, the map \eqref{eqn:stabinfinity} takes the form:
\begin{equation}
\label{eqn:melc}
\stab^+_\infty\left(\ov{f(X^1) g(X^2)}\right) = \ov{\sym \left( f(X^1)g(X^2)  \zeta\left( \frac {X^2}{X^1} \right) \prod_{i=1}^{\bw^2} \Big[\frac {X^1}{qu_i}\Big] \prod_{i=1}^{\bw^1} \Big[\frac {u_i}{qX^2}\Big] \right)} \qquad
\end{equation}
where the notation $\sym$ involves symmetrizing the variables $X^1$ and $X^2$ in all possible ways. Therefore, the right hand side of \eqref{eqn:melc} is a symmetric function in $X = X^1 + X^2$. To prove the equality \eqref{eqn:melc}, one can observe that $\iota^*(\text{RHS}) \cdot [N^-]^{-1}$ gives precisely $\ov{f(X^1) g(X^2)}$. Combining \eqref{eqn:melc} with \eqref{eqn:italian}, we obtain $e_i^+(z) \left( \stab^+_\infty\left(\ov{f(X^1) g(X^2)} \right) \right) = $
$$
\ov{\sym \ f(X^1-z)g(X^2) \zeta\left( \frac {X^2}{X^1-z} \right) \prod_{i=1}^{\bw^2} \Big[\frac {X^1-z}{qu_i}\Big] \prod_{i=1}^{\bw^1} \Big[\frac {u_i}{qX^2}\Big]  \zeta \left( \frac z{X^1} \right)\zeta \left(\frac z{X^2} \right)} \prod_{i=1}^\bw \Big[\frac {u_{i}}{qz} \Big] +
$$
$$
\ov{\sym \ f(X^1)g(X^2-z) \zeta\left(\frac {X^2-z}{X^1} \right) \prod_{i=1}^{\bw^2} \Big[\frac {X^1}{qu_i}\Big] \prod_{i=1}^{\bw^1} \Big[\frac {u_i}{q(X^2-z)}\Big] \zeta \left( \frac z{X^1} \right)\zeta \left(\frac z{X^2} \right)} \prod_{i=1}^\bw \Big[\frac {u_{i}}{qz} \Big]
$$	
Meanwhile, $\stab^+_\infty\left(\Delta^{\op}(e_i^+(z)) \left(\ov{f(X^1)} \cdot \ov{g(X^2)} \right) \right)$ equals:
$$
\stab^+_\infty\Big(e^+_i(z)\ov{f(X^1)} \cdot \ph^+_i(z) \ov{g(X^2)} + \ov{f(X^1)} \cdot e_i^+(z) \ov{g(X^2)} \Big) = \stab^+_\infty (\tau)
$$	
where $\quad \tau = $
$$
\ov{f( X^1 - z ) \zeta\left( \frac z{X^1} \right) g(X^2) \frac {\zeta\left( \frac z{X^2} \right)}{\zeta\left( \frac {X^2}z \right)}} \prod_{i=1}^{\bw^1} \Big[\frac {u_{i}}{qz} \Big] \prod_{i=1}^{\bw^2} \frac {\Big[\frac {u_{i}}{qz} \Big]}{\Big[\frac z{qu_{i}} \Big]} + \overline{f(X^1) g( X^2 - z ) \zeta\left( \frac z{X^2} \right)} \prod_{i=1}^{\bw^2} \Big[\frac {u_{i}}{qz} \Big] 
$$
Once again using \eqref{eqn:melc}, we obtain precisely $e_i^+(z) \left( \stab^+_\infty\left( \ov{f(X^1) g(X^2)} \right) \right)$, as required. The corresponding formulas for $\ph_i^-(z)$ and $e_i^-(z)$ are proved analogously, so we will not review them here in the interest of space. \\
\end{proof}

\section{Lagrangian bases and correspondences}
\label{sec:lagrangian}

\noindent One of the strengths of the stable basis construction is that it can be considered with respect to any symplectic flow $\sigma \curvearrowright X$, and various $\sigma$ will give rise to various stable basis maps. In the case of the cyclic quiver, we have already seen that the choice of \eqref{eqn:torus} gives rise to products of quiver varieties as fixed points of a bigger quiver variety $\CN_{\bv,\bw}$. Let us consider instead the following one-dimensional subtorus:
\begin{equation}
\label{eqn:snow}
\BC^* \ \stackrel{\sigma}\longrightarrow \ T_\bw \qquad \qquad \qquad  a \rightarrow (1,a,a^{N_1},...,a^{N_{\bw}}) 
\end{equation}
where we assume $N_1 \ll ... \ll N_\bw$ are integers. We will encounter this flow in Chapter \ref{chap:pieri}. The fixed point set of $\CN_{\bv,\bw}$ with respect to this torus is minimal, in the sense that it only consists of the fixed points $I_\bla$, as $\bla$ ranges over $\bw-$partitions. Then the stable basis construction gives rise to elements $s_\bla^{\sigma^{\pm 1} ,\CL} \ \in \ K_{\bv,\bw}$. Since $\CL = \CO(\bm)$ for some $\bm = (m_1,...,m_n) \in \qq$, we will write the above as:
\begin{equation}
\label{eqn:stable}
s_\bla^{\pm,\bm} \ \in \ K_{\bv,\bw} 
\end{equation}
and call these the \b{positive} and \b{negative} stable bases, respectively. They are interesting inasmuch as $\bm$ is not integral, since:
\begin{equation}
\label{eqn:linebundle}
s_{\bla}^{\pm, \bm+\bk} \ = \ s_\bla^{\pm, \bm} \cdot \CO(\bk) \qquad \forall \ \bk \in \zz
\end{equation}
The name ``stable basis" reflects the fact that as $\bla$ ranges over all $\bw-$partitions, each of the two collections \eqref{eqn:stable} determines a basis of $K(\bw)_\loc$. Even more so, they give an integral basis of the $\BF_\bw-$module of $K-$theory classes supported on the attracting subvariety of $X$. The basis elements \eqref{eqn:stable} are precisely the classes which appear in Theorem \ref{thm:pieri}. In the very particular case:
$$
n = 1 \quad \text{and} \quad \bw = (1), \qquad \text{when} \qquad K(1) \ \cong \ \text{Fock space} 
$$
the basis elements $s_\la^{\pm,0}$ correspond to modified Schur functions $s_\lambda$ and their duals with respect to the Macdonald inner product. By \eqref{eqn:linebundle}, we have $s_\la^{+,k} = \nabla^k(s_\lambda)$, where the operator $\nabla$ of tensoring by $\CO(1)$ was first discovered in combinatorics by Bergeron and Garsia. Finally, the analogue of \eqref{eqn:stabinfinity} says that $s_\la^{+,\infty}$ should be interpreted as the class of the fixed point $I_\lambda$, renormalized appropriately. Following Haiman, this class corresponds to modified Macdonald polynomials in combinatorics, so we conclude that in the case $n=1$ and $\bw = (1)$, the basis $s_\la^{+,m}$ interpolates between modified Schur functions and modified Macdonald polynomials. 

\tab 
Beside the stable basis, Theorem \ref{thm:pieri} also deals with certain operators denoted therein by $P_{\pm[i;j)}^\bm$ and $Q_{\pm[i;j)}^\bm$. We will construct these operators in the next Chapter as arising from certain correspondences between Nakajima cyclic quiver varieties. Let us say a few words about such operators in the more general setup of two symplectic varieties $X$ and $Y$, both with actions of a torus $T$, in which we fix a one-dimensional symplectic flow $\sigma$. We will consider Lagrangian correspondences:
\begin{equation}
\label{eqn:lag}
\Gamma \in K_T(W) \qquad \qquad \text{where} \qquad \qquad \xymatrix{
& W \ar[ld]_{\pi_1} \ar[rd]^{\pi_2} & \\ 
Y & & X} 
\end{equation}
We always assume that $\pi_1$ is proper with respect to the $T$ fixed locus, and that $\pi_2$ is a lci morphism, so the above correspondence induces an operator:
$$
f \ : \ K_T(X) \longrightarrow K_T(Y) \qquad \qquad f(\alpha) \ = \ \pi_{1*} \left(\Gamma \cdot \pi_2^*(\alpha) \right)
$$ 
We would like to ask how this operator interacts with the stable basis, i.e. whether we can complete the following diagram by a horizontal map that makes it commutative:
\begin{equation}
\label{eqn:diag}
\xymatrix{
& & W \ar[ld] \ar[rd] & & \\ 
& Y & & \ar@{.>}[ll]_f X & \\
\oZ^\sigma \ar[ur] \ar[dr] & & & & \oZ^\sigma \ar[ul] \ar[dl] \\
& Y^\sigma \ar@{.>}[uu]_{\stab_\CL^\sigma} & & X^\sigma \ar@{.>}[uu]^{\stab_\CL^\sigma} \ar@{.>}[ll]^{f^\sigma} & \\ 
& & W^\sigma \ar[lu] \ar[ru] & &} 
\end{equation}
In diagram \eqref{eqn:diag}, the dotted arrows represent maps between $K-$theory that are induced by the correspondences in the corners, and $W^\sigma$ is a correspondence induced by $W$ on the fixed locus. Computing $W^\sigma$, or equivalently the map $f^\sigma$, is one of the main tasks in working with stable bases (in fact, our main Theorem \ref{thm:pieri} is precisely one such computation). In localized $K-$theory, the above diagram requires:
$$
f^\sigma \ = \ \left(\stab^\sigma_\CL\right)^{-1} \circ f \circ \stab^\sigma_\CL 
$$
which determines $f^\sigma$ uniquely. Recall that the attracting set is the locus of points which have a well-defined limit in the direction of $\sigma$. In order for $f^\sigma$ to be well-defined in integral $K-$theory, one needs to decide whether the correspondence $f$ takes classes supported on the attracting set to classes supported on the attracting set. A reasonable geometric condition for this to happen is to have:
\begin{equation}
\label{eqn:reasonable}
\pi_1\left(\pi_2^{-1}(\text{Attracting set})\right) \ \subset \ \text{Attracting set} \end{equation}
Therefore, we need to understand whether the correspondences we write down respect property \eqref{eqn:reasonable} with respect to a certain one dimensional flow on our varieties. Let us describe these loci for Nakajima quiver varieties. When $X = \CN_{v,1}$ is the Hilbert scheme of $v$ points in the plane and $\sigma$ is the one dimensional symplectic flow in the direction of the equivariant parameter $t$, it is well-known that:
$$
\text{Attracting set of }\sigma \ = \ \Big \{\text{ideals } I \subset \BC[x,y] \text{ supported on the line } x = 0 \Big \}
$$
$$
\text{Repelling set of }\sigma \ = \ \Big \{\text{ideals } I \subset \BC[x,y] \text{ supported on the line } y = 0 \Big \}
$$
where the coordinates of the plane are acted on by $(q,t) \curvearrowright (x,y) = (qtx,qt^{-1}y)$. In terms of quadruples of matrices $(X,Y,A,B)$, the condition that an ideal be supported on the line $x=0$ (respectively $y=0$) corresponds to the matrix $X$ (respectively $Y$) being nilpotent. In the higher rank case of moduli spaces of framed sheaves $\CN_{v,w}$, we work with respect to the torus \eqref{eqn:snow}, and so we have:
\begin{equation}
\label{eqn:att0}
\text{Attracting set of }\sigma \ = \ \Big\{ \text{rank }w\text{ sheaves }\CF_w \text{ endowed with a filtration}
\end{equation}
$$
\CF_1 \subset ... \subset \CF_w \ \text{ s.t. } \ \CI_k = \CF_k/\CF_{k-1} \ \text{ has } \ \CI_k|_\infty \cong \omega_k \CO_\infty \ \text{ and supp } \CI_k = \{x=0\} \Big \}
$$
\begin{equation}
\label{eqn:rep0}
\text{Repelling set of }\sigma \ = \ \Big\{ \text{rank }w\text{ sheaves }\CF_w \text{ endowed with a filtration}
\end{equation}
$$
\CF_1 \subset ... \subset \CF_w \ \text{ s.t. } \ \CI_k = \CF_{k}/\CF_{k-1} \ \text{ has } \ \CI_k|_\infty \cong \omega_{w-k+1}  \CO_\infty \ \text{ and supp } \CI_k = \{y=0\} \Big \}
$$
where we write $W = \BC \cdot \omega_1 \oplus ... \oplus \BC \cdot \omega_w$ for the basis of the framing space. In terms of quadruples of matrices $(X,Y,A,B)$, the above conditions can be rephrased as:
\begin{equation}
\label{eqn:att}
\text{Attracting set of }\CN_{v,w} \text{ w.r.t. }\sigma \ = \ \Big\{ (X,Y,A,B) \Big\}
\end{equation}
such that there exists a filtration $V^{1} \subset ... \subset V^{w}$ preserved by $(X,Y,A,B)$, with $X$ nilpotent and $A\cdot \omega_k$ generating $V^{k}/V^{k-1}$, and:
\begin{equation}
\label{eqn:rep}
\text{Repelling set of }\CN_{v,w} \text{ w.r.t. }\sigma \ = \ \Big\{ (X,Y,A,B) \Big\}
\end{equation}
such that there exists a filtration $V^{1} \subset ... \subset V^{w}$ preserved by $(X,Y,A,B)$, with $Y$ nilpotent and $A \cdot \omega_{w-k+1}$ generating $V^{k}/V^{k-1}$. Finally, Nakajima cyclic quiver varieties $\CN_{\bv,\bw}$ are $\BZ/n\BZ$ fixed loci of Gieseker moduli spaces. Therefore, the attracting and repelling sets are simply obtained by taking the $\BZ/n\BZ$ fixed points of \eqref{eqn:att} and \eqref{eqn:rep}, respectively, and replacing $(X,Y,A,B)$ by $(X_i,Y_i,A_i,B_i)_{1 \leq i \leq n}$. In the case of the cyclic quiver, the condition that $X$ is nilpotent must be replaced by the condition that any vector $v \in V_i$ be annihilated by the composition $X_{N} \circ X_{N-1} \circ ... \circ X_i$ for large enough $N$. \\ 


\begin{exercise}
\label{ex:correspondence}

The correspondence $\fZ_i$ of Section \ref{sec:simple} satisfies property \eqref{eqn:reasonable} with respect to either \eqref{eqn:att} or \eqref{eqn:rep}. Hence the operators $e_{i,d}^\pm$ of \eqref{eqn:simpleop} are Lagrangian: they take positive/negative stable bases to integral combinations of stable bases. 

\end{exercise}






\chapter{The Shuffle Algebra}
\label{chap:shuffle}


\section{Definition of the shuffle algebra}
\label{sec:shuffle}


\noindent The purpose of this Section is to review the basic construction of shuffle algebras. Consider a vector space $V$ and define the space of symmetric tensors as:
\begin{equation}
\label{eqn:isovect}
T^{\otimes} V \ \supset \ \sym \ V \ \stackrel{\cong}\longrightarrow \ \frac {T^{\otimes} V}{(v_1 \otimes v_2 - v_2 \otimes v_1)}
\end{equation}
where $T^{\otimes} V$ denotes the tensor algebra of $V$, graded according to the number of tensor factors, and $\sym \ V$ denotes the subalgebra of symmetric tensors. The isomorphism in \eqref{eqn:isovect} is normalized so that its inverse is:
$$
v_1 \otimes ... \otimes v_k \ \longrightarrow \ \sym \left[ v_1 \otimes ... \otimes v_k \right] \ := \ \sum_{\sigma \in S(k)} v_{\sigma(1)} \ ... \ v_{\sigma(k)}
$$
When dealing with symmetric tensors, we will forgo the $\otimes$ sign between them.
The isomorphism \eqref{eqn:isovect} is only one of vector spaces, since the usual tensor product of two symmetric tensors need not be symmetric. Instead, we define the following \textbf{shuffle product} as the algebra structure on $\sym \ V$ which corresponds to the usual multiplication of tensors under the isomorphism \eqref{eqn:isovect}:
$$
\sym \ V \otimes \sym \ V \ \stackrel{\sh}\longrightarrow \ \sym \ V
$$
where:
\begin{equation}
\label{eqn:shufprod}
\left( v_1 \ ... \ v_k \right) \sh \left( v_{k+1} \ ... \ v_{k+l} \right) = \sym \left( \frac {v_1 \ ... \ v_k v_{k+1} \ ... \ v_{k+l}}{k! \cdot l!} \right)
\end{equation}
Note that we can take the sum in the RHS over all permutations $\sigma \in S(k+l)$ because the tensors $v_1 ...  v_k$ and $v_{k+1}  ... v_{k+l}$ are assumed symmetric to begin with. Because of the factorials in the denominator, there are ${k+l \choose k}$ summands $v_{\sigma(1)}...v_{\sigma(k+l)}$ in \eqref{eqn:shufprod}. Note that certain definitions of shuffle algebras do not require the original tensors to be symmetric, as opposed from ours. We will henceforth refer to $\sym \ V$ as the \textbf{shuffle algebra} associated to $V$. It is a bialgebra with coproduct given by:
\begin{equation}
\label{eqn:shufcop}
\Delta(v_1 \ ... \ v_k) = \sum_{i=0}^k \left( v_1 \ ... \ v_i \right) \otimes \left( v_{i+1} \ ... \ v_k \right)
\end{equation}
as well as unit and counit corresponding to the empty tensor. The map $S(v_1...v_k) = (-1)^k v_k...v_1$ is an antipode, although it is rather trivial in our case since our tensors are all symmetric. The tautological pairing between $V$ and $V^*$ gives rise to the following pairing between the corresponding shuffle algebras:
\begin{equation}
\label{eqn:pairshuf0}
\sym \ V \otimes \sym \ V^* \ \longrightarrow \ \BC
\end{equation}
$$
\left( v_1 \ ... \ v_k, \lambda_1 \ ... \ \lambda_k \right) = \lambda_1(v_1)...\lambda_k(v_k)
$$
It is straightforward to check that \eqref{eqn:pairshuf0} satisfies \eqref{eqn:bialg}, and is therefore a bialgebra pairing. We may then construct the \textbf{double shuffle algebra} as its Drinfeld double:
\begin{equation}
\label{eqn:shufdoub}
\sym \ V \otimes \sym \ V^*
\end{equation}
according to the relations imposed in \eqref{eqn:dd}. The shuffle algebra that features in this paper is based on the above construction for $V = \{z_i^d\}_{1\leq i \leq n}^{d \in \BZ}$ and hence:
$$
\sym \ V \ ``=" \ \bigoplus_{\bk = (k_1,...,k_n) \in \nn} \BF[...,z^{\pm 1}_{i1},...,z^{\pm 1}_{ik_i},... ]^\sym_{1\leq i \leq n}
$$
where $\BF = \BZ[q^{\pm 1},t^{\pm 1}]$. The variable $z_{ia}$ will be thought to have \b{color} $i$ modulo $n$, and we will only require our Laurent polynomials to be symmetric with respect to variables of a given color. We enlarge the vector space of symmetric polynomials to include rational functions:
$$
\widetilde{\sym} \ V \ := \ \bigoplus_{\bk = (k_1,...,k_n) \in \nn} \tBF(...,z_{i1},...,z_{ik_i},...)^\sym_{1\leq i \leq n}
$$
where $\tBF = \BQ(q,t)$, and endow it with the following \b{shuffle product}:
\begin{equation}
\label{eqn:shuffle}
F * F' = \textrm{Sym} \left[ \frac {F(...,z_{ia},...) F'(...,z_{jb},...)}{\bk! \cdot \bl!} \prod^{1\leq i \leq n}_{1\leq a\leq k_i} \ \prod^{1\leq j \leq n}_{k_j+1 \leq b \leq k_j+l_j} \zeta \left( \frac {z_{ia}}{z_{jb}} \right) \right]
\end{equation}
where $\bk! = \prod_{i=1}^n k_i!$. Note that the direct analogue of \eqref{eqn:shufprod} would have involved setting $1$ instead of twisting by the rational function $\zeta$ in \eqref{eqn:shuffle}. This twist will be explained in Section \ref{sec:k}, where we recall how \eqref{eqn:shuffle} is related to the $K-$theoretic Hall product as constructed by \citep{SV}. 

\tab 
The reason why we must work with rational functions $F$ is that the function $\zeta$ will produce poles in \eqref{eqn:shuffle}. However, the poles of the form $z_{ia} - z_{ib}$ will be removed by the symmetrization, and effectively the only poles we will have to contend with are of the form $qz_{ia} - q^{-1}z_{ib}$. Therefore, we will work with the following subspaces of $\widetilde{\sym} \ V$:
\begin{equation}
\label{eqn:defshuf}
\CS^+_\bk \ := \ \left\{F(...,z_{ia},...) = \frac {r(...,z_{ia},...)}{\prod^{1\leq i \leq n}_{1\leq a \neq b \leq k_i} (qz_{ia} - q^{-1}z_{ib})} \right\}
\end{equation}
where $r(...,z_{i1},...,z_{ik_i},...)^{1 \leq i \leq n}_{1 \leq a \leq k_i}$ is any symmetric Laurent polynomial that satisfies:
\begin{equation}
\label{eqn:shufelem}
r(...,q^{-1},t^{\pm 1},q,...) \ \equiv \ 0
\end{equation}
The above are called the \b{wheel conditions}, where we are allowed to set any three variables of colors $i,i\pm 1, i$ equal to $q^{-1},t^{\pm 1},q$. They are a natural generalization of the construction of \citep{FO} in the elliptic finite-dimensional case. Imposing the wheel conditions only for variables of a specific color is natural, due to the parameters $q$ and $t$ having color $0$ and $1$, respectively. Then the \b{positive shuffle algebra} is defined as:
\begin{equation}
\label{eqn:shufplus}
\CS^+ \ := \ \bigoplus_{\bk\in \nn} \CS_{\bk}^+
\end{equation}
and its elements, namely symmetric rational functions $F$ of the form \eqref{eqn:defshuf}, will be called \b{positive shuffle elements}. Just like in Proposition 2.7. of \citep{Nshuf}, one sees that $\CS^+ \subset \widetilde{\sym} \ V$ is a subalgebra with respect to the shuffle product. The algebra we wish to study is actually the double of \eqref{eqn:shufplus}, in the sense that we should also consider the \b{negative shuffle algebra}:
\begin{equation}
\label{eqn:shufminus}
\CS^- \ := \ \bigoplus_{\bk\in \nn} \CS_{-\bk}^-
\end{equation}
defined as the same vector space \eqref{eqn:shufplus}, but endowed with the opposite shuffle product $\CS^- = \left(\CS^+\right)^\op$. The elements of $\CS^-$ will be referred to as \b{negative shuffle elements}. We would like to realize $\CS^+$ and $\CS^-$ as bialgebras with a pairing between them, just like we did with the symmetric power $\sym \ V$ in \eqref{eqn:shufdoub}. However, the direct generalization of \eqref{eqn:shufcop} will not give a correct coproduct on either $\CS^+$ or $\CS^-$, because we have twisted the shuffle product by the rational function $\zeta$ in \eqref{eqn:shuffle}. The correct thing to do is to slightly enlarge our shuffle algebras:

$$
\CS^{\geq} = \left\langle \CS^{+}, \ \{\ph^+_{i,d}\}^{d\in \BN_0}_{1\leq i \leq n}\right\rangle \qquad \qquad \qquad \CS^\leq = \left\langle \CS^{-}, \ \{\ph^-_{i,d}\}^{d\in \BN_0}_{1\leq i \leq n}\right\rangle
$$
where the symbols $\ph^\pm_{i,d}$ commute among themselves, and interact with $\CS^\pm$ via:
\begin{equation}
\label{eqn:colombia}
\ph^+_i(w) \cdot F = F \cdot \ph^+_i(w) \prod^{1\leq j \leq n}_{1\leq a \leq k_j} \frac {\zeta(w/z_{ja})}{\zeta(z_{ja}/w)} \qquad \text{where} \qquad \ph^+_i(w) = \sum_{d=0}^\infty \ph^+_{i,d} w^{-d} 
\end{equation}
\begin{equation}
\label{eqn:venezuela}
\ph^-_i(w) \cdot G = G \cdot \ph^-_i(w) \prod^{1\leq j \leq n}_{1\leq a \leq k_j} \frac {\zeta(z_{ja}/w)}{\zeta(w/z_{ja})} \qquad \text{where} \qquad \ph^-_i(w) = \sum_{d=0}^\infty \ph^-_{i,d} w^d 
\end{equation}
for any $i \in \{1,...,n\}$, $F(...,z_{ja},...) \in \CS^+$ and $G(...,z_{ja},...) \in \CS^-$. To make sense of the above formula, we think of $w$ as a variable of color $i$. Then there exist coproducts:
$$
\Delta: \CS^\geq \longrightarrow \CS^\geq \ \widehat{\otimes} \ \CS^\geq \qquad \qquad \qquad \Delta: \CS^\leq \longrightarrow \CS^\leq \ \widehat{\otimes} \ \CS^\leq 
$$
given by $\Delta(\ph^\pm_i(w)) = \ph^\pm_i(w) \otimes \ph^\pm_i(w)$ for all $i\in \{1,...,n\}$, and:
\begin{equation}
\label{eqn:cop1}
\Delta(F) \ = \ \sum_{0 \leq \bl \leq \bk} \frac {\left[ \prod^{b > l_j}_{1\leq j \leq n} \ph^+_j(z_{jb}) \right] F(...,z_{i1},...,z_{il_i} \otimes z_{i,l_i+1},...,z_{ik_i},...)}{\prod^{a \leq l_i}_{1\leq i \leq n} \prod_{1\leq j \leq n}^{b > l_j} \zeta\left( z_{jb} / z_{ia}  \right)}
\end{equation}
\begin{equation}
\label{eqn:cop2}
\Delta(G) \ = \ \sum_{0 \leq \bl \leq \bk} \frac {G(...,z_{i1},...,z_{il_i} \otimes z_{i,l_i+1},...,z_{ik_i},...)\left[ \prod^{a \leq l_i}_{1\leq i \leq n} \ph^-_i(z_{ia}) \right] }{\prod^{a \leq l_i}_{1\leq i \leq n} \prod_{1\leq j \leq n}^{b > l_j} \zeta\left( z_{ia} / z_{jb}  \right)}
\end{equation}
for all $F\in \CS^+$ and $G \in \CS^-$. To make sense of \eqref{eqn:cop1} and \eqref{eqn:cop2}, we expand the right hand sides in the limit $|z_{ia}| \ll |z_{jb}|$ for all $a\leq l_i$ and $b>l_j$, and then place all monomials in $\{z_{ia}\}_{a\leq l_i}$ to the left of the $\otimes$ symbol and all monomials in $\{z_{jb}\}_{b > l_j}$ to the right of the $\otimes$ symbol. Because of the expansion, the right hand sides of \eqref{eqn:cop1} and \eqref{eqn:cop2} will have infinitely many terms, though finitely many in each homogeneous degree. Therefore, the coproduct $\Delta$ naturally takes values in a certain completion. With respect to the bialgebra structure defined above, the following maps $S:\CS^\geq \rightarrow \CS^\geq$ and $S:\CS^\leq \rightarrow \CS^\leq$ are anti-homomorphisms of both algebras and coalgebras:
\begin{equation}
\label{eqn:antshuf1}
S(\ph_i^+(z)) = \left(\ph^+_i(z) \right)^{-1} \qquad \qquad S(F) = \left[\prod^{1\leq i \leq n}_{1\leq a \leq k_i} \left(-\ph^+_i(z_{ia}) \right)^{-1} \right] * F
\end{equation}
\begin{equation}
\label{eqn:antshuf2}
S(\ph_i^-(z)) = \left(\ph^-_i(z) \right)^{-1} \qquad \qquad S(G) = G * \left[\prod^{1\leq i \leq n}_{1\leq a \leq k_i} \left(-\ph^-_i(z_{ia}) \right)^{-1} \right]
\end{equation}
for any $F(...,z_{ia},...)^{1\leq i \leq n}_{1\leq a \leq k_i} \in \CS^+$ and $G(...,z_{ia},...)^{1\leq i \leq n}_{1\leq a \leq k_i} \in \CS^-$. Because the currents $\ph^\pm_i(z)$ have infinitely many terms, the above maps also take values in a completion of $\CS^\pm$. The following Exercise is straightforward, but it is a good toy example for the computations in the next Section: \\

\begin{exercise}
\label{ex:antipode}

The maps $S$ defined above are antipode maps, i.e.
$$
S(F_1) * F_2 \ = \ F_1 * S(F_2) \ = \ S(G_1) * G_2 \ = \ G_1 * S(G_2) \ = \ 0
$$
for all $F \in \CS^+$ and $G \in \CS^-$. 

\end{exercise}

\tab 
Finally, the following Exercise is very important, as it will allow us to construct the double shuffle algebra as a Drinfeld double. \\

\begin{exercise}
\label{ex:pairing}
	
There exists a bialgebra pairing:
\begin{equation}
\label{eqn:pairshuf}
\langle \cdot,\cdot \rangle :\CS^\leq \otimes \CS^\geq \longrightarrow \tBF \qquad \text{given by} \qquad \langle \ph^-_{i}(z), \ph^+_j(w) \rangle = \frac {\zeta(w/z)}{\zeta(z/w)}
\end{equation}
for variables $z,w$ of color $i,j$, as well as the following formula: 
\begin{equation}
\label{eqn:daddypair}
\langle G,F \rangle = \frac {1}{\bk!} \int^{|q| < 1 < |p|}_{|z_{ia}|=1} \frac {G(...,z_{ia},...) F(...,z_{ia},...)}{\prod^{1\leq i,j \leq n}_{a\leq k_i, b\leq k_j} \zeta_{p} (z_{ia}/z_{jb})}\prod^{1\leq i \leq n}_{1\leq a \leq k_i} Dz_{ia} \ \Big |_{p \mapsto q}	
\end{equation}
for all $F\in \CS^+$ and $G\in \CS^-$, where $\zeta_{p}$ is defined as the following cousin of \eqref{eqn:defzeta}: 
\begin{equation}
\label{eqn:defzetap}
\zeta_{p} \left( \frac {x_i}{x_j} \right) \ = \ \zeta\left( \frac {x_i}{x_j} \right) \frac {\left[\frac {x_j}{p^2x_i} \right]^{\delta_j^i}}{\left[\frac {x_j}{q^2x_i} \right]^{\delta_j^i}}
\end{equation}
The reason we use $\zeta_{p}$ instead of $\zeta$, and evaluate the integral \eqref{eqn:daddypair} assuming $|q| < 1 < |p|$ only to set $p \mapsto q$ afterwards, is to avoid having double poles at $qz_{ia}-q^{-1}z_{ib}$. \\	
\end{exercise}




\section{Verma modules for shuffle algebras}
\label{sec:verma}

\noindent The pairing defined in Exercise \ref{ex:pairing} allows us to define the Drinfeld double:
\begin{equation}
\label{eqn:doubleshuf}
\CS \ = \ \frac {\CS^\leq \otimes \CS^\geq}{\ph_{i,0}^+ \ph_{i,0}^- - 1}
\end{equation}
and call it the \b{double shuffle algebra}. The algebra $\CS$ is graded, with: 
$$
\deg \ F(...,z_{i1},...,z_{ik_i},...) = \pm \bk \qquad \text{where} \qquad \bk = (k_1,...,k_n) \in \nn
$$ 
for a positive/negative shuffle element $F \in \CS^\pm$. We also assign $\deg \ph_{i,d}^\pm = 0$. However, there is a second grading we will call the \b{homogenous degree}, given by:
$$
d = \homdeg F \ = \ \text{homogenous degree of the function }F(...,z_{i1},...,z_{ik_i},...) 
$$
for all $F \in \CS^\pm$, and $\homdeg \ph_{i,d}^\pm = \pm d$. The \b{bidegree} of shuffle elements will be denoted by $(\bk,d) \in \zz \times \BZ$. We will denote the graded and bigraded pieces by:
$$
\CS \ = \ \bigoplus_{\bk \in \zz} \CS_\bk \qquad \qquad \qquad \CS_\bk = \bigoplus_{d\in \BZ} \CS_{\bk,d}
$$
respectively. We are now in position to properly state Theorem \ref{thm:iso0}: \\

\begin{theorem}
\label{thm:iso}
There exists a bigraded isomorphism of bialgebras:
\begin{equation}
\label{eqn:upsilon}
\Upsilon \ : \ \UU \longrightarrow \ \CS
\end{equation}
generated by:
$$
\Upsilon(\ph_{i,d}^\pm) \ = \ \ph_{i,d}^\pm \qquad \text{ and }\qquad \Upsilon(e^\pm_{i,d}) \ = \  \frac {z_{i1}^{d}}{[q^{-2}]} \ \in \ \CS^\pm
$$
for all $i\in \{1,...,d\}$ and all $d\in \BZ$.
	
\end{theorem}


\tab
It is easy to see that the map $\Upsilon$ respects the coproduct and pairing of the generators. It is also not hard to show that the map $\Upsilon$ is a monomorphism, on account of the non-degeneracy of the bialgebra pairing on $\UU$ (see \citep{Ntor} for a review of this argument). The main difficulty in proving Theorem \ref{thm:iso} is the surjectivity of $\Upsilon$, which boils down to the fact that the restriction to the positive halves:
\begin{equation}
\label{eqn:upsilonplus}
\Upsilon^+ \ : \ \UUp \longrightarrow \ \CS^+
\end{equation}
is surjective. The analogous claim for the negative halves is proved likewise. The surjectivity of $\Upsilon^+$ is equivalent to the fact that the shuffle algebra is generated by degree one elements $\{z_i^d\}^{d\in \BZ}_{1\leq i \leq n}$, which will be established in the next Chapter.


\tab
In this Section, we will show how to construct Verma modules for the shuffle algebra. The construction is meant to match that of \citep{H} for the quantum toroidal algebra, under the isomorphism $\Upsilon$. To define these, let us consider any $\bw \in \nn$ and fix a collection of colored parameters $u_1,...,u_\bw$, such that there are $w_i$ parameters of color $i$ modulo $n$. As a vector space, the \b{Verma module} coincides with either half of the shuffle algebra itself: 
\begin{equation}
\label{eqn:verma}
M(\bw) \ = \ \CS^\pm \otimes \BZ[u_1^{\pm 1},...,u_\bw^{\pm 1}] 
\end{equation}
$$
M(\bw) \ = \ \bigoplus_{\bv \in \nn} M_{\bv,\bw} \ = \ \bigoplus_{\bv \in \nn}  \Big\{ \text{shuffle element } f(...,z_{i1},...,z_{iv_i},...) \Big\} 
$$
where the shuffle elements $f$ are allowed to have coefficients among the $u_i^{\pm 1}$'s, but otherwise all the formulas in the previous Section apply to them. The positive half $\CS^+$ acts on \eqref{eqn:verma} by the usual shuffle product, twisted by a symmetric polynomial that is constructed from the $u_i$:
\begin{equation}
\label{eqn:act1}
F \ \curvearrowright \ f \ := \ \left( F(z_{i1},...,z_{ik_i}) \prod_{i=1}^\bw \prod_{a=1}^{k_i} \left[\frac {u_i}{qz_{ia}}\right] \right) * f \qquad \forall \ F \in \CS^+_\bk \quad \forall \ f \in M(\bw)
\end{equation}
The Cartan elements are required to act by the following formula for all $i\in \{1,...,n\}$:
\begin{equation}
\label{eqn:act2}
\ph^\pm_i(w) \ \curvearrowright \ f \ := \ \underbrace{\prod_{1\leq j \leq \bw}^{u_j \equiv i} \frac {\Big[ \frac {u_j}{qw} \Big]}{\Big[ \frac {w}{qu_j} \Big]}}_{\text{lowest weight}} \cdot \prod^{1\leq j \leq n}_{1\leq a \leq v_j} \frac {\zeta \Big( \frac w{z_{ja}} \Big)}{\zeta \Big( \frac {z_{ja}}w \Big)} \cdot f \qquad \qquad \qquad \forall \ f \in M_{\bv,\bw}
\end{equation}
which establishes that the ``lowest weight" of $M(\bw)$, interpreted as in the sense of \citep{H}, is the rational function marked by the underbrace. Formulas \eqref{eqn:act1} and \eqref{eqn:act2}, together with relation \eqref{eqn:reldrinfeld2} that holds in any Drinfeld double, force the following action of negative shuffle elements on $M(\bw)$:
\begin{equation}
\label{eqn:act3}
G \ \curvearrowright \ f \ := \ \left \langle S(G_1), \frac {f_1}{\prod_{i=1}^\bw \left[\frac {Z^1_i}{qu_i} \right]} \right \rangle f_2 \left \langle G_2, \frac {f_3}{\prod_{i=1}^\bw \left[\frac {u_i}{qZ^3_i} \right]} \right \rangle 
\end{equation}
for all $G \in \CS^-_{-\bk}$ and $f \in M(\bw)$, where $\Delta^2(f) = f_1 (Z^1) \otimes f_2(Z^2) \otimes f_3(Z^3)$ denotes the coproduct \eqref{eqn:cop1}. We write $Z^k = Z^k_1+...+Z^k_n$, where $Z^k_i = \{z_{i1}^k,z_{i2}^k,...\}$ is the set of variables of the rational functions $f_k$ in each tensor, for all $k\in \{1,2,3\}$ and $i\in \{1,...,n\}$. Because of our conventions for the coproduct of shuffle elements \eqref{eqn:cop1}, the denominator of \eqref{eqn:act3} must be expanded in the range:
\begin{equation}
\label{eqn:limit}
|Z^1| \ \ll \ 1 \qquad \qquad |Z^2| \ = \  1 \qquad \qquad |Z^3| \ \gg \ 1  
\end{equation}
If we use formula \eqref{eqn:daddypair} for the pairing, we can write the action of negative shuffle elements as a bona fide rational function for all $G\in \CS^-$: 
\begin{equation}
\label{eqn:act4}
G \ \curvearrowright \ f \ := \ \sum_{\bv = \bv_1+\bv_2+\bv_3} \frac {(-1)^{|\bv^1|} }{\bv^1! \cdot \bv^3!} \cdot 
\end{equation}
$$
\int^{|q| < 1 < |p|}_{|Z^1| \ll 1} \int^{|q| > 1 > |p|}_{|Z^3| \gg 1} \frac {G(Z^1, Z^3) f(Z^1,Z^2,Z^3)  DZ^1DZ^3}{\zeta_{p}\left(\frac {Z^1+Z^3}{Z^1+Z^3}\right) \zeta\left(\frac {Z^2}{Z^1+Z^3}\right) \prod_{i=1}^\bw \left[\frac {Z^1_i+Z_i^3}{qu_i} \right]}  \ \Big |_{p \mapsto q}	
$$
where $Z^1_i = \{z_{ia}\}_{1\leq a\leq \bv^1_i}$, $Z^2_i = \{z_{ia}\}_{\bv^1_i < a\leq \bv^1_i+\bv^2_i}$, $Z^3_i = \{z_{ia}\}_{\bv^1_i+\bv^2_i < a \leq \bv_i}$. The statement that $q$ is assumed to be either greater than or smaller than 1 depending on which integral we take should be interpreted as follows: factors which involve both $Z^1$ and $Z^3$ are not concerned with the size of $q$, because of the expansion in $|Z^1|\ll|Z^3|$. But in those factors which involve either $Z^1$ or $Z^3$, we must replace $q$ by a parameter $q_\pm$ which is taken either $<1$ or $>1$ depending on the integral, compute the integral via residues, and then set $q_\pm \mapsto q$. To show that formula \eqref{eqn:act4} respects the shuffle product, i.e. that $\left( G \curvearrowright \Big( G' \curvearrowright f \Big) \right) = \Big( G * G' \Big) \curvearrowright f$, one observes that the integrands that arise in \eqref{eqn:act4} are the same, and the contours can be made to coincide by the same argument as in the proof of Exercise \ref{ex:pairing}. Therefore, we obtain the following: \\


\begin{proposition}
\label{prop:verma}
	
Formulas \eqref{eqn:act1}, \eqref{eqn:act2}, \eqref{eqn:act4} give rise to a well-defined action:
$$
\CS \ \curvearrowright \ M(\ebw)
$$
	
\end{proposition}


\noindent In order to merit being called a Verma module, we should prove the existence of a \b{Shapovalov form} on $M(\bw)$. By definition, this is a symmetric bilinear form:
$$
(\cdot,\cdot) \ : \ M(\bw) \otimes M(\bw) \ \longrightarrow \ \tBF_\bw 
$$
where $\BF_\bw = \BZ[q^{\pm 1}, t^{\pm 1},u_1^{\pm 1},...,u_\bw^{\pm 1}]$ and $\tBF_\bw = \BQ(q, t, u_1,...,u_\bw)$, such that:
\begin{equation}
\label{eqn:shap}
(G \curvearrowright f, f') \ = \ (f, G^T \curvearrowright f') \qquad \qquad \forall \ G \in \CS \quad \text{and} \quad \forall \ f,f'\in M(\bw)
\end{equation}
Here, $G \mapsto G^T$ denotes the \b{transposition} involution, namely the antiautomorphism induced by the identity map of vector spaces:
$$
\CS^+ \ \ni \ G \ \stackrel{T}\leftrightarrow \ G \ \in \ \CS^-
$$
The defining property \eqref{eqn:shap} of the Shapovalov form implies that the quotient space:
\begin{equation}
\label{eqn:quotient}
L(\bw) \ := \ \frac {M(\bw)}{\text{kernel of }(\cdot,\cdot)}
\end{equation}
is a representation of $\CS$. In the theory of lowest weight representations of affine algebras or quantum groups, the representations $L(\text{lowest weight})$ obtained in this way are the irreducible building blocks of category $\CO$. Let $Z = \sum_{i=1}^n Z_i = \sum_{i=1}^n \sum_{a=1}^{v_i} z_{ia}$ be a placeholder for our set of variables, similar to \eqref{eqn:placeholder}. \\

\begin{exercise}
\label{ex:shap}

The following assignment gives rise to a Shapovalov form on $M(\bw)$:
\begin{equation}
\label{eqn:shapshuffle}
(f,f') \ = \ \frac {1}{\bv!} \left( \int_{|Z| = 1} - \int_{|Z| \ll 1} \right) \frac {f(Z) f'(Z) \cdot DZ}{\prod^{1\leq i \leq n}_{1\leq j \leq n} \zeta \left(\frac {Z_i}{Z_j} \right) \prod_{i=1}^{\bw} \Big[\frac {Z_i}{qu_i} \Big]\Big[\frac {u_i}{qZ_i} \Big]}
\end{equation}


\end{exercise}

\tab  
Using \eqref{eqn:daddypair}, we will show in Chapter \ref{chap:proofs} that formula \eqref{eqn:shapshuffle} equals the following formula, which is written purely in terms of the Hopf algebra structure of the shuffle algebra:
\begin{equation}
\label{eqn:shapshufflehopf}
(f, f') \ = \ \left \langle \frac {\left(f_1 * S (f_2) \right)^T}{\prod_{i=1}^\bw \left[\frac {Z^1_i}{qu_i} \right] \left[\frac {Z^2_i}{qu_i} \right]}, \frac {f'}{\prod_{i=1}^\bw \left[\frac {u_i}{qZ_i} \right]} \right \rangle 
\end{equation}
where the variables of $f_1$, $f_2$ and $f'$ are denoted generically by $\{Z^1\}$, $\{Z^2\}$ and $\{Z\}$ respectively. Each summand of $\Delta(f) = f_1 \otimes f_2$ is expanded in the range $|Z^1| \ll 1$, $|Z| = 1$, $|Z^2| \gg 1$. Formula \eqref{eqn:shapshufflehopf} is forced upon us by property \eqref{eqn:shap} and the fact that the Verma module is generated by the positive shuffle algebra. Let us note that for $\bw = (0,...,0)$, the integral \eqref{eqn:shapshuffle} equals 0 by the same inclusion-exclusion argument from the proof of Exercise \ref{ex:antipode}, so we conclude that:
$$
(M_{\bv,0}, M_{\bv,0}) \ = \ 0 \qquad \forall \ \bv \neq 0
$$
and hence $L(0)$ is one dimensional. In Section \ref{sec:act}, we will identify the representations $L(\bw)$ with the localized $K-$theory groups of Nakajima quiver varieties $K(\bw)$, and the Shapovalov form \eqref{eqn:shapshuffle} with the integral pairing of \eqref{eqn:pairingmoduli} - \eqref{eqn:integral}. \\

\section{The $K-$theoretic Hall algebra}
\label{sec:k}

\noindent We will now recall how $\CS^+$ and the shuffle product \eqref{eqn:shuffle} can be interpreted as the $K-$theoretic Hall algebra of cyclic quiver varieties, as constructed by \citep{SV} in the case $n=1$. To set up the problem, let us recall the usual Hall algebra for the groups $G_k = GL(\BC^k)$. Consider the representation rings:
$$
\text{Rep}(G_k) \ \cong \ \BZ[z_1^{\pm 1},...,z_k^{\pm 1}]^\sym 
$$
for all $k\in \BN$. The above isomorphism is given by sending a representation to its character. Then we may consider the direct sum of these representation rings:
\begin{equation}
\label{eqn:directsum}
\bigoplus_{k=0}^\infty \text{Rep}(G_k) \ = \ \bigoplus_{k=0}^\infty \BZ[z_1^{\pm 1},...,z_k^{\pm 1}]^\sym
\end{equation}
and ask how to construct a natural algebra structure on \eqref{eqn:directsum}, which preserves the grading by $k$. To accomplish this, consider the embedding $G_k \times G_l \longrightarrow G_{k+l}$ as block-diagonal matrices for all natural numbers $k$ and $l$, and the parabolic subgroup:
\begin{equation}
\label{eqn:par}
\xymatrix{& P_{k,l} \ar@{_{(}->}[dl] \ar@{->>}[dr]  \\
	G_{k+l} & & G_k \times G_l} 
\end{equation}
consisting of matrices that preserve the subspace $\BC^k \hookrightarrow \BC^{k+l}$. Then the assignment:
$$
\rep(G_k) \otimes \rep(G_l) \ \stackrel{*}\longrightarrow \ \rep(G_{k+l})
$$
\begin{equation}
\label{eqn:shalom}
f *f' \ = \ \text{Ind}_{P_{k,l}}^{G_{k+l}} \left( f \otimes f' \right)
\end{equation}
gives rise to an associative algebra structure on \eqref{eqn:directsum}. The above product can be computed explicitly, according to the well-known formula:
\begin{equation}
\label{eqn:induction}
(f*f')(z_1,...,z_{k+l}) = \sym \left[ \frac {f(z_1,...,z_k) f'(z_{k+1},...,z_{k+l})}{k! \cdot l!} \prod^{1\leq i \leq k}_{k+1 \leq j \leq k+l} \left(1-\frac {z_j}{z_i}\right)^{-1} \right] \qquad
\end{equation}
for any symmetric Laurent polynomials $f(z_1,...,z_k)$ and $f'(z_{k+1},...,z_{k+l})$. Let us remark that, historically, the appropriate setup for the above Hall algebra are finite groups. However, we chose to present it as above in order to motivated the following construction, although there is quite a rich history of mathematics in between them.

\tab
The connection to geometry starts by observing that $\text{Rep}(G_k) = K\left([\cdot/G_k]\right)$. Moreover, the diagram \eqref{eqn:par} gives rise to the following morphisms of stacks:
\begin{equation}
\label{eqn:parabolic}
\xymatrix{& [\cdot / P_{k+l}] \ar[dl]_{\pi_1} \ar[dr]^{\pi_2}  \\
[\cdot / G_{k+l}] & & [\cdot / G_k] \times [\cdot / G_l]} 
\end{equation}
Therefore, the Hall product \eqref{eqn:shalom} can be thought of as the correspondence:
$$
K\left(\left[\cdot/G_k \right] \right) \otimes K\left(\left[\cdot/G_l \right] \right) \ \stackrel{*}\longrightarrow \ K\left(\left[\cdot/G_{k+l} \right] \right)
$$
\begin{equation}
\label{eqn:meretz}
f * f' \ = \ \pi_{1*} \left(\pi_2^*( f \otimes f' ) \right)
\end{equation}
which induces an algebra structure on $\bigoplus_{k=0}^\infty K\left(\left[\cdot/G_k \right] \right)$. We can generalize this setup by replacing the quotient stack $[\cdot/G_k]$ with the so-called commuting stack:
\begin{equation}
\label{eqn:stack}
\fC_k \ := \ \Big[\{X,Y \in \text{End}(\BC^k) \ \text{ such that } \ [X,Y] = 0 \}/G_k \Big]
\end{equation}
which is simply the stack quotient that corresponds to the ``moduli of sheaves" without any framing $\CN_{k,0}$. We consider the action of the torus $\BC^* \times \BC^*$ on $\fC_k$ given by $(q,t)\cdot(X,Y) = (qt X, qt^{-1}Y)$. The equivariant $K-$theory of the stack $\fC_k$, and its relation to the shuffle algebra, were pursued in \citep{SV}. In this paper, we will only consider those $K-$theory classes on the above stack which come from equivariant constants, so let us define:
$$
K_{\BC^* \times \BC^*}(\fC_k) \ := \ K_{\BC^* \times \BC^* \times G_k}(\pt) \ = \ \BF[z_1^{\pm 1},...,z_k^{\pm 1}]^\sym
$$
where $\BF = \BZ[q^{\pm 1}, t^{\pm 1}]$. We will soon see that we will need to enlarge the above ring by allowing constants in $\tBF = \BQ(q,t)$, as well as rational functions:
\begin{equation}
\label{eqn:shuf0}
\widetilde{K}_{\BC^* \times \BC^*}(\fC_k) \ := \ K_{\BC^* \times \BC^*}(\fC_k)_\loc \ = \ \widetilde{\BF}(z_1,...,z_k)^\sym
\end{equation}
by analogy with Section \ref{sec:shuffle}, where we replaced $\sym \ V$ by $\widetilde{\sym} \ V$. We consider $\fC_k \times \fC_l \rightarrow \fC_{k+l}$ as the stack of commuting block diagonal matrices corresponding to a fixed subspace $\BC^k \subset \BC^{k+l}$. Then the analogue of \eqref{eqn:parabolic} is the diagram:
$$
\xymatrix{& \fP_{k,l} \ar[dl]_{\pi_1} \ar[dr]^{\pi_2}  \\
\fC_{k+l} & & \fC_k \times \fC_l}
$$
where the stack correspondence is given by: 
$$
\fP_{k,l} \ := \ \Big[\{X,Y \in \text{End}(\BC^{k+l}) \ \text{ that preserve }\BC^k \text{ and } \ [X,Y] = 0 \}/P_{k,l} \Big] 
$$
In other words, $\fP_{k+l}$ parametrizes commuting block triangular matrices, which preserve a fixed embedding $\BC^k \subset \BC^{k+l}$, modulo the action of the parabolic group $P_{k,l}$. Then the appropriate shuffle product is also given by a push-pull diagram:
\begin{equation}
\label{eqn:shufprod0}
\widetilde{K}_{\BC^* \times \BC^*}(\fC_k) \ \otimes \ \widetilde{K}_{\BC^* \times \BC^*}(\fC_l) \ \stackrel*\longrightarrow \ \widetilde{K}_{\BC^* \times \BC^*}(\fC_{k+l}) 
\end{equation}
$$
F * F' \ = \ \widetilde{\pi}_{1*}\Big(\pi_2^*(F \otimes F') \Big)
$$
where the modified push-forward is defined by \eqref{eqn:modpush}. The following explicit formula will be proved in Exercise \ref{ex:stackhall} below, in the more general setting of the cyclic quiver:
$$
(F * F')(z_1,...,z_{k+l}) = \sym \left[ \frac {F(z_1,...,z_k) F'(z_{k+1},...,z_{k+l})}{k! \cdot l!} \prod^{1\leq i \leq k}_{k+1 \leq j \leq k+l} \frac {\left[\frac {z_j}{qtz_i}\right]\left[\frac {tz_j}{qz_i} \right]}{\left[\frac {z_j}{z_i}\right]\left[\frac {z_j}{q^2z_i} \right]} \right]
$$ 
This is precisely the shuffle product considered in \citep{FT} and \citep{SV}. Let us now recall how to generalize this picture to obtain \b{K-theoretic Hall algebras} for the cyclic quiver, although we note that the construction works for general quiver varieties. Consider the quotient stack version of the Nakajima quiver variety $\CN_{\bk,0}$ of any degree $\bk = (k_1,...,k_n) \in \nn$:
$$
\fC_{\bk} \ := \ \left[\left \{ \xymatrix{\BC^{k_i} \ar@/^/[r]^{X_i} & \BC^{k_{i+1}} \ar@/^/[l]^{Y_i}} \ \text{ such that } \ X_{i-1}Y_{i-1} = Y_iX_i \right\}^{1\leq i \leq n}/ \ G_\bk \right]
$$
where $G_\bk = \prod_{i=1}^n GL(\BC^{k_i})$. We define the $K-$theory of this stack following \eqref{eqn:shuf0}:
$$
K_{\BC^* \times \BC^*}(\fC_\bk) \ := \ K_{\BC^* \times \BC^* \times G_\bk}(\pt) \ = \ \BF[...,z_{i1}^{\pm 1},...,z_{ik_i}^{\pm 1},...]^\sym
$$
\begin{equation}
\label{eqn:johnny}
\widetilde{K}_{\BC^* \times \BC^*}(\fC_\bk) \ := \ K_{\BC^* \times \BC^*}(\fC_\bk)_\loc \ = \ \widetilde{\BF}(...,z_{i1},...,z_{ik_i},...)^\sym
\end{equation}
where the rational functions above are only required to be symmetric in the variables $z_{ia}$ of given color $i$ modulo $n$. We will now define the algebra structure on the sum of the above $K-$theories. For any $\bk,\bl \in \nn$, we fix a subspace $\BC^\bk \subset \BC^{\bk+\bl}$, by which we mean a collection of subspaces $\BC^{k_i} \subset \BC^{k_i+l_i}$ for all $1\leq i \leq n$, and define:
\begin{equation}
\label{eqn:parabolic2}
\xymatrix{& \fP_{\bk+\bl} \ar[dl]_{\pi_1} \ar[dr]^{\pi_2}  \\
\fC_{\bk+\bl} & & \fC_{\bk} \times \fC_{\bl}}
\end{equation}
as $\quad \fP_{\bk+\bl} :=$
$$
\left[ \left\{\xymatrix{\BC^{k_i+l_i} \ar@/^/[r]^{X_i} & \BC^{k_{i+1}+l_{i+1}} \ar@/^/[l]^{Y_i}} \text{ s.t. }X,Y \text{ preserve }\BC^\bk \text{ and } X_{i-1}Y_{i-1} = Y_iX_i \right\}^{1 \leq i \leq n} / P_{\bk,\bl} \right] 
$$
where $P_{\bk,\bl} = \prod_{i=1}^n P_{k_i,l_i}$. The shuffle product is defined by analogy with \eqref{eqn:shufprod0}:
$$
\widetilde{K}_{\BC^* \times \BC^*}(\fC_\bk) \ \otimes \ \widetilde{K}_{\BC^* \times \BC^*}(\fC_\bl) \ \stackrel*\longrightarrow \ \widetilde{K}_{\BC^* \times \BC^*}(\fC_{\bk+\bl}) 
$$
$$
F * F' \ = \ \widetilde{\pi}_{1*}\Big(\pi_2^*(F \otimes F') \Big)
$$
The following Exercise proves that the Hall product defined above matches the shuffle product defined in \eqref{eqn:shuffle}, under the identification \eqref{eqn:johnny}. \\

\begin{exercise} 
\label{ex:stackhall}
	
The Hall product for the cyclic quiver is given by:
\begin{equation}
\label{eqn:mult}
F * F'  = \textrm{Sym} \left[ \frac {F(...,z_{ia},...) F'(...,z_{jb},...)}{\bk! \cdot \bl!} \prod^{1\leq i \leq n}_{1\leq a\leq k_i} \ \prod^{1\leq j \leq n}_{k_j+1 \leq b \leq k_j+l_j} \zeta \left( \frac {z_{ia}}{z_{jb}} \right) \right]
\end{equation}
for any $F \in \widetilde{K}_{\BC^* \times \BC^*}(\fC_\bk)$ and $F' \in \widetilde{K}_{\BC^* \times \BC^*}(\fC_\bl)$. \\
\end{exercise}


\section{The $K-$theory action}
\label{sec:act}

\noindent Let us now recall the $K-$theory groups $K(\bw)$ of Nakajima quiver varieties, that were defined in Section \ref{sec:quiver} for any $\bw \in \nn$. If we set:
$$
\Lambda(\bw) \ = \ \bigoplus_{\bv \in \nn} \Lambda_{\bv,\bw} \qquad \text{where} \qquad \Lambda_{\bv,\bw} \ = \ \BF_\bw[..., x_{i1}^\pm,...,x_{iv_i}^\pm,... ]^\sym_{1\leq i \leq n}
$$
then we have seen in \eqref{eqn:kmoduli} that:
$$
K(\bw) \ = \ \frac {\Lambda(\bw)}{\text{kernel of the pairing }(\cdot, \cdot)} 
$$ 
which was defined in \eqref{eqn:pairingmoduli}. Comparing formula \eqref{eqn:integral} with \eqref{eqn:shapshuffle}, we see that the above pairing coincides with the Shapovalov form. Therefore, we may be justified to conclude that $K(\bw) \cong L(\bw)$ as vector spaces. However, there are two subtleties here: one is that $K(\bw)$ was defined starting from Laurent polynomials, while $L(\bw)$ was defined starting from the a priori larger vector space of shuffle elements \eqref{eqn:shufelem}. This will be solved by the fact that we quotient out by the kernels of the pairing $(\cdot,\cdot)$. The second subtlety is that elements of $L(\bw)$ are defined over rational functions in $\BQ(q,t)$, while $K(\bw)$ are defined over Laurent polynomials $\BZ[q^{\pm 1}, t^{\pm 1}]$. Though we will not pursue this here, we believe that one can define an ``integral" form of the shuffle algebra, such that the resulting module $L(\bw)$ become isomorphic to the integral $K-$theory group $K(\bw)$. Instead, we will contend ourselves with an embedding which becomes an isomorphism after localization: \\



\begin{theorem}
\label{thm:act} 

For any $\ebw \in \nn$, we have an embedding of modules: 
$$
\UU \ \curvearrowright \ K(\ebw) \ \hookrightarrow \ L(\ebw) \ \curvearrowleft \ \CS
$$
which becomes an isomorphism after localization $K(\ebw) \hookrightarrow K(\ebw)_{\emph{loc}}$. The two actions above match via the isomorphism $\Upsilon : \UU \cong \CS$ of Theorem \ref{thm:iso}.	\\

\end{theorem}

\begin{proof} The embedding is induced by the natural inclusion $\Lambda(\bw) \hookrightarrow M(\bw)$. To show that it induces a well-defined map on the quotients $K(\bw) \rightarrow L(\bw)$, we need to prove:
$$
(f,g) = 0 \qquad \forall \ g \in \Lambda(\bw) \qquad \Leftrightarrow \qquad (f,G) = 0 \qquad \forall \  G \in M(\bw)
$$
for any Laurent polynomial $f \in \Lambda(\bw)$. We do not yet have a general understanding of this phenomenon (which would probably require one to be more subtle about defining Verma modules for shuffle algebra) so we will instead prove it by localization. Formula \eqref{eqn:residuecount} shows that, with respect to the pairing \eqref{eqn:pairingmoduli}:
$$
(f,g) = 0 \qquad \forall \ g \in \Lambda(\bw) \qquad \Leftrightarrow \qquad f\left(\chi_\bla\right) = 0 \quad \forall \ \bw-\text{partitions }\bla
$$
The analogous formula can be proved for the Shapovalov form, where the analogue of Proposition \ref{prop:integral} allows us to express the residue count for the integral \eqref{eqn:shapshuffle} as a sum over $\bw-$partitions. The poles are made worse by the fact that shuffle elements have denominators of the form $qz_{ia} - q^{-1}z_{ib}$, but the argument (and the day) is saved by the wheel conditions that we imposed on shuffle elements:
\begin{equation}
\label{eqn:showroom}
M(\bw) \ \ni \ F(...,z_{ia},...) \ = \ \frac {r(...,z_{ia},...)}{\prod^{1 \leq i \leq n}_{1\leq a \neq b \leq k_i} (qz_{ia} - q^{-1}z_{ib})}
\end{equation}
where the Laurent polynomial $r$ satisfies $r(...,q^{-1},t^{\pm 1},q,...) \equiv 0$ for any three variables of colors $i,i\pm 1,i$. Therefore, we conclude that a shuffle element $F \in M(\bw)$ vanishes in $L(\bw)$ if and only if $F\left(\chi_\bla\right) = 0$ for all $\bw-$partitions $\bla$, where:
\begin{equation}
\label{eqn:quantity}
F\left(\chi_\bla\right) \ = \ F(...,\chi_\sq,...)_{\sq \in \bla} \ \in \ \tBF
\end{equation}
One must take care when specializing the variables of $F$ at the set of weights of a $\bw-$partition, because of the denominators that appear in \eqref{eqn:showroom}. However, the wheel conditions eliminate all the poles, and this can be seen by the following procedure: 

\begin{itemize}

\item for each box $(x,y) \in \la^j \subset \bla$, insert the input $z_{ia} = s^{(y)}_j q_1^x$ into the function $F$

\item the wheel conditions imply that $r$ of \eqref{eqn:showroom} is divisible by certain linear factors of the form $s^{(y)}_j q_1^x - s^{(y\pm 1)}_jq_1^{x\pm 1}$, as we will see in the proof of Lemma \ref{lem:bound}

\item these linear factors precisely cancel the problematic denominators of \eqref{eqn:showroom}

\item specialize $s^{(y)}_j = qu_j \cdot q_2^y$ and obtain the constant in \eqref{eqn:quantity}.

\end{itemize}

\noindent Therefore, both Laurent polynomials $f\in \La(\bw)$ and shuffle elements $F \in M(\bw)$ vanish in the quotients $K(\bw)$ and $L(\bw)$ when their specializations to the set of weights of any $\bw-$partition vanish. This implies that we have an embedding $K(\bw) \hookrightarrow L(\bw)$.

\tab	
We will now show that this embedding intertwines the actions of $\UU$ and $\CS$. According to Theorem \ref{thm:iso0}, the map $\Upsilon$ is an isomorphism, which implies that the shuffle algebra is generated by degree one elements $z_i^d$. Therefore, it is enough to match the action of $e_{i,d}^\pm$ on $K(\bw)$ from \eqref{eqn:simpleop} with the action of the shuffle elements $F = \frac {z_i^d}{q^{-1}-q} \in \CS^+$ in \eqref{eqn:act1} and $G = \frac {z_i^d}{q^{-1}-q} \in \CS^-$ in \eqref{eqn:act4} on $L(\bw)$. Let us start from formula \eqref{eqn:form1}, which can be written as:
\begin{equation}
\label{eqn:areo}
e_{i,d}^+ (\ov{f}) = \ov{f'} \qquad \text{where} \qquad f'(\chi_{\bla}) = \sum^{\bla = \bmu + \bsq}_{\bsq \text{ of color }i} \frac {\chi_\bsq^d}{[q^{-2}]}  f(\chi_{\bmu}) \zeta \left( \frac {\chi_\bsq}{\chi_{\bmu}} \right)  \prod^{u_j\equiv i}_{j \leq \bw} \Big[\frac {u_j}{q\chi_\bsq} \Big] 
\end{equation}
To match this with \eqref{eqn:act1}, we need to show that $f'$ can be taken to be: \begin{equation}
\label{eqn:hotah}
\left( \frac {z^d}{[q^{-2}]} \prod^{u_j\equiv i}_{j \leq \bw} \Big[\frac {u_j}{q z} \Big] \right) * f  = \sym \left[ \frac {z^d}{[q^{-2}]} f(...,z_{ia},...) \prod_{i,a} \zeta \left(\frac z{z_{ia}} \right) \prod^{u_j\equiv i}_{j \leq \bw}  \Big[\frac {u_j}{q z} \Big] \right] \ 
\end{equation}
This is easy to see from the definition of the shuffle product, and will be explained in detail in more generality in Proposition \ref{prop:act} below. The idea is that when we evaluate the symmetrization \eqref{eqn:hotah} at the set of weights $\{\chi_\sq\}_{\sq \in \bla}$, the variable $z$ must be set equal to the weight of a removable box $\sq \in \bla$ (in other words, $z$ must be set equal to $\frac {\chi_\sq}{q^2}$ for an outer corner $\sq$ of $\bla$). Otherwise, the product of $\zeta\left(\frac {z}{z_{ia}} \right) \Big |_{\{z,...,z_{ia},...\} \mapsto \chi_\bla}$ specializes to 0. Writing $\bmu = \bla - \bsq$ gives us precisely \eqref{eqn:areo}.

\tab
Let us now recall formula \eqref{eqn:form2} for the negative half of $\UU$, and rewrite it as:
\begin{equation}
\label{eqn:arys}
e_{i,d}^- \left( \ov{f} \right) = \left( \int_{|z_i|\gg 1} - \int_{|z_i|\ll 1} \right) z_i^d \cdot \ov{\frac {f(X + z_i)}{\zeta\Big( \frac X{z_i} \Big) \prod^{u_j \equiv i}_{1\leq j \leq \bw} \Big[\frac {z_i}{qu_j} \Big]}} Dz_i
\end{equation}
Meanwhile, when $G = \frac {z_i^d}{q^{-1}-q}$, the only summands which can appear in formula \eqref{eqn:act4} are $Z^1 = \emptyset, Z^2 = X, Z^3 = \{z_i^d\}$ and $\{Z^1\} = z_{i}^d, Z^2 = X, Z^3 = \emptyset$, so we obtain:
\begin{equation}
\label{eqn:oakheart}
\frac {z_i^d}{[q^{-2}]} \curvearrowright f \ = \ \left( \int_{|z_i|\gg 1} - \int_{|z_i|\ll 1} \right) \frac {z_i^d f(X+z_i)  Dz_i}{[q^{-2}] \zeta\left(\frac {z_i}{z_i} \right) \zeta\left(\frac {X}{z_i}\right) \prod^{u_j \equiv i}_{1\leq j \leq \bw}\prod^{u_j \equiv i}_{1\leq j \leq \bw} \left[\frac {z_i}{qu_i} \right]}
\end{equation}
There is no need to consider the various assumptions $|q|<1<|p|$, or the rational function $\zeta_p$ instead of $\zeta$ for that matter, because these only affect the residue computation when there is more than one $z$ variable. Since $\zeta \left(\frac {z_i}{z_i} \right)$ gives rise to a factor of $\frac 1{[q^{-2}]}$, formulas \eqref{eqn:arys} and \eqref{eqn:oakheart} match. 
\end{proof}

\tab 
According to the above Theorem, shuffle elements act on the $K-$theory groups of Nakajima cyclic quiver varieties, something that was expected after Section \ref{sec:k}. The following Proposition shows that shuffle elements precisely encode the matrix coefficients of the corresponding operators in the basis of fixed points $|\bla \rangle$ of $K-$theory. \\

\begin{proposition}
\label{prop:act}

For any $F \in \CS^+_{\ebk}$ and $G \in \CS^-_{-\ebk}$, we have:
\begin{equation}
\label{eqn:fixedplus}
\langle \bla | F |\bmu \rangle = F(\chi_\blamu) \prod_{\bsq \in \blamu} \left[\prod_{\square \in \bmu} \zeta \left( \frac {\chi_\bsq}{\chi_\square}\right)  \prod_{i=1}^\ebw \Big[\frac {u_i}{q\chi_\bsq} \Big]  \right]
\end{equation}
\begin{equation}
\label{eqn:fixedminus}
\langle \bmu | G |\bla \rangle  = G(\chi_\blamu) \prod_{\bsq \in \blamu} \left[\prod_{\square \in \bla} \zeta \left( \frac {\chi_\square}{\chi_\bsq}\right)  \prod_{i=1}^\ebw \Big[\frac {\chi_\bsq}{qu_i} \Big]  \right]^{-1}
\end{equation}
where the sum is over all skew diagrams $\blamu$ of size $\ebk$, and: 
$$
F(\chi_\blamu) \ := \ F \left(\{\chi_\sq\}_{\sq \in \blamu} \right)
$$
is defined for any shuffle elements as in \eqref{eqn:quantity}. \\


\end{proposition}

\begin{proof} Note that the case when $F$ or $G$ have a single variable (in other words, degree 1) was proved in \eqref{eqn:fixedpoints1} and \eqref{eqn:fixedpoints2}. Since such shuffle elements generate the whole of $\CS$, according to Theorem \ref{thm:iso}, it is enough to show that formulas \eqref{eqn:fixedplus} and \eqref{eqn:fixedminus} respect the shuffle product. Let us take care of the former, since the latter is treated analogously. For any positive shuffle elements $F * F'$ and skew diagram $\blamu$, we have:
$$
(F*F')(\blamu) \prod_{\bsq \in \blamu} \left[\prod_{\square \in \bmu} \zeta \left( \frac {\chi_\bsq}{\chi_\square}\right) \prod_{i=1}^\bw \Big[\frac {u_i}{q\chi_\bsq} \Big]  \right] = 
$$
\begin{equation}
\label{eqn:alfie}
= \sum_{\blamu = D \sqcup D'} F(D) F'(D') \prod^{\sq \in D}_{\sq' \in D'} \zeta \left( \frac {\chi_\sq}{\chi_{\sq'}} \right) \prod_{\bsq \in \blamu} \left[ \prod_{\square \in \bmu} \zeta \left( \frac {\chi_\bsq}{\chi_\square}\right) \prod_{i=1}^\bw \Big[\frac {u_i}{q\chi_\bsq} \Big] \right]
\end{equation}
where the sum is over all partitions of the set of boxes of $\blamu$. Because $\zeta(q_1^{-1}) = \zeta(q_2^{-1}) = 0$, the only summands which are non-zero are those where no box of $D$ is below or to the right of any box from $D'$. In other words, there must be a partition $\bnu$ nestled between $\bmu$ and $\bla$ such that $D = \blanu$ and $D' = \bnumu$. Hence \eqref{eqn:alfie} equals:
$$
\sum_{\bla \geq \bnu \geq \bmu} F(\blanu) F'(\bnumu) \cdot
$$
$$
\prod^{\bsq \in \blanu}_{\bsq' \in \bnumu} \zeta \left( \frac {\chi_\bsq}{\chi_{\bsq'}} \right) \prod^{\bsq \in \blanu}_{\square \in \bmu} \zeta \left( \frac {\chi_\bsq}{\chi_\square}\right)  \prod^{\bsq' \in \bnumu}_{\square \in \bmu} \zeta \left( \frac {\chi_{\bsq'}}{\chi_\square}\right) \prod_{\bsq \in \blanu}^{1\leq i \leq \bw} \Big[\frac {u_i}{q\chi_\bsq} \Big] \prod_{\bsq' \in \bnumu}^{1\leq i \leq \bw} \Big[\frac {u_i}{q\chi_{\bsq'}} \Big] 
$$ 	
The above formula is precisely what one obtains by iterating \eqref{eqn:fixedplus} twice. 
\end{proof}

\tab
Theorem \ref{thm:act} claims that the operators on $K(\bw)$ induced by the line bundle $\CL^{\otimes d}$ on the simple geometric correspondences $\fZ_i$ (see Section \ref{sec:simple}) are realized by the shuffle element $\frac {z_i^d}{q^{-1}-q}$. This begs the question as to how to construct shuffle realizations of other line bundles on geometric correspondences. The first idea one has to construct more complicated correspondences is to iterate several simple correspondences:
$$
\fZ_{i_1}^{d_1} \circ ... \circ \fZ_{i_m}^{d_m} \quad \text{corresponds to} \quad \frac {z_{i_1}^{d_1} * ... * z_{i_m}^{d_m}}{[q^{-2}]^m} \ = \ \sym \left[ \frac {z_{i_1}^{d_1}...z_{i_m}^{d_m}}{[q^{-2}]^m} \prod_{1\leq a < b \leq m} \zeta \left(\frac {z_{i_a}}{z_{i_{b}}} \right) \right]
$$
To be precise, we ought to write the shuffle element in the right hand side by replacing the variables $\{z_{i_a}, i_a \equiv c\}$ by $\{z_{c1},z_{c2},...\}$ for every residue $c$ modulo $n$, but we will gloss over this imprecision in order to maintain the elegance of shuffle formulas. Another set of geometric operators is given by the correspondences:
\begin{equation}
\label{eqn:corr1}
\fV_{\bv^+,\bv^-,\bw}^{q_1} \quad  \subset \quad \CN_{\bv^+,\bw} \times \CN_{\bv^-,\bw}
\end{equation}
defined for all $\bv^+ \geq \bv^-$ to consist of the space of linear maps:
\begin{equation}
\label{eqn:preserve}
\xymatrix{  & W_i \ar[d]^{A_i} & W_{i+1} \ar[d]^{A_{i+1}} &  \\
...  & V_i^+ \ar[l]_{Y_{i-1}} \ar@{->>}[d] & V_{i+1}^+  \ar[l]_{Y_{i}} \ar@{->>}[d]  & \ar[l]_{Y_{i+1}} ... \\
...  \ar[ru]^{X_{i-1}} & V_i^- \ar[l]^{Y_{i-1}} \ar[d]^{B_i} \ar[ru]^{X_{i}} & V_{i+1}^- \ar[l]^{Y_i} \ar[d]^{B_{i+1}} \ar[ru]^{X_{i+1}} &  \ar[l]^{Y_{i+1}} ...  \\
 & W_i  & W_{i+1} & }
\end{equation}
for a fixed collection of quotients  $\{V^+ \twoheadrightarrow V^-\}  = \{V_i^+ \stackrel{p_i}\twoheadrightarrow V_i^-\}_{1\leq i \leq n}$. In other words, both the $X$ and $Y$ maps are required to preserve the quotients, but the $X$ maps are required to act by 0 on $\text{Ker } p_i$ (and thus factor through $V_{i}^- \rightarrow V_{i+1}^+$). One defines $\fV_{\bv^+,\bv^-,\bw}^{q_2}$ similarly, by switching the roles of $X$ and $Y$.


\tab
Since we lose degrees of freedom by the requirement that $X|_{\text{Ker }p_i} = 0$, the naive guess is that the dimension of $\fV_{\bv^+,\bv^-,\bw}^{q_1}$ is smaller than middle dimension in \eqref{eqn:corr1}. However, we also lose a number of degrees of freedom in the moment map, since:
$$
\mu(X,Y,A,B) = X_{i-1}Y_{i-1} - Y_iX_i + A_iB_i \quad \text{annihilates} \quad \text{Ker }p_i
$$
and thus factors through a map $V_i^- \rightarrow V_i^+$. The correspondence $\fV^{q_1}_{\bv^+,\bv^-,\bw}$ is well-known to be smooth and middle dimensional, and in fact Lagrangian. Therefore, the $T-$character in its tangent spaces at the torus fixed points can be computed by analogy with \eqref{eqn:tansim}. Note that fixed points are parametrized by skew $\bw-$diagrams $\bla^+\backslash \bla^-$, such that no two boxes are on the same row (the latter condition is required by the vanishing on the $X$ maps):
$$
T _{\bla^+ \geq \bla^-} \fV^{q_1} = \sum_{j=1}^\bw \left(\sum_{\sq \in \bla^+}^{c_\sq \equiv j} \frac {\chi_\sq}{qu_j} + \sum_{\sq' \in \bla^-}^{c_{\sq'} \equiv j} \frac {u_j}{q\chi_{\sq'}} \right) + \sum_{\bsq, \bsq' \in \bla^+ \backslash \bla^-} \left( \delta_{c_{\bsq'}}^{c_{\bsq}+1} \cdot \frac {t\chi_\bsq}{q\chi_{\bsq'}} - \delta_{c_{\bsq'}}^{c_{\bsq}} \cdot \frac {\chi_\bsq}{\chi_{\bsq'}}\right) + 
$$
$$
+ \sum^{\sq \in \bla^+}_{\sq' \in \bla^-} \left(\delta_{c_{\sq'}}^{c_{\sq}-1} \cdot \frac {\chi_{\square}}{qt\chi_{\square'}} + \delta_{c_{\sq'}}^{c_{\sq}+1} \cdot \frac {t\chi_{\square}}{q\chi_{\square'}} - \delta_{c_{\sq'}}^{c_{\sq}} \cdot \frac {\chi_{\square}}{\chi_{\square'}} - \delta_{c_{\sq'}}^{c_{\sq}} \cdot \frac {\chi_{\square}}{q^2\chi_{\square'}} \right)
$$
With this in mind, we can use \eqref{eqn:formula} and \eqref{eqn:fixedplus} to obtain the following direct generalization of Proposition 6.7. of \citep{Nflag}. Let $\bk = \bv^+ - \bv^-$. \\
		
		
\begin{proposition}
\label{prop:smooth}
The operators induced by $\fV_{\ebv^+,\ebv^-,\ebw}^{q_1}$ and $\fV_{\ebv^+,\ebv^-,\ebw}^{q_2}$ on $K(\ebw)$ correspond to the operators induced by the following shuffle elements, respectively:
$$
B^{q_1}_\ebk \ := \ \prod_{i=1}^n  \frac {\prod_{a=1}^{k_i} \prod_{b=1}^{k_{i+1}}  \Big[ \frac {z_{i+1,b}}{qtz_{ia}} \Big]}{\prod_{a=1}^{k_i} \prod_{b=1}^{k_{i}}  \Big[ \frac {z_{ib}}{q^2z_{ia}} \Big]} \qquad \text{and} \qquad B^{q_2}_\ebk \ := \ \prod_{i=1}^n  \frac {\prod_{a=1}^{k_i} \prod_{b=1}^{k_{i-1}}  \Big[ \frac {tz_{i-1,b}}{qz_{ia}} \Big]}{\prod_{a=1}^{k_i} \prod_{b=1}^{k_{i}}  \Big[ \frac {z_{ib}}{q^2z_{ia}} \Big]}
$$
\end{proposition}

\tab
One can also consider the operator on $K(\bw)$ induced by the vector bundle $\CE \rightarrow \CN_{\bv^+,\bw} \times \CN_{\bv^-,\bw}$ that was defined in \eqref{eqn:defe}:

$$
\CE|_{\CF^+,\CF^-} \ = \ \Ext^1(\CF^+,\CF^-(-\infty))^{\BZ/n\BZ}
$$
This vector bundle has a section:
$$
s|_{\CF^+,\CF^-} \ = \ \psi(\text{tautological homomorphism}) 
$$
where the map $\psi$ is the following composition:
$$
\Hom(\CF^+, \CF^+|_\infty)^{\BZ/n\BZ} \cong \Hom(\CF^+, \CF^-|_\infty)^{\BZ/n\BZ} \longrightarrow \Ext^1(\CF^+,\CF^-(-\infty))^{\BZ/n\BZ} 
$$
The left isomorphism is due to the fact that our sheaves are framed, i.e. trivialized at $\infty$, while the middle map is part of the long exact sequence for $\text{RHom}$. From the long exact sequence, we see that the section $s$ vanishes precisely when $\CF^+ \subset \CF^-$. Therefore, the operator induced by the top exterior power of $\CE$:
\begin{equation}
\label{eqn:co}
K_{\bv^-,\bw} \ \stackrel{\fA}\longrightarrow \ K_{\bv^+,\bw} \qquad \qquad \alpha \longrightarrow \widetilde{\pi}_{+*} \Big([\CE] \cdot \pi^*_-(\alpha) \Big)
\end{equation}
vanishes unless $\bv^+ \geq \bv^-$. The following Exercise is a straightforward generalization of Proposition 6.7. of \citep{Nflag}: \\

\begin{exercise}
\label{ex:e}
The operator $\fA : K_{\bv^-,\bw} \rightarrow K_{\bv^+,\bw}$ corresponds to that induced by the shuffle element $1\in \CS_\bk^+$, where $\bk = \bv^+ - \bv^-$. 
\end{exercise}

\tab 
In the special case $\bk = k\bth$ for a natural number $k$, the shuffle element $1\in \CS_\bk^+$ was studied in \citep{FT2}, where it was denoted by $F_k$. It was shown in \loccit that the shuffle elements $\{F_k\}_{k\in \BN}$ correspond to the central Heisenberg subalgebra of $\uu \subset \UU$ at ``slope 0", as we will recall in Chapter \ref{chap:subalgebras}. 


\section{Eccentric correspondences}
\label{sec:eccentric}

\noindent For us, the richest class of geometric operators will be given by the so-called \b{eccentric correspondences}, which we define below. Particular instances of these correspondences will be shown to generate the shuffle algebra, in a sense which will be made clear in the next Chapter, and they are the operators that feature in Theorem \ref{thm:pieri}. Geometrically, these correspondences are defined for all $i<j$ and all degree vectors $\bv^+,\bv^- \in \nn$ such that $\bv^+ = \bv^- + [i;j)$:
\begin{equation}
\label{eqn:fine}
\fZ^{q_1}_{[i;j)} \ \hookrightarrow \ \CN_{\bv^+,\bw} \times \CN_{\bv^-,\bw} 
\end{equation}
To define the locus \eqref{eqn:fine}, consider a surjective map between two $n-$tuples of vector spaces $\{V^+ \twoheadrightarrow V^- \} = \{V^+_k \twoheadrightarrow V^-_k\}_{1 \leq k \leq n}$ and assume we are given a full flag of the kernel of these surjections:
\begin{equation}
\label{eqn:flag}
0 = U^{[j;j)} \subset U^{[j-1;j)} \subset ... \subset U^{[i+1;j)} \subset U^{[i;j)} = \text{Ker} \left( V^+ \twoheadrightarrow V^- \right) 
\end{equation}
The meaning of this notation is that $U^{[a;j)} = U^{[a;j)}_1 \oplus ... \oplus U^{[a;j)}_n$ is an $n-$tuple of vector spaces, such that $U^{[a;j)}_k$ has dimension equal to the number of integers $\equiv k$ mod $n$ in the interval $[a;j)$. The inclusion $U^{[a;j)}_k \subset U^{[a+1;j)}_k$ is required to have codimension $\delta_k^a$ modulo $n$. Recall the vector space $N_{\bv^+,\bw}$ of \eqref{eqn:quad} and define the subspace:
$$
Z^{q_1}_{[i;j)} \ \subset \ N_{\bv^+,\bw}
$$
of quadruples $(X,Y,A,B)$ of linear maps on the vector spaces $V^+ = \{V^+_k\}_{1\leq k \leq n}$, which preserve the quotient $V^+ \twoheadrightarrow V^-$ and: 

\begin{itemize}

\item the $X$ maps are ``nilpotent" on the flag \eqref{eqn:flag}: $\ X(U^{[a;j)}) \subset U^{[a+1;j)} \qquad \quad \ (*)$ 

\item the $Y$ maps are ``anti-nilpotent" on the flag \eqref{eqn:flag}: $\ Y(U^{[a+1;j)}) \subset U^{[a;j)} \quad (**)$

\end{itemize}



\noindent Note that the vector space $Z^{q_1}_{[i;j)}$ is invariant under the conjugation action of the parabolic subgroup $P_{[i;j)} \subset G_{\bv^+}$ which preserves both the quotients $V^+ \twoheadrightarrow V^-$ and the flag \eqref{eqn:flag}. Moreover, the moment map makes the following diagram commute:
\begin{equation}
\label{eqn:moment2}
\xymatrix{
Z^{q_1}_{[i;j)} \ar@{^{(}->}[d] \ar[r]^\nu & \fp_{[i;j)} \ar@{^{(}->}[d] \\
N_{[i;j)} \ar[r]^\mu & \fg_{\bv^+}}
\end{equation}
where $\fp_{[i;j)}$ is the vector space of traceless endomorphisms of $V^+$ which preserve the quotients $V^+ \twoheadrightarrow V^-$ and the flag \eqref{eqn:flag}. \\

\begin{definition}
\label{def:eccentric}
	
Define the \b{eccentric correspondence} as the quotient:
\begin{equation}
\label{eqn:eccentric}
\fZ^{q_1}_{[i;j)} \ = \ \nu^{-1}(0) \sslash_\bth P_{[i;j)} 
\end{equation}

\end{definition}

\tab
We will now give a naive dimension count for $\fZ^{q_1}_{[i;j)}$, by first computing the dimension of the affine space $Z^{q_1}_{[i;j)}$ and then subtracting the dimensions of $P_{[i;j)}$ and $\fp_{[i;j)}$. It is easy to see that:
$$
\dim \Big( X\text{ maps} \Big) \quad = \quad \sum_{k=1}^n v_k^- v_{k+1}^+ + \dim \Big( \text{nilpotent maps }(*) \text{ on the flags \eqref{eqn:flag}} \Big)
$$
$$
\dim \Big( Y\text{ maps}\Big) = \sum_{k=1}^n v_{k+1}^- v_k^+ + \dim \Big( \text{anti-nilpotent maps }(**) \text{ on the flags \eqref{eqn:flag}} \Big)
$$
It is straightforward to compute the dimension of the spaces of nilpotent and anti-nilpotent linear maps between two flags of vector spaces, so we give the answer here:
$$
\dim \Big( \text{nilpotent maps }(*) \Big) \ \quad = \quad \ \sum_{i \leq a \leq b - 1 < j - 1} \delta^{a}_{b-1} 
$$
$$
\dim \Big( \text{anti-nilpotent maps }(**) \Big) \ = \sum_{i \leq a \leq b + 1 < j + 1 } \delta^{a}_{b+1} 
$$
where the congruences are taken modulo $n$. We conclude that:
$$
\dim Z^{q_1}_{[i;j)} = \sum_{k=1}^n \Big( v_k^- v_{k+1}^+ + v_{k+1}^- v_k^+ + w_kv_k^+ + w_k v_k^-\Big) + \sum_{i \leq a \leq b - 1 < j - 1} \delta^{a}_{b-1} + \sum_{i \leq a \leq b + 1 < j + 1} \delta^{a}_{b+1} 
$$
Moreover, we have:
$$
\dim P_{[i;j)} \ = \ \dim \fp_{[i;j)}+1 \ = \ \sum_{k=1}^n v^-_kv^+_k + \sum_{i\leq a \leq b < j} \delta^a_b 
$$
where we note that the dimension of $\fp_{[i;j)}$ is one smaller than that of $P_{[i;j)}$ because we have imposed the traceless condition. We have the identity:
$$
\sum_{i \leq a \leq b - 1 < j - 1} \delta^a_{b-1} + \sum_{i \leq a \leq b + 1 < j + 1} \delta^{a}_{b+1} - 2 \sum_{i\leq a \leq b < j} \delta^a_b + 1 = \sum_{k=1}^n (v_k^+ - v_k^-)(v_{k+1}^+ - v_{k+1}^-) - (v_k^+ - v_k^-)^2
$$
which can be proved by first for $j\in \{i,...,i+n-1\}$ and then showing that the two sides are invariant under $[i;j)\mapsto [i;j+n)$. We conclude that $\quad \dim \fZ^{q_1}_{[i;j)} \geq$
$$
\sum_{k=1}^n \left(- \frac {(v_k^+-v_{k+1}^+)^2}2 + w_kv_k^+ - \frac {(v_k^--v_{k+1}^-)^2}2 + w_k v_k^- \right) \ = \ \frac {\dim \CN_{\bv^+,\bw} + \dim \CN_{\bv^-,\bw}}2
$$
Equality is attained if and only if the derivative of the moment map $\nu$ of \eqref{eqn:moment2} is surjective on the semistable locus, i.e. the equations $\nu = 0$ cut out a subvariety of minimal possible dimension. We conjecture that this is the case, which implies that $\fZ^{q_1}_{[i;j)}$ is a local complete intersection, and hence we can make sense of its tangent space by subtracting the $\dim P_{[i;j)} + \dim \fp_{[i;j)}$ relations from the $\dim Z^{q_1}_{[i;j)}$ affine coordinates. 

\tab 
Fixed points of $\fZ^{q_1}_{[i;j)}$ consist of two pieces of data, the first of which is a skew $\bw-$diagram $\cray$ of size $[i;j)$. However, the condition that the parabolic subgroup $P_{[i;j)}$ preserve the flag \eqref{eqn:flag} determines that our second piece of data is a labeling:
$$
\psi \ : \ \{\sq \in \cray\} \ \stackrel{\text{bijection}}\longrightarrow \ \{i,...,j-1\} \qquad \text{s.t.} \quad \psi(\sq) \equiv \col \sq, \ \ \forall \ \sq \in \cray
$$
and the labels of the boxes $\sq$ increase as we go up and to the right in the skew $\bw-$diagram $\cray$, with the possible exception that the box with label $a-1$ is allowed to be directly above the box labeled $a$. Without this latter relaxation, the labeling would have given rise to what is known as a standard Young tableau. We therefore call $(\cray,\psi)$ a \b{positive almost standard Young tableau}, and shorten this terminology by $\asyt^+$. The following is proved similarly with Exercise \ref{ex:tangent2}: \\  

\begin{exercise}
\label{ex:tangent3}

The $T-$character in the tangent spaces to $\fZ^{q_1}_{[i;j)}$ are given by:
\begin{equation}
\label{eqn:tywin}
T_{(\bla^+ \geq \bla^-, \psi)} \fZ^{q_1}_{[i;j)} = \sum_{k = 1}^{\bw} \left( \sum_{\square \in \bla^+}^{c_\square \equiv k} \frac {\chi_\square}{qu_k} + \sum_{\square' \in \bla^-}^{c_{\square'} \equiv k} \frac {u_k}{q\chi_{\square'}} \right) + \frac 1{q^2} +
\end{equation}
$$
+ \sum_{i \leq a \leq b - 1 < j - 1}^{a \equiv b-1} \frac {\chi_{b}}{qt\chi_{a}} + \sum_{i \leq a \leq b + 1 < j + 1}^{a \equiv b+1} \frac {t\chi_{b}}{q\chi_{a}} - \sum_{i \leq a \leq b < j}^{a \equiv b} \frac {\chi_{b}}{\chi_{a}} - \sum_{i \leq a \leq b < j}^{a\equiv b} \frac {\chi_{b}}{q^2\chi_{a}}  +
$$
$$
+ \sum^{\sq \in \bla^+}_{\sq' \in \bla^-} \left(\delta_{c_{\sq'}}^{c_{\sq}-1} \cdot \frac {\chi_{\square}}{qt\chi_{\square'}} + \delta_{c_{\sq'}}^{c_{\sq}+1} \cdot \frac {t\chi_{\square}}{q\chi_{\square'}} - \delta_{c_{\sq'}}^{c_{\sq}} \cdot \frac {\chi_{\square}}{\chi_{\square'}} - \delta_{c_{\sq'}}^{c_{\sq}} \cdot \frac {\chi_{\square}}{q^2\chi_{\square'}} \right)
$$
where $\chi_a$ denotes the weight of the box labeled $a$ in the $\asyt^+$ defined by $\psi$. 
\end{exercise}



\tab 
Besides the natural projection maps $\pi_\pm:\fZ_{[i;j)}^{q_1} \rightarrow \CN_{\bv^\pm, \bw}$, the fact that eccentric correspondences are $P_{[i;j)}$ quotients implies that they are endowed with line bundles:
$$
\CL_i,..., \CL_{j-1} \ \in \ \text{Pic} \left( \fZ_{[i;j)}^{q_1} \right)
$$
where the fibers of $\CL_a$ consist of the one-dimensional vector spaces $U^{[a;j)}/U^{[a+1;j)}$ corresponding to the flag \eqref{eqn:flag}. Then we may consider the operator
$$
P^{(d)}_{\pm [i;j)} \ : \ K_\bullet(\bw) \longrightarrow K_{\bullet\pm [i;j)}(\bw) \qquad \qquad \alpha \mapsto \widetilde{\pi}_{\pm *} \left( \CL_{i}^{d_{i}} \cdot ... \cdot \CL_{j-1}^{d_{j-1}} \cdot \pi_\mp^{*}(\alpha) \right)
$$
for any vector of integers $(d) = (d_i,...,d_{j-1})$. Then formula \eqref{eqn:tywin} allows us to compute the matrix coefficients of the above operators in the fixed point basis, via \eqref{eqn:formula}. \\

\begin{exercise}
\label{ex:eccentric}
The operators $P^{(d)}_{\pm [i;j)}$ act on $K(\bw)$ according to the shuffle elements:
\begin{equation}
\label{eqn:defp0}
P^{(d)}_{\pm [i;j)} \ = \ \textrm{Sym} \left[\frac {z_i^{d_i} \ ... \ z_{j-1}^{d_{j-1}}}{[q^{-2}] \prod_{a=i+1}^{j-1} \Big[\frac {tz_{a-1}}{q z_{a}} \Big]} \prod_{i \leq a < b < j} \zeta \left( \frac {z_{b}}{z_{a}} \right) \right] \ \in \ \CS^\pm 
\end{equation}
We abuse notation and denote the geometric operator on $K-$theory and the algebraic shuffle element $\in \CS^\pm$ by the same symbol. 
\end{exercise}

\tab 
If we switch the roles of $X$ and $Y$ in the definition of the eccentric correspondences, we obtain the analogous:
$$
\fZ^{q_2}_{[i;j)} \ \hookrightarrow \ \CN_{\bv^+,\bw} \times \CN_{\bv^-,\bw} 
$$
which parametrize quadruples $(X,Y,A,B)$ such that the $X$ maps are now anti-nilpotent and the $Y$ maps are nilpotent. Fixed points of this correspondence consist of a $\bw-$diagram $\cray$, together with a labeling:
$$
\psi \ : \ \{\sq \in \cray\} \ \stackrel{\text{bijection}}\longrightarrow \ \{i,...,j-1\} \qquad \text{s.t.} \quad \psi(\sq) \equiv \col \sq, \ \ \forall \ \sq \in \cray
$$
and the labels decrease as we go up and to the right in the skew $\bw-$diagram $\cray$, with the possible exception that the box with label $a$ is allowed to be directly to the right of the box labeled $a-1$. Such a labeling will be called a \b{negative almost standard Young tableau}, and we shorten this terminology to $\asyt^-$. The corresponding operators on $K-$theory are induced by the following shuffle elements:
\begin{equation}
\label{eqn:defq0}
Q^{(d)}_{\pm[i;j)} \ = \ \textrm{Sym} \left[\frac {z_i^{d_i} ... z_{j-1}^{d_{j-1}}}{[q^{-2}] \prod_{a=i+1}^{j-1} \Big[\frac {z_{a}}{qtz_{a-1}} \Big]} \prod_{i \leq a < b < j} \zeta \left( \frac {z_{a}}{z_{b}} \right) \right] \ \in \ \CS^\pm
\end{equation}
which is proved by analogy with Exercise \ref{ex:eccentric}. The following is a generalization of Exercise \ref{ex:correspondence}, and proved in a similar fashion. \\

\begin{exercise}
\label{ex:lagragian}

For any vector $d = (d_i,...,d_{j-1})$ of integers, the operators $P_{[i;j)}^{(d)}$ and $Q_{-[i;j)}^{(d)}$ are Lagrangian in the positive stable basis for any $\bm \in \qq$:
$$
P_{[i;j)}^{(d)} \cdot s_\bmu^{+,\bm} = \sum_\bla s_\bla^{+,\bm} \cdot \text{coeffs in }\BF_\bw \qquad \qquad Q_{-[i;j)}^{(d)} \cdot s_\bla^{+,\bm} = \sum_\bmu s_\bmu^{+,\bm} \cdot \text{coeffs in }\BF_\bw 
$$
while their transposes $P_{-[i;j)}^{(d)}$ and $Q_{[i;j)}^{(d)}$ are Lagrangian in the negative stable basis. 


\end{exercise}

\tab 
While we claimed that $P_{\pm [i;j)}^{(d)}$ and $Q_{\pm [i;j)}^{(d)}$ given by formulas \eqref{eqn:defp0} and \eqref{eqn:defq0} are shuffle elements, we have not yet proved that they indeed satisfy the wheel conditions. This will be proved in the following Exercise, together with the fact that these special shuffle elements lie in the image of $\Upsilon : \UU \rightarrow \CS$. We will show in Chapter \ref{chap:subalgebras} that the various $P_{\pm [i;j)}^{(d)}$ and $Q_{\pm [i;j)}^{(d)}$ generate $\CS$, which will prove Theorem \ref{thm:iso}. \\

\begin{exercise}
\label{ex:gen}
For any vector of integers $(d) = (d_i,...,d_{j-1})$, the rational functions: 
$$
P_{\pm[i;j)}^{(d)} \quad \text{and} \quad Q_{\pm[i;j)}^{(d)} \quad \text{lie in} \quad \CS^\pm
$$
i.e. they satisfy the wheel conditions of \eqref{eqn:shufelem}. Moreover, these shuffle elements lie in the subalgebra of $\CS^\pm$ generated by $\{z_i^d\}_{1\leq i \leq n}^{d\in \BZ}$, i.e. in the image of $\Upsilon$ of Theorem \ref{thm:iso}.  	
\end{exercise}

\chapter{Subalgebras and Root $R-$matrices}
\label{chap:subalgebras}

\section{The quantum group}
\label{sec:quantum}


\noindent The famous $RTT$ presentation of the quantum group $U_q(\fgl_n)$ was introduced by \citep{FRT}. We will be concerned with a certain presentation of the affine quantum group $\uu$, which was developed by \citep{FRT2}, \citep{RS} and \citep{DF}. Note that while our construction is very close to that of \emph{loc. cit.}, there will be differences in notation. Let $E_{ij}$ denote the elementary $n \times n$ matrix with a single entry 1 at row $i$ and column $j$, and zero everywhere else. Consider the following tensor product of matrices:
$$
R\left(\frac zw \right) = \sum_{1\leq i,j \leq n} E_{ii} \otimes E_{jj} \left( \frac {zq-wq^{-1}}{w-z} \right)^{\delta_j^i} + (q-q^{-1}) \sum_{1\leq i \neq j \leq n} E_{ij} \otimes E_{ji} \frac {w^{\delta_{i>j}}z^{\delta_{i<j}}}{w-z} 
$$
which differs by the one in \loccit by a scalar multiple, the transformation $q\mapsto -q^{-1}$, and taking the transpose. By appealing to \loccit we conclude that $R$ satisfies the \b{quantum Yang-Baxter equation}:
\begin{equation}
\label{eqn:qybe}
R_{12}\left(\frac zw \right) R_{13}\left(\frac zy \right) R_{23}\left( \frac wy \right) \ = \ R_{23}\left( \frac wy \right) R_{13}\left(\frac zy \right) R_{12}\left(\frac zw \right)
\end{equation}
where $R_{12} = R \otimes \text{Id} \in \text{Mat} \otimes \text{Mat} \otimes \text{Mat}$ etc. The quantum group $\uu$ is defined as the Hopf algebra generated by symbols: 
\begin{equation}
\label{eqn:quantumgroup}
\uu \ = \ \left \langle e_{\pm [i;j)}, \psi_k, c \right \rangle_{i<j}^{k\in \BZ} 
\end{equation}
where the Cartan elements $\psi_k$ are commutative, $c$ is central, and we have $\psi_{k+n} = c\psi_k$ for all $k\in \BZ$. In order to define the relations between the remaining generators, we identify $e_{\pm [i;j)} = e_{\pm [i+n;j+n)}$ and place them in two matrix valued power series:
\begin{equation}
\label{eqn:t+}
T^+(z) \ \ = \ \ \sum^{i\leq j}_{1\leq i \leq n} e_{[i;j)} \psi_i \cdot E_{j\text{ mod } n,i} z^{\left \lceil \frac jn \right \rceil - 1} 
\end{equation}
\begin{equation}
\label{eqn:t-}
T^-(z) = \sum^{i\leq j}_{1\leq i \leq n} e_{-[i;j)} \psi_i^{-1} \cdot E_{i, j\text{ mod }n} z^{- \left \lceil \frac jn \right \rceil + 1} 
\end{equation}
where the residue of $j$ modulo $n$ is taken to lie in $\{1,...,n\}$, in order to match our conventions on indexing elementary matrices. We write $e_{\pm [i;i)} = 1$ and $e_{\pm [i;j)} = 0$ if $i>j$. We impose the following \b{RTT relations}:
\begin{equation}
\label{eqn:rtt1}
R\left(\frac zw \right) T^+_1(z) T^+_2(w) = T^+_2(w) T^+_1(z)  R\left(\frac zw \right)
\end{equation}
\begin{equation}
\label{eqn:rtt2}
R\left(\frac zw \right) T^-_1(z) T^-_2(w) = T^-_2(w) T^-_1(z)  R\left(\frac zw \right)
\end{equation}
and:
\begin{equation}
\label{eqn:rtt3}
R\left(\frac z{wc} \right) T^-_2(w) T^+_1(z)  = T^+_1(z) T^-_2(w)  R\left(\frac {zc}w \right) 
\end{equation}
between the generators of $\uu$, where we set $X_1 = X \otimes \text{Id}$ and $X_2 = \text{Id} \otimes X$ for any matrix $X$. In particular, we have the relations:
\begin{equation}
\label{eqn:car}
\psi_k \cdot e_{\pm [i;j)}   = (-q)^{\pm(\delta_k^i - \delta_k^j)} e_{\pm [i;j)} \cdot \psi_k \qquad \forall \ \text{arcs } [i;j) \text{ and } k\in \BZ
\end{equation}
For arbitrary $a\leq c$ and $b\leq d$, unwinding relations \eqref{eqn:rtt1} and \eqref{eqn:rtt2} gives us:
$$
\frac {e_{\pm [a;c)} e_{\pm [b,d)}}{(-q)^{\delta^b_a - \delta_b^d + \delta^d_a}} - \frac {e_{\pm [b,d)} e_{\pm [a;c)}}{(-q)^{\delta^b_c - \delta_b^d + \delta^d_c}}  = (q-q^{-1}) \left[ \sum_{a \leq x < c}^{x\equiv d} e_{\pm [b,c+d-x)} e_{\pm [a;x)} - \sum_{a<x\leq c}^{x\equiv b} e_{\pm [x;c)} e_{\pm [a+b-x,d)}  \right]
$$
while unwinding relation \eqref{eqn:rtt3} implies:
$$
\left[ e_{[a;c)}, e_{-[b,d)} \right] = (q-q^{-1}) \left[ \sum_{a \leq x < c}^{x\equiv b} \frac {e_{-[c+b-x,d)} e_{[a;x)}}{(-q)^{\delta^b_c + \delta^a_c - \delta_b^a}} \cdot \frac {\psi_x}{\psi_c} - \sum_{a<x\leq c}^{x\equiv d} \frac {e_{[x;c)} e_{-[b,a+d-x)}}{(-q)^{- \delta^b_a + \delta^d_b - \delta_a^d}} \cdot \frac {\psi_x}{\psi_a} \right]
$$
We can introduce a bialgebra structure on $\uu$ by setting $\Delta(\psi_k) = \psi_k \otimes \psi_k$ and:
$$
\Delta(T^+(z)) \ = \ T^+(z) \otimes T^+(z c_1) \qquad \qquad \Delta(T^-(z)) \ = \ T^-(z c_2) \otimes T^-(z)
$$
where we write $c_1 = c \otimes 1$ and $c_2 = 1 \otimes c$. These formulas imply the relations: 
\begin{equation}
\label{eqn:ray1}
\Delta(e_{[i;j)}) \ = \ \sum_{a=i}^j e_{[a;j)} \frac {\psi_a}{\psi_i} \otimes e_{[i;a)} \qquad \quad \Delta(e_{-[i;j)}) \ = \ \sum_{a=i}^j e_{-[i;a)} \otimes e_{-[a;j)} \frac {\psi_i}{\psi_a}  
\end{equation}for all arcs $[i;j)$. Finally, we can define the antipode map:
$$
S^\pm(z) \ := \ S(T^\pm(z)) \ = \ \left(T^\pm(z) \right)^{-1} 
$$
which satisfies the opposite coproduct relations from $T^\pm(z)$ above. We will consider the following elements $f_{\pm [i;j)} \in \uu$ obtained from the antipode by:
\begin{equation}
\label{eqn:s+}
S^+(z) \ = \ \sum^{i\leq j}_{1\leq i \leq n} (-1)^{j-i} f_{[i;j)} \psi_j \cdot E_{j\text{ mod } n,i} z^{\left \lceil \frac jn \right \rceil - 1} 
\end{equation}
\begin{equation}
\label{eqn:s-}
S^-(z) = \sum^{i\leq j}_{1\leq i \leq n} (-1)^{j-i} f_{-[i;j)} \psi_j^{-1} \cdot E_{i,j\text{ mod }n} z^{- \left \lceil \frac jn \right \rceil + 1} 
\end{equation}
which will satisfy the following coproduct relations:
\begin{equation}
\label{eqn:ray2}
\Delta(f_{[i;j)}) \ = \ \sum_{a=i}^j f_{[i;a)} \frac {\psi_j}{\psi_a} \otimes f_{[a;j)} \qquad \quad \Delta(f_{-[i;j)}) \ = \ \sum_{a=i}^j f_{-[a;j)} \otimes f_{-[i;a)} \frac {\psi_a}{\psi_j}
\end{equation}
Consider the subalgebras $U_q^\pm(\dot{\fgl}_n) \subset \uu$ generated by $\{e_{\pm [i;j)}\}_{i<j}$, or alternatively by the antipodes $\{f_{\pm[i;j)}\}_{i<j}$. If we enlarge these subalgebras to:
\begin{equation}
\label{eqn:bialgebras}
\uug \ = \ \left \langle \uup, \{\psi_k\}_{k\in \BZ}\right \rangle \qquad \uul \ = \ \left \langle \uum, \{\psi_k\}_{k\in \BZ} \right \rangle
\end{equation}
then \eqref{eqn:ray1} and \eqref{eqn:ray2} establish the fact that \eqref{eqn:bialgebras} are sub-Hopf algebras of $\uu$. The following Exercise proves that $\uu$ is the double of the subalgebras \eqref{eqn:bialgebras}. \\

\begin{exercise}
\label{ex:hopf}

There is a Hopf pairing $\uug \otimes \uul \rightarrow \BQ(q)$ generated by:
\begin{equation}
\label{eqn:rpair}
\langle T^+_1(z),  T^-_2(w) \rangle \ = \ R\left(\frac zw \right)
\end{equation}
and \eqref{eqn:bialg}, which makes $\uu$ into a Drinfeld double.

\end{exercise}


\tab 
On general grounds, one can use \eqref{eqn:bialg} to obtain $\langle S_1^+(z), S_2^-(w) \rangle = R\left(z/w\right)$ from \eqref{eqn:rpair}. Going back to \eqref{eqn:t+}, \eqref{eqn:t-}, \eqref{eqn:s+}, \eqref{eqn:s-}, we infer the following relations in terms of the generators of the algebra $\uu$:
\begin{equation}
\label{eqn:charles}
\langle \psi_i^{-1}, \psi_j \rangle = (-q)^{-\delta_j^i} \qquad \quad \langle e_{-[i;j)}, e_{[i;j)} \rangle \ = \ \langle f_{-[i;j)}, f_{[i;j)} \rangle \ = \ 1-q^2
\end{equation}
Comparing the structure above with the definition of the quantum group $\su$ defined in Section \ref{sec:basicquantum}, we observe that we have a homomorphism:
$$
\su \ \stackrel{\tau}\hookrightarrow \ \uu \qquad \qquad e_i^+ \mapsto \frac {e_{[i;i+1)}}{1-q^2} \qquad e_i^- \mapsto \frac {e_{-[i;i+1)}}{q^{-1}-q} \qquad \ph_i \mapsto \frac {\psi_{i+1}}{\psi_i}
$$
which preserves the Hopf algebra structure. In fact, we will often work with the slightly smaller subalgebra of $\uu$ generated by the $e_{\pm [i;j)}$ and the ratios $\ph_i := \frac {\psi_{i+1}}{\psi_i}$. Since the Hopf structure structure preserves the smaller subalgebra, we will often neglect the extra information on the individual $\psi_i$, and use the notation $\uu$ to refer to the slightly smaller subalgebra that is generated by $e_{\pm[i;j)}$'s and all $\ph_i$'s. Note that the embedding $\tau$ preserves the $\zz-$grading on the two algebras involved, where:
$$
e_{\pm [i;j)} \text{ and }f_{\pm [i;j)} \ \in \ \uu_{\pm [i;j)} \qquad \text{and} \qquad \psi_k \ \in \ \uu_0
$$
It is possible to complete the embedding $\tau$ to an isomorphism of Hopf algebras:
\begin{equation}
\label{eqn:omega}
\su \otimes U_q(\dot{\fgl}_1) \ \stackrel{\widetilde{\tau}}\cong \ \uu
\end{equation}
The isomorphism $\widetilde{\tau}$ preserves the grading on the two sides, where the Heisenberg generators $p_k$ of $\uu$ are given degree $k\bth = (k,...,k)$. Therefore, we may use \eqref{eqn:omega} to obtain the following estimate on the dimension of the graded pieces of the positive half of the quantum group $\uup := \langle e_{[i;j)} \rangle_{i<j} \subset \uu$:
\begin{equation}
\label{eqn:dimestimate}
\dim \ \uup_\bk \ = \ \# \{\rho \vdash \bk\}
\end{equation}
where a \b{partition} of a degree vector $\bk$ is any unordered sum of arcs inside $\nn$:
\begin{equation}
\label{eqn:partition}
\rho \ : \ \bk = [i_1;j_1)+...+[i_t;j_t)
\end{equation}
If we did not wish to appeal to the existence of the isomorphism $\widetilde{\tau}$, one could obtain the estimate \eqref{eqn:dimestimate} by obtaining a PBW basis of the positive subalgebra $\uup$ (refer to Theorem 2.11 of \citep{GM} for a construction). \\

\section{Subalgebras of the shuffle algebra}
\label{sec:slopefiltration}

\noindent For any vector of rational numbers $\bm = (m_1,...,m_n) \in \BQ^n$, we consider the following subspaces of the positive and negative shuffle algebras, respectively:
\begin{equation}
\label{eqn:limit1}
\CS^+_{\leq \bm} \ = \ \Big \{F \ \Big| \lim_{\xi \rightarrow \infty} \frac {F(z_{i1},...,z_{i a_i},\xi z_{i,a_i+1},..., \xi z_{i,a_i+b_i})}{\xi^{\bm \cdot \bb - \frac {(\ba, \bb )}2}} \ \text{ exists } \forall \ \ba,\bb \Big\} \ \subset \ \CS^+ 
\end{equation}
\begin{equation}
\label{eqn:limit2}
\CS^-_{\leq \bm} \ = \ \Big \{G \ \Big| \lim_{\xi \rightarrow 0} \frac {G(\xi z_{i1},..., \xi z_{i a_i}, z_{i,a_i+1},..., z_{i,a_i+b_i})}{\xi^{-\bm \cdot \ba + \frac {(\ba,\bb)}2}} \ \text{ exists } \forall \ \ba,\bb \Big\} \ \subset \ \CS^- 
\end{equation}
Shuffle elements in \eqref{eqn:limit1} and \eqref{eqn:limit2} will be said to have \b{slope} $\leq \bm$. \\

\begin{proposition}

For all $\ebm\in \qq$, the subspaces $\CS^\pm_{\leq \ebm} \subset \CS^\pm$ are subalgebras. \\

\end{proposition}

\begin{proof} The proposition is an immediate consequence of the definition of the shuffle multiplication in \eqref{eqn:shuffle}, as well as the simple fact that:
\begin{equation}
\label{eqn:limitrof1}
\lim_{\xi \rightarrow \infty} \ \zeta\left(\frac {\xi z_i}{z_j} \right) \cdot \xi^{\frac {(\vs^i,\vs^j)}2} \qquad \text{exists and equals} \quad \left(- \sqrt{\frac {qz_i}{z_j}} \right)^{-(\vs^i,\vs^j)} t^{\frac {\delta_{j-1}^i - \delta_{j+1}^i}2}
\end{equation}
\begin{equation}
\label{eqn:limitrof2}
\lim_{\xi \rightarrow 0} \ \zeta\left(\frac {\xi z_i}{z_j} \right) \cdot \xi^{-\frac {(\vs^i,\vs^j)}2} \qquad \text{exists and equals}  \quad \left(\sqrt{\frac {qz_i}{z_j}} \right)^{(\vs^i,\vs^j)} t^{\frac {\delta_{j+1}^i - \delta_{j-1}^i}2}
\end{equation}
for any variables $z_i$ and $z_j$ of colors $i$ and $j$, respectively.
\end{proof}

\tab
By the definition of the coproduct \eqref{eqn:cop1} and \eqref{eqn:cop2}, we have:
\begin{equation}
\label{eqn:topdegree0}
\Delta(F) = (\text{anything}) \otimes (\text{slope} \leq \bm) \qquad \qquad \forall \ F \in \CS^+_{\leq \bm} 
\end{equation}
\begin{equation}
\label{eqn:botdegree0}
\Delta(G) \ = \ (\text{slope} \leq \bm) \otimes (\text{anything}) \qquad \qquad \forall \ G \in \CS^-_{\leq \bm} 
\end{equation}
In other words, having slope $\leq \bm$ is a property that is preserved in one of the tensor factors of the coproduct. We will say that a shuffle element has slope $<\bm$ if it has slope $\leq \bm - \varepsilon$ for some small $\varepsilon \in \BQ_+^n$. Then formulas \eqref{eqn:topdegree0} and \eqref{eqn:botdegree0} can be made more precise by writing:
\begin{equation}
\label{eqn:topdegree}
\Delta(F) \ = \ \Delta_\bm(F) + (\text{anything}) \otimes (\text{slope} < \bm) \qquad \qquad \forall \ F \in \CS^+_{\leq \bm} 
\end{equation}
\begin{equation}
\label{eqn:botdegree}
\Delta(G) \ = \ \Delta_\bm(G) + (\text{slope} < \bm) \otimes (\text{anything}) \qquad \qquad \forall \ G \in \CS^-_{\leq \bm} 
\end{equation}
where the leading term is defined to be:
\begin{equation}
\label{eqn:leadingdelta1}
\Delta_\bm(F) \ := \ \sum_{\ba+\bb = \bk} \lim_{\xi \rightarrow \infty} \frac {\Delta(F)_{\ba \otimes \bb}}{\xi^{\bm \cdot \bb}} \ = \ \sum_{\ba+\bb = \bk} \ \lim_{\xi \rightarrow \infty} \frac {\ph_{\bb} F(z_{i1},..., z_{i a_i} \otimes \xi z_{i,a_i+1},..., \xi z_{i,a_i+b_i})}{\xi^{\bm \cdot \bb -  \frac {(\ba, \bb)}2} \cdot \text{monomials}}
\end{equation}
\begin{equation}
\label{eqn:leadingdelta2}
\Delta_\bm(G) \ := \ \sum_{\ba+\bb = \bk} \lim_{\xi \rightarrow 0} \frac {\Delta(G)_{\ba\otimes \bb}}{\xi^{-\bm\cdot \ba}} \ =  \ \sum_{\ba+\bb = \bk} \lim_{\xi \rightarrow 0} \frac {G(..., \xi z_{i1},..., \xi z_{i a_i}, z_{i,a_i+1},..., z_{i,a_i+b_i},...) \ph_{-\ba}}{\xi^{-\bm \cdot \ba + \frac {(\ba,\bb)}2} \cdot \text{monomials}}
\end{equation}
where we employ the notation $\ph_\bk = \prod_{i=1}^n \ph_i^{k_i}$ for all $\bk = (k_1,...,k_n) \in \nn$. In the above formulas, the word ``monomials" refers to certain products of $q,t$ and powers of the $z_{is}$ which arise from the right hand sides of the various limits \eqref{eqn:limitrof1} and \eqref{eqn:limitrof2} which appear when computing the limits \eqref{eqn:leadingdelta1} and \eqref{eqn:leadingdelta2}. As we will not need them explicitly, we will not bother writing them out. We will use the notation:
$$
\CS^+_{\leq \bm|\bk} \ = \ \CS^+_{\leq \bm} \cap \CS^+_\bk \qquad \qquad \CS^+_{\leq \bm|\bk,d} \ = \ \CS^+_{\leq \bm} \cap \CS^+_{\bk,d} 
$$
$$
\CS^-_{\leq \bm|-\bk} \ = \ \CS^-_{\leq \bm} \cap \CS^-_{-\bk} \qquad \qquad \CS^-_{\leq \bm|-\bk,d} \ = \ \CS^-_{\leq \bm} \cap \CS^-_{-\bk,d} 
$$
for the various graded pieces. By considering property \eqref{eqn:limit1} when $\ba = 0$, we observe that the bigraded piece $\CS^+_{\leq \bm|\bk,d}$ is empty unless $d\leq \bm \cdot \bk$. Similarly, the bigraded piece $\CS^-_{\leq \bm|-\bk,d}$ is empty unless $d \leq - \bm \cdot \bk$. We will clump together all of these bigraded pieces where the second degree $d$ is as large or small as possible into the subalgebras:
\begin{equation}
\label{eqn:emb1}
\CB^+_\bm \ = \ \bigoplus^{\bm \cdot \bk \in \BZ}_{\bk \in \BN^n} \CB^+_{\bm|\bk} \ := \ \bigoplus^{\bm \cdot \bk  = d}_{\bk \in \BN^n, d \in \BZ} \CS^+_{\leq \bm|\bk,d} \subset \CS^+
\end{equation}
\begin{equation}
\label{eqn:emb2}
\CB^-_\bm \ = \ \bigoplus^{\bm \cdot \bk \in \BZ}_{\bk \in \BN^n} \CB^-_{\bm|-\bk} \ := \ \bigoplus^{\bm \cdot \bk  = - d}_{\bk \in \BN^n, d \in \BZ} \CS^-_{\leq \bm|-\bk,d} \subset \CS^-
\end{equation}
Shuffle elements in \eqref{eqn:emb1} and \eqref{eqn:emb2} will be said to have \b{slope} $ = \bm$. It is easy to see that the coproduct $\Delta_\bm$ descends to a coproduct on the slightly larger subalgebras:
\begin{equation}
\label{eqn:enh1}
\CB^\geq_\bm \ := \ \left \langle \CB^+_\bm, \ \{\ph_i\}_{1\leq i \leq n} \right \rangle \ \subset \ \CS^\geq
\end{equation}
\begin{equation}
\label{eqn:enh2}
\CB^\leq_\bm \ := \ \left \langle \CB^-_\bm, \ \{\ph^{-1}_i\}_{1\leq i \leq n} \right \rangle \ \subset \ \CS^\leq
\end{equation}
The following Proposition establishes much of our interest in these subalgebras: \\

\begin{proposition}
\label{prop:silly}

The pairing \eqref{eqn:daddypair} still satisfies the bialgebra property between the subalgebras:
$$
\CB^\leq_\ebm \otimes \CB^\geq_\ebm \ \longrightarrow \ \BQ(q)
$$
with respect to the coproduct $\Delta_\ebm$. This allows us to construct the Drinfeld double:
$$
\CB_\ebm = \CB^\leq_\ebm \otimes \CB^\geq_\ebm
$$
The embeddings \eqref{eqn:enh1} and \eqref{eqn:enh2} extend to an embedding of the doubles:
\begin{equation}
\label{eqn:emb3}
\CB_\ebm \ \subset \ \CS
\end{equation}

\end{proposition}

\begin{proof} By Exercise \ref{ex:pairing}, we have:
$$
\langle \Delta^\op(G), F \otimes F' \rangle = \langle G, F * F' \rangle \qquad \qquad \langle G \otimes G', \Delta(F) \rangle = \langle G * G', F \rangle
$$
for any $F,F' \in \CS^+$ and $G,G' \in \CS^-$. If we further assume $F,F' \in \CB^+_\bm$ and $G,G' \in \CB^-_\bm$, we need to show that the above properties still hold if we replace $\Delta$ by $\Delta_\bm$. However, this is clear from \eqref{eqn:topdegree} and \eqref{eqn:botdegree} since:
$$
\Big \langle (\text{anything}) \otimes (\text{slope} < \bm), F \otimes F' \Big \rangle \ = \ \Big \langle G \otimes G', (\text{anything}) \otimes (\text{slope} < \bm) \Big \rangle \ = \ 0
$$
for any $F,F',G,G'$ of slope $= \bm$. This follows from the definition of the slope, and because $\langle \cdot, \cdot \rangle$ only pairs non-trivially shuffle elements of opposite bidegrees. A similar argument proves that \eqref{eqn:emb3} is an algebra homomorphism: one needs to show that whenever \eqref{eqn:reldrinfeld} holds in $\CB_\bm$ with respect to the coproduct $\Delta_\bm$, it holds in $\CS$ with respect to the coproduct $\Delta$. \\
\end{proof}

\section{Arcs and subalgebras}
\label{sec:arcs}
	
\noindent One of the main results of this Chapter is to identify the subalgebras of \eqref{eqn:emb3}: \\

\begin{theorem}
\label{thm:subalg}
For any rational slope $\ebm \in \qq$, we have an isomorphism:
\begin{equation}
\label{eqn:subalgebra}
\CB_\ebm \ \cong \ \bigotimes_{h=1}^g U_q(\dot{\fgl}_{l_h})
\end{equation}
where the natural numbers $g,l_1,...,l_g$ depend on $\ebm$.
\end{theorem}

\tab
Let us now show how to construct the natural numbers $g,l_1,...,l_g$ which appear in the above Theorem. We say that an arc $[i;j)$ in the cyclic quiver is $\bm-$\textbf{integral} if:
$$
\bm \cdot [i;j) = m_i+...+m_{j-1} \in \BZ
$$
We leave it as a simple exercise for the reader to show that there exists a $\bm-$integral arc as above, starting at each vertex $i$. Therefore, there exists a well-defined minimal $\bm-$integral arc interval starting at each vertex $i$, and we will denote it by:
\begin{equation}
\label{eqn:minimal}
[i;\upsilon_\bm(i))
\end{equation}
We consider $n$ points corresponding to the residues modulo $n$, and draw an oriented edge between the point $i$ and the point $\upsilon_\bm(i)$. The uniqueness of the minimal $\bm-$integral arc implies that the resulting graph will be a union of oriented cycles:
\begin{equation}
\label{eqn:cycle}
C_1 \sqcup ... \sqcup C_g
\end{equation}
where $C_h = \{i_1,i_2,...,i_{l_h}\}$ is such that:
$$
i_{e+1} = \upsilon_\bm(i_e) \quad \forall \ e \in \{1,...,l_h\} \qquad \quad \Leftrightarrow \qquad \quad [i_1;i_2), \ [i_2;i_3), \ ..., \ [i_{l_h};i_1)
$$
are all minimal $\bm-$integral arcs. This determines the numbers $g,l_1,...,l_h$ from Theorem \ref{thm:subalg}. Given any arc $[i;j)$ in the cycle $C_h$, the labels on its vertices will be $i_1$, $i_2$,..., $i_k$ mod $n$, such that $i_{e+1} = \upsilon_\bm(i_e)$, $i = i_1$ and $j = i_{k+1}$. To $[i;j)$ we associate the following $\bm-$integral arc in the cyclic quiver:
\begin{equation}
\label{eqn:arc}
[i;j)_{h} \ := \ [i_1;i_2) \cup ... \cup [i_k;i_{k+1})
\end{equation}

\begin{example}
\label{ex:basic}

\noindent When $\bm = \left(\frac ab,\frac ab,...,\frac ab\right)$ for coprime natural numbers $a,b$, then we have $\upsilon_\bm(i) = i+b$ for all $i$ modulo $n$. Therefore, we have: 
$$
g = \gcd(b,n) \qquad \text{and} \qquad l_h = \frac ng \qquad \forall \ h\in \{1,...,g\} 
$$
and the graph \eqref{eqn:cycle} is a disjoint union of $g$ cycles of length $\frac ng$.
\end{example}

\noindent 
The isomorphism \eqref{eqn:subalgebra} will be given by sending the root generators $e_{\pm [i;j)}^{(h)}$ of the quantum groups in the right hand side to a shuffle element of degree $[i;j)_{h}$ in the left hand side. Therefore, the restriction of this isomorphism to the positive halves of the algebras in question should also be an isomorphism:
\begin{equation}
\label{eqn:polar}
\CB^+_\bm \ \cong \ \bigotimes_{h=1}^g U^+_q(\dot{\fgl}_{l_h}) 
\end{equation}
Our main technical challenge is to make sure that the two sides of \eqref{eqn:polar} have the same dimension in every degree $\bk$. By \eqref{eqn:dimestimate}, the dimension of the graded pieces of quantum groups equals the number of partitions of the degree vector into arcs, i.e.
\begin{equation}
\label{eqn:bound0}
\dim \left(\bigotimes_{h=1}^g U^+_q(\dot{\fgl}_{l_h}) \right) \text{ in degree } \bk \ = \ \# \text{ of }\bm-\text{integral partitions of }\bk
\end{equation}
where a partition into arcs $[i;j)$ is called $\bm-$integral if all constituent arcs are $\bm-$integral. The main technical step in establishing the isomorphism \eqref{eqn:polar} is to obtain the estimate:
\begin{equation}
\label{eqn:bbound}
\dim \CB_{\bm|\bk}^+ \ \leq \ \#\text{ of }\bm-\text{integral partitions of }\bk
\end{equation}
In fact, we will prove a stronger estimate than \eqref{eqn:bbound}, by placing an upper bound on the dimension of the bigraded pieces of the subalgebra:
$$
\CS^+_{\leq \bm|\bk,d} \ \subset \ \CS^+_{\leq \bm}
$$
for any $\bm \in \qq$ and bidegree $(\bk,d) \in \nn \times \BZ$. In order to say what the upper bound is, let us define a $\BZ-$\textbf{indexed partition} $(\rho,\delta) \vdash (\bk,d)$ as an unordered decomposition:
$$
\Big(\bk,d \Big) \ = \ \Big( [i_1;j_1), d_1 \Big) + ... + \Big( [i_t;j_t), d_t \Big) \qquad \text{where} \qquad \delta = (d_1,...,d_t) \in \BZ^t
$$
is a vector of integers which sum up to $d$. We say that $(\rho,\delta)$ has slope $\leq \bm$ if: 
\begin{equation}
\label{eqn:slope}
d_s \leq \bm \cdot [i_s;j_s), \qquad \forall \ s\in \{1,...,t\}
\end{equation}



\begin{lemma} 
\label{lem:bound}

For any $\ebm,\ebk,d$, we have:
\begin{equation}
\label{eqn:bound}
\dim \CS^+_{\leq \ebm|\ebk, d} \ \leq \ \# \Big\{(\rho, \delta) \vdash (\ebk,d) \ \text{of slope} \leq \ebm \Big \}
\end{equation}

\end{lemma}


\begin{proof} We will use the argument of \citep{Ntor}, itself a generalization of the argument in \citep{FHHSY}. Consider an arbitrary shuffle element: 
$$
F(...,z_{i1},...,z_{ik_i},...)\in \CS^+_{\leq \bm|\bk,d}
$$ 
and any partition $\rho \vdash \bk$ into arcs:
$$
\rho \ : \ [i_1;j_1) + ... + [i_t;j_t) \ = \ \bk
$$	
We will assume that the parts of $\rho$ are ordered in terms of their lengths $l_1 \geq ... \geq l_t$, where $l_s = j_s-i_s$. To this data, we may assign the specialization:
\begin{equation}
\label{eqn:specialize}
F_\rho(y_1,...,y_t) \ = \ F(...,y_sq_1^{i_s},...,y_sq_1^{j_s-1},...) 
\end{equation}
where recall that $q_1 = qt$. Let us now consider the map:
$$
\ph_\rho \ : \ \CS^+_{\leq \bm|\bk,d} \ \longrightarrow \ \tBF(y_1,...,y_t)^\sym \qquad \qquad \ph_\rho(F) = F_\rho
$$
defined above, where $\tBF = \BQ(q,t)$. The superscript $\sym$ refers to the fact that $F_\rho$ must be symmetric in the variables $y_s$ and $y_{s'}$ if $[i_s;j_s) = [i_{s'};j_{s'})$ modulo $\BZ(n,n)$. We will consider the filtration:
\begin{equation}
\label{eqn:gordon}
S_\rho \ := \ \bigcap_{\rho'>\rho} \text{Ker } \ph_{\rho'}
\end{equation}
where partitions are ordered lexicographically in the lengths of their parts:
$$
\rho \geq \rho' \quad \Leftrightarrow \quad \left(\text{the lengths }l_1,...,l_t \text{ of }\rho \right) \geq \left( \text{the lengths }l'_{1},...,l'_{t'} \text{ of }\rho'\right)
$$
When $\rho$ is the finest partition (whose parts are just simple roots $\bs^i$), we have $S_\rho = 0$. When $\rho$ is a maximal partition, we have $S_\rho = \CS^+_{\leq \bm|\bk,d}$. Therefore, we conclude that:
$$
\dim \CS^+_{\leq \bm|\bk,d} \ \leq \ \sum_{\rho \vdash \bk} \dim \ph_\rho(S_\rho)
$$
so it is enough to show that:
\begin{equation}
\label{eqn:bali}
\dim \ph_\rho(S_\rho) \ \leq \ \# \Big\{(d_1,...,d_t)\in \BZ^t \ \text{ such that } \ d_s \leq \bm \cdot [i_s;j_s) \quad \forall \ s \Big\}^\sym
\end{equation}
The superscript $\sym$ refers to the fact that the collections $(d_1,...,d_t)$ are unordered in those entries $d_s, d_{s'}$ such that $[i_s;j_s) = [i_{s'};j_{s'})$. To prove \eqref{eqn:bali}, let us choose any: 
$$
F_\rho \ = \ \ph_\rho(F) \ \in \ \ph_\rho(S_\rho)
$$
By the very definition of shuffle elements in \eqref{eqn:defshuf}, we can write: 
\begin{equation}
\label{eqn:special}
F_\rho = \frac {r(...,y_sq_1^{i_s},...,y_sq_1^{j_s-1},...)}{\prod^{a\equiv a'}_{s \neq s'} \prod^{i_s \leq a < j_s}_{i_{s'} \leq a'< j_{s'}} \left[ \frac {y_{s'}q_1^{a'}}{y_sq_1^a q^2} \right]}
\end{equation}
for some symmetric Laurent polynomial $r$ which satisfies the wheel conditions \eqref{eqn:shufelem}. These conditions imply that $r$ is divisible the following linear factors:
$$
y_s q_1^{a} = y_{s'} q_1^{a'} q^{-2} \qquad \text{for} \quad a\equiv a' \quad \text{and} \quad i_s \leq a < j_s - 1, \quad i_{s'} \leq a' \leq  j_{s'} - 1
$$
$$
y_s q_1^{a}  = y_{s'} q_1^{a'}q^2 \qquad \text{for} \quad a\equiv a' \quad \text{and} \quad i_s < a \leq j_s - 1, \quad i_{s'} \leq a' \leq j_{s'} - 1
$$
for any $s<s'$. The above zeroes are counted with the correct multiplicities, so we may cancel them out with the factors in the denominator. We conclude that:
\begin{equation}
\label{eqn:funky}
F_\rho(y_1,...,y_t) = \frac {p(y_1,...,y_t)}{ \prod^{s<s', a \equiv j_s - 1}_{i_{s'} \leq a < j_{s'}} \left[\frac {y_{s'}q_1^{a}}{y_s q_1^{j_s-1} q^2} \right] \prod^{s<s', a \equiv i_s}_{i_{s'} \leq a < j_{s'}} \left[\frac {y_{s'} q_1^{a}}{y_s q_1^{i_s} q^{-2}} \right]} 
\end{equation}
for a Laurent polynomial $p$, which is symmetric in the variables $y_s$ and $y_{s'}$ whenever $[i_s;j_s) = [i_{s'};j_{s'})$. Moreover, under any one of the specializations:
\begin{equation}
\label{eqn:jorah}
y_s q_1^{j_s} = y_{s'} q_1^{a} \qquad \text{for some} \quad a \equiv j_s \qquad \qquad i_{s'} \leq a' < j_{s'}
\end{equation}
\begin{equation}
\label{eqn:jeor}
y_s q_1^{i_s-1} = y_{s'} q_1^{a} \qquad \text{for some} \quad a \equiv i_s-1 \qquad i_{s'} \leq a' < j_{s'}
\end{equation}
we have $F_\rho = F_\rho'$ for a partition $\rho'>\rho$ in lexicographic ordering. By the assumption that $F\in S_\rho$, we conclude that $F_\rho$, and hence also $p$, vanishes at each of the specializations \eqref{eqn:jorah} and \eqref{eqn:jeor}. Therefore, we can rewrite \eqref{eqn:funky} as:
\begin{equation}
\label{eqn:fresh}
F_\rho(y_1,...,y_t) = e(y_1,...,y_t) \cdot \frac {\prod^{s<s', a \equiv j_s}_{i_{s'} \leq a < j_{s'}} \left[\frac {y_{s'}q_1^{a}}{y_s q_1^{j_s}} \right] \prod^{s<s', a \equiv i_s-1}_{i_{s'}  \leq a < j_{s'} } \left[\frac {y_{s'} q_1^{a}}{y_s q_1^{i_s-1}} \right]}{\prod^{s<s', a \equiv j_s - 1}_{i_{s'} \leq a < j_{s'}} \left[\frac {y_{s'}q_1^{a}}{y_s q_1^{j_s-1} q^2} \right] \prod^{s<s', a \equiv i_s}_{i_{s'} \leq a < j_{s'}} \left[\frac {y_{s'} q_1^{a}}{y_s q_1^{i_s} q^{-2}} \right]} 
\end{equation}
for a Laurent polynomial $e$, which is symmetric in the variables $y_s$ and $y_{s'}$ if $[i_s;j_s) = [i_{s'};j_{s'})$ modulo $\BZ(n,n)$. Let us now recall \eqref{eqn:limit1}, which gives us estimates on the degree of $F_\rho$ in every variable:
\begin{equation}
\label{eqn:honhonhon}
\deg_{y_s} F_\rho \leq \bm\cdot [i_s;j_s) - \frac {\Big([i_s;j_s), \bk - [i_s;j_s) \Big)}2
\end{equation}
for all $s\in \{1,...,t\}$, and hence:
$$
\deg_{y_s} e \leq \bm\cdot [i_s;j_s) - \frac {\Big([i_s;j_s), \bk - [i_s;j_s) \Big)}2
- \sum^{s<s', a \equiv j_s}_{i_{s'} \leq a < j_{s'} } \frac 12 -
\sum^{s<s', a \equiv i_s-1}_{i_{s'} \leq a < j_{s'}} \frac 12 +
$$
$$
+ \sum^{s<s', a \equiv j_s - 1}_{i_{s'} \leq a < j_{s'}} \frac 12 + \sum^{s<s', a \equiv i_s}_{i_{s'} \leq a < j_{s'}} \frac 12 - \sum^{s>s', a \equiv j_{s'}}_{i_{s} \leq a < j_{s}} \frac 12 - \sum^{s>s', a \equiv i_{s'}-1}_{i_{s} \leq a < j_{s}} \frac 12 + \sum^{s>s', a \equiv j_{s'} - 1}_{i_{s} \leq a < j_{s}} \frac 12 + \sum^{s>s', a \equiv i_{s'}}_{i_{s} \leq a < j_{s}} \frac 12
$$
Adding together the above sums yields $\frac 12\sum_{s\neq s'}([i_s;j_s), [i_{s'},j_{s'}))$, so we conclude that $\deg_{y_s} e \leq \bm \cdot [i_s;j_s)$. Since the homogenous degree of $F$ equals $d$, we conclude that:
$$
e(y_1,...,y_t) \ = \ \sum_{d_1+...+d_t = d}^{d_s \leq \bm \cdot [i_s;j_s)} c_{d_1,...,d_t} \cdot y_1^{d_1}...y_t^{d_t} \qquad  \text{for some constants } \quad  c_{d_1,...,d_t} \in \tBF
$$
Moreover, $e$ must be symmetric in the variables $y_s$ and $y_{s'}$ whenever $[i_s;j_s) = [i_{s'};j_{s'})$. Since $e$ determines $F_\rho$ completely by \eqref{eqn:fresh}, this precisely proves the bound \eqref{eqn:bali}. \\
\end{proof}

\begin{remark}
\label{rem:b}

When $d = \bm \cdot \bk$, any $\BZ-$indexed partition of slope $\leq \bm$ is forced to satisfy $d_s = \bm \cdot [i_s;j_s) \in \BZ$ for all $s\in \{1,...,t\}$. Therefore, \eqref{eqn:bound} implies \eqref{eqn:bbound}. 

\tab
Running the argument of Lemma \ref{lem:bound} gives us a useful criterion for when a shuffle element is 0. Assume that bounds \eqref{eqn:honhonhon} are strict for any partition $\rho$ with more than one part, and that $F_\rho = 0$ for all partitions $\rho$ with exactly one part; then $F=0$. In terms of the coproduct, this can be translated into:
\begin{equation}
\label{eqn:principle}
\text{if }F\in \CB_{\bm|\bk}^+ \text{ is primitive,} \ \text{ i.e. } \ \Delta_\bm(F) = \ph_{\deg F} \otimes F + F \otimes 1
\end{equation}
$$
\text{and }F(q_1^i,...,q_1^{j-1}) = 0 \text{ for all } [i;j) = \bk \qquad \Rightarrow \qquad F = 0
$$
A similar principle holds when we replace $q_1$ by $q_2$.
\end{remark}

\section{Explicit shuffle elements}
\label{sec:explicit} 


\noindent We will now construct elements of $\CB^+_\bm$ and $\CB^-_\bm$, which will allow us to prove that the bounds in \eqref{eqn:bbound}, and later \eqref{eqn:bound}, are effective. As predicted by Lemma \ref{lem:bound}, there should be one generator of these algebras in every bidegree:
$$
(\pm[i;j), d) \quad \text{such that} \quad d = \pm \ \bm \cdot [i;j) \in \BZ
$$
To construct such shuffle elements, we will consider special cases of \eqref{eqn:defp0} and \eqref{eqn:defq0}:
\begin{equation}
\label{eqn:defpm1}
P^{\pm \bm}_{\pm [i;j)} \ = \ \textrm{Sym} \left[\frac {\prod_{a=i}^{j-1} z_a^{\lfloor m_i + ... + m_{a} \rfloor - \lfloor m_i + ... + m_{a-1} \rfloor}}{t^{\ind^\bm_{[i;j)}} q^{i-j}\prod_{a=i+1}^{j-1} \left(1 - \frac {q_2 z_{a}}{z_{a-1}} \right)} \prod_{i \leq a < b < j} \zeta \left( \frac {z_{b}}{z_{a}} \right) \right] \ \in \ \CS^\pm
\end{equation}
\begin{equation}
\label{eqn:defpm2}
Q^{\pm \bm}_{\mp [i;j)} \ = \ \textrm{Sym} \left[\frac {\prod_{a=i}^{j-1} z_a^{\lfloor m_i + ... + m_{a-1} \rfloor - \lfloor m_i + ... + m_{a} \rfloor}}{t^{-\ind^\bm_{[i;j)}} \prod_{a=i+1}^{j-1} \left(1 - \frac {q_1 z_{a-1}}{z_{a}} \right)} \prod_{i \leq a < b < j} \zeta \left( \frac {z_{a}}{z_{b}} \right) \right] \ \in \ \CS^\pm
\end{equation}
where we recall that $q_1 = qt$ and $q_2 = \frac qt$, and define the constant:
\begin{equation}
\label{eqn:cohen}
\ind^\bm_{[i;j)} \ = \ \sum_{a=i}^{j-1} \Big( m_i + ... + m_a - \lfloor m_i + ... + m_{a-1} \rfloor \Big) 
\end{equation}
The notation is reminiscent of \eqref{eqn:ind}, which was a generalization of \eqref{eqn:cohen} above that depended on the shape of a ribbon. Finally, set:
$$
P^{\bm}_{\pm [i;j)} = Q^{\bm}_{\pm [i;j)} = 0 \qquad \text{if} \qquad \bm \cdot [i;j) \notin \BZ
$$
As we showed in Exercise \ref{ex:gen}, the rational functions $P_{\pm [i;j)}^\bm$ and $Q_{\pm [i;j)}^\bm$ are shuffle elements, and in fact, they lie in the image of the map $\Upsilon$. However, at this stage we are more interested in the fact that these elements have slope $=\bm$: \\


\begin{exercise} 
\label{ex:mich}

We have $P^\bm_{\pm [i;j)}, Q^\bm_{\pm [i;j)} \in \CB^\pm_\bm$, and their coproduct is given by:
\begin{equation}
\label{eqn:copy}
\Delta_\bm(P^\bm_{[i;j)}) = \sum_{a=i}^j P^\bm_{[a;j)} \ph_{[i;a)}  \otimes P^\bm_{[i;a)} \qquad \Delta_\bm(Q^\bm_{[i;j)}) = \sum_{a=i}^j Q^\bm_{[i;a)} \ph_{[a;j)} \otimes Q^\bm_{[a;j)} \qquad
\end{equation}
$$
\Delta_\bm(P^\bm_{-[i;j)}) = \sum_{a=i}^j  P^\bm_{-[a;j)} \otimes P^\bm_{-[i;a)} \ph_{-[a;j)} \ \qquad \Delta_\bm(Q^\bm_{-[i;j)}) = \sum_{a=i}^j  Q^\bm_{-[i;a)} \otimes Q^\bm_{-[a;j)} \ph_{-[i;a)} 
$$



\end{exercise}

\tab
The above coproduct relations match \eqref{eqn:ray1} and \eqref{eqn:ray2} for the quantum affine algebra. In order to draw a complete parallel between the two pictures, we need to compute the various pairings of the shuffle elements $P^\bm_{\pm[i;j)}$ and $Q^\bm_{\pm[i;j)}$ in the subalgebras $\CB_\bm^\pm$. To this end, we introduce the \b{norm maps}:
\begin{equation}
\label{eqn:eta1}
\eta^{q_1}_{\pm [i;j)}:\CS_{\pm [i;j)} \longrightarrow \tBF, \qquad \qquad \eta^{q_1}_{\pm [i;j)}(F) = \frac {F(q_1^{i},...,q_1^{j-1})}{\prod_{i \leq a < b < j} \zeta(q_1^{b-a})}
\end{equation}
\begin{equation}
\label{eqn:eta2}
\eta^{q_2}_{\pm [i;j)}:\CS_{\pm [i;j)} \longrightarrow \tBF, \qquad \qquad \eta^{q_2}_{\pm [i;j)}(G) = \frac {G(q_2^{-i},...,q_2^{-j+1})}{\prod_{i \leq a < b < j} \zeta(q_2^{b-a})}
\end{equation}
The above specializations are very important, because $\zeta(q_1^{-1}) = \zeta(q_2^{-1}) = 0$, which will force most summands in the shuffle product will vanish. This is made precise by the following Exercise: \\


\begin{exercise}
\label{ex:pseudo}
The norm maps are \textbf{pseudo-multiplicative}, in the sense that any product $F*F' \in \CS_{[i;j)}$ or $G*G' \in \CS_{-[i;j)}$, we have:
$$
\eta^{q_1}_{[i;j)}(F*F') = 
\begin{cases}
\eta^{q_1}_{[a;j)}(F)\eta^{q_1}_{[i;a)}(F') & \text{ if } \ \ \exists \ a \text{ s.t. } \deg F = [a;j) \text{ and } \deg F' = [i;a) \\
0 & \text{ otherwise}
\end{cases}
$$
$$
\eta^{q_2}_{-[i;j)}(G*G') = 
\begin{cases}
\eta^{q_2}_{[a;j)}(G)\eta^{q_2}_{[i;a)}(G') & \text{ if } \ \ \exists \ a \text{ s.t. } \deg G = -[a;j) \text{ and } \deg G' = -[i;a) \\
0 & \text{ otherwise}
\end{cases}
$$
Similar results hold for $\eta^{q_1}_{-[i;j)}$ and $\eta^{q_2}_{[i;j)}$, although one has to switch the order of the shuffle product $F*F' \leftrightarrow F'*F$ and $G*G' \leftrightarrow G'*G$ in the above formulas. Moreover:
\begin{equation}
\label{eqn:normp1}
\eta^{q_1}_{[i;j)} \left( P_{[i';j')}^{\bm} \right) \ = \ \delta_{[i';j')}^{[i;j)} \cdot \frac {q_1^{\sum_{a=i}^{j-1} am_a} q^{\ind_{[i;j)}^\bm}}{q^{i-j}(1-q^{2})^{j-i-1}} 
\end{equation}
\begin{equation}
\label{eqn:normp2}
\eta^{q_2}_{-[i;j)} \left( Q_{-[i';j')}^{\bm} \right)  =  \delta_{[i';j')}^{[i;j)} \cdot \frac {q_2^{\sum_{a=i}^{j-1} am_a} q^{\ind^\bm_{[i;j)}}}{(1-q^2)^{j-i-1}}  
\end{equation}
whenever $[i;j)$ is $\bm-$integral arc. 



\end{exercise} 

\tab
The following Exercise proves that the norm maps \eqref{eqn:normp1} and \eqref{eqn:normp2} coincide with the linear form obtained by contracting the Hopf pairing \eqref{eqn:daddypair} with $Q_{-[i;j)}^\bm$ and $P_{[i;j)}^\bm$: \\

\begin{exercise}
\label{ex:comppair}
For any $f\in \CB^+_\bm$ and $g\in \CB^-_\bm$, we have:
\begin{equation}
\label{eqn:functional1}
\langle Q_{-[i;j)}^{\bm}, F \rangle \ = \ (1-q^2)^{j-i} \cdot q_1^{-\sum_{a=i}^{j-1} am_a} q^{i-j-\ind_{[i;j)}^\bm} \cdot \eta^{q_1}_{[i;j)}(F) 
\end{equation}
\begin{equation}
\label{eqn:functional2}
\langle G, P_{[i;j)}^\bm \rangle \quad = \quad (1-q^2)^{j-i} \cdot q_2^{-\sum_{a=i}^{j-1} am_a} q^{-\ind_{[i;j)}^\bm} \cdot \eta^{q_2}_{-[i;j)}(G) 
\end{equation}
for any $\bm-$integral arc $[i;j)$. The pairings $\langle F, Q_{[i;j)}^{-\bm} \rangle$ and $\langle P_{-[i;j)}^{-\bm}, G \rangle$ are given by the same formulas \eqref{eqn:functional1} and \eqref{eqn:functional2}, respectively.
\end{exercise}


\tab
As a consequence of the above Exercise and either \eqref{eqn:normp1} or \eqref{eqn:normp2}, we obtain:
\begin{equation}
\label{eqn:pairy}
\langle Q_{-[i';j')}^\bm, P_{[i;j)}^\bm \rangle \ = \ \delta_{[i';j')}^{[i;j)} \cdot (1-q^2) \ = \ \langle P_{-[i';j')}^\bm, Q_{[i;j)}^\bm \rangle
\end{equation}
for all $\bm-$integral arcs $[i;j)$ and $[i';j')$. Even more so, formula \eqref{eqn:functional1} and the fact that $\langle \cdot, \cdot \rangle$ is a bialgebra pairing imply that:
$$
\langle Q_{-[i_1;j_1)}^\bm...Q_{-[i_t;j_t)}^\bm, F \rangle = \text{const}\cdot \eta^{q_1}_{[i_1;j_1)} \circ ... \circ \eta^{q_1}_{[i_t;j_t)} \Big( \Delta^{t-1}(F) \Big)
$$
This formula allows us to prove that the pairing is non-degenerate on $\CB_\bm$:
$$
\langle \cdot , \cdot \rangle \ : \ \CB^-_\bm \otimes \CB^+_\bm \ \longrightarrow \ \BC(q) \qquad \qquad \text{is non-degenerate}
$$
Indeed, if there existed an element $F \in \CB^+_{\bm|\bk}$ such that $\langle G, F \rangle = 0$ for all $G \in \CB_{\bm|-\bk}^-$, the above formula would imply that:
$$
\eta^{q_1}_{[i_1;j_1)} \circ ... \circ \eta^{q_1}_{[i_t;j_t)} \Big( \Delta^{t-1}(F) \Big) = 0
$$
for all partitions $\rho = [i_1;j_1) + ... + [i_t;j_t) \vdash \bk$. This would allow us to prove that $F$ is 0, inductively in $\bk$, based on the principle \eqref{eqn:principle}. Indeed, the induction hypothesis implies that any intermediate factor in the coproduct of $F$ must be 0, hence $F$ must be primitive. Then the fact that any linear map $\eta^{q_1}_{[i;j)}$ annihilates $F$ implies that $F=0$. \\


\begin{proof} \emph{of Theorem \ref{thm:subalg}:} The Theorem is a consequence of a general principle which was proved as part of Lemma 4.4. in \citep{Ntor}. The principle says in order to define a bialgebra homomorphism:
\begin{equation}
\label{eqn:billie}
\bigotimes_{h=1}^g U_q(\dot{\fgl}_{l_h}) \ \stackrel{\cong}\longrightarrow \ A
\end{equation}
to a Drinfeld double $A = A^\leq \otimes A^\geq$ defined with respect to a \textbf{non-degenerate} bialgebra pairing, it is enough to produce elements:
$$
\widetilde{e}_{\pm[i;j)}, \ \widetilde{\ph}_k \ \in \ A 
$$
for any arc $[i;j)$ and any vertex $k$ in the cycles $C_1,...,C_g$ of \eqref{eqn:cycle}, which satisfy the same coproduct and pairing relations as the generators of the quantum affine group \eqref{eqn:ray1} and \eqref{eqn:charles}. When $A = \CB_\bm$, we set:
\begin{equation}
\label{eqn:choices}
\widetilde{e}_{[i;j)} \ = \ P^\bm_{[i;j)_h} \qquad \qquad \widetilde{e}_{-[i;j)} \ = \ Q^\bm_{-[i;j)_h} \qquad \qquad \widetilde{\ph}_k = \ph_{[k;\upsilon_\bm(k))}
\end{equation}
for any arc $[i;j)$ and any vertex $k$ in each cycle $C_h$, where the arc $[i;j)_h$ in the cyclic quiver is defined in \eqref{eqn:arc} and the minimal arc $[k;\upsilon_\bm(k))$ in the cyclic quiver is defined in \eqref{eqn:minimal}. The fact that the above choices \eqref{eqn:choices} satisfy the required coproduct and pairing relations follows from \eqref{eqn:copy} and \eqref{eqn:pairy}, which completes the proof of the Theorem. Although we will not need it, we claim that the isomorphism \eqref{eqn:billie} sends:
$$
f_{[i;j)} \ \longrightarrow \ Q^\bm_{[i;j)_h} \qquad \text{and} \qquad f_{-[i;j)} \ \longrightarrow \ P^\bm_{-[i;j)_h}
$$
The main technical requirement in proving this statement is computing $\eta^{q_1}_{-[i;j)} ( Q_{[i';j')}^{\bm} )$ or $\eta^{q_2}_{[i;j)} ( P_{-[i';j')}^{\bm} )$ by analogy with \eqref{eqn:normp1} and \eqref{eqn:normp2}. This computation is messy and akin to the proof of Proposition 4.18 of \citep{Ntor}, and so we will not pursue it. \\
\end{proof}

\section{Algebra factorizations and $R-$matrices}
\label{sec:factorizing}

\noindent Theorem \ref{thm:subalg} allows us to identify the subalgebras $\CB_{\bm}$ for various $\bm \in \qq$. According to the following Proposition, these are the building blocks of the whole shuffle algebra: \\

\begin{proposition}
\label{prop:factorizing}

For any fixed slope vector $\ebm = (m_1,...,m_n) \in \qq$, we have:
\begin{equation}
\label{eqn:factor1}
\CS^+ \ = \ \bigotimes_{r \in \BQ}^\rightarrow \CB^+_{\ebm + r\bth}
\end{equation}
where $\bth = (1,...,1)$. \\

\end{proposition}

\begin{proof} Relation \eqref{eqn:factor1} requires us to show that any shuffle element $F \in \CS^+$ can be uniquely written as:
\begin{equation}
\label{eqn:Rtrivial}
F = \mathop{\sum^{\bk_1, ..., \bk_t}_{(\bm+r_i\bth)\cdot \bk^i \in \BZ}}^{r_1<...<r_t \in \BQ} F_{1} * ... * F_{t}, \qquad \text{where} \quad F_{i} \in \CB_{\bm+r_i\bth|\bk_i}^+
\end{equation}
In other words, for any bidegree $(\bk,d) \in \nn \times \BZ$, the product map:
$$
\mathop{\bigoplus^{d = d_1+...+d_t}_{d_i := (\bm+r_i\bth)\cdot \bk^i \in \BZ}}_{r_1<...<r_t \in \BQ}^{\bk = \bk_1 + ... + \bk_t} \CB^+_{\bm+r_1\bth} \otimes ... \otimes \CB^+_{\bm+r_t\bth} \ \longrightarrow \ \CS^+_{\bk,d}
$$ 
is an isomorphism. We would like to do this by a dimension estimate, but a quick glance shows us that both the left and right hand sides are infinite dimensional vector spaces. However, as we observed in Section \ref{sec:slopefiltration}, the graded piece $\CS^+_{\bk,d}$ is filtered by the finite-dimensional vector spaces $\CS^+_{\leq \bm + r \bth |\bk,d}$, as $r$ goes over all of $\BQ$. Then we will prove that the product map induces an isomorphism on all pieces of this filtration:
\begin{equation}
\label{eqn:psi}
\mathop{\bigoplus^{d = d_1+...+d_t}_{d_i := (\bm+r_i\bth)\cdot \bk^i \in \BZ}}_{r_1<...<r_t \leq r \in \BQ}^{\bk = \bk_1 + ... + \bk_t} \CB^+_{\bm+r_1\bth} \otimes ... \otimes \CB^+_{\bm+r_t\bth} \ \stackrel{\Psi}\longrightarrow \ \CS^+_{\leq \bm+r\bth|\bk,d}
\end{equation}
for any $\bk \in \nn$, $d\in \BZ$ and $r\in \BQ$. Note that there are now finitely many summands in the left hand side: if $r_1$ were arbitrarily small, this would force $d_1$ to be arbitrarily small, which would force one of the other $d_i$'s to be arbitrarily large, which would force some $r_i$ to be arbitrarily large, which would contradict the fact that $r_1<...<r_t\leq r$ in the above sum. Theorem \ref{thm:subalg} implies that the dimension of the domain of $\Psi$ equals the number of $\BZ-$indexed partitions (see Section \ref{sec:arcs} for the terminology):
$$
(\rho,\delta) \ \vdash \ (\bk,d) \qquad \text{where} \qquad \rho = [i_1;j_1) + ... + [i_t;j_t), \quad d = d_1 + ... + d_t
$$
where each constituent arc $[i_s;j_s)$ of $\rho$ and its index $d_s$ satisfy:
\begin{equation}
\label{eqn:rey}
(\bm+r_s\bth)\cdot [i_s;j_s) = d_s
\end{equation}
for some rational numbers $r_1 < ... < r_t \leq r$. Note that each $r_s$ is uniquely determined by the condition \eqref{eqn:rey}, since $\bth\cdot[i;j) = j-i$ is non-zero for all arcs $[i;j)$. The number of partitions computed above is precisely equal to the upper bound for $\CS^+_{\leq \bm+r\bth|\bk,d}$ we found in \eqref{eqn:bound}, so we have just showed that the domain of the map $\Psi$ in \eqref{eqn:psi} has dimension $\geq$ the dimension of the codomain. In order to prove that $\Psi$ is an isomorphism, it is therefore enough to prove that $\Psi$ is injective. This is equivalent to showing that no shuffle element of the form \eqref{eqn:Rtrivial} can be 0 unless the sum is trivial. To do this, assume that:
\begin{equation}
\label{eqn:rtrivial}
\left( \sum_{\text{various summands}} F_1 * ... * F_l \right) * G \ = \ \sum_{\text{various summands}} F_1' * ... * F_{l'}'
\end{equation}
where $G \in \CB_{\bm+r\bth|\bk}^+$ is maximal in the following sense. It lies in $\CB_{\bm+r\bth}^+$, all other factors in \eqref{eqn:rtrivial} lie in $\CB_{\bm+r'\bth}^+$ for some $r'\leq r$, and $G$ mas maximal possible degree $\bk$ among all factors from \eqref{eqn:rtrivial} which lie in $\CB_{\bm+r\bth}^+$. Finally, we assume that $G$ is linearly independent from any other factors in the right hand side which have the same maximality properties as $G$, which we can do after further simplifying the sum in the right hand side of \eqref{eqn:rtrivial}. Then we take the coproduct $\Delta_{\bm+r\bth}$ of both sides of \eqref{eqn:rtrivial}: in the left hand side, there will be a single summand in degree:
\begin{equation}
\label{eqn:bidegrees}
(\text{anything}) \otimes \CS^+_{\bk, (\bm+r\bth)\cdot \bk}, \qquad \text{namely} \quad 
\left( \sum_{\text{various summands}} F_1 * ... * F_l \right) \otimes G
\end{equation}
Meanwhile, the coproduct of the right hand side of \eqref{eqn:rtrivial} has no summands in bidegree \eqref{eqn:bidegrees}, and so we have proved that non-trivial relations such as \eqref{eqn:rtrivial} cannot exist. Therefore, the map $\Psi$ is injective, and our estimates on the dimension of its domain and codomain imply that it is an isomorphism. \\
\end{proof}

\begin{proof} \emph{of Theorem \ref{thm:iso}:} Proposition \ref{prop:factorizing} implies that $\CS^+$, the positive half of the shuffle algebra, is generated by the root generators $P_{[i;j)}^\bm$ of \eqref{eqn:defpm1}. Since such elements lie in $\text{Im } \Upsilon^+ : \UUp \rightarrow \CS^+$ by Exercise \ref{ex:gen}, we conclude that $\Upsilon^+$ is surjective.
\end{proof}


\tab 
Analogously, we have a similar factorization result for the negative shuffle algebra:
\begin{equation}
\label{eqn:factor2}
\CS^- \ = \ \bigotimes_{r \in \BQ}^\rightarrow \CB^-_{\bm + r\bth}
\end{equation}
We can combine \eqref{eqn:factor1} and \eqref{eqn:factor2} to obtain the following factorization of the double shuffle algebra:
\begin{equation}
\label{eqn:factor}
\CS \ = \  \bigotimes_{r \in \BQ}^\rightarrow \CB^+_{\bm + r\bth} \otimes \CB_{\infty \cdot \bth} \otimes \bigotimes_{r \in \BQ}^\rightarrow \CB^-_{\bm + r\bth}
\end{equation}
where $\CB_{\infty \cdot \bth} = \langle \ph_{i,d}^\pm \rangle^{d\in \BN_0}_{1\leq i \leq n}$. Recall that the above factorization should be interpreted by saying that any element of $\CS$ can be written uniquely as a sum of products:
$$
F_1 * ...* F_s * \ph * G_1 * ... * G_t
$$
where $F_i \in \CB^+_{\bm+r_i\bth}$, $\ph \in \CB_{\infty \cdot \bth}$ and $G_i \in \CB^-_{\bm+r_i'\bth}$ for certain rational numbers $r_1<...<r_s$ and $r_1'<...<r_t'$. Finally, let us observe that the argument that established non-trivial relations \eqref{eqn:rtrivial} cannot exist also implies that the isomorphisms \eqref{eqn:factor1} and \eqref{eqn:factor2} respect the pairing. This means that in order to compute the pairing between:
$$
G = \mathop{\sum^{\bk_1, ..., \bk_t}_{(\bm+r_i\bth)\cdot \bk^i \in \BZ}}^{r_1<...<r_t \in \BQ} G_{1} * ... * G_{t} \qquad \text{and} \qquad F = \mathop{\sum^{\bk_1, ..., \bk_t}_{(\bm+r_i\bth)\cdot \bk^i \in \BZ}}^{r_1<...<r_t \in \BQ} F_{1} * ... * F_{t}
$$
where $F_{i} \in \CB_{\bm+r_i\bth|\bk_i}^+$ and $G_{i} \in \CB_{\bm+r_i\bth|\bk_i}^-$, we need only multiply the pairings of their constituent factors:
\begin{equation}
\label{eqn:pair}
\langle G, F \rangle \ = \ \mathop{\sum^{\bk_1, ..., \bk_t}_{(\bm+r_i\bth)\cdot \bk^i \in \BZ}}^{r_1<...<r_t \in \BQ} \langle G_1,F_1\rangle... \langle G_t, F_t \rangle
\end{equation}
This observation, together with the factorization \eqref{eqn:factor}, will allow us to construct the factorization \eqref{eqn:rfac1} of the $R-$matrix of $\CS$. Given a bialgebra $A$, recall that the universal $R-$matrix is an element $\CR \in A \woo A$ such that:
\begin{equation}
\label{eqn:prop1}
\CR \cdot \Delta(a) = \Delta^{\op}(a) \cdot \CR, \qquad \forall a\in A
\end{equation}
\begin{equation}
\label{eqn:prop2}
(\Delta \otimes 1) \CR = \CR_{13} \CR_{23}, \qquad (1 \otimes \Delta) \CR = \CR_{13} \CR_{12}
\end{equation}
We will write $\CR = \CR_A$ in order to emphasize to which algebra the $R-$matrix belongs. Property \eqref{eqn:prop1} implies that for any $V,W\in \text{Rep}(A)$, the operator $R_{VW}$ given by: 
$$
A \ \woo \ A \longrightarrow \text{End}(V \otimes W), \qquad \CR \longrightarrow R_{VW}
$$ 
intertwines the tensor product representations $V \otimes W$ and $W \otimes V$, which explains the terminology ``universal" and ``matrix". When $A = A^- \otimes A^+$ is presented as a Drinfeld double, then a universal $R-$matrix always exists: 
$$$$

\begin{exercise} 
\label{ex:universal}
	
Let $\{F_i\}$ and $\{G_i\}$ be dual bases of $A^+$ and $A^-$ with respect to the bialgebra pairing. Then the canonical tensor is a universal $R-$matrix:
$$
\CR \ = \ \sum_i F_i \otimes G_i \ \in \ A^+ \ \woo \ A^- \ \subset \ A \ \woo \ A
$$

\end{exercise}

\tab
We will use the above Proposition to factor the universal $R-$matrix of the algebra $\CS$. As was proved \eqref{eqn:pair}, the factorizations \eqref{eqn:factor1} and \eqref{eqn:factor2} respect the bialgebra pairing. This implies that dual bases $\{F_i\}$ and $\{G_i\}$ of $\CS^+$ and $\CS^-$ can be defined as:
$$
F_i = \prod_{r\in \BQ}^\rightarrow F_i^{(r)} \qquad \qquad G_i = \prod_{r\in \BQ}^\rightarrow G_i^{(r)} 
$$
where $\{F_i^{(r)}\}$ and $\{G_i^{(r)}\}$ are dual bases of $\CB_{\bm+r\bth}^+$ and $\CB_{\bm+r\bth}^-$, respectively. Together with Exercise \ref{ex:universal}, this implies the following factorization formula for the universal $R-$matrix of the double shuffle algebra:
\begin{equation}
\label{eqn:rfac3}
\CR_{\CS} \ = \ \left( \prod^{\rightarrow}_{r \in \BQ} \CR_{\CB_{\bm+r\bth}} \right) \cdot \CR_{\CB_{\infty \cdot \bth}} \ \in \ \CS \ \widehat{\otimes} \ \CS
\end{equation}
The above result mirrors such factorization formulas for universal $R-$matrices of quantum groups, as featured for example in \citep{KT}. In all cases, the philosophy is to break up the universal $R-$matrix of the quantum toroidal algebra $\CS$ into universal $R-$matrices of the simpler subalgebras:
$$
\CB_{\bm + r \bth} \ \cong \ \bigotimes_{h=1}^{g^{(r)}} U_q(\dot{\fgl}_{l^{(r)}_h})
$$
according to the decomposition of Theorem \ref{thm:subalg}, where the numbers $g^{(r)}$ and $l_1^{(r)},...,l_{g^{(r)}}^{(r)}$ are associated to the slope $\bm+r\bth$ in Section \ref{sec:arcs}. Since $\CS \cong \UU$ by Theorem \ref{thm:iso0}, we conclude that:
\begin{equation}
\label{eqn:rfac4}
\CR_{\UU} \ = \ \left( \prod^{\rightarrow}_{r \in \BQ} \prod_{h=1}^{g^{(r)}} \CR_{U_q\left(\dot{\fgl}_{l_h}^{(r)}\right)} \right) \cdot \CR_{\text{Heisenberg}}^{\otimes n}
\end{equation}

\chapter{Pieri Rules for Stable Bases}
\label{chap:pieri}
\section{Stable bases for the cyclic quiver}
\label{sec:stabcyc}


\noindent In the previous Chapters, we studied the elements $P_{[i;j)}^\bm, Q_{-[i;j)}^\bm \in \CS$ for all $\bm \in \qq$ and all $\bm-$integral arcs $[i;j)$. Now we will consider the operators they induce on the $K-$theory of Nakajima cyclic quiver varieties by Theorem \ref{thm:act}:
$$
P_{[i;j)}^{\bm}, \ Q_{-[i;j)}^{\bm} \ \curvearrowright \ K(\bw)
$$ 
Recall that a basis of $K(\bw)$ is given by the classes of the torus fixed points $|\bla \rangle$, as $\bla$ ranges over all $\bw-$partitions. In terms of this basis, Proposition \ref{prop:act} gives us:
\begin{equation}
\label{eqn:fixedp}
\langle \bla | \oP_{[i;j)}^\bm |\bmu \rangle \ \ \ = \ \ \ \oP_{[i;j)}^\bm(\blamu) \prod_{\bsq \in \blamu} \left[\prod_{\square \in \bmu} \zeta \left( \frac {\chi_\bsq}{\chi_\square}\right) \prod_{k=1}^\bw \Big[\frac {u_k}{q\chi_\bsq} \Big]  \right]
\end{equation}
\begin{equation}
\label{eqn:fixedq}
\langle \bmu | \oQ_{-[i;j)}^\bm |\bla \rangle \ = \ \oQ_{-[i;j)}^\bm(\blamu) \prod_{\bsq \in \blamu} \left[\prod_{\square \in \bla} \zeta \left( \frac {\chi_\sq}{\chi_\bsq}\right) \prod_{k=1}^\bw \Big[\frac {\chi_\bsq}{qu_k} \Big]  \right]^{-1} 
\end{equation}
where the sums go over all skew diagrams $\blamu$ of size $[i;j)$. Up to the product of factors $\zeta$ and $[\cdot]$, the above formulas rely on formulas \eqref{eqn:defpm1} and \eqref{eqn:defpm2} that present the shuffle elements $\oP_{[i;j)}^\bm$ and $\oQ_{-[i;j)}^\bm$ as a symmetrization. When evaluating these shuffle elements at the set of weights of a skew diagram $\blamu$, the only summands which survive correspond to almost standard tableaux of shape $\blamu$, as described in Section \ref{sec:eccentric}. However, there are many such tableaux for a given skew diagram, so it seems that formulas \eqref{eqn:fixedp} and \eqref{eqn:fixedq} turn out to be quite complicated computationally.

\tab
The purpose of this Chapter is to find a basis in which the operators $\oP_{[i;j)}^\bm$ and $\oQ_{-[i;j)}^\bm$ are ``nice", and specifically to prove Theorem \ref{thm:pieri}. In Section \ref{sec:eccentric}, we conjectured that these operators are given by the so-called eccentric Lagrangian correspondences (the conjecture hinges on the yet unproved lci-ness of the eccentric correspondences). Therefore, it is natural to look for Lagrangian classes in which the operators $\oP_{[i;j)}^\bm$ and $\oQ_{-[i;j)}^\bm$ act, and these will be given by the stable basis of \eqref{eqn:stable}. Specifically, let:
$$
A \ \subset \ T_\bw \qquad \text{where} \qquad A \ = \ \BC^*_t \times \prod_{k=1}^\bw \BC^*_{u_k}
$$
denote the codimension one Hamiltonian torus that acts on $\CN_{\bv,\bw}$, and let us consider a generic one-dimensional subtorus:
\begin{equation}
\label{eqn:hamiltonian}
\BC^* \ \stackrel{\sigma}\longrightarrow \ A \qquad \qquad \qquad  a \rightarrow (a,a^{N_1},...,a^{N_{\bw}}) 
\end{equation}
where $N_1 \ll ... \ll N_{\bw}$ are integers. Specifically, the inequalities are chosen so that at any fixed point of $\CN_{\bv,\bw}$, the decomposition into attracting and repelling directions of the tangent space matches \eqref{eqn:floc1} and \eqref{eqn:floc2}, and those $\sigma$ with this property determine a chamber in the Lie algebra $\fa$. Fixed points for $\sigma$ are the same as fixed points for the whole of $T_\bw$, and thus:
$$
\CN_{\bv,\bw}^\sigma \ = \ \{I_\bla \quad \forall \ \bla \text{ a }\bw-\text{partition of size }\bv\}
$$
Because $N_1 \ll ... \ll N_{\bw}$, the flow ordering \eqref{eqn:flow} on fixed points induced by $\sigma$ coincides with the dominance ordering on $\bw-$partitions, which is also the reason why we have used the same symbol $\unlhd$ to denote them both. Then the Maulik-Okounkov stable basis is defined as in \eqref{eqn:stable}:
$$
s^{\bm}_\bla \ = \ \stab^\sigma_\bm(|\bla\rangle)  \ \in \ K_{\bv,\bw} 
$$
We suppress the superscript $+$, because we will only work with the positive stable basis in this Chapter. Letting $T^-$ denote the repelling part of the tangent bundle with respect to the flow $\sigma$, then $\{s^{\bm}_\bla\}$ is the unique integral basis which is triangular in the dominance ordering on $\bw-$partitions:
\begin{equation}
\label{eqn:tri1}
s^{\bm}_\bla = |\bla\rangle \cdot \left[T^-_\bla \CN_{\bv,\bw} \right] + \sum_{\bmu \lhd \bla} |\bmu\rangle \cdot s^{\bm}_{\bla|\bmu} 
\end{equation}
where the coefficients $s_{\bla|\bmu}^{\bm}$ satisfy the smallness property on the Newton polytopes:
\begin{equation}
\label{eqn:small1}
P_A(s^{\bm}_{\bla|\bmu}) \ \subset \ P^\circ_A (s^{\bm}_{\bmu|\bmu}) + \bm \cdot (\bc_\bmu-\bc_\bla) \qquad \forall \ \bmu \lhd \bla
\end{equation}
where for any partition $\bla$, we write $\bc_\bla = (c_\bla^1,...,c_\bla^n)$ with:
$$
c_\bla^i \ = \ \sum_{\square \in \bla}^{c_\square \equiv i} c_\square
$$
Recall that the content of a box was defined in \eqref{eqn:content}. Note that $o_\bla^{\bs^i} = \exp(c_\bla^i) \in A$ is nothing but the torus weight of the one-dimensional space $\CO_i(1)|_\bla$, and this matches \eqref{eqn:ind3}. As we explained in Section \ref{sec:torusaction}, proving inclusions of Newton polytopes is equivalent to proving them upon projection to any one dimensional subtorus $\sigma$ as in \eqref{eqn:hamiltonian}. The advantage of this is that we specialize $u_i \mapsto t^{N_i}$, and so the restrictions of the stable basis become Laurent polynomials in two variables: 
$$
s^{\bm}_{\bla|\bmu} |_{u_i \mapsto t^{N_i}} \ \in \ \BZ[q^{\pm 1}, t^{\pm 1}]
$$
The smallest and largest $t-$degree of the above Laurent polynomials will be denoted:
$$
\mindeg s^{\bm}_{\bla|\bmu} \qquad \text{and} \qquad \maxdeg s^{\bm}_{\bla|\bmu} \quad \in \ \BZ
$$
When $\bla = \bmu$, the leading term of \eqref{eqn:tri1} does not depend on $\bm$, so we denote it by: 
\begin{equation}
\label{eqn:lead}
\fK_\bla \ := \ s_{\bla|\bla}^{\bm}
\end{equation}
According to \eqref{eqn:tanmod0}, we have:
\begin{equation}
\label{eqn:lead2}
\fK_\bla \ = \ \prod_{\sq \in \bla} \left[\prod_{\bsq \in \bla} \zeta\left(\frac {\chi_\bsq}{\chi_\sq}\right) \prod_{k=1}^{\bw} \Big[\frac {u_k}{q\chi_\sq} \Big] \Big[\frac {\chi_\sq}{qu_k} \Big]  \right]^{(-)}
\end{equation} 
To explain the notation $(-)$, recall that $\zeta$ is a product of quantum numbers $[x] = x^{\frac 12} - x^{-\frac 12}$. Therefore, the right hand side of \eqref{eqn:lead2} is a product of such $[x]$, and the superscripts $(-)$ or $(0)$ or $(+)$ will refer to the subproduct consisting of those factors $[x]$ such that $\deg x < 0$ or $\deg x = 0$ or $\deg x > 0$. In the specialization $u_i \mapsto t^{N_i}$, the minimal $t-$degree of $\fK_\bla$ is the opposite of its maximal $t-$degree, so let us write:
\begin{equation}
\label{eqn:kap}
\maxdeg \fK_\bla \ =: \ k_\bla \ := \ - \mindeg \fK_\bla
\end{equation}
Moreover, let us define the term of lowest $t-$degree in $\fK_\bla$, namely those terms where $\mindeg$ is attained:
\begin{equation}
\label{eqn:kappa}
\kappa_\bla \ := \ \ld \fK_\bla 
\end{equation}
By the principle explained in Section \ref{sec:torusaction}, it is sufficient to check condition \eqref{eqn:small1} on the inclusion of two Newton polytopes after projecting $u_i \mapsto t^{N_i}$ onto the line determined by \eqref{eqn:hamiltonian}, for any choice of $N_1\ll... \ll N_\bw$. Then the inclusion of two intervals becomes equivalent to the two inequalities: 

\begin{equation}
\label{eqn:smallish1}
\maxdeg s^{\bm}_{\bla|\bmu} \ \leq \ k_\bmu + \bm\cdot (\bc_\bmu - \bc_\bla)
\end{equation}
\begin{equation}
\label{eqn:smallish2}
\mindeg s^{\bm}_{\bla|\bmu} \ > \ - k_\bmu + \bm\cdot (\bc_\bmu - \bc_\bla)
\end{equation}
for all $\bmu \lhd \bla$. Throughout the remainder of this paper, we will always work under the specializations $u_i \mapsto t^{N_i}$, or equivalently $a_i \mapsto N_i$, and so $\bc_\bla$ belongs to $\zz$. \\

\section{Operators in the stable basis}
\label{sec:operators}

\noindent Suppose we are given an operator $F : K_{\bv',\bw} \rightarrow K_{\bv,\bw}$ whose matrix coefficients in terms of fixed points are known:
$$
F |\bmu \rangle = \sum_{|\bla| = \bv} F^\bla_\bmu \cdot |\bla \rangle \qquad \text{in other words} \qquad F^\bla_\bmu = \langle \bla |F| \bmu \rangle
$$ 
Fix $\bm \in \qq$ and suppose that the operator $F$ is Lagrangian, meaning that it takes any stable basis vector to an integral combination of stable basis vectors:
\begin{equation}
\label{eqn:crucial}
F \cdot s_\bmu^{\bm} = \sum_{|\bla| = \bv} \gamma^{\bla}_\bmu \cdot s_\bla^{\bm} \qquad \qquad \forall \ |\bmu| = \bv'
\end{equation}
for certain Laurent polynomials $\gamma^\bla_\bmu(q,t,u_1,...,u_{\bw})$. The following Lemma tells us what these coefficients have to be, provided that we know certain bounds on the $\mindeg$ and $\maxdeg$ of the matrix coefficients $F^\bla_\bmu$. Note that the degree will always be measured with respect to the specialization provided by the torus \eqref{eqn:hamiltonian}. \\

\begin{lemma}
\label{lem:compute}
If we assume that:
\begin{equation}
\label{eqn:giovanni1}
\emaxdeg F^\bla_\bmu \ \leq \ k_\bla - k_\bmu + \ebm \cdot (\bc_\bla - \bc_\bmu)
\end{equation}
\begin{equation}
\label{eqn:giovanni2}
\emindeg F^\bla_\bmu \ \geq \ k_\bmu - k_\bla +  \ebm \cdot (\bc_\bla - \bc_\bmu)
\end{equation}
then:
\begin{equation}
\label{eqn:moroder}
\gamma^{\bla}_\bmu \ = \ \left( \eld F_\bmu^\bla \right) \cdot \frac {\kappa_\bmu}{\kappa_{\bla}} 
\end{equation}
where the term of lowest degree of $F_\bmu^\bla$ is defined as consisting of those monomials of degree equal to the right hand side of \eqref{eqn:giovanni2}. If the right hand side of \eqref{eqn:giovanni2} is not integral, then the lowest degree term of $F_\bmu^\bla$ is taken to be 0, and hence $\gamma^\bla_\bmu = 0$. \\ 

\end{lemma}

\begin{proof} We will prove the Lemma for any fixed $\bmu$, by descending induction on $\bla$. The induction step will also explain how to take care of the base case, so assume that \eqref{eqn:moroder} is known for any $\bla' \rhd \bla$ and let us prove it for $\bla$. Consider the expression:
$$
F \cdot s_\bmu^{\bm} - \sum^{\bla' \rhd \bla}_{|\bla'| = \bv} \gamma^{\bla'}_\bmu \cdot s_{\bla'}^{\bm} \ = \ \theta_\bmu \ = \ \sum^{\bla' \not \rhd \bla}_{|\bla'| = \bv} \gamma^{\bla'}_\bmu \cdot s_{\bla'}^{\bm} \quad \in \ K_{\bv,\bw}
$$
Because of the the above equality, we can evaluate the restriction of $\theta_\bmu$ to the fixed point $\bla$ in two different ways:
\begin{equation}
\label{eqn:donna}
\langle \bla | F \cdot s_\bmu^{\bm} - \sum^{\bla' \rhd \bla}_{|\bla'| = \bv} \gamma^{\bla'}_\bmu \cdot s_{\bla'|\bla}^{\bm} \ = \ \langle \bla|\theta_\bmu \ = \ \gamma^{\bla}_\bmu \cdot \fK_\bla
\end{equation}
In itself, the above formula is not particularly useful in computing $\gamma_\bmu^{\bla}$, because it would involve knowing the restrictions $s_{\bla'|\bla}^{\bm}$ for any $\bla' \rhd \bla$. But in the situation at hand, we may use the given assumptions \eqref{eqn:giovanni1} and \eqref{eqn:giovanni2}, as well as the first equality of \eqref{eqn:donna} to obtain the following estimates:
$$
\maxdeg \langle \bla|\theta_\bmu \leq \ma_{\bla' \rhd \bla} \left(\maxdeg \langle \bla | F \cdot s_\bmu^{\bm}, \ \maxdeg \gamma_\bmu^{\bla'} + \maxdeg s_{\bla'|\bla}^{\bm} \right)
$$
and the analogous relation for min deg. If we write $s_\bmu^{\bm} = \sum_{\bmu' \unlhd \bmu} s_{\bmu|\bmu'}^{\bm}\cdot |\bmu'\rangle$, the above implies that $\maxdeg \langle \bla|\theta_\bmu \leq$
$$
\ma^{\bmu' \unlhd \bmu}_{\bla' \rhd \bla} \left(\maxdeg F^\bla_{\bmu'} + \maxdeg s_{\bmu|\bmu'}^{\bm}, \ \maxdeg \gamma_\bmu^{\bla'} + \maxdeg s_{\bla'|\bla}^{\bm} \right)
$$
and the analogous relation for min deg. Using \eqref{eqn:small1}, \eqref{eqn:giovanni1} and \eqref{eqn:moroder}, we obtain:
$$
\maxdeg \langle \bla|\theta_\bmu \ \leq \ k_\bla  + \bm \cdot (\bc_\bla - \bc_\bmu)
$$
and the analogous result for min deg:
$$
\mindeg \langle \bla|\theta_\bmu \ \geq \ - k_\bla  + \bm \cdot (\bc_\bla - \bc_\bmu)
$$
Using the second equality of \eqref{eqn:donna}, we obtain:
\begin{equation}
\label{eqn:inec}
\maxdeg \gamma^{\bla}_\bmu \ \leq \ \bm \cdot (\bc_\bla - \bc_\bmu) \ \leq \ \mindeg \gamma^{\bla}_\bmu
\end{equation}
The only way \eqref{eqn:inec} can happen is if all constituent monomials of the Laurent polynomial $\gamma^{\bla}_\bmu$ are concentrated in $t-$degree $m = \bm \cdot (\bc_\bla - \bc_\bmu)$. If $m \not \in \BZ$, this forces $\gamma^\bla_\bmu = 0$. If $m \in \BZ$, we conclude that:
$$
\gamma^{\bla}_\bmu = \ld \gamma^{\bla}_\bmu = \frac {\ld \gamma^{\bla}_\bmu}{\kappa_\bla} = \frac {\ld \langle \bla | \theta_\bmu}{\kappa_\bla} = \frac {\ld \left( \langle \bla| F \cdot s_\bmu^{\bm} - \sum^{\bla' \rhd \bla}_{|\bla'| = \bv} \gamma^{\bla'}_\bmu \cdot s_{\bla'|\bla}^{\bm} \right)}{\kappa_\bla}
$$
However, because the inequality \eqref{eqn:smallish2} is strict, the sum over $\bla' \rhd \bla$ does not contribute anything to the term of lowest degree. We conclude that:
$$
\gamma^{\bla}_\bmu = \frac {\ld \langle \bla| F \cdot s_\bmu^{\bm}}{\kappa_\bla} = \frac {\ld \left( \langle \bla| F |\bmu \rangle \cdot \fK_\bmu \right) + \sum^{\bmu' \lhd \bmu}_{|\bmu| = \bv'} \ld \left( \langle \bla| F |\bmu' \rangle \cdot s_{\bmu|\bmu'}^{\bm} \right)}{\kappa_\bla}
$$
Similarly, because the inequality \eqref{eqn:smallish2} is strict, the sum over $\bmu' \lhd \bmu$ does not contribute anything to the term of lowest degree, hence \eqref{eqn:moroder} follows.  
\end{proof}

\noindent We wish to apply the above Lemma in order to prove Theorem \ref{thm:pieri}, and therefore we need to estimate the matrix coefficients in the fixed point basis $|\bla\rangle$ of the operators $F = \oP_{[i;j)}^\bm$ or $\oQ_{-[i;j)}^\bm$. As these matrix coefficients are given by \eqref{eqn:fixedp} and \eqref{eqn:fixedq}, we need to evaluate the min deg and max deg of the quantities:
\begin{equation}
\label{eqn:a1}
\oP_{[i;j)}^\bm(\blamu) \qquad \text{and} \qquad \oQ_{-[i;j)}^\bm(\blamu)
\end{equation}
as well as:
\begin{equation}
\label{eqn:a2}
\fG^+_{\blamu} \ := \ \prod_{\bsq \in \blamu} \left[\prod_{\square \in \bmu} \zeta \left( \frac {\chi_\bsq}{\chi_\square}\right) \prod_{k=1}^\bw \Big[\frac {u_k}{q\chi_\bsq} \Big] \right]
\end{equation}
\begin{equation}
\label{eqn:a3}
\fG^-_{\blamu} \ := \ \prod_{\bsq \in \blamu} \left[\prod_{\square \in \bla} \zeta \left( \frac {\chi_\sq}{\chi_\bsq}\right) \prod_{k=1}^\bw \Big[\frac {\chi_\bsq}{qu_k} \Big] \right]^{-1}
\end{equation}
The formulas for \eqref{eqn:a1} will be discussed in the next Section. In the remainder of the present Section, we will take care of the rather tedious computation of the quantities $\fG^\pm_\blamu$. As in formulas \eqref{eqn:moroder}, it makes more sense to study the ratios:
\begin{equation}
\label{eqn:ratio} 
\fR^{+}_\blamu \ := \ \fG^+_\blamu \cdot \frac {\fK_\bmu}{\fK_\bla} \qquad \qquad \fR^{-}_\blamu \ := \ \fG^-_\blamu \cdot \frac {\fK_\bla}{\fK_\bmu} 
\end{equation}
We consider the minimal and maximal degrees:
\begin{equation}
\label{eqn:top}
\maxdeg \fR^\pm_\blamu \ =: \ r_\blamu \ = \mindeg \fR^\pm_\blamu \ 
\end{equation}
which will be shown in Exercise \ref{ex:ohboy} to not depend on $\pm$. The equality between the $\maxdeg$ and the opposite $\mindeg$ is a consequence of the fact that $\fR^\pm_\blamu$ is a product of quantum numbers $[x] = x^{\frac 12} - x^{-\frac 12}$. Consider the terms of lowest degree:
\begin{equation}
\label{eqn:lowest} 
\rho^\pm_\blamu \ = \ \ld \fR^\pm_\blamu 
\end{equation}
Exercise \ref{ex:ohboy} below will compute the quantities $r_\blamu$ and $\rho^\pm_\blamu$ in a way which will be used in the next Section. With the convention that Kronecker delta symbols are taken mod $n$, define the integer valued function:
$$
z(x) = \frac {|x|}2 \Big( \delta_x^1 + \delta_x^{-1} - 2\delta_x^0 \Big) 
$$
and note that:
\begin{equation}
\label{eqn:thing}
\maxdeg \zeta \left(\frac {\chi_\sq}{\chi_\bsq} \right) \ = \ z(c_\sq - c_\bsq) \ = \ - \mindeg \zeta \left(\frac {\chi_\sq}{\chi_\bsq} \right)
\end{equation}
for all boxes $\sq$ and $\bsq$. It is clear that $z(x) = z(-x)$, and therefore:
\begin{equation}
\label{eqn:ossum}
\maxdeg \frac {\zeta \left(\frac {\chi_\sq}{\chi_\bsq} \right)}{\zeta \left(\frac {\chi_\bsq}{\chi_\sq} \right)} \ = \ \mindeg \frac {\zeta \left(\frac {\chi_\sq}{\chi_\bsq} \right)}{\zeta \left(\frac {\chi_\bsq}{\chi_\sq} \right)} \ = \ 0
\end{equation}

\begin{exercise}
\label{ex:ohboy}
For skew diagrams $\blamu$, we have:
\begin{equation}
\label{eqn:stannis}
r_\blamu \ = \ - \frac 12 \sum_{\sq , \bsq \in \blamu} z(\sq - \bsq) 
\end{equation}
Moreover, when $\blamu = C$ is a cavalcade of ribbons as in Section \ref{sec:basicpar}, we have:
\begin{equation}
\label{eqn:cal}
\rho^+_\blamu = \frac {[q^{-2}]^{\#_C} \cdot (-q)^{N_{C}^+}}{\prod^{a \leftrightarrow b}_{i \leq a < b < j} \ \ld \zeta \left( \frac {\chi_{\sq_b}}{\chi_{\sq_a}} \right)}
\end{equation}
whereas if $\blamu = S$ is a stampede of ribbons as in Section \ref{sec:basicpar}, we have:
\begin{equation}
\label{eqn:stamp}
\rho^-_\blamu = \frac {[q^{-2}]^{\#_S} \cdot  (-q)^{N_{S}^-}}{\prod^{a\leftrightarrow b}_{i \leq a < b < j} \ \ld \zeta \left( \frac {\chi_{\sq_a}}{\chi_{\sq_b}} \right)}
\end{equation}

\end{exercise}

\tab In the above Exercise, we index the boxes of a cavalcade $\sq_i,...,\sq_{j-1}$ in order from the northwest to the southeast (see Figure \ref{fig:cavalcade}). The ribbons of a stampede are ordered from the southwest to the northeast, though the boxes of each individual ribbon in a stampede will still be ordered from the northwest to the southeast (see Figure \ref{fig:stampede}). We write $a \leftrightarrow b$ to indicate that $\sq_a$ and $\sq_b$ should not be next to each other in the same ribbon, and note that such factors should be removed from \eqref{eqn:cal} and \eqref{eqn:stamp} because they could produce factors of $1-1=0$ in the denominator. Recall that the number $N_{\bsq|\bla}^\pm$ denotes the signed number of corners ($\pm$ inner $\mp$ outer) of $\bla$ of the same color as $\bsq$ and with content larger or smaller than that of $\bsq$, depending on whether the sign is $+$ or $-$. We set $N_C^+ = \sum_{\bsq \in C} N^+_{\bsq|\bla}$ for a cavalcade $C$, while for a stampede $S$ we define $N^-_S$ by the slightly more complicated formula \eqref{eqn:slightly}.

\tab
Let us abbreviate $\bc_\blamu = \bc_\bla - \bc_\bmu$. Lemma \ref{lem:compute} reduces Theorem \ref{thm:pieri} to proving that:
\begin{equation}
\label{eqn:stat1}
\maxdeg \oP_{[i;j)}^\bm(\blamu) \ \ \leq \ \ \bm \cdot \bc_\blamu - r_\blamu
\end{equation}
\begin{equation}
\label{eqn:stat2}
\mindeg \oP_{[i;j)}^\bm(\blamu) \ \ \geq \ \ \bm \cdot \bc_\blamu + r_\blamu
\end{equation}
\begin{equation}
\label{eqn:stat3}
\maxdeg \oQ_{-[i;j)}^\bm(\blamu) \ \leq \ \bm \cdot \bc_\blamu - r_\blamu
\end{equation}
\begin{equation}
\label{eqn:stat4}
\mindeg \oQ_{-[i;j)}^\bm(\blamu) \ \geq \ \bm \cdot \bc_\blamu + r_\blamu
\end{equation}
and that equality holds in \eqref{eqn:stat2} and in \eqref{eqn:stat4} whenever $\blamu$ is a cavalcade $C$ of $\bm-$integral ribbons or a stampede $S$ of $\bm-$integral ribbons, respectively. In these cases, the lowest degree terms corresponding to equality in \eqref{eqn:stat2} and \eqref{eqn:stat4} are:
\begin{equation}
\label{eqn:stat5}
\ld \oP_{[i;j)}^\bm(C) = o_C^\bm  \cdot q^{\#_C + \high C + \ind_{C}^{\bm}} \prod^{a\leftrightarrow b}_{i \leq a < b < j} \ \ld \zeta \left( \frac {\chi_{\sq_b}}{\chi_{\sq_a}} \right) 
\end{equation}
\begin{equation}
\label{eqn:stat6}
\ld \oQ_{-[i;j)}^\bm(S) = o_S^{-\bm} \cdot q^{\#_S + \wide S - \ind_{S}^{\bm}-j+i} \prod^{a\leftrightarrow b}_{i \leq a < b < j} \ \ld \zeta \left( \frac {\chi_{\sq_a}}{\chi_{\sq_b}} \right)
\end{equation}
Recall the notation $\ind_{C}^\bm$ from \eqref{eqn:ind}. Note that, while there exists at most one cavalcade $C$ on any given skew diagram $\blamu$, there may be more than one stampede $S$. Therefore, the value for the lowest degree term of $\ld \oQ_{-[i;j)}^\bm(\blamu)$ is actually the sum of the right hand sides of \eqref{eqn:stat6} over all underlying stampedes $S$ of $\blamu$. \\

\section{The shuffle elements $\oP_{[i;j)}^\bm$ and $\oQ_{-[i;j)}^\bm$}
\label{sec:final}


\noindent The purpose of this Section is to compute the maximal and minimal degrees of the quantities $\oP_{[i;j)}^\bm(\blamu)$ and $\oQ_{-[i;j)}^\bm(\blamu)$, as well as their lowest degree terms, and use them to prove formulas \eqref{eqn:stat1} - \eqref{eqn:stat6}. Recall that $q_1 = qt$ and $q_2 = qt^{-1}$. By the definition \eqref{eqn:defpm1} and \eqref{eqn:defpm2}, as well as the renormalization \eqref{eqn:renormalized}, we have:
\begin{equation}
\label{eqn:vaes}
\oP_{[i;j)}^\bm(\blamu) \ = \ \sum_{\text{ASYT}^+} \prod_{a=i}^{j-1} \chi_a^{m_a} \cdot \frac {\prod_{a=i}^{j-1} (\chi_a t^{-a})^{s_a}}{q^{-1} \prod_{a=i+1}^{j-1} \left(\frac 1q - \frac {\chi_{a}}{t\chi_{a-1}} \right)} \prod_{i \leq a < b < j} \zeta \left( \frac {\chi_{b}}{\chi_{a}} \right)
\end{equation}
\begin{equation}
\label{eqn:dothrak}
\oQ_{-[i;j)}^\bm(\blamu) = \sum_{\text{ASYT}^-} \prod_{a=i}^{j-1} \chi_a^{-m_a} \cdot \frac {\prod_{a=i}^{j-1} (\chi_a t^{-a})^{-s_a}}{q^{j-i-1} \prod_{a=i+1}^{j-1} \left(\frac 1q - \frac {t\chi_{a-1}}{\chi_{a}} \right)} \prod_{i \leq a < b < j} \zeta \left( \frac {\chi_{a}}{\chi_{b}} \right)
\end{equation}
where we abbreviate:
\begin{equation}
\label{eqn:numbers} 
s_a \ = \ \lfloor m_i + ... + m_a \rfloor - \lfloor m_i + ... + m_{a-1} \rfloor - m_a
\end{equation}
and we recall that almost standard Young tableaux, denoted by $\asyt$, were defined in Section \ref{sec:eccentric}. Recall that a $\asyt^+$ (respectively $\asyt^-$) refers to a way to label the boxes of $\blamu$ with the numbers $i,...,j-1$, such that the labels increase (respectively decrease) as we go up and to the right, with the possible exception that the box labeled by $a-1$ is allowed to be directly above (respectively to the left of) the box labeled by $a$. The monomial $\chi_a$ denotes the weight of the box labeled by $a$. \\


\begin{exercise}
\label{ex:final}	

For any real numbers $x_i,...,x_{j-1}$, we have:
\begin{equation}
\label{eqn:finalment}
\sum_{a=i+1}^{j-1} \min(0,x_{a} - x_{a-1}) \ \leq \ \sum_{a=i}^{j-1} x_a s_a \ \leq \ \sum_{a=i+1}^{j-1} \max(0,x_{a} - x_{a-1})  
\end{equation}
where the first inequality becomes an equality if and only if:
$$
x_i = ... = x_{a_1-1} < x_{a_1} = ... = x_{a_2-1} < ... < x_{a_t} = ... = x_{j-1}
$$
while the second inequality becomes an equality if and only if:
$$
x_i = ... = x_{a_1-1} > x_{a_1} = ... = x_{a_2-1} > ... > x_{a_t} = ... = x_{j-1}
$$
for any chain of $\bm-$integral arcs $[i;a_1) \cup [a_1;a_2) \cup ... \cup [a_t;j)$. 
\end{exercise}

\tab 
Let us write $c_a$ for the content of the $a-$th box in an $\asyt^\pm$, and let us write $x_a = c_a - a$. With the above exercise in mind, relations \eqref{eqn:vaes} and \eqref{eqn:dothrak} imply that:
$$
\maxdeg \oP_{[i;j)}^\bm(\blamu) \text{ or } \oQ_{-[i;j)}^\bm(\blamu) \ \leq \ \sum_{a=i}^{j-1} m_a c_a  + \sum_{i \leq a < b < j} z(c_a - c_b) = \bm \cdot \bc_\blamu - r_\blamu
$$
$$
\mindeg \oP_{[i;j)}^\bm(\blamu) \text{ or } \oQ_{-[i;j)}^\bm(\blamu) \ \geq \ \sum_{a=i}^{j-1} m_a c_a  - \sum_{i \leq a < b < j} z(c_a - c_b) = \bm \cdot \bc_\blamu + r_\blamu
$$
We have used \eqref{eqn:thing} to estimate the min deg and the max deg of the product of $\zeta$'s. The above formulas imply \eqref{eqn:stat1} - \eqref{eqn:stat4}. Exercise \ref{ex:final} also tells us when the inequality for $\mindeg$ becomes an equality. In the case of $\oP_{[i;j)}^\bm(\blamu)$ (respectively $\oQ_{-[i;j)}^\bm(\blamu)$), the $\geq$ sign is an equality if and only if the contents are given by:
$$
c_b = c_{a_k}+ (b-a_k) \qquad \forall \ b \ \in \ \text{the }\bm-\text{integral arc} \ [a_k;a_{k+1}) 
$$
and such that $c_{a_k-1} < c_{a_{k}} - 1$ (respectively  $c_{a_k-1} > c_{a_k} - 1$). This implies that the skew $\bw-$diagram $\blamu$ is a cavalcade (respectively stampede) of $\bm-$integral ribbons of types $[i;a_1), [a_1;a_2),...,[a_t;j)$. Moreover, each $\asyt^\pm$ in \eqref{eqn:vaes} and \eqref{eqn:dothrak} for which equality is attained is the standard labeling of a cavalcade or stampede, as described in Section \ref{sec:basicpar}. Then for the cavalcade $C = \blamu$, we have:
$$
\ld \oP_{[i;j)}^\bm(C) = \prod_{a=i}^{j-1} \chi_a^{m_a} \cdot \frac {\prod_{a=i}^{j-1} (\chi_a t^{-a})^{s_a}}{q^{-1} \prod_{a=i+1}^{j-1} \ \ld \left(\frac 1q - \frac {\chi_{a}}{t\chi_{a-1}} \right)} \prod_{i \leq a < b < j} \ld \zeta \left( \frac {\chi_{b}}{\chi_{a}} \right)
$$
Concerning the denominators $\left( \frac 1q - \frac {\chi_{a}}{t\chi_{a-1}} \right)$, there are three possibilities: 

\begin{itemize}

\item if $\sq_{a-1}$ and $\sq_a$ lie in different ribbons, then the contribution of the denominator to the $\ld$ is $\frac 1q$

\item if $\sq_{a-1}$ is one box above $\sq_a$, then we obtain a factor of $\frac 1q - \frac 1q = 0$, which is canceled by a factor of $[1]$ in the numerator of $\prod \zeta \left( \frac {\chi_{b}}{\chi_{a}} \right)$

\item if $\sq_{a-1}$ is one box to the left of $\sq_{a}$, then we obtain a factor of $\frac 1q-q$, which is canceled by a factor of $[q^{-2}]$ in the numerator of $\prod \zeta \left( \frac {\chi_{b}}{\chi_{a}} \right)$
	
\end{itemize}


\noindent Recalling the notation \eqref{eqn:ind} and \eqref{eqn:ind3} for the cavalcade of ribbons $C = \blamu$, we conclude that:
$$
\ld \oP_{[i;j)}^\bm(C) = o_C^\bm  \cdot q^{\#_C + \high C + \ind^\bm_{C}} \prod^{a \leftrightarrow b}_{i \leq a < b < j} \ld \zeta \left( \frac {\chi_{b}}{\chi_{a}} \right)
$$
The reason why we do not encounter the factors $\zeta$ when $a \leftrightarrow b$ are successive boxes in the same ribbon is that they have already been used to cancel various factors from the denominator, in the second and third bullets above. The above relation establishes \eqref{eqn:stat5}. As for $\oQ_{-[i;j)}^\bm$, for any stampede $S = \blamu$, we have:
$$
\ld \oQ_{-[i;j)}^\bm(S) = o_S^{-\bm} \cdot \frac {\prod_{a=i}^{j-1} (\chi_at^{-a})^{-s_a}}{q^{j-i-1} \prod_{a=i+1}^{j-1} \ \ld \left(\frac 1q - \frac {t\chi_{a-1}}{\chi_{a}} \right)} \prod_{i \leq a < b < j} \ld \zeta \left( \frac {\chi_{a}}{\chi_{b}} \right)
$$
Concerning the denominators $\left( \frac 1q - \frac {t\chi_{a-1}}{\chi_{a}} \right)$, there are three possibilities: 

\begin{itemize}
	
\item if $\sq_{a-1}$ and $\sq_a$ lie in different ribbons, then the contribution of the denominator to the $\ld$ is $\frac 1q$
	
\item if $\sq_{a-1}$ is one box to the left of $\sq_a$, then we obtain a factor of $\frac 1q - \frac 1q = 0$, which is canceled by a factor of $[1]$ in the numerator of $\prod \zeta \left( \frac {\chi_{a}}{\chi_b} \right)$
	
\item if $\sq_{a-1}$ is one box above $\sq_{a}$, then we obtain a factor of $\frac 1q-q$, which is canceled by a factor of $[q^{-2}]$ in the numerator of $\prod \zeta \left( \frac {\chi_{a}}{\chi_{b}} \right)$
	
\end{itemize}


\noindent Recalling the notation \eqref{eqn:ind} and \eqref{eqn:ind3} for the stampede of ribbons $S = \blamu$, we conclude that:
$$
\ld Q_{[i;j)}^\bm(S) = o_S^{-\bm} \cdot q^{\#_S + \wide S - \ind_{S}^\bm - j + i} \prod^{a \leftrightarrow b}_{i \leq a < b < j} \ld \zeta \left( \frac {\chi_{a}}{\chi_{b}} \right)
$$
The reason why we do not encounter the factors $\zeta$ when $a \leftrightarrow b$ are successive boxes in the same ribbon is that they have already been used to cancel various factors from the denominator, in the second and third bullets above. The above establishes \eqref{eqn:stat6}. \\

\section{Interpreting ribbons via Maya diagrams}
\label{sec:maya}

\noindent Having finished the proof of Theorem \ref{thm:pieri}, let us interpret it combinatorially. The first thing to observe is that while the action $\UU \curvearrowright K(\bw)$ depends on all the equivariant parameters $q,t,u_1,...,u_\bw$, the coefficients of both \eqref{eqn:pieri1} and \eqref{eqn:pieri2} depend only on $q$. In other words, for fixed $\bm \in \qq$ we may renormalize the basis by setting $v_\bla := s_\bla^{+,\bm} / o_\bla^{\bm}$, and Theorem \ref{thm:pieri} implies that the formula:
\begin{equation}
\label{eqn:fock1}
e_{[i;j)} \cdot v_\bmu \ = \ \sum^{\blamu \ = \ C \text{ is a }[i;j)}_{\text{cavalcade of }\bm-\text{ribbons }} v_\bla \cdot (1 - q^2)^{\#_C} (-q)^{N_{C}^+} q^{\high C + \ind_{C}^\bm}
\end{equation}
gives rise to an action of $\bigotimes_{h=1}^g U^+_q(\dot{\fgl}_{l_h}) \curvearrowright K(\bw)$. The numbers $g,l_1,...,l_g$ were associated to $\bm \in \qq$ in Section \ref{sec:arcs}. For the remainder of this Section, for simplicity we will only work with the case of creation operators $e_{[i;j)}$ and with $\bw = \bs^n$, hence:
$$
\uup \ \curvearrowright \ K(\bs^n) \ = \ \Lambda \ = \ \text{Fock space}
$$
In this case, $\bla = (\la)$ is a single partition. When $\bm = (0,...,0)$, formula \eqref{eqn:fock1} reads:
\begin{equation}
\label{eqn:easyfock1}
e_{[i;j)} \cdot v_\mu \ = \ \sum^{\lamu \ = \ C \text{ is a }[i;j)}_{\text{cavalcade of ribbons }} v_\la \cdot (1 - q^2)^{\#_C} (-q)^{N_{C}^+} q^{\high C}
\end{equation}
and it gives rise to the action $\uup \curvearrowright \Lambda$ constructed in \citep{LLT} based on the action $\sup \curvearrowright \Lambda$ of Hayashi and Misra-Miwa. However, let us note that while \loccit presents the action in terms of the simple and loop generators of $\uup$, we present it in terms of the root generators of Section \ref{sec:quantum}.  

\tab 
In the case of arbitrary $\bm \in \qq$, the simple generators are of the form $e_{[i;j)}$, where $[i;j)$ is a minimal $\bm-$integral arc. These generators act on $v_\la$ by adding a whole ribbon of type $[i;j)$. At first, it would seem like this is very different from the situation in the previous paragraph, where we added a single box. However, we will now show that the two situations are equivalent, by an assignment:
$$
\Big\{ \text{partition }\la \Big\} \ \longrightarrow \ \Big \{\bm-\text{core}, \ \bm-\text{quotient} \Big\}
$$
which we will now define. When $m = \left( \frac an, ..., \frac an \right)$ for $\gcd(a,n)=1$, it will precisely reduce to the well-known $n-$core and $n-$quotient construction in combinatorics. The idea is to find a language in which one can think of entire ribbons as single boxes, and a convenient way to do so is by using Maya diagrams. Specifically, one can take the Young diagram of a partition $\lambda$ and rotate it by $45^\circ$:
\begin{figure}[H]
	
\begin{picture}(300,170)(-20,-10)
	
\put(40,0){\line(1,0){320}}
\put(200,0){\line(1,1){120}}
\put(260,100){\line(-1,-1){20}}
\put(240,80){\line(-1,1){20}}
\put(220,100){\line(-1,-1){40}}
\put(180,60){\line(-1,1){20}}
\put(160,80){\line(-1,-1){20}}
\put(200,0){\line(-1,1){120}}
\put(280,80){\line(-1,1){20}}
	
\put(80,0){\circle{5}}
\multiput(80,5)(0,10){14}{\line(0,1){5}}
	
\put(100,0){\circle{5}}
\multiput(100,5)(0,10){14}{\line(0,1){5}}
	
\put(120,0){\circle{5}}
\multiput(120,5)(0,10){14}{\line(0,1){5}}
	
\put(140,0){\circle{5}}
\multiput(140,5)(0,10){14}{\line(0,1){5}}
	
\put(160,0){\circle*{5}}
\multiput(160,5)(0,10){14}{\line(0,1){5}}
	
\put(180,0){\circle{5}}
\multiput(180,5)(0,10){14}{\line(0,1){5}}
	
\put(200,0){\circle*{5}}
\multiput(200,5)(0,10){14}{\line(0,1){5}}
	
\put(220,0){\circle*{5}}
\multiput(220,5)(0,10){14}{\line(0,1){5}}
	
\put(240,0){\circle{5}}
\multiput(240,5)(0,10){14}{\line(0,1){5}}
	
\put(260,0){\circle*{5}}
\multiput(260,5)(0,10){14}{\line(0,1){5}}
	
\put(280,0){\circle{5}}
\multiput(280,5)(0,10){14}{\line(0,1){5}}
	
\put(300,0){\circle*{5}}
\multiput(300,5)(0,10){14}{\line(0,1){5}}
	
\put(320,0){\circle*{5}}
\multiput(320,5)(0,10){14}{\line(0,1){5}}
	
	
\end{picture}
	
\caption[The Maya diagram associated to a partition]
	
\label{fig:maya}
	
\end{figure} 

\noindent We rescale the Young diagram so that the vertical lines passing through the corners of the boxes have integer $x$ coordinate. If a vertical line passes through an outer corner or a northeast pointing edge (respectively inner corner or northwest pointing edge) of $\la$, we paint the corresponding $x$ intercept black (respectively white). We will refer to $\bullet_i$ (respectively $\circ_i$) as the black (respectively white) point situated at coordinate $i\in \BZ$. The resulting collection of black and white points on the horizontal axis is called the \b{Maya diagram} corresponding to $\la$, and note that it satisfies:
\begin{equation}
\label{eqn:selyse}
0 \ = \ \Big( \# \text{ white points } > 0 \Big) \ - \ \Big( \# \text{ black points } \leq 0 \Big)
\end{equation}
Conversely, there is a unique Young diagram corresponding to any partition of $\BZ$ into black and white points which satisfies \eqref{eqn:selyse} and is all black far enough to the right and all white far enough to the left. The size of the Young diagram can be read off the Maya diagram by the formula:
$$
|\la| \ = \ \sum_{i=1}^\infty i - \Big(\text{coordinate of the }i-\text{th black integer} \Big)
$$
Let us now reinterpret the combinatorial objects of \eqref{eqn:fock1} in terms of Maya diagrams. Adding a ribbon to a partition $\mu$ is equivalent to switching the white point $\circ$ situated at position $i$ with the black point $\bullet$ situated at position $j$ in the Maya diagram of $\mu$, for some $i<j$. This will be denoted by $\circ_i \leftrightarrow \bullet_j$, and it corresponds to adding a ribbon of type $[i;j)$. Adding a cavalcade of ribbons to a Maya diagram $\mu$ corresponds to performing a collection of such switches:
\begin{equation}
\label{eqn:cavalcademaya}
C \ = \ \{\circ_{i_1} \leftrightarrow \bullet_{j_1},...,\circ_{i_k} \leftrightarrow \bullet_{j_k}\}
\end{equation}
for some $i_1<j_1<i_2<j_2<...<i_k<j_k$. Adding a stampede of ribbons is a similar succession of switches:
\begin{equation}
\label{eqn:stampedemaya}
S \ = \ \{\circ_{i_1} \leftrightarrow \bullet_{j_1},...,\circ_{i_k} \leftrightarrow \bullet_{j_k}\}
\end{equation}
such that $i_{a+1}<j_a$ for all $a$. In the case of a stampede, the order of the switches matters, and so we perform the switches \eqref{eqn:stampedemaya} in order from $k$ to $1$. 



\tab 
Let us fix $\bm \in \qq$, and consider the corresponding union of cycles $C_1 \sqcup ... \sqcup C_g$ that we assigned to $\bm$ in Section \ref{sec:arcs}. Specifically, $C_h = \{i_1,...,i_{l_h}\} \subset \{1,...,n\}$ such that:
$$
[i_1;i_2), \ [i_2;i_3), \ ..., \ [i_{l_h},i_1) \quad \text{form a chain of minimal }\bm-\text{integral arcs}
$$
We will write $\omega_h n$ for the total length $\sum_{s=1}^{l_h} (i_{s+1}-i_s)$ of the above chain of arcs, and $\omega_h \in \BN$ will be called the \b{winding number}. Take the Maya diagram of a partition $\la$, and for all $h\in \{1,...,g\}$ and $k \in \{1,...,\omega_h\}$ consider the bi-infinite sequence $S^{h,k} \subset \BZ$ consisting of points situated at coordinates $\in C_h + kn$ modulo $\omega_h n$. Let:
$$
c^{h,k} \ := \  \# \{\circ_a, \ a>0, \ a \in C_h + kn \text{ mod }\omega_h n\}  \ - \ \# \{\bullet_a, \ a\leq 0, \ a \in C_h + kn \text{ mod }\omega_h n\} 
$$
and identify the sequence $S^{h,k} \cong \BZ$ with the Maya diagram of a partition $\la^{h,k}$. Indeed, there is a single way to do this, since property \eqref{eqn:selyse} must be satisfied for $\la^{h,k}$ to be a partition. Doing the above construction for all $h\in \{1,...,g\}$ and all $k\in \{1,...,\omega_h\}$ gives us a collection of partitions: 
\begin{equation}
\label{eqn:mquotient}
\{\la^{h,k}\}^{1\leq h \leq g}_{1\leq k \leq \omega_h} \quad \text{which will be called the }\bm-\b{quotient} \text{ of } \ \la 
\end{equation}
Let $N = N_\bm := \sum_{h=1}^g \omega_h$ be the total number of constituent partitions of the $\bm-$quotient. The numbers $c^{h,k}$ must add up to 0 in virtue of \eqref{eqn:selyse}, and:
\begin{equation}
\label{eqn:mcore}
\text{the collection} \quad \{c^{h,k}\}^{1\leq h \leq g}_{1\leq k \leq \omega_h} \in \BZ^{N-1} \quad \text{will be called the }\bm-\b{core} \text{ of } \ \la 
\end{equation}
From the construction, it is clear that the assignment:
\begin{equation}
\label{eqn:bij}
\Big\{ \text{partitions } \Big\} \ \stackrel{\Psi_\bm}\longrightarrow \ \BZ^{N-1} \times \Big\{ \text{partitions } \Big\}^N, \qquad \Psi_\bm(\la) \ = \ (c^{h,k}, \la^{h,k})^{1\leq h \leq g}_{1\leq k \leq \omega_h}
\end{equation}
is a bijection. For us, this bijection is useful because it allows us to describe $\bm-$integral ribbons. Specifically, adding a box to some partition $\la^{h,k}$ corresponds to switching a white point $\circ_{i}$ with the black point $\bullet_{i+1}$ in the Maya diagram of $\la^{h,k}$. In the Maya diagram of the original partition $\la$, this corresponds to switching the white point $\circ_i$ with the black point $\bullet_{\upsilon_\bm(i)}$. This precisely corresponds to adding a $\bm-$integral ribbon of minimal type $[i;\upsilon_\bm(i))$ to the Young diagram $\la$. 

\tab 
Adding a general ribbon to $\la^{h,k}$ corresponds to adding a general $\bm-$integral ribbon to $\la$. Therefore, we see that the bijection $\Psi_\bm$ intertwines the operator \eqref{eqn:fock1} on the left hand side of \eqref{eqn:bij} with the operator \eqref{eqn:easyfock1} on the right hand side of \eqref{eqn:bij}. This statement is more philosophical than precise: indeed, to obtain an actual intertwiner, one would have to match the various powers of $\pm q$ in the right hand sides of \eqref{eqn:fock1} and \eqref{eqn:easyfock1}, and this requires choosing an appropriate renormalization of the basis $v_\mu$. When $n=1$, this is achieved by Proposition 5.5. of \citep{Npieri}.

\chapter{Proofs of the Exercises}
\label{chap:proofs}
\begin{proof} \emph{of Exercise \ref{ex:unstable}:} We will prove the case when $\th > 0$, as the remaining case is proved by dualizing all the maps. We need to show that a quadruple $(X,Y,A,B)$ is semistable if and only if it is cyclic, by which we mean that $V$ is generated by successive applications of $X$ and $Y$ on the image of $A$. Let us first show that a cyclic quadruple is semistable. By Lemma \ref{lem:ss}, it is enough to exhibit a certain $\det^{-1}$ covariant function which does not vanish on the quadruple $(X,Y,A,B)$. Being cyclic implies that there exist $v = \dim V$ vectors $z_1,...,z_v \in W$, and certain polynomials $P_1(X,Y),...,P_v(X,Y)$ such that:
$$
V = \text{span } \Big\langle P_1(X,Y)Az_1,...,P_v(X,Y)Az_v \Big\rangle	
$$
Then a $\det^{-1}$ covariant function which does not vanish on $(X,Y,A,B)$ is precisely:
$$
\det\left(P_1(X,Y)Az_1,...,P_v(X,Y)Az_v\right)
$$	
Note that the above is $\det^{-1}$ and not $\det$ covariant, because the action of the determinant character on functions of quadruples is inverse to the action on the vector space of quadruples. Conversely, let us show that a semistable quadruple is cyclic, by proving the contrapositive: a non-cyclic quadruple is unstable. In other words, let us assume that there exists a proper subspace $V' \subset V$ which contains the image of $A$ and is preserved by $X,Y$. We will assume $V' = \BC^{v'} \subset \BC^v = V$ as the standard inclusion of the first $v'$ basis vectors. Then both $X$ and $Y$ are block triangular with the topmost block of size $v'$, while $A$ only has non-zero entries in the first $v'$ rows. If we consider the following one parameter subgroup in the $G_v-$orbit of $(X,Y,A,B)$:
$$
\text{diag}(\underbrace{1,...,1}_{v'},\underbrace{t,...,t}_{v-v'}) \cdot (X,Y,A,B)
$$ 
then we observe that it has a well-defined limit as $t\rightarrow \infty$. In this limit, $\det^{-1}$ of the above one parameter subgroup goes to 0, so we conclude that the closure of the orbit of $(X,Y,A,B)$ intersects the zero section. Therefore, the quadruple is unstable. \\
\end{proof}

\begin{proof} \emph{of Exercise \ref{ex:tangent}:} Since $\CN_{v,w}$ is defined as the Hamiltonian reduction \eqref{eqn:nak0}, we conclude that we have short exact sequences:
\begin{equation}
\label{eqn:lana}
\xymatrix{
\fg_v \ar@{^{(}->}[r] & T \left( \mu^{-1}(0) \right) \ar@{^{(}->}[d] \ar@{->>}[r]
& T\CN_{v,w} \\
& N_{v,w} \ar@{->>}[d]^{d\mu} & \\
& \fg_v^\vee &} 
\end{equation}
where $N_{v,w} = \Hom(V,V) \oplus \Hom(V,V) \oplus \Hom(W,V) \oplus \Hom(V,W)$ and we identify the tangent bundle to a vector space with the vector space itself:
\begin{equation}
\label{eqn:kane}
TN_{v,w} =  \frac 1{qt} \cdot \CV \otimes \CV^\vee + \frac tq \cdot \CV \otimes \CV^\vee + \sum_{j=1}^w \frac {1}{qu_j} \cdot \CV + \sum_{j=1}^w  \frac {u_j}q \cdot \CV^\vee
\end{equation}
where $\CV$ denotes the tautological vector bundle with fibers $V$. We have introduced equivariant parameters in front of each summand of \eqref{eqn:kane} so the above becomes an equality of $T-$equivariant vector bundles. The left horizontal map of \eqref{eqn:lana} is the infinitesimal action of $G_v$ on $\mu^{-1}(0)$, which is injective because the $G_v$ action is free on the semistable locus. Meanwhile, the surjectivity of the bottom vertical map is equivalent to the fact that Nakajima quiver varieties have precisely the expected dimension. Then we conclude that the $K-$theory class of the tangent bundle to Nakajima quiver varieties is given by subtracting the tangent directions to $G_v$ and the normal directions to $\fg_v^\vee$ from \eqref{eqn:kane}:
$$
T\CN_{v,w} = \sum_{j=1}^w \Big( \frac {1}{qu_j} \cdot \CV +  \frac {u_j}q \cdot \CV^\vee \Big) + \left(\frac 1{qt} + \frac tq - 1 - \frac 1{q^2} \right) \CV \otimes \CV^\vee
$$
where $\CV$ is the tautological vector bundle on $\CN_{v,w}$ with fibers given by the vector space $V$. Then the proof of the Exercise reduces to the statement that:
\begin{equation}
\label{eqn:fibertaut}
\CV|_\bla \ = \ \sum_{\sq \in \bla} \chi_\sq
\end{equation}
This follows from semistability, which implies that the fiber $\CV|_\bla$ is spanned by $v_\sq = (X^iY^jA) \cdot \omega_k$ over all boxes $\sq = (i,j)$ in the $k-$th constituent partition of $\bla$, where $\omega_1,...,\omega_w$ denote the standard basis of $W$. The character of $T$ on the vector $v_\sq$ therefore equals $q_1^iq_2^j qu_k$, which is precisely the weight $\chi_\sq$ defined in \eqref{eqn:weight}. \\
\end{proof}

\begin{proof} \emph{of Exercise \ref{ex:tangent2}:} This Exercise is proved much like the previous one. Points of $\fZ_{\bv^+,\bv^-,\bw}$ are collections of linear maps as in the commutative diagram \eqref{eqn:bigdiag}:
\begin{equation}
\label{eqn:bigdiagram}
\xymatrix{ & & \ar@/^/[dl]^{Y_{i-1}^+} V_i^+ \ar@{->>}[dd] \ar@/^/[dr]^{X_i^+} & \\
... \ V_{i-2}^\pm \ar@/^/[r]^{X^\pm_{i-2}} & \ar@/^/[l]^{Y^\pm_{i-2}} V_{i-1}^\pm \ar@/^/[ru]^{X_{i-1}^+} \ar@/_/[rd]_{X_{i-1}^-} & & V_{i+1}^\pm \ar@/^/[r]^{X^\pm_{i+1}} \ar@/^/[lu]^{Y_i^+} \ar@/_/[ld]_{Y_{i}^-} & \ar@/^/[l]^{Y^\pm_{i+1}} V_{i+2}^\pm \ ... \\
 &	& \ar@/_/[lu]_{Y_{i-1}^-} V_i^- \ar@/_/[ru]_{X_{i}^-} &}
\end{equation}	
We suppress the $A$ and $B$ maps in order to make the above diagram more readable. The above linear maps must satisfy the moment map conditions:
$$
X_{j-1}^+Y_{j-1}^- - Y_{j}^+X_{j}^- + A_{j}^+B_{j}^- \ = \ 0 \ \in \ \Hom(V_{j}^-,V_{j}^+) \qquad \forall \ j \in \{1,...,n\}
$$
and must be taken modulo the subgroup $P \subset G_{\bv^+} \times G_{\bv^-}$ which preserves the collection of quotients $V^+ \twoheadrightarrow V^-$. This subgroup has Lie algebra: 
$$
\fp \ = \  \BC \oplus \bigoplus_{k=1}^n \text{Hom}(V_{k}^-,V_{k}^+)
$$	
where the first copy of $\BC$ simply rescales the one-dimensional kernel of $V_i^+ \twoheadrightarrow V_i^-$. As in the proof of the previous Exercise, the fact that $\fZ_{\bv^+,\bv^-,\bw}$ has the expected dimension implies that its tangent space is given by the affine space of quadruples minus the moment map conditions, and minus the gauge transormations in $\fp$. We obtain the following equality of $K-$theory classes:
\begin{equation}
\label{eqn:bb}
T \fZ_{\bv^+,\bv^-,\bw} = \sum_{j=1}^\bw \Big( \frac 1{qu_j} \cdot \CV_{j}^+ + \frac {u_j}q \cdot \CV_{j}^{-\vee} \Big) + 
\end{equation}
$$
+ \sum_{k=1}^n \Big( \frac 1{qt} \cdot \CV_{k+1}^{+} \otimes \CV_k^{-\vee} + \frac tq \cdot \CV_{k}^+ \otimes \CV_{k-1}^{-\vee} - \CV^+_k \otimes \CV^{- \vee}_k - \frac 1{q^2} \cdot \CV^+_k \otimes \CV^{- \vee}_k \Big) - 1
$$
where $\CV^\pm_j$ denotes the pull-back of the $j-$th tautological vector bundle from the quiver variety $\CN_{\bv^\pm,\bw}$, for all $j$ modulo $n$. When $j\neq i$ modulo $n$, we have $\CV^+_j = \CV^-_j$, while when $j = i$ modulo $n$, the tautological line bundle $\CL$ coincides with $\text{Ker}(\CV^+_i \twoheadrightarrow \CV^-_i)$. Together with \eqref{eqn:fibertaut}, \eqref{eqn:bb} implies \eqref{eqn:tansim}. \\	
\end{proof} 


\begin{proof} \text{of Exercise \ref{ex:intcorr}:} We will prove that the right hand side of \eqref{eqn:formula1} equals \eqref{eqn:form1}, and leave the analogous case of $e_{i,d}^-$ to the interested reader. We may expand the right hand side in terms of fixed points:
$$
\int z^d \cdot \ov{f(X - z) \zeta\Big(\frac zX \Big)} \prod_{1\leq j \leq \bw}^{u_j \equiv i} \Big[\frac {u_{i}}{qz} \Big] Dz = 
$$
$$
=\sum_{\bla^+} |\bla^+ \rangle \int z^d f\left(\chi_{\bla^+} - z\right) \zeta\left(\frac z{\chi_{\bla^+}} \right) \prod_{1\leq j \leq \bw}^{u_j \equiv i} \Big[\frac {u_{i}}{qz} \Big] Dz
$$
If we recall the definition of $\zeta$ in \eqref{eqn:defzeta}, we see that:
\begin{equation}
\label{eqn:tarth}
\zeta\left(\frac z{\chi_{\bla}} \right) \prod_{1\leq j \leq \bw}^{u_j \equiv i} \Big[\frac {u_{i}}{qz} \Big] = \frac {\prod^{\text{inner corners}}_{\sq \text{ of }\bla\text{ of color }i} \left[\frac {\chi_\sq}{q^2 z} \right]}
{\prod^{\text{outer corners}}_{\sq \text{ of }\bla\text{ of color }i} \left[ \frac {\chi_\sq}{q^2 z} \right]}
\end{equation}
for any $\bw-$partition $\bla$. Therefore, we have:
$$
\sum_{\bla^+} |\bla^+ \rangle \int z^d f(\chi_{\bla^+} - z ) \zeta\left(\frac z{\chi_{\bla^+}} \right) \prod_{1\leq j \leq \bw}^{u_j \equiv i} \Big[\frac {u_{i}}{qz} \Big] Dz =
$$
$$
= \sum_{\bla^+} |\bla^+ \rangle \int z^d f(\chi_{\bla^+} - z) \frac {\prod^{\text{inner corners}}_{\sq \text{ of }\bla^+\text{ of color }i} \left[\frac {\chi_\sq}{q^2 z} \right]}
{\prod^{\text{outer corners}}_{\sq \text{ of }\bla^+\text{ of color }i} \left[ \frac {\chi_\sq}{q^2 z} \right]} Dz
$$
Instead of integrating around $0$ and $\infty$, we could integrate over small contours around all the other poles. These poles are $z = \chi_\sq q^{-2}$ for an outer corner $\sq \in \bla^+$ of color $i$. Such $z$ is precisely the weight of a box $\bsq$ which can be removed from $\bla^+$ in order to produce a smaller partition $\bla^- \leq_i \bla^+$. We conclude that:
$$
\sum_{\bla^+} |\bla^+ \rangle \int z^d \cdot f(\chi_{\bla^+}-z) \frac {\prod^{\text{inner corners}}_{\sq \text{ of }\bla^+\text{ of color }i} \left[\frac {\chi_\sq}{q^2 z} \right]}
{\prod^{\text{outer corners}}_{\sq \text{ of }\bla^+\text{ of color }i} \left[ \frac {\chi_\sq}{q^2 z} \right]} Dz = 
$$
$$
= \sum_{\bla^+ \geq_i \bla^-}^{\bsq = \bla^+/\bla^-} |\bla^+ \rangle \cdot \chi_\bsq^d \cdot f(\chi_{\bla^-}) \cdot (1-1) \cdot \frac {\prod^{\text{inner corners}}_{\sq \text{ of }\bla^+\text{ of color }i} \left[\frac {\chi_\sq}{q^2 \chi_\bsq} \right]}
{\prod^{\text{outer corners}}_{\sq \text{ of }\bla^+\text{ of color }i} \left[ \frac {\chi_\sq}{q^2 \chi_\bsq} \right]} = 
$$
$$
=  \sum_{\bla^+ \geq_i \bla^-}^{\bsq = \bla^+/\bla^-} |\bla^+ \rangle \cdot \frac {\chi_\bsq^d}{[q^{-2}]} \cdot f(\chi_{\bla^-}) \cdot \frac {\prod^{\text{inner corners}}_{\sq \text{ of }\bla^-\text{ of color }i} \left[\frac {\chi_\sq}{q^2 \chi_\bsq} \right]}
{\prod^{\text{outer corners}}_{\sq \text{ of }\bla^-\text{ of color }i} \left[ \frac {\chi_\sq}{q^2 \chi_\bsq} \right]} = 
$$
$$
= \sum_{\bla^+ \geq_i \bla^-}^{\bsq = \bla^+/\bla^-} | \bla^+ \rangle \cdot \frac {\chi_\bsq^d}{[q^{-2}]} \cdot f(\chi_{\bla^-}) \cdot \zeta \left( \frac {\chi_\bsq}{\chi_{\bla^-}} \right)  \prod^{u_j\equiv i}_{1\leq j \leq \bw} \Big[\frac {u_j}{q\chi_\bsq} \Big] 
$$
where in the last equality we have applied \eqref{eqn:tarth} again. Comparing this with \eqref{eqn:form1} gives us the required result. \\
\end{proof}

\begin{proof} \text{of Exercise \ref{ex:opposite}:} Let us assume that $\alpha \in K_T(F)$ and $\beta \in K_T(F')$ for two connected components of the fixed locus $F,F' \subset X^A$, and let us first consider the case when $F = F'$. Note that the order of any two fixed components with respect to $\sigma$ and $-\sigma$ are exact opposites of each other:
$$
F_1 \unlhd_\sigma F_2 \qquad \Leftrightarrow \qquad F_1 \unrhd_{-\sigma} F_2
$$
Therefore, the correspondences $Z^\sigma_F$ and $Z^{-\sigma}_F$ only intersect on the diagonal $\Delta_F \subset F \times F$, and they do so properly. Indeed, property \eqref{eqn:leading} gives us: 
$$
\stab_\CL^\sigma(\alpha)|_F = \alpha \cdot \left[N_{F \subset X}^-\right] \qquad \qquad \stab_{\CL^{-1}}^{\sigma^{-1}}(\beta)|_F = \beta \cdot \left[N_{F \subset X}^+\right]
$$
If we let $\iota^F:F \hookrightarrow X$ be the standard embedding, equivariant localization gives us:
$$
\stab_\CL^\sigma(\alpha) = \widetilde{\iota}^F_*\left( \frac {\alpha}{\left[N_{F \subset X}^+\right]} \right) + \text{terms supported on }F_0 \unlhd F
$$
$$
\stab_{\CL^{-1}}^{\sigma^{-1}}(\beta) = \widetilde{\iota}^F_*\left( \frac {\beta}{\left[N_{F \subset X}^-\right]} \right) + \text{terms supported on }F_0 \unrhd F
$$
which implies \eqref{eqn:pairing}. When $F \neq F'$, the RHS of \eqref{eqn:pairing} is trivially zero, and so we must show the same for the LHS. Because $\stab_\CL^\sigma$ (respectively $\stab_{\CL^{-1}}^{\sigma^{-1}}$) is supported on the attracting set of $F$ (respectively the repelling set of $F'$), we conclude that the left hand side of \eqref{eqn:pairing} can be non-zero only if $F' \lhd F$. If this is the case, then equivariant localization gives us:
\begin{equation}
\label{eqn:enigma}
K_T(\pt) \ \ni \ \left(\stab_\CL^\sigma(\alpha), \stab_{\CL^{-1}}^{\sigma^{-1}}(\beta) \right)_X = \sum_{F' \unlhd F_0 \unlhd F} \frac {\stab^\sigma_\CL|_{F \times F_0} \cdot \stab^{\sigma^{-1}}_{\CL^{-1}}|_{F' \times F_0} \cdot \alpha \cdot \beta}{[N_{F_0 \subset X}]}
\end{equation}
Let us consider the minimal and maximal degree of the above expression in the direction of $\sigma \in A \subset T$. Without loss of generality, we may assume that $\alpha$ and $\beta$ have equivariant parameters concentrated in a single degree, which is possible since they are $K-$theory classes on a variety with trivial action. Set:
$$
a \ = \ \mindeg \alpha \ = \ \maxdeg \alpha \qquad \qquad b \ = \ \mindeg \beta \ = \ \maxdeg \beta
$$
Then condition \eqref{eqn:small} tells us that:
$$
\maxdeg \stab^{\sigma}_\CL|_{F \times F_0} \ \leq \ \maxdeg [N^-_{F_0 \subset X}] + \CL|_{F_0} - \CL|_{F}
$$ 
$$
\maxdeg \stab^{\sigma^{-1}}_{\CL^{-1}}|_{F' \times F_0} \ \leq \ \maxdeg [N^+_{F_0 \subset X}] - \CL|_{F_0} + \CL|_{F'}
$$ 
$$
\mindeg \stab^{\sigma}_\CL|_{F \times F_0} \ \geq^* \ \mindeg [N^-_{F_0 \subset X}] + \CL|_{F_0} - \CL|_{F}
$$
$$
\mindeg \stab^{\sigma^{-1}}_{\CL^{-1}}|_{F' \times F_0} \ \geq^\circ \ \mindeg [N^+_{F_0 \subset X}] - \CL|_{F_0} + \CL|_{F'}
$$
where equality holds in $\geq^*$ only if $F_0 = F$, and equality holds in $\geq^\circ$ only if $F_0 = F'$. Since $F \neq F'$ these inequalities cannot both be equalities, hence adding the four inequalities above gives us:
$$
\maxdeg \text{of \eqref{eqn:enigma}} \ \leq \ [N^+] + [N^-] + \CL|_{F'} - \CL|_F + a + b - [N] = \CL|_{F'} - \CL|_F + a + b
$$
$$
\mindeg \text{of \eqref{eqn:enigma}} \ > \ [N^+] + [N^-] + \CL|_{F'} - \CL|_F + a + b - [N] = \CL|_{F'} - \CL|_F + a + b
$$
where we use the shorthand notation $N = N_{F_0 \subset X}$ for any fixed component $F_0$. The only way the minimal degree of \eqref{eqn:enigma} can be strictly bigger than the maximal degree is if the Laurent polynomial \eqref{eqn:enigma} equals zero. \\
\end{proof}

\begin{proof} \emph{of Exercise \ref{ex:correspondence}:} Points of $\fZ_i$ are quadruples of linear maps that preserve an collection of quotients $\{V^+_j \twoheadrightarrow V^-_j\}$ of codimension $\delta_j^i$. To prove \eqref{eqn:reasonable}, we must show that if the maps satisfy properties \eqref{eqn:att} or \eqref{eqn:rep} on the collection of vector spaces $\{V^-_j\}_{1\leq j \leq n}$, they also satisfy the same properties on the collection of vector spaces $\{V^+_j\}_{1\leq j \leq n}$. \footnote{This would establish the fact that the operators $e^+_{i,d}$ are Lagrangian. The case of the operators $e^-_{i,d}$ is proved by switching $+$ with $-$, and the argument is analogous} Without loss of generality, let us study the attracting case, i.e. property \eqref{eqn:att}. The assumption tells us that there exists a filtration of $\{V^-_j\}_{1\leq j \leq n}$ whose associated graded vector spaces are generated by $A \cdot (\text{the basis vectors of }W)$ and on which the $X$ maps are nilpotent. To extend this filtration to the vector spaces $\{V^+_j\}_{1\leq j \leq n}$, we must decide in which filtration degree to put $l \in \text{Ker} (V_i^+ \twoheadrightarrow V_i^-)$. By semistability, we may write:
\begin{equation}
\label{eqn:sum}
l \ = \ \sum_{j=1}^\bw P_j(X,Y)\cdot A\omega_j	
\end{equation}
for various polynomials $P_j$, and we simply define the filtration on $\{V^+_j\}_{1\leq j \leq n}$ by placing $l$ in filtration degree equal to the highest $j$ which can appear non-trivially in sums of the form \eqref{eqn:sum}. Finally, the $X$ maps are nilpotent on $\{V^+_j\}_{1\leq j \leq n}$ as on $\{V^-_j\}_{1\leq j \leq n}$, because $X_i \left( V_i^+ \right) \subset V_{i+1}^-$. The repelling case is treated by replacing the words ``highest $j$" with ``lowest $j$" and ``$X$ nilpotent" with ``$Y$ nilpotent". \\	
\end{proof}

\begin{proof} \emph{of Exercise \ref{ex:antipode}:} Let us prove only the first of the required identities, as the rest are completely analogous. For any $F \in \CS_\bk^+$, let us expand \eqref{eqn:cop1}:
$$
\Delta(F) \ = \ \sum_{\text{expansion in }z_{ia} \ll z_{jb}}^{0 \leq \bl \leq \bk} \frac {\left[ \prod^{b > l_j}_{1\leq j \leq n} \ph^+_j(z_{jb}) \right] F'(z_{i1},...,z_{il_i}) \otimes F''(z_{i,l_i+1},...,z_{ik_i})}{\prod^{a \leq l_i}_{1\leq i \leq n} \prod_{1\leq j \leq n}^{b > l_j} \zeta\left( z_{jb} / z_{ia}  \right)}	
$$	
where we use the notation $F',F''$ to separate the variables of $F$ into two groups, with regard to the expansion in $z_{ia} \ll z_{jb}$ for all $a\leq l_i$ and $b>l_j$. Applying the antipode to the above coproduct gives us $\qquad (S \otimes \text{Id}) \circ \Delta(F) =$	
$$
= \ \sum_{\text{expand in }z_{ia} \ll z_{jb}}^{0 \leq \bl \leq \bk} \frac {S\Big[F'(z_{i1},...,z_{il_i}) \Big] S\left[ \prod^{b > l_j}_{1\leq j \leq n} \ph^+_j(z_{jb}) \right] \otimes F''(z_{i,l_i+1},...,z_{ik_i})}{\prod^{a \leq l_i}_{1\leq i \leq n} \prod_{1\leq j \leq n}^{b > l_j} \zeta\left( z_{jb} / z_{ia}  \right)}	
$$
since $S$ is an anti-homomorphism. Multiplying the tensor factors together gives us:
$$
S(F_1)F_2 = \text{shuffle product applied to }(S \otimes \text{Id}) \circ \Delta(F) = \sum_{\text{expand in }z_{ia} \ll z_{jb}}^{0 \leq \bl \leq \bk} 
$$	
$$
\frac {\left[ \prod^{a \leq l_i}_{1\leq i \leq n} \left(-\ph^+_i(z_{ia}) \right)^{-1} \right] * F'(z_{i1},...,z_{il_i}) * \left[ \prod^{b > l_j}_{1\leq j \leq n} \ph^+_j(z_{jb})^{-1} \right] * F''(z_{i,l_i+1},...,z_{ik_i})}{\prod^{a \leq l_i}_{1\leq i \leq n} \prod_{1\leq j \leq n}^{b > l_j} \zeta\left( z_{jb} / z_{ia}  \right)}	
$$
We can use \eqref{eqn:colombia} to commute all the $\ph$'s to the front, hence $\qquad S(F_1)F_2 = $
$$
= \ \left[ \prod^{1\leq a \leq k_i}_{1\leq i \leq n} \ph^+_i(z_{ia}) ^{-1} \right] * \sum_{\text{expand in }z_{ia} \ll z_{jb}}^{0 \leq \bl \leq \bk}  (-1)^{|\bl|} \frac {F'(z_{i1},...,z_{il_i}) * F''(z_{i,l_i+1},...,z_{ik_i})}{\prod^{a \leq l_i}_{1\leq i \leq n} \prod_{1\leq j \leq n}^{b > l_j} \zeta\left( z_{ia} / z_{jb}  \right)} =
$$
$$
=  \left[ \prod^{1\leq a \leq k_i}_{1\leq i \leq n} \ph^+_i(z_{ia}) ^{-1} \right] * \sum_{\text{expand in }z_{ia} \ll z_{jb}}^{0 \leq \bl \leq \bk}  (-1)^{|\bl|} \cdot \sym \ F(...,z_{i1},...,z_{ik_i},...) 
$$
where the last equality follows by the very definition of $F'$ and $F''$, and $\sym$ refers to symmetrization with respect to the groups of variables $\{z_{ia}\}_{a\leq l_i}$ and $\{z_{jb}\}_{b>l_j}$. We can package this symmetrization by rewriting the sum as:
\begin{equation}
\label{eqn:conclude}
S(F_1)F_2 =  \left[ \prod^{1\leq a \leq k_i}_{1\leq i \leq n} \ph^+_i(z_{ia}) ^{-1} \right] * \sum^{V \subset \{...,z_{ia},...\}^{1\leq i \leq n}_{1\leq a \leq k_i}}_{\text{expand in }z_{ia} \ll z_{jb} \text{ for }z_{ia}\in V, z_{jb} \notin V} (-1)^{|V|} F(...,z_{i1},...,z_{ik_i},...) \qquad
\end{equation}
where the sum goes over all subsets of the set of variables. We claim that the sum over all subsets $V$ vanishes, on account of the power $(-1)^{|V|}$ and the inclusion-exclusion principle. While expressions such as $S(F_1)F_2$ make sense in a formal completion, to ensure convergence one needs to evaluate them in certain representations where $\ph^\pm_i(z)$ act via rational functions, such as $K(\bw)$. \\
\end{proof}

\begin{proof} \emph{of Exercise \ref{ex:pairing}:} We need to prove the following formulas:
\begin{equation}
\label{eqn:for1}
\left \langle \ph_i^-(w) \otimes G, \Delta(F) \right \rangle = \left \langle G \prod^{1\leq j \leq n}_{1\leq a \leq k_j} \frac {\zeta(z_{ja}/w)}{\zeta(w/z_{ja})} \otimes \ph_i^-(w), \Delta(F) \right \rangle 
\end{equation}
\begin{equation}
\label{eqn:for2}
\left \langle \Delta^{\op}(G), \ph^+_i(w) \otimes F \right \rangle = \left \langle \Delta^{\op}(G), F \prod^{1\leq j \leq n}_{1\leq a \leq k_j} \frac {\zeta(w/z_{ja})}{\zeta(z_{ja}/w)} \otimes \ph_i^+(w)  \right \rangle
\end{equation}
for any $F \in \CS^+_{\bk}$ and $G \in \CS^-_{-\bk}$, as well as:
\begin{equation}
\label{eqn:for3}
\left \langle G * G', F \right \rangle = \left \langle G \otimes G', \Delta(F) \right \rangle \qquad \quad \forall \ F \in \CS^+_{\bk+\bl}, \ G \in \CS^-_{-\bk}, \ G' \in \CS^-_{-\bl}
\end{equation}
\begin{equation}
\label{eqn:for4}
\left \langle G, F * F' \right \rangle = \left \langle \Delta^{\op}(G), F \otimes F' \right \rangle \qquad \quad \forall \ F \in \CS^+_{\bk}, \ F' \in \CS^+_{\bl}, \ G \in \CS^-_{-\bk-\bl}
\end{equation}
We will only prove \eqref{eqn:for2} and \eqref{eqn:for4}, since the other two are analogous. As the pairing $\langle \cdot, \cdot \rangle$ only pairs non-trivially shuffle elements of opposite bidegrees, and since $\Delta^{\op}(G) = G \otimes 1 + \prod^{1 \leq j \leq n}_{1\leq a \leq k_j} \ph^-_j(z_{ja}) \otimes G + \text{intermediate terms}$, relation \eqref{eqn:for2} becomes:
$$
\left \langle \prod^{1 \leq j \leq n}_{1\leq a \leq k_j} \ph^-_j(z_{ia}),  \ph_i^+(w) \right \rangle \cdot \langle G,  F \rangle =  \left \langle G,F \cdot \frac {\zeta(w/z_{ja})}{\zeta(z_{ja}/w)} \right \rangle \cdot \langle 1, \ph_i^+(w) \rangle 
$$
Since $\langle 1, \cdot \rangle$ is just the counit of the algebra, it is enough to prove:
$$
\left \langle \prod^{1 \leq j \leq n}_{1\leq a \leq k_j} \ph^-_j(z_{ia}),  \ph_i^+(w) \right \rangle = \prod^{1\leq j \leq n}_{1\leq a \leq k_j} \frac {\zeta(w/z_{ja})} {\zeta(z_{ja}/w)}
$$ 
This follows from the fact that the coproduct of currents $\ph_j^\pm(w)$ is group-like and from \eqref{eqn:pairshuf}. In order to prove \eqref{eqn:for4}, note that its left hand side equals:
$$
\frac {1}{(\bk+\bl)!} \int^{|q| < 1 < |p|}_{|z_{ia}|=|w_{jb}|=1} \frac {G(z_{ia},w_{jb}) \cdot \sym \left[ F(z_{ia})F'(w_{jb}) \prod \zeta \left( \frac {z_{ia}}{w_{jb}} \right) \right]}{\prod \zeta_{p} \left( \frac {z_{ia}}{z_{jb}} \right) \prod \zeta_{p} \left( \frac {w_{ia}}{w_{jb}} \right) \prod \zeta_{p} \left( \frac {w_{jb}}{z_{ia}} \right) \prod \zeta_{p} \left( \frac {z_{ia}}{w_{jb}} \right) } \prod Dz_{ia} Dw_{jb} \ \Big |_{p \mapsto q}
$$
where we denote the set of variables of $F$ by $\{z_{ia}\}$ and that of $F'$ by $\{w_{jb}\}$. To keep our formulas legible, we assume that all products go over all $i,j \in \{1,...,n\}$ and over all possible indices $a,b$. Since the contours are symmetric in all the variables, we may eliminate the symmetrization, so the above equals:
$$
\frac {1}{\bk! \cdot \bl!} \int^{|q| < 1 < |p|}_{|z_{ia}|=|w_{jb}|=1} \frac {G(z_{ia},w_{jb}) F(z_{ia})F'(w_{jb}) \prod \zeta \left( \frac {z_{ia}}{w_{jb}} \right)}{\prod \zeta_{p} \left( \frac {z_{ia}}{z_{jb}} \right) \prod \zeta_{p} \left( \frac {w_{ia}}{w_{jb}} \right) \prod \zeta_{p} \left( \frac {w_{jb}}{z_{ia}} \right) \prod \zeta_{p} \left( \frac {z_{ia}}{w_{jb}} \right) } \prod Dz_{ia} Dw_{jb} \ \Big |_{p \mapsto q}
$$
The above integrand has the following poles that involve both $z$'s and $w$'s:
$$
z_{ia} = qt w_{i-1,b} \qquad z_{ia} = qt^{-1} w_{i+1,b} \qquad z_{ia} = p^{-2} w_{ia} \qquad z_{ia} = p^{2} w_{ia}
$$
and a factor of $z_{ia} = q^2w_{ia}$ in the numerator. Because of this factor, any residues obtained at the simple pole $z_{ia} = p^{2} w_{ia}$ will vanish upon setting $p \mapsto q$. As for the other poles, we encounter none of them as we move the contours to ensure $|z_{ia}| \gg |w_{jb}|$ for all possible indices, because we assumed that $|q|<1<|p|$. We conclude that:
\begin{equation}
\label{eqn:lhs}
\text{LHS of \eqref{eqn:for4} } \ = \ \frac {1}{\bk! \cdot \bl!} \int^{|q| < 1 < |p|}_{|z_{ia}| \gg |w_{jb}|} \frac {G(z_{ia},w_{jb}) F(z_{ia})F'(w_{jb}) \prod Dz_{ia} Dw_{jb}}{\prod \zeta_{p}\left( \frac {z_{ia}}{z_{jb}} \right) \prod \zeta_{p}\left( \frac {w_{ia}}{w_{jb}} \right) \prod \zeta \left( \frac {w_{ia}}{z_{jb}} \right)} \ \Big |_{p \mapsto q}
\end{equation}
The reason why we were able to replace $\zeta_{p}$ by $\zeta$ in the denominator is that they coincide in the limit $|z_{ia}| \gg |w_{jb}|$, after one sets $p \mapsto q$. Meanwhile, to compute the right hand side of \eqref{eqn:for4}, we use the definition of the coproduct in \eqref{eqn:cop2}:
$$
= \langle \Delta(G), F' \otimes F \rangle = \frac {1}{\bk! \cdot \bl!} \int^{|q| < 1 < |p|}_{|w_{jb}| \ll  |z_{ia}|} \frac {G(w_{jb},z_{ia}) F'(w_{jb})F(z_{ia}) \prod Dz_{ia} Dw_{jb}}{\prod \zeta_{p} \left( \frac {z_{ia}}{z_{jb}} \right) \prod \zeta_{p} \left( \frac {w_{ia}}{w_{jb}} \right) \prod \zeta \left( \frac {w_{ia}}{z_{jb}} \right)} \ \Big |_{p \mapsto q}
$$
Since $G$ is symmetric, this matches with \eqref{eqn:lhs} and so the proof is complete. \\
\end{proof}

\begin{proof} \emph{of Exercise \ref{ex:shap}:} We will first prove that formula \eqref{eqn:shapshuffle} is equivalent to \eqref{eqn:shapshufflehopf}. The proof will follow closely that of Exercise \ref{ex:antipode}. Specifically, we have:
$$
f_1 \otimes f_2 = \Delta(f) = \sum_{\bv = \bv^1+\bv^2} \frac {\ph(Z^2)f(Z^1 \otimes Z^2)}{\zeta \left( \frac {Z^2}{Z^1} \right)}
$$
where the sum goes over all ways to partition the degree vector $\bv$ of $f$ into two parts, $\bv^1 = \deg f_1$ and $\bv^2 = \deg f_2$, and $Z$, $Z^1$, $Z^2$ are place-holders for the variables of $f$, $f_1$, $f_2$, respectively. The above sum is expanded in the range $|Z^1| \ll |Z^2|$. Recalling the definition of the antipode $S$ from \eqref{eqn:antshuf1}, we have:
\begin{equation}
\label{eqn:kraft}
f_1 * S(f_2) = \sum_{\bv = \bv^1+\bv^2} (-1)^{|Z^2|} \frac {f(Z^1 * Z^2)}{\zeta \left( \frac {Z^1}{Z^2} \right)} = \sum_{Z = Z^1 + Z^2} (-1)^{|Z^2|} f(Z)
\end{equation}
The last equality is due to the definition of the shuffle product, and it involves symmetrizing over all variables involved. Therefore, in the right hand side we must sum over all partitions of the set of variables $Z$ into two sets $Z^1$ and $Z^2$. Just like in the proof of Exercise \ref{ex:antipode}, relation \eqref{eqn:kraft} pairs trivially with any shuffle element $f'$. However, this is no longer true if we divide it by the polynomial with coefficient among the $u_i$'s that appears in \eqref{eqn:shapshufflehopf}:
$$
\left \langle \frac {\left(f_1 * S (f_2)\right)^T}{\prod_{i=1}^\bw \left[\frac {Z^1_i}{qu_i} \right]\left[\frac {Z^2_i}{qu_i} \right]}, \frac {f'}{\prod_{i=1}^\bw \left[\frac {u_i}{qZ_i} \right]} \right \rangle = \sum_{Z = Z^1 + Z^2} (-1)^{|Z^2|} \int_{|Z^1| \ll 1 = |Z^2|}  \frac {f(Z) f'(Z) \  DZ}{\zeta\left(\frac ZZ \right) \prod_{i=1}^\bw \left[\frac {Z_i}{qu_i} \right] \left[\frac {u_i}{qZ_i} \right]}
$$
where in the above equality we used formula \eqref{eqn:daddypair} for the Hopf pairing. Note that it was no longer necessary to invoke $\zeta_p$ instead of $\zeta$, because $f$ and $f'$ are assumed to be Laurent polynomials and thus do not have any denominators at $qz_{ia} - q^{-1}z_{ib}$. The integrand does not depend on the partition $Z = Z^1+Z^2$, but the alternating sum over all partitions has the effect of removing the poles at $\infty$. Hence we obtain \eqref{eqn:shapshuffle}. \\	
\end{proof}


\begin{proof} \emph{of Exercise \ref{ex:stackhall}:} Given rational functions $F$ and $F'$, we think of:
\begin{equation}
\label{eqn:class}
F \otimes F' = F(...,z_{ia},...) F'(...,z_{jb},...)
\end{equation}
as a representation of $G_{\bk} \times G_{\bl}$, namely an element in the $K-$theory of the stack $\fC_{\bk} \times \fC_{\bl}$ which is pulled back from a point. The pull-back $\pi_2^*(F \otimes F')$ to the $K-$theory of the stack $\fP_{\bk,\bl}$ is also given by the class \eqref{eqn:class}. However, before pushing this class forward to $\fC_{\bk+\bl}$, we need to understand the map $\pi_1$. Explicitly, we have the diagram:
\begin{equation}
\label{eqn:dia}
\xymatrix{
\nu^{-1}(0) \ar@{^{(}->}[r] \ar@{^{(}->}[d]^-\alpha & (X_i,Y_i)^{\text{preserve}}_{1\leq i \leq n} \ar[r]^-{\nu} \ar@{^{(}->}[d]^-\beta & \text{End}(\BC^{\bk+\bl})^{\text{preserve}} \ar@{^{(}->}[d]^-\gamma \\
\mu^{-1}(0) \ar@{^{(}->}[r] & (X_i,Y_i)_{1\leq i \leq n} \ar[r]^-{\mu} & \text{End}(\BC^{\bk+\bl})} 
\end{equation}
The bottom right entry is the affine space of all endomorphisms of the collection of vector spaces $\{\BC^{k_i+l_i}\}_{1\leq i \leq n}$, while the bottom middle entry is the affine space consisting of all collections of linear maps $(X_i,Y_i)^{1\leq i \leq n}$ between vector spaces:
$$
\xymatrix{\BC^{k_i+l_i} \ar@/^/[r]^{X_i} & \BC^{k_{i+1}+l_{i+1}} \ar@/^/[l]^{Y_i}}
$$
The two entries directly above them refer to those collections which preserve the collection of subspaces $\BC^{\bk} \subset \BC^{\bk+\bl}$. Therefore, we conclude that:
$$
\widetilde{\beta}_*(1) = \prod_{i=1}^n \left[\frac 1{qt} \cdot \Hom(\BC^{k_i},\BC^{l_{i+1}})\right] \left[\frac tq \cdot \Hom(\BC^{k_{i+1}},\BC^{l_i})\right]
$$
where the equivariant parameters $qt$ and $qt^{-1}$ are necessary as they scale the linear maps $X$ and $Y$, respectively. As representations of $G_\bk \times G_\bl$, the above equals:
$$
\widetilde{\beta}_*(1) = \prod_{i=1}^n \left( \prod^{1\leq a \leq k_i}_{k_{i+1} < b \leq k_{i+1}+l_{i+1}} \left[\frac {z_{i+1,b}}{qtz_{ia}} \right] \prod^{1\leq a \leq k_i}_{k_{i-1} < b \leq k_{i-1}+l_{i-1}} \left[\frac {tz_{i-1,b}}{qz_{ia}} \right] \right)
$$
The same computation establishes the push-forward under $\alpha$, with the caveat that we imposed too many relations. More specifically, we do not need to impose those relations that lie in the cokernel of $\gamma$, so we obtain:
$$
\widetilde{\alpha}_*(1) = \prod_{i=1}^n \frac {\prod^{1\leq a \leq k_i}_{k_{i+1} < b \leq k_{i+1}+l_{i+1}} \left[\frac {z_{i+1,b}}{qtz_{ia}} \right] \prod^{1\leq a \leq k_i}_{k_{i-1} < b \leq k_{i-1}+l_{i-1}} \left[\frac {tz_{i-1,b}}{qz_{ia}} \right]} {\prod^{1\leq a \leq k_i}_{k_{i} < b \leq k_{i}+l_{i}} \left[\frac {z_{ib}}{q^2z_{ia}} \right]}
$$
Since the elements \eqref{eqn:class} are equivariant constants, we obtain a similar formula:
$$
\widetilde{\alpha}_*(F \otimes F') = F(...,z_{ia},...) F'(...,z_{jb},...) \cdot
$$
$$
\prod_{i=1}^n \frac {\prod^{1\leq a \leq k_i}_{k_{i+1} < b \leq k_{i+1}+l_{i+1}} \left[\frac {z_{i+1,b}}{qtz_{ia}} \right] \prod^{1\leq a \leq k_i}_{k_{i-1} < b \leq k_{i-1}+l_{i-1}} \left[\frac {tz_{i-1,b}}{qz_{ia}} \right]} {\prod^{1\leq a \leq k_i}_{k_{i} < b \leq k_{i}+l_{i}} \left[\frac {z_{ib}}{q^2z_{ia}} \right]}
$$
The same formula holds if we interpret $\alpha$ as the map:
$$
\fP_{\bk,\bl} = \nu^{-1}(0)/P_{\bk,\bl} \ \stackrel{\alpha}\longrightarrow \ \Gamma = \mu^{-1}(0)/P_{\bk,\bl}
$$
In order to obtain the required result, we need to compose $\alpha$ above with the map: 
$$
\Gamma = \mu^{-1}(0)/P_{\bk,\bl} \ \stackrel{\tau}\longrightarrow \ \fC_{\bk+\bl} = \mu^{-1}(0)/G_{\bk+\bl}
$$
Since push-forward under $\tau$ is induction from $P_{\bk,\bl}$ to $G_{\bk+\bl}$, we may use \eqref{eqn:induction} to conclude that $\widetilde{\tau \circ \alpha}_*(f \boxtimes g)$ equals:
$$
\sym \left( \frac {F(...,z_{ia},...) F'(...,z_{jb},...)}{\bk! \cdot \bl!} \prod_{i=1}^n \frac {\prod^{1\leq a \leq k_i}_{k_{i+1} < b \leq k_{i+1}+l_{i+1}} \left[\frac {z_{i+1,b}}{qtz_{ia}} \right] \prod^{1\leq a \leq k_i}_{k_{i-1} < b \leq k_{i-1}+l_{i-1}} \left[\frac {tz_{i-1,b}}{qz_{ia}} \right]} {\prod^{1\leq a \leq k_i}_{k_{i} < b \leq k_{i}+l_{i}} \left[\frac {z_{ib}}{q^2z_{ia}} \right] \prod^{1\leq a \leq k_i}_{k_{i} < b \leq k_{i}+l_{i}} \left[\frac {z_{ib}}{z_{ia}} \right]} \right)
$$
Since $\widetilde{\tau \circ \alpha}_*(F \otimes F') = F * F'$ by definition, we conclude \eqref{eqn:mult}. \\
\end{proof}

\begin{proof} \emph{of Exercise \ref{ex:e}:} For any $\bw-$partition $\bla^-$ of size $\bv^-$, we have:
$$
\fA\left(|\bla^-\rangle \right) = \widetilde{\pi}_{1*}\Big( \sum_{\bla^+ \in \CN_{\bv^+,\bw}^{\text{fixed}}} [\CE|_{\bla^+,\bla^-}] \cdot  |(\bla^+,\bla^-)\rangle \Big) = \sum_{\bla^+} |\bla^+\rangle \cdot \frac {[\CE|_{\bla^+,\bla^-}]}{[T_{\bla^-} \CN_{\bv^-,\bw}]}  
$$
We may use \eqref{eqn:league} to express the $[\cdot]$ class of the vector bundle $\CE$, and \eqref{eqn:tanmod0} to express the $[\cdot]$ class of the tangent bundle to Nakajima quiver varieties. We obtain:
$$
\langle \bla^+ | \fA |\bla^-\rangle = \frac { \prod_{i=1}^n \frac {\Big[ \frac {\CV^-_{i+1}}{qt \CV^+_i} \Big] \Big[ \frac {t\CV^-_{i-1}}{q\CV^+_i} \Big] }{\Big[ \frac {\CV^-_i}{\CV^+_i} \Big] \Big[ \frac {\CV^-_i}{q^2\CV^+_i} \Big]} \prod_{j=1}^{\bw} \Big[ \frac {\CV^-_{j}}{qu_j} \Big] \Big[\frac {u_j}{q \CV^+_{j}} \Big]}{ \prod_{i=1}^n \frac {\Big[ \frac {\CV^-_{i+1}}{qt \CV^-_i} \Big] \Big[ \frac {t\CV^-_{i-1}}{q\CV^-_i} \Big] }{\Big[ \frac {\CV^-_i}{\CV^-_i} \Big] \Big[ \frac {\CV^-_i}{q^2\CV^-_i} \Big]} \prod_{j=1}^{\bw} \Big[ \frac {\CV^-_{j}}{qu_j} \Big] \Big[\frac {u_j}{q \CV^-_{j}} \Big]} \quad \Big|_{\{\CV^\pm_i\}_{1 \leq i \leq n} \mapsto \chi_{\bla^\pm}} \quad = 
$$
$$
= \prod_{\bsq \in \cray} \left[\prod_{\sq \in \bla^-} \frac {\Big[ \frac {\chi_\sq}{qt \chi_\bsq} \Big]^{\delta_{c_{\bsq}+1}^{c_\sq}} \Big[ \frac {t\chi_\sq}{q\chi_\bsq} \Big]^{\delta_{c_{\bsq}-1}^{c_\sq}} }{\Big[ \frac {\chi_\sq}{\chi_\bsq} \Big]^{\delta_{c_\bsq}^{c_\sq}} \Big[ \frac {\chi_\sq}{q^2\chi_\bsq} \Big]^{\delta_{c_\bsq}^{c_\sq}}} \prod_{1\leq j \leq \bw}^{u_j \equiv c_\bsq} \Big[\frac {u_j}{q\chi_\bsq} \Big] \right] = 
$$
$$
= \prod_{\bsq \in \cray} \left[\prod_{\sq \in \bla^-} \zeta\left(\frac {\chi_\bsq}{\chi_\sq}\right) \prod_{1\leq j \leq \bw}^{u_j \equiv c_\bsq} \Big[\frac {u_j}{q\chi_\bsq} \Big] \right]
$$
Comparing the above with \eqref{eqn:fixedplus} for $F = 1$, we obtain the desired conclusion. \\
\end{proof}

\begin{proof} \emph{of Exercise \ref{ex:eccentric}:} Let us treat the case of positive shuffle elements, i.e. when the sign is $+$, since the opposite case is analogous. We start from the fact that fixed points of $\fZ_{[i;j)}^{q_1}$ are given by positive almost standard Young tableaux $(\bla^+ \geq \bla^-,\psi)$. Let us recall that the datum $\psi$ is a labeling $\sq_i, ..., \sq_{j-1}$ of the boxes of $\cray$, where the labels increase as we go up and to the right, with the exception that $\sq_{a-1}$ is allowed to be located directly above $\sq_{a}$. Therefore, \eqref{eqn:formula} implies:
$$
P^{(d)}_{[i;j)} \left(|\bla^- \rangle \right) = \sum_{(\bla^+\geq \bla^-,\psi)} |\bla^+\rangle \cdot \chi_i^{d_i}... \chi_{j-1}^{d_{j-1}} \cdot \frac {[T_{\bla^+}\CN_{\bv^+,\bw}]}{[T_{(\bla^+ \geq \bla^-,\psi)} \fZ_{[i;j)}^{q_1}]}
$$ 
where $\chi_a$ denotes the weight of the box labeled $a$ in an $\asyt$ $\psi$. We may use \eqref{eqn:tanmod0} and \eqref{eqn:tywin} to evaluate the tangent spaces in question:
$$
P^{(d)}_{[i;j)} \left(|\bla^- \rangle \right) = \sum_{\bla^+ \geq \bla^-} \frac {|\bla^+\rangle}{[q^{-2}]}  \sum^\psi_{\asyt^+} \chi_i^{d_i}... \chi_{j-1}^{d_{j-1}} \cdot \frac {\prod_{i \leq a < b + 1 < j + 1}^{a \equiv b+1} \Big[\frac {\chi_{a}}{qt\chi_{b}} \Big]  \prod_{i \leq a < b - 1 < j - 1}^{a \equiv b-1} \Big[ \frac {t\chi_{a}}{q\chi_{b}} \Big]}{\prod_{i \leq a < b < j}^{a \equiv b} \Big[ \frac {\chi_{a}}{\chi_{b}} \Big] \prod_{i \leq a < b < j}^{a\equiv b} \Big[ \frac {\chi_{a}}{q^2\chi_{b}} \Big]} \cdot
$$
$$
\cdot \prod_{a=i}^{j-1} \left( \prod_{\sq \in \bla^-} \frac {\Big[\frac {\chi_\sq}{qt\chi_a} \Big]^{\delta^{c_\sq}_{c_a+1}}\Big[\frac {t\chi_\sq}{q\chi_a} \Big]^{\delta^{c_\sq}_{c_a-1}}}{\Big[\frac {\chi_\sq}{\chi_a} \Big]^{\delta^{c_\sq}_{c_a}}\Big[\frac {\chi_\sq}{q^2\chi_a} \Big]^{\delta^{c_\sq}_{c_a}}} \prod_{1\leq k \leq \bw}^{u_k \equiv a} \Big[ \frac {u_k}{q\chi_a} \Big] \right)
$$
where in each summand, we write $\chi_a$ for the weight of the box labeled by $a$ in the $\asyt^+$ given by $\psi$. Recalling the definition of $\zeta$ in \eqref{eqn:defzeta}, we obtain:
$$
\langle \bla^+ | P^{(d)}_{[i;j)} |\bla^- \rangle = \frac 1{[q^{-2}]} \sum^\psi_{\asyt^+} \chi_i^{d_i}... \chi_{j-1}^{d_{j-1}} \cdot
$$
$$
\frac {\prod_{i\leq a < b \leq j} \zeta\left(\frac {\chi_b}{\chi_a} \right)}{\prod_{a=i+1}^{j-1} \Big[\frac {t\chi_{a-1}}{q\chi_a}\Big]} \prod_{a=i}^{j-1} \left( \prod_{\sq \in \bla^-} \zeta \left(\frac {\chi_a}{\chi_\sq} \right) \prod_{1\leq k \leq \bw}^{u_k \equiv a} \Big[ \frac {u_k}{q\chi_a} \Big] \right)
$$
In the right hand side above, we claim that:
$$
\sum^\psi_{\asyt^+} \chi_i^{d_i}... \chi_{j-1}^{d_{j-1}} \frac {\prod_{i\leq a < b \leq j} \zeta\left(\frac {\chi_b}{\chi_a} \right)}{\prod_{a=i+1}^{j-1} \Big[\frac {t\chi_{a-1}}{q\chi_a}\Big]} = \sym \left[z_i^{d_i}... z_{j-1}^{d_{j-1}} \frac {\prod_{i\leq a < b \leq j} \zeta\left(\frac {z_b}{z_a} \right)}{\prod_{a=i+1}^{j-1} \Big[\frac {tz_{a-1}}{qz_a}\Big]} \right] \Big|_{z_i,...,z_{j-1} \mapsto \chi_\cray}
$$
Indeed, the symmetrization is a sum over all ways to labeling the boxes of the skew tableau $\cray$, but as we saw in the proof of Proposition \ref{prop:act}, the corresponding summand is zero unless the labeling is an $\asyt^+$. Then \eqref{eqn:fixedplus} implies the Exercise. \\	
\end{proof}

\begin{proof} \emph{of Exercise \ref{ex:gen}:} We will prove the required result for the rational function $P_{[i;j)}^{(d)}$, since the other cases are analogous. We have:
$$
P_{[i;j)}^{(d)} = \frac {r(z_i,...,z_{j-1})}{\prod_{i \leq a\equiv b < j} \Big[\frac {z_b}{q^2z_a} \Big]}
$$
where:
\begin{equation}
\label{eqn:rr}
r = \textrm{Sym} \left[ z_i^{d_i}...z_{j-1}^{d_{j-1}} \prod_{a=i+1}^{j-1} \frac {[q^{-2}]}{\Big[\frac {t z_{a-1}}{qz_{a}} \Big]} \prod_{i\leq a < b < j} \frac {\Big[\frac {z_a}{qtz_b} \Big]^{\delta_{b+1}^a}\Big[\frac {tz_a}{qz_b} \Big]^{\delta_{b-1}^a}\Big[\frac {z_b}{q^2z_a} \Big]^{\delta_{a}^b}}{\Big[\frac {z_a}{z_b} \Big]^{\delta^{a}_b}} \right] 
\end{equation}
It is clear that $r$ is a symmetric Laurent polynomial, since the denominator of the first product divides the numerator of the second product. Meanwhile, the denominator of the second product has poles of order $\leq 1$ at $z_a-z_b$ when $a \equiv b$, and these poles will be eliminated by the symmetrization. To show that $r$ satisfies the wheel conditions, we must specialize three of the variables to: 
\begin{equation}
\label{eqn:specialization}
z_a = q^{-1}, z_b = t^{\pm 1}, z_c = q \qquad \text{for some} \quad a \equiv c \equiv b \mp 1
\end{equation}
and show that $r$ vanishes. Let us take care of the case when the above sign is $+$, since the other case is analogous. Then we will show not only that $r$ vanishes, but that each summand of the symmetrization \eqref{eqn:rr} vanishes. Indeed, 
$$
\Big[\frac {z_a}{qtz_b} \Big] = 0 \qquad \text{unless} \quad a < b
$$ 
$$
\frac {\Big[\frac {tz_c}{qz_b} \Big]}{ \Big[ \frac {t z_i}{q z_{i+1}} \Big] ... \Big[\frac {tz_{j-2}}{qz_{j-1}} \Big]} = 0 \qquad \text{unless} \quad b < c \quad \text{or} \quad b = c+1
$$
$$
\Big[\frac {z_c}{q^2z_a} \Big] = 0 \qquad \text{unless} \quad c < a
$$
Since the three inequalities on $a,b,c$ cannot hold simultaneously, we conclude that each summand of \eqref{eqn:rr} vanishes at the specialization \eqref{eqn:specialization}, and thus $P_{[i;j)}^{(d)} \in \CS^+$. Let us now show that $P_{[i;j)}^{(d)}$ lies in the image of $\Upsilon$ for any vector of integers $d = (d_i,...,d_{j-1})$. We will proceed by induction of $j-i$, and consider the vector space $I$ of Laurent polynomials $r(z_i,...,z_{j-1})$ such that:
$$
\sym \left[ \frac {r(z_i,...,z_{j-1})}{\prod_{a=i+1}^{j-1} \Big[\frac {t z_{a-1}}{qz_{a}} \Big]} \prod_{i\leq a < b < j} \frac {\Big[\frac {z_a}{qtz_b} \Big]^{\delta_{b+1}^a}\Big[\frac {tz_a}{qz_b} \Big]^{\delta_{b-1}^a}}{\Big[\frac {z_a}{z_b} \Big]^{\delta^{a}_b} \Big[\frac {z_a}{q^2z_b} \Big]^{\delta^{a}_b} } \right] \ \in \ \text{Im } \Upsilon
$$
The induction hypothesis implies that if $tz_{a-1}-qz_a$ divides $r$, then $r\in I$. To establish the fact that $I$ consists of all Laurent polynomials, we therefore need only find a single Laurent polynomial $r \in I$ such that $r(q_2^{-i},...,q_2^{-j+1}) \neq 0$. Such an example is provided by a linear combination of the shuffle elements:
$$
\text{Im } \Upsilon \ \ni \ z_i^{e_i} * ... * z_{j-1}^{e_{j-1}} \ = \ \sym \left( z_i^{e_i}...z_{j-1}^{e_{j-1}} \prod_{i \leq a < b < j} \frac {\Big[\frac {z_b}{qtz_a} \Big]^{\delta_{a+1}^b}\Big[\frac {tz_b}{qz_a} \Big]^{\delta_{a-1}^b}}{\Big[\frac {z_b}{z_a} \Big]^{\delta^{b}_a} \Big[\frac {z_b}{q^2z_a} \Big]^{\delta^{b}_a} }  \right)
$$
Indeed, an appropriate linear combination of the monomials $z_i^{e_i}...z_{j-1}^{e_{j-1}}$ ensures that:
$$
\text{Im } \Upsilon \ \ni \ \sym \left[ \frac 1{\prod_{a=i+1}^{j-1} \Big[\frac {t z_{a-1}}{qz_{a}} \Big]} \prod_{i \leq a < b < j} \frac {\Big[\frac {z_a}{qtz_b} \Big]^{\delta_{b+1}^a} \Big[\frac {tz_a}{qz_b} \Big]^{\delta_{b-1}^a} \Big[\frac {z_a}{q^2z_b} \Big]^{\delta^{a}_b}  \Big[\frac {z_b}{qtz_a} \Big]^{\delta_{a+1}^b}   \Big[\frac {tz_b}{qz_a} \Big]^{\delta_{a-1}^b}}{\Big[\frac {z_b}{z_a} \Big]^{\delta^{b}_a}  \Big[\frac {z_a}{q^2z_b} \Big]^{\delta^{a}_b}  } \right]
$$
Indeed, the first denominator is not an issue, since it is canceled by the second factor in the numerator. The above factors can be reshuffled around so that:
$$
\text{Im } \Upsilon \ \ni \ \sym \left[ \frac {\prod_{i \leq a < b < j} \Big[\frac {z_b}{qtz_a} \Big]^{\delta_{a+1}^b}   \Big[\frac {tz_b}{qz_a} \Big]^{\delta_{a-1}^b}  \Big[\frac {z_a}{q^2z_b} \Big]^{\delta^{a}_b}  }{\prod_{a=i+1}^{j-1} \Big[\frac {t z_{a-1}}{qz_{a}} \Big]} \prod_{i \leq a < b < j} \frac {\Big[\frac {z_a}{qtz_b} \Big]^{\delta_{b+1}^a} \Big[\frac {tz_a}{qz_b} \Big]^{\delta_{b-1}^a}}{\Big[\frac {z_a}{z_b} \Big]^{\delta^{b}_a}  \Big[\frac {z_a}{q^2z_b} \Big]^{\delta^{a}_b}  } \right]
$$
Since the numerator of the first fraction does not vanish when we set $z_i = q_2^{-i}$,..., $z_{j-1} = q_2^{-j+1}$, the proof is complete. \\
\end{proof}


\begin{proof} \emph{of Exercise \ref{ex:hopf}:} To prove that $\langle \cdot , \cdot \rangle$ extends to a well-defined Hopf pairing, we need to show that it preserves the relations \eqref{eqn:rtt1} and \eqref{eqn:rtt2}. Specifically, we need to prove that:
$$
\left \langle T_1^-(z),  T_3^+(y) T_2^+(w) R_{23} \left(\frac wy\right) \right \rangle = \left \langle T_1^-(z), R_{23} \left(\frac wy\right) T_2^+(w) T_3^+(y) \right \rangle
$$	
and:
$$
\left \langle R_{12}\left(\frac zw\right) T_1^-(z) T_2^-(w), T_3^+(y) \right \rangle = \left \langle T_2^-(w) T_1^-(z) R_{12}\left(\frac zw\right) , T_3^+(y) \right \rangle
$$
where we write $X_1 = X \otimes \text{Id} \otimes \text{Id}$, $X_2 = \text{Id} \otimes X \otimes \text{Id}$ and $X_3 = \text{Id} \otimes \text{Id} \otimes X$ for any matrix $X$. In the interest of space, let us only prove the first formula. Since the central charge $c$ pairs trivially with anything, the bialgebra property \eqref{eqn:bialg} gives us:
$$
\left \langle T_1^-(z),  T_3^+(y) T_2^+(w) R_{23} \left(\frac wy\right) \right \rangle = \left \langle  T_1^-(zc_2) \otimes T_1^-(z) ,  T_2^+(w) \otimes T_3^+(y)  R_{23} \left(\frac wy\right) \right \rangle =
$$
$$
= R_{12} \left(\frac zw \right) R_{13} \left(\frac zy \right) R_{23} \left(\frac wy\right)  = R_{23} \left(\frac wy\right) R_{13} \left(\frac zy \right) R_{12} \left(\frac zw \right) = 
$$
$$
= \left \langle T_1^-(zc_2) \otimes T_1^-(z), R_{23} \left(\frac wy\right)  T_3^+(y) \otimes T_2^+(w) \right \rangle = \left \langle T_1^-(z), R_{23} \left(\frac wy\right) T_2^+(w) T_3^+(y) \right \rangle
$$
where the middle equality is the quantum Yang-Baxter equation \eqref{eqn:qybe}. In any Drinfeld double, the positive and negative halves interact by relation \eqref{eqn:reldrinfeld}, which for $a = T_2^-(w)$ and $b = T_1^+(z)$ implies:
$$
\left \langle T_1^+(z), T_2^-(wc_2) \right \rangle T_2^-(w) T_1^+(zc_1)  = T_1^+(z) T_2^-(wc_2) \left \langle T_1^+(zc_1), T_2^-(w) \right \rangle
$$
Replacing the pairing by $R_{12}$, we obtain precisely relation \eqref{eqn:rtt3}, as expected. \\
\end{proof}

\begin{proof} \emph{of Exercise \ref{ex:mich}:} Let us prove the required result for $P_{[i;j)}^\bm$, since the cases of $P_{-[i;j)}^\bm$, $Q^\bm_{[i;j)}$ and $Q^\bm_{-[i;j)}$ are analogous. To prove the slope condition, we need to send the subset $\{z_a\}_{a\in A}$ of variables to infinity, for any $A \subset \{i,...,j\}$, and compute the degree of the resulting expression:
$$
\text{degree in } \{z_a\}_{a\in A} = \sum_{a\in A} \Big( \lfloor m_i+...+m_a \rfloor - \lfloor m_i+...+m_{a-1} \rfloor \Big) - 
$$
$$
- \#\{a\in A \text{ such that } a - 1 \notin A\} 
$$
From \eqref{eqn:limit1}, we see that $P^\bm_{[i;j)}$ has slope $\bm$ if for all such subsets $A$, we have:
$$
\sum_{a\in A} \Big( \lfloor m_i+...+m_{a} \rfloor - \lfloor m_i+...+m_{a-1} \rfloor \Big) - \#\{a\in A \text{ s.t. } a - 1 \notin A\} \leq \sum_{a\in A} m_a  \Leftrightarrow
$$
\begin{equation}
\label{eqn:ineq}
\Leftrightarrow \sum_{a\in A} \Big (\{ m_i+...+m_{a-1} \} - \{ m_i+...+m_{a} \} \Big) \leq \#\{a\in A \text{ s.t. } a - 1 \notin A\} 
\end{equation}
where $\{x\} = x - \lfloor x \rfloor$ denotes the fractional part of $x$. Proving \eqref{eqn:ineq} is an elementary exercise, but let us trace through it, because it will be important for us when equality is attained. Let us assume that $A$ is divided into subsets of consecutive integers as:
$$
A = \{a_1+1,...,b_1\} \sqcup \{a_2+1,...,b_2\} \sqcup ... \{a_t+1,...,b_t\}
$$
for some $i-1\leq a_1<b_1<...<a_t<b_t\leq j$. Then the inequality \eqref{eqn:ineq} becomes:
$$
\sum_{c=1}^t \left(\{ m_i+...+m_{a_c} \} - \{ m_i+...+m_{b_c} \} \right) \leq t - \delta_{a_1}^{i-1}
$$
This inequality holds term-wise, meaning that it holds for each individual $c$, as:
$$
\{\alpha\} - \{\beta\} < 1 \quad \forall \ \alpha,\beta \in \BR \ \text{ proves the case }c \in \{2,...,t\}
$$
$$
0 - \{\beta\} \ \leq \ 0 \qquad \qquad \forall \ \ \beta \in \BR \quad \text{ proves the case } c = 1	
$$	
Equality can only hold if $t=1$, in other words when $A=\{i,...,a-1\}$ consists of consecutive integers starting at $i$, such that $m_{i}+...+m_{a-1} = \bm \cdot [i;a) \in \BZ$. Collecting all the terms of top degree gives us \eqref{eqn:copy}. \\
\end{proof}

\begin{proof} \emph{of Exercise \ref{ex:pseudo}:} We will prove the computation of $\eta_{[i;j)}^{q_1}(F*F')$, since that of $\eta_{-[i;j)}^{q_2}(G*G')$ is proved analogously. Let us write $\deg F = \bk$ and $\deg F' = \bl$, where $\bk+\bl = [i;j)$, and we will think of this as a partition of the set of variables:
$$
\{z_i,...,z_{j-1}\} = X \sqcup Y \qquad \text{where} \quad |X| = \bk \qquad |Y| = \bl
$$	
By applying the definition of the shuffle product in \eqref{eqn:shuffle}, we have:
\begin{equation}
\label{eqn:model}
(F * F')(z_i,....,z_j) = \sum_{\sigma} \left[ \frac {F\left(z_{\sigma(x)}\right)_{x\in X} F'\left(z_{\sigma(y)}\right)_{y\in Y}}{\bk! \cdot \bl!} \prod^{x\in X}_{y\in Y} \zeta \left(\frac {z_{\sigma(x)}}{z_{\sigma(y)}} \right) \right]
\end{equation}
where the sum is over all permutations $\sigma$ which only permute variables whose indices are congruent modulo $n$. Applying $\eta^{q_1}_{[i;j)}$ to the above entails specializing the variables at $z_x = q_1^x$. Since $\zeta(q_1^{-1}) = 0$, the only summands which do not vanish are the ones for which:
$$
\sigma(x) - \sigma(y) \neq - 1 \qquad \forall \ x \in X \ \ y \in Y
$$
This implies that $\sigma(x) > \sigma(y)$ for all $x \in X$ and $y \in Y$, which in turn implies that $\bk = [a;j)$ and $\bl = [i;a)$ for some $a\in \{i,...,j\}$. Moreover, the only summands of \eqref{eqn:model} which survive the specialization $z_x = q_1^x$ are the ones where $\sigma$ permutes $[i;a)$ and $[a;j)$ independently of each other. Therefore, $(F * F')(q_1^i,....,q_1^{j-1})$ equals:
$$
\sum_{\sigma' \in S(\{i,...,a-1\})}^{\sigma \in S(\{a,...,j-1\})} \left[ \frac {F\left(q_1^{\sigma(a)},...,q_1^{\sigma(j-1)} \right) F'\left(q_1^{\sigma(i)},...,q_1^{\sigma(a-1)}\right)}{[i;a)! \cdot [a;j)!}  \prod^{x\in [a;j)}_{y\in [i;a)} \zeta (q_1^{x-y}) \right]
$$
Since $F$ and $F'$ are symmetric, all summands in the above expression are equal to each other, hence $\eta^{q_1}_{[i;j)} (F*F') = \eta^{q_1}_{[a;j)}(F) \eta^{q_1}_{[i;a)}(F')$. As for \eqref{eqn:normp1}, $ \quad P^{\bm}_{[i;j)}(q_1^{i},...,q_1^{j-1})=$
$$
= \sum_{\sigma : \{i,...,j-1\} \stackrel{\cong}\rightarrow \{i,...,j-1\}}^{\sigma(a) \equiv a \ \forall a} \frac {\prod_{a=i}^{j-1} q_1^{\sigma(a) \left( \lfloor m_i + ... + m_a \rfloor - \lfloor m_i + ... + m_{a-1} \rfloor \right)}}{t^{\ind_{[i;j)}^\bm} q^{i-j} \prod_{a=i+1}^{j-1} \left(1 - q_1^{\sigma(a) - \sigma(a-1)}q_2 \right)} \prod_{i \leq a < b < j} \zeta \left( q_1^{\sigma(b)-\sigma(a)} \right) 
$$
Since $\zeta(q_1^{-1}) = 0$, the only permutation for which the above does not vanish is the identity permutation $\sigma(a) = a$, hence \eqref{eqn:normp1}. Formula \eqref{eqn:normp2} is proved analogously. \\
\end{proof}

\begin{proof} \emph{of Exercise \ref{ex:comppair}:} Let us first prove \eqref{eqn:functional1}, by applying relation \eqref{eqn:daddypair}:
$$
\langle Q_{-[i;j)}^\bm, F \rangle = \int_{|z_a|=1}^{|q| < 1 < |p|} \frac {F(z_i,...,z_{j-1}) \prod_{a = i}^{j-1} z_a^{\lfloor m_i+...+m_{a-1} \rfloor - \lfloor m_i + ... + m_{a} \rfloor} \prod_{a < b} \zeta\left(\frac {z_a}{z_b}\right)}{t^{-\ind_{[i;j)}^\bm} \left(1 - \frac {q_1 z_{i}}{z_{i+1}} \right) ... \left(1 - \frac {q_1 z_{j-1}}{z_{j}} \right)\prod_{a \leq b} \zeta_p\left(\frac {z_b}{z_a}\right)\prod_{a < b} \zeta_p\left(\frac {z_a}{z_b}\right)} \Big |_{p \mapsto q}
$$
for all $F\in \CS^+$. The integrand has the following poles involving $z_a$ and $z_b$ for $a<b$:
$$
z_a = p^{2} z_b, \qquad z_a = p^{-2} z_b, \qquad z_a = q_1 z_b, \qquad z_a = q_2 z_b, \qquad z_{a+1} = q_1 z_{a}
$$
although the first pole is offset by a factor of $z_a - q^{2}z_b$ in the numerator. The other poles do not hinder moving the contours such that $z_{a} \gg z_{a+1}$, except for $z_{a+1} = q_1 z_{a}$. Therefore, we pick up residues whenever:
$$
\{z_i,...,z_{j-1}\} = \{y_1 q_1^{i_1},...,y_1q_1^{j_1-1},...,y_tq_1^{i_t},...,y_t q_1^{j_t-1}\} \qquad \text{for} \qquad y_1 \gg ... \gg y_t
$$
over all partitions of the degree vector into arcs: 
$$
[i;j) \ = \ [i_1;j_1) + ... + [i_t;j_t)
$$
Since $F \in \CB_\bm^+$, then in the limit $y_1 \gg ... \gg y_t$, the estimates \eqref{eqn:limit1} imply that the residue has order:
\begin{equation}
\label{eqn:quant2}
\sim \prod_{s=1}^t y_s^{m_{i_s} + ... + m_{j_s-1} + \lfloor m_i+...+m_{i_s-1} \rfloor - \lfloor m_i+...+m_{j_s-1} \rfloor - \delta_s^1 + \delta_s^t} 
\end{equation}
It is easy to see that the exponent of $y_1$ is $< 0$ unless $t=1$, which implies that the integral vanishes unless each exponent in \eqref{eqn:quant2} is 0. This implies that the only residue which contributes to the above sum is $z_a = q_1^{a}$ for all $a$, hence:
$$
\langle Q_{-[i;j)}^\bm, F \rangle = (q^{-1}-q)^{j-i} \cdot q_1^{\sum_{a=i}^{j-1}a\left( \lfloor m_i+...+m_{a-1} \rfloor - \lfloor m_i+...+m_{a} \rfloor \right)}t^{\ind_{[i;j)}^\bm} \cdot \frac {F(q_1^{i},...,q_1^{j-1})}{\prod_{i\leq a < b < j} \zeta(q_1^{b-a})}
$$
This implies \eqref{eqn:functional1}. As for \eqref{eqn:functional2}, relation \eqref{eqn:daddypair} yields:
$$
\langle G, P_{[i;j)}^\bm \rangle = \int_{|z_a|=1}^{|q|<1<|p|} \frac {g(z_i,...,z_{j-1}) \prod_{a=i}^{j-1} z_a^{\lfloor m_i+...+m_a \rfloor - \lfloor m_i + ... + m_{a-1} \rfloor} \prod_{a < b} \zeta\left(\frac {z_b}{z_a}\right)}{t^{\ind_{[i;j)}^\bm} q^{i-j} \left(1 - \frac {q_2z_{i+1}}{z_{i}} \right) ... \left(1 - \frac {q_2z_{j}}{z_{j-1}} \right)\prod_{a \leq b} \zeta_p\left(\frac {z_a}{z_b}\right)\prod_{a < b} \zeta_p\left(\frac {z_b}{z_a}\right)} \Big |_{p \mapsto q}
$$
for any $g\in \CS^-$. The integrand has the following poles between $z_a$ and $z_b$ for $a<b$:
$$
z_b = p^{2} z_a, \qquad z_b = p^{-2} z_a, \qquad z_b = q_1 z_a, \qquad z_b = q_2 z_a, \qquad z_{a+1} = q_2^{-1} z_{a}
$$
although the first pole is offset by a factor of $z_b - q^{2}z_a$ in the numerator. The other poles do not hinder moving the contours such that $z_{a} \ll z_{a+1}$, except for $z_{a+1} = q_2^{-1} z_{a}$. Therefore, we pick up residues whenever:
$$
\{z_i,...,z_{j-1}\} = \{y_1 q_2^{-i_1},...,y_1q_2^{-j_1+1},...,y_tq_2^{-i_t},...,y_t q_2^{-j_t+1}\} \qquad \text{for} \qquad y_1 \ll ... \ll y_t
$$
for any partition of the degree vector into arcs: 
$$
[i;j) \ = \ [i_1;j_1) + ... + [i_t;j_t)
$$
If $G \in \CB_\bm^-$, then in the limit $y_1 \ll ... \ll y_t$, the estimates \eqref{eqn:limit2} imply that the residue has order:
\begin{equation}
\label{eqn:quant1}
\sim \prod_{s=1}^t y_s^{\lfloor m_i+...+m_{j_s-1} \rfloor - \lfloor m_i+...+m_{i_{s}-1} \rfloor - m_{i_s}-...-m_{j_s-1} - \delta_s^t + \delta_s^1} 
\end{equation}
It is easy to see that the exponent of $y_t$ is $< 0$ unless $t=1$, which implies that the integral vanishes unless each exponent in \eqref{eqn:quant1} is 0. This implies that the only residue which contributes to the above sum is $z_a = q_2^{-a}$ for all $a$, hence:
$$
\langle G, P_{[i;j)}^\bm \rangle = (1-q^2)^{j-i} q_2^{\sum_{a=i}^{j-1} a\left( \lfloor m_i+...+m_{a-1} \rfloor - \lfloor m_i+...+m_{a} \rfloor \right)}t^{-\ind_{[i;j)}^\bm} \cdot \frac {g(q_2^{-i},...,q_2^{-j+1})}{\prod_{i\leq a < b < j} \zeta(q_2^{b-a})}
$$
This yields \eqref{eqn:functional2}. \\
\end{proof}

\begin{proof} \emph{of Exercise \ref{ex:universal}:} It is enough to check property \eqref{eqn:prop1} for $a\in A^+$, as the case $a\in A^-$ is analogous and the property is multiplicative in $a$. In other words, we need to check that for any basis element $F_x \in A^+$ we have:
\begin{equation}
\label{eqn:lion}
\CR \cdot \Delta(F_x) = \Delta^{\op}(F_x)\cdot \CR
\end{equation}
Because the $F_i$ and $G_i$ are dual bases of $A^+$ and $A^-$, relation \eqref{eqn:bialg} implies that the structure constants for their multiplication and comultiplication are the same, i.e.
$$
\Delta(F_x)  = \sum_{y,z} F_y \otimes F_z c^{yz}_x, \qquad \text{where} \qquad G_y G_z = \sum_x G_x c^{yz}_x
$$
$$
\Delta(G_x)  = \sum_{y,z} G_y \otimes G_z d^{yz}_x, \qquad \text{where} \qquad F_z F_y = \sum_x F_x d^{yz}_x
$$
Therefore, the desired relation \eqref{eqn:lion} becomes equivalent to:
$$
\sum_{i,y,z} F_i F_y \otimes G_i F_z c^{yz}_x \ = \ \sum_{i,y,z} F_z F_i \otimes F_y G_i c^{yz}_x \ \Leftrightarrow
$$
$$
\Leftrightarrow \ \sum_{i,j,y,z} F_j \otimes G_i F_z d^{yi}_j c^{yz}_x \ = \ \sum_{i,j,y,z} F_j \otimes F_y G_i  d^{iz}_j c^{yz}_x
$$
For any fixed $j$ and $x$, the above equality follows from:
$$
\sum_{i,y,z} G_i F_z d^{yi}_j c^{yz}_x \ = \ \sum_{i,y,z} F_y G_i d^{iz}_j c^{yz}_x 
$$
which is simply \eqref{eqn:reldrinfeld} for $a=G_j$ and $b=F_x$. As for \eqref{eqn:prop2}, we have:
$$
(\Delta \otimes 1)\CR \ = \ \sum_{i,x,y} F_x \otimes F_y \otimes G_i c^{xy}_i \ = \ \sum_{x,y} F_x \otimes F_y \otimes G_x G_y \ = \ \CR_{13} \CR_{23}
$$
$$
(1 \otimes \Delta)\CR \ = \ \sum_{i,x,y} F_i \otimes G_x \otimes G_y d^{xy}_i \ = \ \sum_{x,y} F_yF_x \otimes G_x \otimes G_y \ = \ \CR_{13}\CR_{12}
$$
\end{proof}


\begin{proof} \emph{of Exercise \ref{ex:ohboy}:} By definition, we have:
$$
\fR_\blamu^+ = \prod_{\bsq \in \blamu} \left( \prod_{\square \in \bmu} \zeta \left( \frac {\chi_\bsq}{\chi_\square}\right) \prod_{i=1}^\bw \Big[\frac {u_i}{q\chi_\bsq} \Big] \right) \frac {\prod_{\sq \in \bmu} \left(\prod_{\bsq \in \bmu} \zeta\left(\frac {\chi_\bsq}{\chi_\sq}\right) \prod_{i=1}^{\bw} \Big[\frac {u_i}{q\chi_\sq} \Big] \Big[\frac {\chi_\sq}{qu_i} \Big]  \right)^{(-)}}{\prod_{\sq \in \bla} \left( \prod_{\bsq \in \bla} \zeta\left(\frac {\chi_\bsq}{\chi_\sq}\right) \prod_{i=1}^{\bw} \Big[\frac {u_i}{q\chi_\sq} \Big] \Big[\frac {\chi_\sq}{qu_i} \Big] \right)^{(-)}} 
$$
where we recall that the superscripts $(+)$, $(0)$ and $(-)$ refer to the fact that we only retain those factors $[x]$ for $\deg x > 0$, $\deg x = 0$ and $\deg x < 0$, respectively. We can simplify the above formula to:
$$
\fR_\blamu^+ = \prod_{\bsq \in \blamu}  \frac {\left(\prod_{\square \in \bmu} \zeta \left( \frac {\chi_\bsq}{\chi_\square}\right) \prod_{i=1}^\bw \Big[\frac {u_i}{q\chi_\bsq} \Big]  \right)^{(+) \text{ and }(0) \text{ and }(-)} }{\left(\prod_{\square \in \bmu} \zeta \left( \frac {\chi_\bsq}{\chi_\square}\right) \prod_{i=1}^\bw \Big[\frac {u_i}{q\chi_\bsq} \Big]  \prod_{\sq \in \bla} \zeta\left(\frac {\chi_\sq}{\chi_\bsq}\right) \prod_{i=1}^{\bw} \Big[\frac {\chi_\bsq}{qu_i} \Big]  \right)^{(-)}} = 
$$
\begin{equation}
\label{eqn:cavalcade}
=  \frac {\left( \prod^{\bsq \in \blamu}_{\square \in \bmu} \zeta \left( \frac {\chi_\bsq}{\chi_\square}\right) \prod_{i=1}^\bw \Big[\frac {u_i}{q\chi_\bsq} \Big] \right)^{(+)} }{\left(\prod^{\bsq \in \blamu}_{\sq \in \bmu} \zeta\left(\frac {\chi_\sq}{\chi_\bsq}\right) \prod_{i=1}^{\bw} \Big[\frac {\chi_\bsq}{qu_i} \Big]  \right)^{(-)}} \cdot \frac {\prod_{\bsq \in \blamu} \left( \prod_{\square \in \bmu} \zeta \left( \frac {\chi_\bsq}{\chi_\square}\right) \prod_{i=1}^\bw \Big[\frac {u_i}{q\chi_\bsq} \Big] \right)^{(0)}}{\prod_{\bsq,\bsq' \in \blamu} \zeta\left(\frac {\chi_\bsq}{\chi_{\bsq'}} \right)^{(-)}}
\end{equation}
The degree 0 factors do not contribute to the max deg and min deg. Because of \eqref{eqn:ossum}, the first fraction does not contribute anything to the maximal degree, so we conclude that:
$$
r_\blamu = \maxdeg \fR^+_\blamu = \maxdeg \frac 1{\prod_{\bsq,\bsq' \in \blamu} \zeta \left(\frac {\chi_{\bsq'}}{\chi_\bsq}\right)^{(-)}} = -\frac 12 \sum_{\bsq,\bsq' \in \blamu} z(c_\bsq - c_{\bsq'})
$$
The reason behind the last equality is that for any pair of boxes $\bsq, \bsq'$, precisely one of the factors $\Big[ \frac {\chi_\bsq}{q_1\chi_{\bsq'}}\Big]$ and $\Big[ \frac {\chi_{\bsq'}}{q_2\chi_{\bsq}}\Big]$ appears in the above product (unless these factors have degree 0, in which case they do not appear at all). One proves the statement for $\mindeg$ analogously. Finally, we need to estimate the lowest degree term of \eqref{eqn:cavalcade} when $\blamu$ is a cavalcade. According to \eqref{eqn:tarth}, for any skew partition $\bmu$ and any box $\bsq$ we have:
\begin{equation}
\label{eqn:ics}
\prod_{\sq \in \bmu} \zeta\left(\frac {\chi_\bsq}{\chi_\sq} \right) \prod_{i=1}^{\bw} \Big[\frac {u_i}{q\chi_\bsq} \Big] \ = \ \frac {\prod^{\text{inner corners}}_{\sq \text{ of }\bmu} \Big[ \frac {\chi_\sq}{q^2 \chi_\bsq} \Big]}{\prod^{\text{outer corners}}_{\sq \text{ of }\bmu} \Big[ \frac {\chi_\sq}{q^2 \chi_\bsq} \Big]} 
\end{equation}
\begin{equation}
\label{eqn:igrec}
\prod_{\sq \in \bmu} \zeta\left(\frac {\chi_\sq}{\chi_\bsq} \right) \prod_{i=1}^{\bw} \Big[\frac {\chi_\bsq}{qu_i} \Big]  \ = \ \frac {\prod^{\text{inner corners}}_{\sq \text{ of }\bmu} \Big[ \frac {\chi_\bsq}{\chi_\sq} \Big]}{\prod^{\text{outer corners}}_{\sq \text{ of }\bmu} \Big[ \frac {\chi_\bsq}{\chi_\sq} \Big]} 
\end{equation}
and so we can write \eqref{eqn:cavalcade} as:

\begin{equation}
\label{eqn:ofribbons}
\fR_\blamu^+ = \frac {\prod_{\bsq \in \blamu} \prod^{\text{inner corners } \sq \text{ of }\bmu}_{\text{of content }c_\sq > c_\bsq} \frac {\Big[ \frac {\chi_\sq}{q^2 \chi_\bsq} \Big]}{\Big[ \frac {\chi_\bsq}{\chi_\sq} \Big]}}{\prod_{\bsq \in \blamu} \prod^{\text{outer corners } \sq \text{ of }\bmu}_{\text{of content }c_\sq > c_\bsq} \frac {\Big[ \frac {\chi_\sq}{q^2 \chi_\bsq} \Big]}{\Big[ \frac {\chi_\bsq}{\chi_\sq} \Big]}} \cdot \frac {\prod_{\bsq \in \blamu} \frac {\prod^{\text{inner corners } \sq \text{ of }\bmu}_{\text{of content }c_\sq = c_\bsq} \Big[ \frac {\chi_\sq}{q^2 \chi_\bsq} \Big]}{\prod^{\text{outer corners } \sq \text{ of }\bmu}_{\text{of content }c_\sq = c_\bsq} \Big[ \frac {\chi_\sq}{q^2 \chi_\bsq} \Big]}}{\prod_{\bsq,\bsq' \in \blamu} \zeta\left(\frac {\chi_\bsq}{\chi_{\bsq'}} \right)^{(-)}}
\end{equation}
To evaluate the term of lowest degree in the first product, observe that:
\begin{equation}
\label{eqn:elde1}
\ld \frac {\Big[ \frac {\chi}{q^2 \chi'} \Big]}{\Big[ \frac {\chi'}{\chi} \Big]} \ = \ -q 
\end{equation}
for any $\chi>\chi'$. There are as many such factors of $-q$ in \eqref{eqn:ofribbons} as there are boxes $\bsq \in \blamu$ to the northwest of a signed corner (count with sign $+$ if an inner corner and with sign $-$ if an outer corner) in $\bmu$. For $C$ a cavalcade of ribbons, this number is precisely $N_{C}^+$. As for the second product in \eqref{eqn:ofribbons}, the numerator picks up a factor of $[q^{-2}]$ (respectively $[q^{-2}]^{-1}$) whenever the cavalcade $C$ intersects an inner (respectively outer) corner of the Young diagram $\bmu$. Since a ribbon always intersects one more inner corner than outer corners, the contribution of the numerator is $[q^{-2}]^{\# \text{ of ribbons}} = [q^{-2}]^{\#_C}$. This concludes our estimate of \eqref{eqn:cal}. Let us now study the case of $-$:
$$
\fR_\blamu^- = \frac 1{\prod_{\bsq \in \blamu} \left( \prod_{\square \in \bla} \zeta \left( \frac {\chi_\square}{\chi_\bsq}\right) \prod_{i=1}^\bw \Big[\frac {\chi_\bsq}{qu_i} \Big] \right)} \cdot \frac {\prod_{\sq \in \bla} \left( \prod_{\bsq \in \bla} \zeta\left(\frac {\chi_\bsq}{\chi_\sq}\right) \prod_{i=1}^{\bw} \Big[\frac {u_i}{q\chi_\sq} \Big] \Big[\frac {\chi_\sq}{qu_i} \Big] \right)^{(-)}}{\prod_{\sq \in \bmu} \left(\prod_{\bsq \in \bmu} \zeta\left(\frac {\chi_\bsq}{\chi_\sq}\right) \prod_{i=1}^{\bw} \Big[\frac {u_i}{q\chi_\sq} \Big] \Big[\frac {\chi_\sq}{qu_i} \Big]  \right)^{(-)}}
$$
$$
= \prod_{\bsq \in \blamu}  \frac {\left( \prod_{\sq \in \bla} \zeta\left(\frac {\chi_\sq}{\chi_\bsq}\right) \prod_{i=1}^{\bw} \Big[\frac {\chi_\bsq}{qu_i} \Big] \prod_{\sq \in \bmu} \zeta\left(\frac {\chi_\bsq}{\chi_\sq}\right)  \prod_{i=1}^\bw \Big[\frac {u_i}{q\chi_\bsq} \Big] \right)^{(-)}}{\left( \prod_{\square \in \bla} \zeta \left( \frac {\chi_\square}{\chi_\bsq}\right) \prod_{i=1}^\bw \Big[\frac {\chi_\bsq}{qu_i} \Big] \right)^{(+) \text{ and }(0) \text{ and }(-)}} = 
$$
$$
=  \frac {\left( \prod^{\bsq \in \blamu}_{\sq \in \bla} \zeta\left(\frac {\chi_\bsq}{\chi_\sq}\right)  \prod_{i=1}^\bw \Big[\frac {u_i}{q\chi_\bsq} \Big] \right)^{(-)}}{\left( \prod^{\bsq \in \blamu}_{\square \in \bla} \zeta \left( \frac {\chi_\square}{\chi_\bsq}\right) \prod_{i=1}^\bw \Big[\frac {\chi_\bsq}{qu_i} \Big] \right)^{(+)}} \cdot \frac 1{\left( \prod^{\bsq \in \blamu}_{\square \in \bla} \zeta \left( \frac {\chi_\square}{\chi_\bsq}\right) \prod_{i=1}^\bw \Big[\frac {\chi_\bsq}{qu_i} \Big] \right)^{(0)} \prod^{\bsq \in \blamu}_{\bsq' \in \blamu} \zeta\left(\frac {\chi_\bsq}{\chi_\bsq'}\right)^{(-)}}
$$
The first factor does not contribute anything to the computation of the maximal degree $r_\blamu$, because of \eqref{eqn:ossum}. Meanwhile, the degree 0 terms also contribute nothing, so we conclude that the maximal degree of $\fR_\blamu^-$ equals: 
$$
\maxdeg \fR^-_\blamu = \maxdeg \frac 1{\prod_{\bsq,\bsq' \in \blamu} \zeta \left(\frac {\chi_{\bsq'}}{\chi_\bsq}\right)^{(-)}} = -\frac 12 \sum_{\bsq,\bsq' \in \blamu} z(c_\bsq - c_{\bsq'})
$$
which is precisely \eqref{eqn:stannis}. The case of $\mindeg$ is treated analogously. Let us now compute the term of lowest degree, and we will do so by rewriting the above expression according to \eqref{eqn:ics} and \eqref{eqn:igrec}:
\begin{equation}
\label{eqn:manyribbons}
\fR_\blamu^- = \frac {\prod_{\bsq \in \blamu} \prod^{\text{inner corners } \sq \text{ of }\bla}_{\text{of content }c_\sq < c_\bsq} \frac {\Big[ \frac {\chi_\sq}{q^2 \chi_\bsq} \Big]}{\Big[ \frac {\chi_\bsq}{\chi_\sq} \Big]}}{\prod_{\bsq \in \blamu} \prod^{\text{outer corners } \sq \text{ of }\bla}_{\text{of content }c_\sq < c_\bsq} \frac {\Big[ \frac {\chi_\sq}{q^2 \chi_\bsq} \Big]}{\Big[ \frac {\chi_\bsq}{\chi_\sq} \Big]}} \cdot \frac {\prod_{\bsq \in \blamu} \frac {\prod^{\text{outer corners } \sq \text{ of }\bla}_{\text{of content }c_\sq = c_\bsq} \Big[ \frac {\chi_\bsq}{\chi_\sq} \Big]}{\prod^{\text{inner corners } \sq \text{ of }\bla}_{\text{of content }c_\sq = c_\bsq} \Big[ \frac {\chi_\bsq}{\chi_\sq} \Big]}}{\prod^{\bsq \in \blamu}_{\bsq' \in \blamu} \zeta\left(\frac {\chi_\bsq}{\chi_\bsq'}\right)^{(-)}} 
\end{equation}
Let us compute $\rho^-_\blamu$, namely the lowest degree terms of the above expression. As:
\begin{equation}
\label{eqn:elde2}
\ld \frac {\Big[ \frac {\chi}{q^2 \chi'} \Big]}{\Big[ \frac {\chi'}{\chi} \Big]} \ = \ (-q)^{-1} 
\end{equation}
for any $\chi<\chi'$, the lowest degree term of the first factor of \eqref{eqn:manyribbons} consists of as many factors of $(-q)$ as there are boxes $\bsq \in \blamu$ to the southeast of a signed corner (count with sign $-$ for an inner corner and with sign $+$ for an outer corner) in $\bla$. The lowest degree term of the numerator of the second factor of \eqref{eqn:manyribbons} equals the product $[q^{-2}]...[q^{-2d_\sq}]$ for any outer corner of $\bla$ which has $d_\sq$ boxes of $\blamu$ diagonally southwest of it, times the product $[q^{-2}]^{-1}...[q^{-2d_\sq}]^{-1}$ for any inner corner of $\bla$ which has $d_\sq$ boxes of $\blamu$ diagonally southwest of it. Therefore, to prove \eqref{eqn:stamp}, we need to show that if $\blamu = S$ is a stampede of ribbons, we have:

\begin{equation}
\label{eqn:finalmente}
\frac {\prod_{i\leq a < b <j}^{a\leftrightarrow b} \ld \zeta \left(\frac {\chi_{\sq_a}}{\chi_{\sq_b}}\right)}{\prod_{i\leq a, b < j} \ld \zeta\left(\frac {\chi_{\sq_a}}{\chi_{\sq_b}}\right)^{(-)}} = \frac {\prod^{\sq \text{ inner}}_{\text{corner of }\bla} [q^{-2}]...[q^{-2d_\sq}]}{\prod^{\sq \text{ outer}}_{\text{corner of }\bla} [q^{-2}]...[q^{-2d_\sq}]} \cdot [q^{-2}]^{\#_R} (-q)^{N_S^- - \widetilde{\#}}
\end{equation}
where $\widetilde{\#}$ denotes the number of signed corners ($-$ inner $+$ outer) of $\bla$ to the northwest of any box $\sq_i,...,\sq_{j-1}$, counted with multiplicities. The left hand side of \eqref{eqn:finalmente} equals:
$$
\frac {\prod_{i\leq a < b <j}^{a\leftrightarrow b} \ld \zeta \left(\frac {\chi_{\sq_a}}{\chi_{\sq_b}}\right)^{(-) \text{ or } (0) \text{ or }(+)}}{\prod_{a < b} \ld \zeta\left(\frac {\chi_{\sq_a}}{\chi_{\sq_b}}\right)^{(-)}\prod_{a < b} \ld \zeta\left(\frac {\chi_{\sq_b}}{\chi_{\sq_a}}\right)^{(-)}} = \prod_{i \leq a < b < j} \zeta\left(\frac {\chi_{\sq_a}}{\chi_{\sq_b}}\right)^{(0)}  \prod_{i\leq a < b <j}^{a\leftrightarrow b} \ld \frac {\zeta \left(\frac {\chi_{\sq_a}}{\chi_{\sq_b}}\right)^{(+)}}{\zeta\left(\frac {\chi_{\sq_b}}{\chi_{\sq_a}}\right)^{(-)}}
$$
By dividing out \eqref{eqn:ics}/\eqref{eqn:igrec} for the diagram $\bla$ by \eqref{eqn:ics}/\eqref{eqn:igrec} for $\bmu$, we obtain:
\begin{equation}
\label{eqn:division}
\prod_{\sq \in \blamu} \zeta\left(\frac {\chi_\bsq}{\chi_\sq} \right) \ = \ \frac {\prod^{\text{inner corners}}_{\sq \text{ of }\blamu} \Big[ \frac {\chi_\sq}{q^2 \chi_\bsq} \Big]}{\prod^{\text{outer corners}}_{\sq \text{ of }\blamu} \Big[ \frac {\chi_\sq}{q^2 \chi_\bsq} \Big]} 
\end{equation}
\begin{equation}
\label{eqn:bell}
\prod_{\sq \in \blamu} \zeta\left(\frac {\chi_\sq}{\chi_\bsq} \right) \ = \ \frac {\prod^{\text{inner corners}}_{\sq \text{ of }\blamu} \Big[ \frac {\chi_\bsq}{\chi_\sq} \Big]}{\prod^{\text{outer corners}}_{\sq \text{ of }\blamu} \Big[ \frac {\chi_\bsq}{\chi_\sq} \Big]} 
\end{equation}
where an inner (respectively outer) corner of $\blamu$ is defined as either an inner (respectively outer) corner of $\bla$ or as an outer (respectively inner) corner of $\bmu$. Recall from Section \ref{sec:basicpar} that a stampede of ribbons traces out a collection of intermediate diagrams between $\bmu$ and $\bla$:
$$
\bla = \bnu_0 \geq \bnu_1 \geq ... \geq \bnu_k = \bmu \qquad \qquad R_s = \bnu_{s-1} / \bnu_{s}
$$
Therefore, \eqref{eqn:bell} implies that:
$$
\prod_{i \leq a < b < j} \zeta\left(\frac {\chi_{\sq_a}}{\chi_{\sq_b}}\right)^{(0)} = \prod_{s=1}^k \prod_{\bsq \in R_s} \prod_{\sq \in R_{1} \sqcup ... \sqcup R_{s-1}} \zeta\left(\frac {\chi_\sq}{\chi_\bsq}\right)^{(0)} = \prod_{s=1}^k \prod_{\bsq \in R_s} \frac {\prod^{\text{inner corners}}_{\sq \text{ of }\bla \backslash \bnu_{s-1}} \Big[ \frac {\chi_\bsq}{\chi_\sq} \Big]}{\prod^{\text{outer corners}}_{\sq \text{ of }\bla \backslash \bnu_{s-1}} \Big[ \frac {\chi_\bsq}{\chi_\sq} \Big]} 
$$
The right hand side contributes $[q^{-2}]^{\#_R}$ from the inner/outer corners of each $\bnu_{s-1}$, and precisely $\frac {\prod^{\sq \text{ inner}}_{\text{corner of }\bla} [q^{-2}]...[q^{-2d_\sq}]}{\prod^{\sq \text{ outer}}_{\text{corner of }\bla} [q^{-2}]...[q^{-2d_\sq}]}$ from the inner/outer corners of $\bla$. Applying \eqref{eqn:division} and \eqref{eqn:bell} implies that:
$$
\prod_{i\leq a < b <j}^{a\leftrightarrow b} \frac {\zeta \left(\frac {\chi_{\sq_a}}{\chi_{\sq_b}}\right)^{(+)}}{\zeta\left(\frac {\chi_{\sq_b}}{\chi_{\sq_a}}\right)^{(-)}} = \prod_{s=1}^k \prod_{\bsq \in R_s} \prod_{\sq \in R_{1} \sqcup ... \sqcup R_{s}} \frac {\zeta \left(\frac {\chi_{\sq}}{\chi_{\bsq}}\right)^{(+)}}{\zeta\left(\frac {\chi_{\bsq}}{\chi_{\sq}}\right)^{(-)}} = \prod_{s=1}^k \prod_{\bsq \in R_s} \frac {\prod^{\text{inner corners}}_{\sq \text{ of }\bla \backslash \bnu_{s}} \frac {\Big[ \frac {\chi_\bsq}{\chi_\sq} \Big]^{(+)}}{\Big[ \frac {\chi_\sq}{q^2\chi_\bsq} \Big]^{(-)}} }{\prod^{\text{outer corners}}_{\sq \text{ of }\bla \backslash \bnu_{s}} \frac {\Big[ \frac {\chi_\bsq}{\chi_\sq} \Big]^{(+)}}{\Big[ \frac {\chi_\sq}{q^2\chi_\bsq} \Big]^{(-)}}} 
$$
We may compute the lowest degree term of the above by using \eqref{eqn:elde1}, and we obtain:
$$
\prod_{i\leq a < b <j}^{a\leftrightarrow b} \ld \frac {\zeta \left(\frac {\chi_{\sq_a}}{\chi_{\sq_b}}\right)^{(+)}}{\zeta\left(\frac {\chi_{\sq_b}}{\chi_{\sq_a}}\right)^{(-)}} = (-q)^{N_S^- - \widetilde{\#}}
$$
This count completes the proof of \eqref{eqn:finalmente}, and hence \eqref{eqn:stamp}. \\
\end{proof}

\begin{proof} \emph{of Exercise \ref{ex:final}:} Write $x_a = \delta_i+...+\delta_{a}$, and the inequality becomes:
$$
\sum_{a=i+1}^{j-1} \min(0,\delta_{a}) \leq \sum_{a=i+1}^{j} \delta_{a} \left(m_i+...+m_{a-1} - \lfloor m_i+...+m_{a-1} \rfloor \right)  \leq \sum_{a=i+1}^{j-1} \max(0,\delta_{a})
$$
The above inequalities hold for all real numbers $\delta_i,...,\delta_{j-1}$, and the first inequality can be an equality only if $\delta_{a} = 0$ or if $m_i+...+m_{a-1} \in \BZ$ and $\delta_{a}>0$. The second inequality can be an equality only if $\delta_a = 0$ or $m_i+...+m_{a-1} \in \BZ$ and $\delta_a<0$. 

\end{proof}

\startbibliography
 \begin{singlespace} 
  \bibliography{References} 
 \end{singlespace}

\end{document}